\pdfoutput=1
\documentclass[12pt]{book}    

\usepackage{savesym}
\usepackage{scrextend}  
\usepackage{amsthm,amsmath}
\usepackage{amssymb,latexsym,graphicx}
\usepackage[round]{natbib} 
\usepackage{amscd}
 \usepackage[all,cmtip]{xy}
\usepackage{tikz-cd}
  \usetikzlibrary{arrows}
\tikzset{
  commutative diagrams/.cd,
  arrow style=tikz
  }
\usepackage{tikz}

\usepackage{accents}
\usepackage{mathtools}
\usepackage{stmaryrd} 

\usepackage{calligra}
\DeclareMathAlphabet{\mathcalligra}{T1}{calligra}{m}{n}
\DeclareMathAlphabet{\mathpzc}{OT1}{pzc}{m}{it}
\usepackage{pdfpages}
\usepackage{hycolor}
\usepackage{xcolor}
\usepackage[
                       colorlinks=true,
                       linkcolor=black, 
                       citecolor=black, 
                       urlcolor=blue,
                     ]{hyperref}
\usepackage{soul} 

\usepackage{makeidx}
\makeindex

\usepackage{stackengine} 
\usepackage{booktabs} 

\newtheorem{theorem}{Theorem}[section]
\newtheorem{corollary}[theorem]{Corollary}

\newtheorem{lemma}[theorem]{Lemma}
\newtheorem{proposition}[theorem]{Proposition}

\theoremstyle{definition}
\newtheorem{definition}[theorem]{Definition}

\newtheorem{convention}[theorem]{Convention}
\newtheorem{remark}[theorem]{Remark}
\newtheorem{exercise}[theorem]{Exercise}
\newtheorem{example}[theorem]{Example}

\theoremstyle{remark}

\newenvironment{dedication}
  {\clearpage           
   \thispagestyle{empty}
   \vspace*{\stretch{1}}
   \itshape             
   \raggedleft          
  }
  {\par 
   \vspace{\stretch{3}} 
   \clearpage           
  }

\newcommand{\comp}[1]{#1^{\rm C}}  
\newcommand{\N}{{\mathbb{N}}}

\newcommand{\R}{{\mathbb{R}}}
\renewcommand{\SS}{{\mathbb{S}}}

\newcommand{\Z}{{\mathbb{Z}}}
\newcommand{\Aa}{{\mathcal{A}}}   
\newcommand{\Bb}{{\mathcal{B}}}
\newcommand{\Cc}{{\mathcal{C}}}   

\newcommand{\Ee}{{\mathcal{E}}}
\newcommand{\Ff}{{\mathcal{F}}}

\newcommand{\Ll}{{\mathcal{L}}}   
\newcommand{\Nn}{{\mathcal{N}}}

\newcommand{\Ss}{{\mathcal{S}}}
\newcommand{\Tt}{{\mathcal{T}}}
\newcommand{\Uu}{{\mathcal{U}}}
\newcommand{\Vv}{{\mathcal{V}}}

\newcommand{\coker}{{\rm coker\, }}  
\newcommand{\im}{{\rm im\, }}             
\newcommand{\SPAN}{{\rm span\, }}         
\newcommand{\id}{{\rm id}}                
\newcommand{\Id}{{\rm Id}}
\newcommand{\codim}{{\rm codim\, }}       
\newcommand{\INDEX}{\mathop{\mathrm{index}}}     
\newcommand{\Map}{{\rm Map}}          
\newcommand{\Fix}{{\rm Fix}}          
\newcommand{\Hess}{\mathrm{Hess}}          
\newcommand{\CUP}{\mathop{\cup}}           
\newcommand{\CAP}{\mathop{\cap}}           
 
\renewcommand{\L}{{\rm L}}

\newcommand{\norm}{{\rm norm}}

\newcommand{\eps}{{\varepsilon}}

\newcommand{\MM}{{\mathfrak M}}
\newcommand{\mm}{{\mathfrak m}}
\newcommand{\SSfrak}{{\mathfrak S}}

\newcommand{\INNER}[2]{\left\langle #1, #2\right\rangle}

\newcommand{\SC}{{\mathrm{sc}}}                  
\newcommand{\SCz}{{\mathrm{sc}^0}}            
\newcommand{\SCo}{{\mathrm{sc}^1}}            
\newcommand{\SCk}{{\mathrm{sc}^k}}            
\newcommand{\SSC}{{\mathrm{ssc}}}               
\newcommand{\Llm}[1]{\Ll^{#1}}   
\newcommand{\Llsc}{\mathcal{L}_\mathrm{sc}}          
\newcommand{\Llscp}{\mathcal{L}_\mathrm{sc}^+}   
\newcommand{\Llpo}{\mathcal{L}_\mathrm{p}}           
\newcommand{\Llco}{\mathcal{L}_\mathrm{c}}           
\newcommand{\Llbo}{\mathcal{L}_\mathrm{b}}           
\newcommand{\Cco}{C_\mathrm{c}}          
\newcommand{\Ttco}{\mathcal{T}_\mathrm{c}}         
\newcommand{\Ttbo}{\mathcal{T}_\mathrm{b}}         
\newcommand{\LlSSfrak}{\mathcal{L}_\mathfrak S}    

\newcommand{\mbf}[1]{\text{\boldmath $#1$}}  

\newcommand{\bs}{{\mbf{s}}}

\def\NABLA#1{{\mathop{\nabla\kern-.5ex\lower1ex\hbox{$#1$}}}}
\def\Nabla#1{\nabla\kern-.5ex{}_{#1}}
\def\Tabla#1{\Tilde\nabla\kern-.5ex{}_{#1}}
\def\abs#1{\mathopen|#1\mathclose|}   
\def\Abs#1{\left|#1\right|}            
\def\norm#1{\mathopen\|#1\mathclose\|}
\def\Norm#1{\left\|#1\right\|}

\renewcommand{\Tilde}{\widetilde}

\newcommand{\p}{{\partial}}

\newcommand{\INTO}{\hookrightarrow}              
\newcommand{\TO}{\longrightarrow}
\renewcommand{\1}{{{\mathchoice {\rm 1\mskip-4mu l} {\rm 1\mskip-4mu l}
{\rm 1\mskip-4.5mu l} {\rm 1\mskip-5mu l}}}}

\newcommand{\Index}[1]{#1\index{#1}}

\newlength\eqshift
\renewcommand\theequation{\thesection.\arabic{equation}}
\let\savetheequation\theequation

\makeatletter
\renewcommand*\env@matrix[1][\arraystretch]{%
  \edef\arraystretch{#1}%
  \hskip -\arraycolsep
  \let\@ifnextchar\new@ifnextchar
  \array{*\c@MaxMatrixCols c}}
\makeatother

\makeatletter
\let\save@mathaccent\mathaccent
\newcommand*\if@single[3]{%
  \setbox0\hbox{${\mathaccent"0362{#1}}^H$}%
  \setbox2\hbox{${\mathaccent"0362{\kern0pt#1}}^H$}%
  \ifdim\ht0=\ht2 #3\else #2\fi
  }
\newcommand*\rel@kern[1]{\kern#1\dimexpr\macc@kerna}
\newcommand*\widebar[1]{\@ifnextchar^{{\wide@bar{#1}{0}}}{\wide@bar{#1}{1}}}
\newcommand*\wide@bar[2]{\if@single{#1}{\wide@bar@{#1}{#2}{1}}{\wide@bar@{#1}{#2}{2}}}
\newcommand*\wide@bar@[3]{%
  \begingroup
  \def\mathaccent##1##2{%
    \let\mathaccent\save@mathaccent
    \if#32 \let\macc@nucleus\first@char \fi
    \setbox\z@\hbox{$\macc@style{\macc@nucleus}_{}$}%
    \setbox\tw@\hbox{$\macc@style{\macc@nucleus}{}_{}$}%
    \dimen@\wd\tw@
    \advance\dimen@-\wd\z@
    \divide\dimen@ 3
    \@tempdima\wd\tw@
    \advance\@tempdima-\scriptspace
    \divide\@tempdima 10
    \advance\dimen@-\@tempdima
    \ifdim\dimen@>\z@ \dimen@0pt\fi
    \rel@kern{0.6}\kern-\dimen@
    \if#31
      \overline{\rel@kern{-0.6}\kern\dimen@\macc@nucleus\rel@kern{0.4}\kern\dimen@}%
      \advance\dimen@0.4\dimexpr\macc@kerna
      \let\final@kern#2%
      \ifdim\dimen@<\z@ \let\final@kern1\fi
      \if\final@kern1 \kern-\dimen@\fi
    \else
      \overline{\rel@kern{-0.6}\kern\dimen@#1}%
    \fi
  }%
  \macc@depth\@ne
  \let\math@bgroup\@empty \let\math@egroup\macc@set@skewchar
  \mathsurround\z@ \frozen@everymath{\mathgroup\macc@group\relax}%
  \macc@set@skewchar\relax
  \let\mathaccentV\macc@nested@a
  \if#31
    \macc@nested@a\relax111{#1}%
  \else
    \def\gobble@till@marker##1\endmarker{}%
    \futurelet\first@char\gobble@till@marker#1\endmarker
    \ifcat\noexpand\first@char A\else
      \def\first@char{}%
    \fi
    \macc@nested@a\relax111{\first@char}%
  \fi
  \endgroup
}
\makeatother

\long\def\symbolfootnote[#1]#2{\begingroup%
\def\thefootnote{\fnsymbol{footnote}}\footnote[#1]{#2}\endgroup}

\begin{document}

\bibliographystyle{plainnat}   
\sloppy
\author{Joa Weber \\ UNICAMP }
\title{Scale Calculus and $\mathrm{M}$-Polyfolds \\
         {\large An Introduction}
       }
\date{\today}

\thispagestyle{empty}
\includepdf[pages={1}]{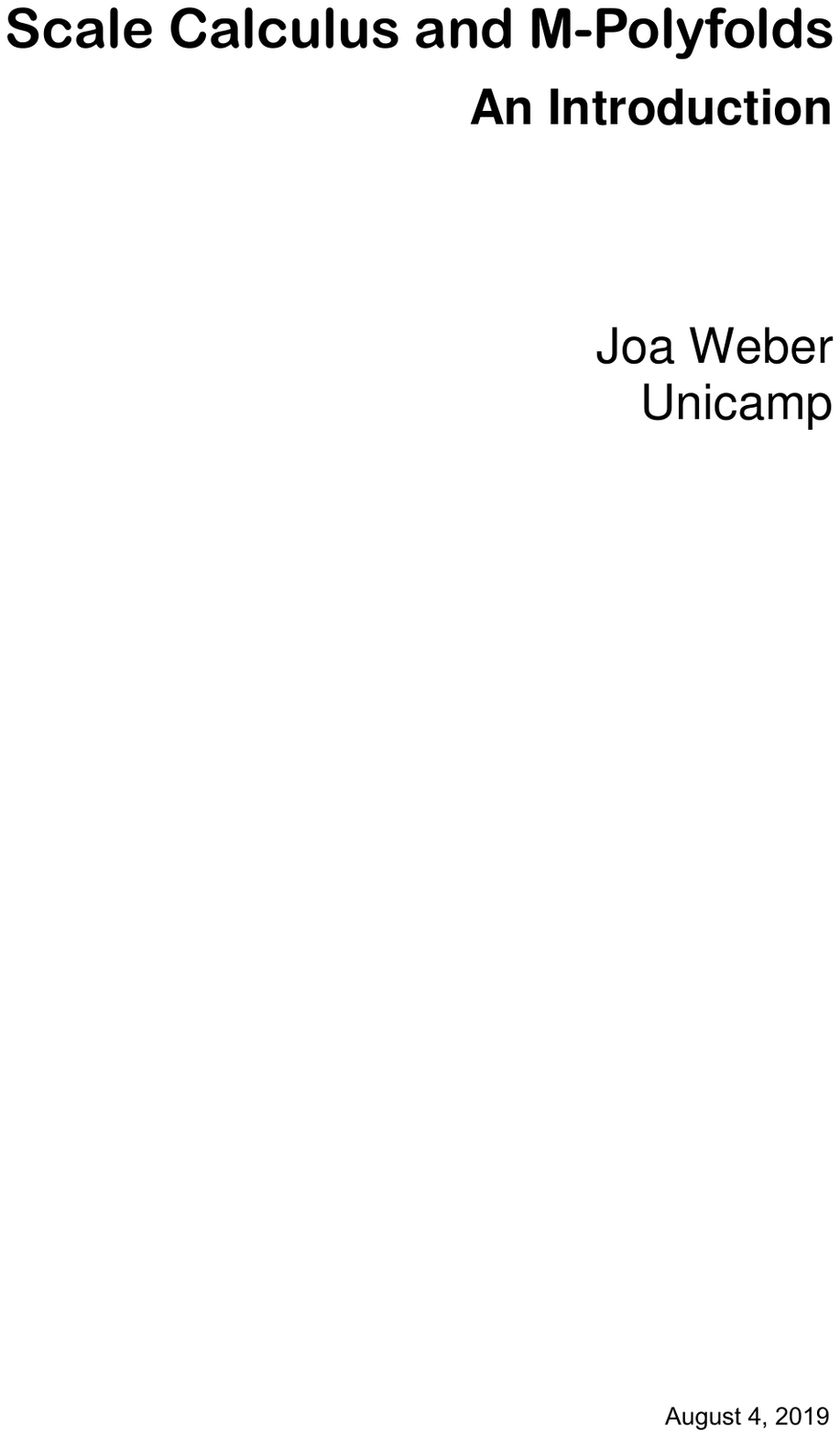}
\thispagestyle{empty}


%
%

\frontmatter 

%
%
%

\begin{dedication}
Dies ist das Geheimnis der  Liebe, \\
{da\ss} sie solche verbinde, \\
deren  jedes f\"ur sich sein k\"onnte \\
und doch nichts ist und sein kann 
ohne das andere\\
\mbox{}\\ \footnotesize\rm
Friedrich Wilhelm Joseph von Schelling\\
1775--1854


\end{dedication}

\newpage
\thispagestyle{empty}

%
%

\chapter{Preface}




This text originates from lecture notes written
during the graduate course
``MM805 T\'{o}picos de An\'{a}lise I''
held from March through June 2018 at UNICAMP.
The manuscript has then been slightly modified in order to serve
as 
\href{https://impa.br/publicacoes/coloquios/}{accompanying text}
for an advanced mini-course during the
\textit{$32^{\rm nd}$ Col\'{o}quio Brasileiro de Matem\'{a}tica},
\href{https://impa.br/eventos-do-impa/eventos-2019/32o-coloquio-brasileiro-de-matematica/}{CBM-32},
IMPA, Rio de Janeiro, in August 2019.

\subsection*{Scope}
Our aim is to give an introduction to the new calculus, called \emph{scale calculus},
and the generalized manifolds, called \emph{M-polyfolds},
that were introduced 
by~\citet*{Hofer:2007a,Hofer:2009a,Hofer:2009b,Hofer:2010b} in their
construction of a generalized differential geometry
in infinite dimensions, called \emph{polyfold theory}.
In this respect we recall and survey in the appendix the incarnations
of the usual (Fr\'{e}chet) calculus in various contexts - from
topological vector spaces (TVS) to complete normed vector spaces,
that is Banach spaces.

Recently the construction of abstract polyfold theory 
has been concluded and made available in the form of a
book by~\citet{Hofer:2017a}.
The door is now open, not only to reformulate and reprove
past moduli space problems using the new language and tools,
but to approach open or new problems.

\subsection*{Content}
There are two parts plus an appendix. Part one introduces scale calculus
starting with the linear theory (scale Banach spaces, scale linear
maps, in particular, scale Fredholm operators -- these are related to scale shifts),
then we define scale continuity and scale differentiability.
The latter is compared to usual (Fr\'{e}chet) differentiability,
then the chain rule is established for scale calculus.
Part one concludes with boundary and, more surprisingly, corner
recognition in scale calculus and with the construction
of scale manifolds.

Part two is concerned with the construction of M-polyfolds
in analogy to manifolds, just locally modeled not only on
Banach space (Banach manifolds), neither only on scale Banach space
(scale manifolds), but on a generalization of retracts called scale
retracts. This choice of local model spaces is motivated by Cartan's
last theorem which we therefore review first.
Part two concludes with the construction of the scale version of
vector bundles, called strong bundles over M-polyfolds, whose local
models are strong trivial-bundle retracts.
To accommodate Fredholm sections one introduces
a double scale structure from which one then extracts
two individual scales.

The appendix recalls and reviews relevant background and results
in topology and analysis, particularly standard calculus.

\subsection*{Audience}
The intended audience are graduate students.
Recommended background is basic knowledge of functional analysis
including the definition of Sobolev spaces such as $W^{k,p}(\R,\R^n)$.

\subsection*{Acknowledgements}
It is a pleasure to thank Brazilian tax payers for the
excellent research and teaching opportunities at UNICAMP
and for generous financial support
through 
Funda\c{c}\~{a}o de Amparo
\`{a} Pesquisa do Estado de S\~{a}o Paulo (FAPESP),
processo $\mathrm{n}^{\rm o}$ 2017/19725-6.
I am indebted to the selection committee of  
\href{https://impa.br/eventos-do-impa/eventos-2019/32o-coloquio-brasileiro-de-matematica/}{CBM-32}
for giving me once more the opportunity
to teach an advanced mini-course at
these bi-annual meetings.

I'd like to thank
Daniel Ferreira Machado,
Darwin Gregorio Villar Salinas, and
Jos\'{e} Lucas Pereira Luiz
for their interest and a very pleasant atmosphere
in the graduate course
``MM805 T\'{o}picos de An\'{a}lise I''
held in the first semester of 2018 at UNICAMP.

Last not least, I am very grateful to Kai Cieliebak 
for handing me out his fine Lecture Notes
\citet{Cieliebak:2018a} while they were still in progress
and to Helmut Hofer for a useful comment.

\vspace{\baselineskip}
\begin{flushright}\noindent
Campinas, \hfill \mbox{ }  {\it Joa Weber} \\
\today \hfill \mbox{ }  
\end{flushright}

\tableofcontents 


\mainmatter 

\chapter{Introduction}\label{sec:intro}


The central problem in areas of global analysis
such as Morse, Floer, or Gromov--Witten theory is to study
spaces of solutions to nonlinear ordinary or partial
differential equations~$\Ff(v)=0$.
The\index{moduli spaces}
so-called \emph{moduli spaces}
\[
     \MM:=\{\Ff=0\}
     ,\qquad
     \mm:=\MM/G
\]
consist\index{parametrized!solutions}\index{solutions!parametrized}
in case of~$\MM$ of 
\emph{parametrized solutions}~$v\colon \Sigma\to S$ taking values in
a manifold or -- after localization --
in a vector space~$S$, often divided out by a group~$G$ that acts
on~$\MM$ by reparametrizing the domain
manifold~$\Sigma$. The elements 
of~$\mm$\index{unparametrized!solutions}\index{solutions!unparametrized}
are then called \emph{unparametrized solutions}.
In case of Morse and Floer homology an element~$\tau$
of the group~$G=(\R,+)$ acts on the domain~$\Sigma=\R$
by \emph{time-shift}\index{time shift}
\[
     (\tau_* v)(t):=v(t+\tau)
\]
for $t\in\R$. 
The shift map $\Psi:\R\times\Map(\R,S)\to\Map(\R,S)$
is defined by $(\tau,v)\mapsto \tau_*v$.
The peculiar different behavior in $\tau$ and in $v$
of this simple map, namely linearity, hence smoothness, in~$v$,
whereas differentiation with respect to $t$ causes $v$
to loose a derivative, eventually led to the discovery
of a new notion of smoothness in infinite dimensions
-- \emph{scale smoothness} due\index{scale calculus!history} 
to~\citet*{Hofer:2007a,Hofer:2017a}.
Scale smoothness is connected to interpolation theory~\citet{Triebel:1978a}.
It was the crucial insight of Hofer, Wysocki, and Zehnder
that requiring compactness of the scale embeddings causes that
scale smoothness satisfies the chain rule
and therefore is suitable to patch together pieces
of scale Banach spaces to obtain scale manifolds, 
or more generally M-polyfolds -- new spaces in infinite dimensions.

\subsubsection*{\boldmath From holomorphic curves to polyfold theory.}

In 1985 \citet{gromov:1985a}
generalized holomorphic curves from complex
analysis to symplectic geometry and thereby
discovered that there is a symplectic topology.
Right after Gromov's seminal ideas
\citet{Floer:1986a,floer:1988c,floer:1989a}
``morsified'' holomorphic curves. He used
a perturbed holomorphic curve equation to construct
a semi-infinite dimensional Morse homology,
called \emph{Floer homology},\index{Floer homology}
which meanwhile has a huge
range of applications, from Hamiltonian and contact dynamics
through symplectic topology to topological field theories; cf. the 
survey~\citet{Abbondandolo:2018f}.
Floer's construction also motivated further developments
like the discovery of Fukaya 
$A_\infty$-categories
\citep{Fukaya:1993a,Fukaya:2009a} and
Symplectic Field Theory \citep{Eliashberg:2000a}.

All these applications face difficult transversality and
compactness issues largely caused by the fact that
one does not find oneself working in a single Banach manifold, but rather 
in a union of such and one has to deal with each strata
individually and even do analysis across neighboring ones.
To deal with these problems \citet{Fukaya:1999a}
discovered the notion of Kuranishi structures
based on finite dimensional approximation.
\\
In contrast Hofer, Wysocki, and Zehnder
stay in infinite dimension and generalize calculus.
Traditionally moduli spaces were studied by cumbersome ad-hoc methods
all of whose steps had to be carried out, although rather analogous,
for each moduli problem from scratch, usually filling hundreds of
pages.
Even in one specific setup,
the differential operator $\Ff$ might act on maps $\gamma$ whose
domains and targets vary, in general. Consequently $\Ff$ cannot be
defined on some single Banach manifold $\Bb$ of maps
with values in some single Banach bundle $\Ee$ over $\Bb$.
Therefore the occurring singular limits, e.g. broken trajectories or bubbling
off phenomena, cause difficult compactness/gluing and transversality
problems for $\Ff$ when defined on many individual Banach manifolds
$\Bb_\beta$ that are at most strata of a common ambient space $\Bb$.
While in traditional approaches the ambient spaces $\Bb$ itself are usually
inaccessible to calculus, in a series of papers 
\citet*{Hofer:2007a,Hofer:2009a,Hofer:2009b,Hofer:2017a}
construct ambient spaces in the form of generalized manifolds,
called {\bf M-polyfolds},\footnote{
  The ``M'' is a reminder that M-polyfolds are
  constructed in analogy to Banach \underline{m}anifolds,
  just replace the local model Banach space by some ($\SC$-retract
  of a) scale Banach space.
  The more general polyfolds are useful in problems
  having local symmetries.
  }
which are accessible to a customized generalized calculus called
{\bf scale} or {\bf sc-calculus}.\index{polyfolds}
Now \textbf{polyfolds} generalize M-polyfolds
like orbifolds generalize manifolds.

Roughly speaking, \textbf{polyfold theory} is a mixture of a generalized differential
geometry, a generalized non-linear analysis, and some category theory.

\subsubsection*{Shift map motivates scale calculus.}
The discovery of scale calculus was triggered by the
properties of the shift map. That map shows up already for one of the
simplest non-trivial scenarios, namely, the downward gradient equation
$\Ff(\gamma):=\dot \gamma+(\nabla f)\circ \gamma=0$
for paths $\gamma\colon\R\to M$ and
associated to a given Morse function $f\colon M\to\R$ on a closed Riemannian
manifold. Given critical points $x\not= y$ of $f$,
the moduli space $\MM_{xy}$ consists of all solutions
$\gamma\colon \R\to M$ of $\Ff(\gamma)=0$ which asymptotically
connect $x$ to $y$, i.e. $\lim_{t\to-/+\infty}\gamma(t)=x/y$. Time
shift by $\tau\in\R$ produces again a solution
\[
     (\tau_*\gamma)(t):=\gamma(t+\tau).
\]
Having the same image in $M$ one
calls $\gamma$ and $\tau_*\gamma$ equivalent and denotes the space of
equivalence classes by $\mm_{xy}:=\MM_{xy}/\R$.
While the quotient of a manifold by a free and smooth
action inherits a manifold
structure,\index{scale calculus!motivated by shift map}
unfortunately, the time shift action is not smooth at all.

To illustrate non-smoothness let us simplify the scenario
in that we consider the time shift action of $\R$
on the compact\footnote{
  Compactness of the domain $\SS^1$ is crucial that inclusion $C^{k+1}\INTO
  C^k$ is compact.
  }
domain $\SS^1=\R/\Z$ of $v\in C^k=C^k(\SS^1,\R)$
where $k\in\N_0$. The derivative of the \textbf{\Index{shift map}}
\begin{equation}\label{eq:shift-map-0}
     \Psi\colon \R\times C^{k+1}\to C^{k+1},\quad
     (\tau,v)\mapsto \tau_*v
\end{equation}
taken at $(\tau,v)\in \R\times C^{k+1}$ does not respect the target
space $C^{k+1}$. Indeed
\[
     d\Psi_{(\tau,v)} (T,V)=\left(\tau_* \dot v\right) T+ \tau_*V\in C^k ,
     \qquad (T,V)\in\R\times C^{k+1}
\]
takes values only in $C^k$, because $\dot v:=\frac{d}{dt} v$ does.
But then there is no reason to ask
the second summand $\tau_*V$ to be better than $C^k$
and for this the assumption $V\in C^k$ suffices.
While $\Psi(\tau,v)$ behaves terribly in $\tau$ it is
extremely tame in $v$, namely linear.

If one accepts different differentiability classes of domain and
target spaces, the shift map has the following still respectable properties
for $k\in\N_0$.
\begin{itemize}
\item[(a)]
  The shift map $\Psi\colon \R\times C^k\to C^k$ is continuous.
\item[(b)]
  The shift map as a map $\Psi\colon \R\times C^{k+1}\to C^k$ is pointwise
  differentiable in the usual sense with (Fr\'{e}chet) derivative
  $d\Psi_{(\tau,v)}\in\Ll(\R\times C^{k+1}, C^k)$.
\item[(c)]
  At $(\tau,v)\in \R\times C^{k+1}$ the derivative $d\Psi_{(\tau,v)}$
  extends uniquely ($C^{k+1}$ is dense in $C^k$) 
  from $\R\times C^{k+1}$ to a continuous linear map $\R\times C^k\to C^k$,
  denoted by $D\Psi_{(\tau,v)}\in\Ll(\R\times C^k, C^k)$ and called
  the \textbf{scale derivative}.
\item[(d)]
  The extension $D\Psi\colon  \R\times C^{k+1}\to\Llco(\R\times C^k,C^k)$
  is continuous in the compact-open topology,\footnote{
     But it is not continuous in the norm topology on $\Ll(\R\times C^k,C^k)$;
     see Remark~\ref{rem:noncont-norm-top}.
     }
  equivalently, it is continuous as a map
  \[
     D\Psi\colon (\R\times C^{k+1})\times (\R\times C^k)\to C^k,\quad
     (\tau,v,T,V)\mapsto D\Psi_{(\tau,v)}(T,V).
  \]
\end{itemize}
Properties (a--d) suggest that instead of considering $\Psi$ as a map
between one domain and one target, both of the same regularity
(the same level), one should use the whole nested sequence (scale) of Banach spaces
and consider $\Psi$ as a map $(\R\times C^k)_{k\in\N_0}\to (C^k)_{k\in\N_0}$
between scales.

The proof of (a--d) hinges on (i) compactness of the linear operator
$C^{k+1}\INTO C^k$ given by inclusion and (ii)
on density of the intersection $E_\infty:=\bigcap_{k=1}^\infty C^k$
in each of the Banach spaces (levels) $E_k:=C^k$.
A nested sequence of Banach spaces $E=(E_k)$
satisfying (i) and (ii) is called a \textbf{Banach scale} or an
\textbf{\boldmath$\SC$-Banach space} and $E_k$ is called
\textbf{level \boldmath$k$} of the scale.

Now one turns properties (a--d) into a definition
calling maps between $\SC$-Banach spaces
satisfying them \textbf{continuously \boldmath$\SC$-differentiable}
or of \textbf{class \boldmath$\SC^1$};
cf. Remark~\ref{rmk:sc-smoothness}
and Definition~\ref{def:sc-differentiability}.
The new class $\SC^1$ generalizes the usual class $C^1$
in the following sense: Suppose that $f\colon E\to F$ is a map between
Banach scales whose restriction to any domain level $E_m$
actually takes values in the corresponding level $F_m$ of the target
and all the so-called \textbf{level maps}
$f_m:=f|_{E_m}\colon E_m\to F_m$ are of class $C^1$. Then $f$ is of
class $\SC^1$; see Lemma~\ref{le:nec_suff_cond_for_sc_by_Ck}.

\subsubsection*{Sc-manifolds are modeled on scale Banach spaces.}

In complete analogy to manifolds
a scale or $\SC$-manifold is a paracompact Hausdorff space $X$
just locally modeled on a scale Banach space $E$, as opposed to an
ordinary Banach space, and requiring the transition maps
to be $\SC$-diffeomorphisms.
In finite dimension $\SC$-calculus
specializes to standard calculus and $\SC$-manifolds
are manifolds.

\subsubsection*{M-polyfolds are modeled on sc-retracts.}
Motivated by Cartan's last theorem~\citeyearpar{Cartan:1986a} M-polyfolds
are described locally by retracts in scale Banach spaces, replacing
the open sets of Banach spaces in the familiar local description of manifolds.
As a consequence M-polyfolds may have locally varying dimensions;
see Figure~\ref{fig:fig-sc-retract-schippe}.
Enlarging the class of smooth maps one risks loosing
vital analysis tools such as the implicit function theorem --
which indeed is not available for sc-smooth maps;
see~\citet{Filippenko:2018b}.
However, for moduli space
problems one only needs to work in the subclass of sc-Fredholm maps
on which an implicit function theorem is available.

\subsubsection*{Outlook.}

Given abstract polyfold theory~\citep{Hofer:2017a},
it is now up to the scientific community
to work out and provide modules, or black boxes,
also called LEGO pieces,
that uniformly cover large classes of applications,
say in Morse and Floer theory.
A shift map LEGO has been provided by \citet{Frauenfelder:2018a}.

\subsubsection*{Appendix on topology and analysis.}

In the appendix we review the incarnations
of the usual (Fr\'{e}chet) calculus in various contexts - from
topological vector spaces (TVS) to Banach spaces.
For self-consistency of the text we recall many results of standard
calculus in topology and analysis which are used in the main body.

\subsubsection*{Notes to the Reader.}

Each of the two chapters begins with a detailed summary and survey
of its contents. Read both of these two chapter summaries first 
to get an idea of what about is this text.

In the end the present lecture notes
only grew to two chapters plus an appendix providing
some background of calculus -- from topology to functional analysis.
In class we also treated, though briefly,
scale Fredholm theory and, as an application, the shift map
LEGO \citep{Frauenfelder:2018a} for Morse and Floer path spaces.
In a planned extension we shall add these topics in the form of two
additional chapters.

Unless mentioned differently, we (closely) follow~\citet*{Hofer:2017a}.
Two other great sources are~\citet*{Fabert:2016a} and~\citet{Cieliebak:2018a}.



\cleardoublepage
\phantomsection

\chapter{Scale calculus}\label{sec:scale-calculus}

The\index{sc abbreviates!scale}
ubiquitous ``sc'' a-priori abbreviates \emph{scale},
but in the context of scale linear operators and maps it stands for
\emph{\underline{s}cale \underline{c}ontinuous}.
The\index{sc abbreviates!\underline{s}cale \underline{c}ontinuous}
latter is denoted in the context of general, possibly non-linear,
maps by $\SCz$ or by $\SCk$ for $k$ times
scale continuously differentiable maps.
In a linear context \emph{subspace} means linear subspace.

\vspace{0.1cm} 
Section~\ref{sec:scale-structures} ``Scale structures''
introduces the notion of a Banach scale
which is a nested sequence of sets $E=E_0\supset E_1\supset\dots$
called levels -- each one being actually a Banach space --
and subject to two more axioms.
A subset $A\subset E$ of the top level
generates, we also say induces, naturally a new nested sequence
$A^{\cap E}$ by intersecting $A$ with each level $E_m$.
The new levels $A_m:=A\CAP E_m$ form the
nested sequence $A^{\cap E}=\left( A=A_0\supset A_1\supset\dots\right)$.
Of course, not every nested sequence is of the form $A^{\cap E}$.

The three axioms for a \textbf{Banach scale} $E$, also called a scale Banach
space or an \textbf{\boldmath$\SC$-Banach space},
are the following: Each level $E_m$ is a Banach space under its own norm
$\abs{\cdot}_m$, all inclusions $E_m\INTO E_{m-1}$ are \emph{compact} linear
operators, and the intersection $E_\infty:=\bigcap_m E_m$ of all
levels is \emph{dense} in every level Banach space $E_m$.
The points of $E_m$ are called points of regularity $m$
and those of $E_\infty$ smooth points.
A \textbf{Banach subscale} of $E$ is a Banach scale $B$ whose levels
are Banach subspaces of the corresponding levels of $E$. 
Is every Banach subscale $B$ generated by its top level $B_0$, i.e. is
$B=(B_0)^{\cap E}$? You bet.
However, not every closed subspace $A$ of a scale
Banach space $E$ generates a Banach subscale. In general, there is no
reason that $A^{\cap E}$ satisfies the density axiom, consider
e.g. cases of trivial intersection $A\CAP E_\infty=\{0\}$.
Those closed subspaces that do generate a Banach subscale
are of crucial significance, they are called
\textbf{\boldmath$\SC$-subspaces}. 

Because Fredholm theory is a fundamental tool in the analysis of
solution spaces of differential equations, $\SC$-subspaces $K$
of finite dimension will be key players, as well as
$\SC$-subspaces $Y$ of finite codimension.
\emph{Finite} dimensional $\SC$-subspaces $K$ of an $\SC$-Banach space $E$
are characterized as follows. For finite dimensional
subspaces $K$ of $E$ it holds:
  \[
     \text{$K\subset E_\infty$}
   \quad\Leftrightarrow\quad
     \textsf{$K$ is an sc-subspace (generates a Banach subscale) of
       $E$.}
  \]
Although simple to prove, this equivalence is far reaching.
In particular, since due to finite dimension the generated
Banach subscale is constant (all levels $K_m=K$ are necessarily equal).

\vspace{0.1cm} 
Section~\ref{sec:examples} ``Examples''
presents a number of examples of
Banach scales, e.g. Sobolev scales and weighted Sobolev scales,
that arise frequently in the study of solution spaces of
differential equations on manifolds. The desire to simplify
and, most importantly, to unify the many cumbersome steps
of the classical treatment of analyzing solution spaces
actually was the motivation to invent scale calculus;
see e.g. the introductions to~\citet{Hofer:2005a}
and~\citet{Hofer:2006a}.

\vspace{0.1cm} 
Section~\ref{sec:scale-linear-theory} ``Scale linear theory''
carries over fundamental notions of linear operators
to Banach scales. For example a \textbf{scale linear} operator
is a linear operator $T\colon E\to F$ between $\SC$-Banach spaces which
preserves levels, that is $T(E_m)\subset F_m$ $\forall m$.
For such $T$ the restriction to level $E_m$ takes values in $F_m$.
The restriction as a map $T_m:=T|_{E_m}\colon E_m\to F_m$
is called a \textbf{level operator}.\index{level!operator}
Now one can carry over (some) standard notions and properties of linear
operators, say continuity, compactness, projections, and so on,
by requiring each level operator to have that property.
For instance, a \underline{s}cale \underline{c}ontinuous
operator, called \textbf{\boldmath$\SC$-operator}, is a scale linear
operator $T\colon E\to F$ such that all level operators are continuous,
that is $T_m\in\Ll(E_m,F_m)$ $\forall m$.

However, as soon as it comes to $\SC$-Fredholm operators, not level
preservation $T(E_m)\subset F_m$, but level -- better regularity --
\emph{improvement} $S(E_m)\subset F_{m+1}$ $\forall m$ becomes a key property.
The latter are called \textbf{\boldmath$\SC^+$-operators}.
They have the property that all their level operators are compact.
\newline
Similarly, as mentioned earlier for Fredholm operators in the usual sense,
finite dimensional and finite codimensional $\SC$-subspaces
will enter the definition of $\SC$-Fredholm operators.
Thus one needs the following two notions:

  Firstly, the notion of \textbf{\boldmath$\SC$-splitting}
of $E=F\oplus G$ into an $\SC$-direct sum
of $\SC$-subspaces $F$ and $G$ called
\textbf{\boldmath$\SC$-complements} of one another.
Just as for Banach spaces any finite dimensional $\SC$-subspace
admits an $\SC$-complement.

  Secondly, the notion of \textbf{\boldmath$\SC$-quotient} $E/A$.
This allows to establish for finite \emph{co}dimensional $\SC$-subspaces
existence of an $\SC$-complement 
(Proposition~\ref{prop:fin-codiml-sc-subspaces})
and characterize them as follows (Lemma~\ref{le:Crit-fin-codiml-subspace-SC}).
For finite co\-dimen\-sional subspaces $A$ of $E$ it holds:
  \[
     \text{$A$ closed in $E$}
   \quad\Leftrightarrow\quad
     \textsf{$A$ is an sc-subspace of $E$.}
  \]
It seems that so far the literature missed 
to spell out these two facts explicitly.

An \textbf{\boldmath$\SC$-Fredholm operator} is an $\SC$-operator
$T\colon E\to F$ such that there are $\SC$-splittings
$E=K\oplus X$ with levels $E_m=K\oplus X_m$ and
$F=Y\oplus C$ with levels $F_m=Y_m\oplus C$
where $K=\ker T$ is the kernel and $Y=\im T$
is the image of $T$ and both $K$ and $C$ are of finite dimension.
Looks fine already?
Well, there is one condition missing yet.\footnote{
  While $T$ as a map $X\to Y:=\im T$
  is an isomorphism, this is not yet
  guaranteed for the level operators as maps
  $T_m\colon X_m=X\CAP E_m\to (\im T)\CAP F_m= Y_m$.
  Their images $T(X_m)\subset Y_m$ a-priori are only
  subspaces. To get isomorphisms one needs to exclude
  elements of higher levels $X\setminus X_m$
  getting mapped under $T$ to level $F_m$,
  in symbols $T(X\setminus X_m)\CAP F_m=\emptyset$.
  }
The operator $T$ as a map $T\colon X\to Y$ must be
an \textbf{\boldmath{$\SC$-isomorphism}}
(a bijective $\SC$-operator whose inverse
is level preserving).
This enforces \emph{level regularity of $T$}
in the sense that $Te\in F_m$ implies $e\in E_m$.
And it assures that the levels $Y\CAP F_m$ generated by the
$\SC$-subspace $Y=\im T$ coincide with the images
$T_m(X_m)$ of the level operators.
\newline
It is then a consequence that the level operators
$T_m\colon E_m\to F_m$ are all Fredholm
with the same kernel $K$ and the same Fredholm index.
Vice versa, if the level operators of an $\SC$-operator $T\colon E\to F$
are Fredholm and $T$ is level regular in the above sense,
then $T$ is $\SC$-Fredholm.
\newline
The classical stability result that Fredholm property and index
are preserved under addition of a compact linear operator carries over
this way: The $\SC$-Fredholm property is preserved under addition of
$\SC^+$-operators.

\vspace{0.1cm} 
Section~\ref{sec:scale-differentiability} ``Scale differentiability''
is\index{revolution!the -- happens here}
where the revolution happens. 
Free difference quotients! Away with Fr\'{e}chet mainstream suppression! 
\citet*{Hofer:2017a} just did it, at least in infinite dimensions..

\begin{remark}\label{rmk:sc-smoothness}
Let $U\subset E$ and $V\subset F$ be open
subsets of $\SC$-Banach spaces. An open subset of $E_m$ is given by
$U_m:=U\cap E_m$. A \textbf{scale continuous}\,\footnote{
  Also called {\textbf{of class \boldmath$\SC^0$}}
  which by definition means level preserving and continuity of all restrictions
  as maps $f_m:=f|_{U_m}\colon U_m\to V_m$, called \textbf{level maps}.
  }
map $f\colon U\to V$ is called \textbf{continuously scale differentiable}
or {\textbf{of class \boldmath$\SC^1$}} if
\begin{itemize}
\item
  the upmost so-called \textbf{diagonal map} (of height one), namely $f$
  as a map $f\colon U_1\to V_0$ is pointwise differentiable and
\item
  its derivative $df(x)\in\Ll(E_1,F_0)$ admits a continuous
  linear extension 
  \begin{equation*}\label{eq:sc-deriv-extends-hh}
  \begin{tikzcd} 
     E_0\arrow[rr, dashed, "Df(x)"]
     &&
       F_0
     \\
     E_1
     \arrow[u, hook, "I_1"]
     \arrow[rru, "{df(x)\,\in\,\Ll(E_1,F_0),\;x\in U_{1}}"']
     &&
  \end{tikzcd} 
  \end{equation*}
  from the dense subset $E_1$ to $E_0$ itself, called the
  \textbf{\boldmath$\SC$-derivative of $f$ at $x\in U_1$}
  and denoted by $Df(x)$. Furthermore, it is required that
\item
  the \textbf{tangent map} $Tf\colon TU\to TV$ defined by
  \[
     Tf(x,\xi):=(f(x),Df(x)\xi)
  \]
  is of class $\SC^0$.
  Here the \textbf{tangent bundle of \boldmath$U$} is the open
  subset\footnote{
    To get the shifted scale $U^k$ forget the first $k$ levels:
    Its $m^{\rm th}$ level is $(U^k)_m:=U_{m+k}$.
    }
  \[
     TU:=U^1\oplus E^0
  \]
  of the Banach scale $E^1\oplus E^0$.
\end{itemize}
\end{remark}

The third axiom, the one requiring level preservation and continuity
of the level maps associated to the tangent map $Tf$,
has a lot of consequences caused by the shift in the definition
of the tangent bundle $TU:=U^1\oplus E^0$.
For instance $Df(x)\colon E_0\to F_0$ restricts at points
of better regularity, say $x\in U_{m+1}\subset U_1$, to (continuous)
\textbf{level operators} $D_\ell f(x):=Df(x)|_{E_\ell}\colon E_\ell\to F_\ell$ for all
levels between $0$ and down to level $m$.
\begin{quote}
     \textsf{In general, the scale derivative only admits level operators $D_\ell f(x)$
     for all levels $\ell$ down to the level right above the $x$-level!}
\end{quote}
The $\SC$-derivative $U_{m+1}\ni x\mapsto Df(x)$
viewed (horizontally) between equal levels $E_m\to F_m$
enjoys only \emph{continuity with respect to the compact open
topology}\footnote{
  If $x_\nu\to x$ in $U_{m+1}$, then for each fixed $\zeta\in E_m$
  one has $Df(x_\nu)\zeta\to Df(x)\zeta$ in $F_m$.
  }
whereas viewed as a diagonal map $Df$ is \emph{continuous 
with respect to the operator norm topology}, i.e. $C^0$ as a map
\[
     U_{m+1}\to\Ll(E_{m+1},F_m)
\]
where the target carries the operator
norm. But for these domains
$Df=df$ pointwise, so $\SC^1$ implies that all
diagonal maps of height one $f\colon U_{m+1}\to V_m$
are of class $C^1$ in the usual sense and this brings us to

Section~\ref{sec:props-sc-differentiability} ``Differentiability --
Scale vs Fr\'{e}chet''.
Here we will see that higher scale differentiability
$f\in\SC^{\color{cyan}k}(U,V)$ implies that all ($\forall m$) diagonal
maps $f\colon U_{m+\ell}\to V_m$ of height
$\ell\in\{0,\dots, {\color{cyan} k}\}$
are of class $C^\ell$ in the usual Fr\'{e}chet 
sense.\footnote{
  Note: Level maps $f_m:U_m\to V_m$ ($\ell=0$)
  of an $\SC^k$-map are only guaranteed to be continuous ($C^0$)
  no matter what is the value of $k$.
  }
Vice versa, for a map $f\colon U\to V$ there is the following
\textbf{criterion to be of class \boldmath$\SC^{k{\color{red}+1}}$}:
For each $\ell\in\{0,\dots,k\}$ restriction produces height $\ell$ diagonal
maps $f\colon U_{m+\ell}\to V_m$ $\forall m$ that are of class $C^{\ell {\color{red}+1}}$.

Section~\ref{sec:chain-rule} ``Chain rule''
proves this building block of calculus.
It allows to construct scale \emph{manifolds} by patching together
local pieces of $\SC$-Banach spaces.
If $f\colon U\to V$ and $g\colon V\to W$ are both of class $\SC^1$, then the
composition $g\circ f$ is, too!
The exclamation mark is due to the fact that
applying an $\SC$-derivative one looses one
level (of regularity), so one might expect
to loose two levels when composing two
$\SC^1$ maps. One doesn't!
This relies on the compactness axiom for the inclusions
$E_k\INTO E_{k+1}$ in a Banach scale.

Section~\ref{sec:boundary-recognition} ``Boundary recognition''
introduces the \textbf{degeneracy index}
$d_C(x)$ of a point $x$ in what is called a partial
quadrant $C$ in a Banach scale $E$.
It takes the value $0$ on \emph{interior points} $x$,
value $1$ on \emph{boundary points} in the usual sense,
and points with $d_C(x)\ge 2$ are \emph{corner points}.
We state without proof invariance of $d_C(x)$ under
\textbf{\boldmath{$\SC^1$-diffeomorphisms}}, that is
$\SC^1$-maps with $\SC^1$-inverses.
It is remarkable that $\SC$-smooth diffeomorphisms
recognize boundary points and corners.
In contrast, homeomorphisms also recognize boundaries,
but not corners.

Section~\ref{sec:sc-manifolds} ``$\SC$-manifolds ''
defines an \textbf{\boldmath$\SC$-manifold}
as a paracompact Hausdorff space $X$
endowed with an equivalence class of $\SC$-smooth atlases.
A continuous map $f:X\to Y$ between $\SC$-manifolds
is called $\SC$-smooth if so are all representatives
with respect to $\SC$-charts of $X$ and $Y$.
An $\SC$-chart of $X$ takes values in an $\SC$-Banach space $E$
and so, due to compatibility of $\SC$-charts through
$\SC$-diffeomorphisms,
the level structure of $E$ is inherited by the $\SC$-manifold $X$.
An important class of $\SC$-manifolds consists of loop spaces
$X:=W^{1,2}(\SS^1,M)$ for finite dimensional manifolds $M$.
These are even \textbf{\boldmath strong $\SC$-manifolds},
or \textbf{\boldmath$\SSC^\infty$-manifolds}, in the sense that
already level maps are smooth,
as opposed to only the diagonal maps as is required for $\SC^\infty$.
Given an $\SC$-manifold $X$, its \textbf{tangent bundle} is
a map of the form 
$
     p\colon TX\to X^1
$ 
that projects on the shifted $\SC$-manifold (forget level $0$ of $X$).

\smallskip
After this survey of Chapter~\ref{sec:scale-calculus} you could, upon
first reading, skip the remainder of Chapter~\ref{sec:scale-calculus} and
proceed with the introduction to Chapter~\ref{sec:local-models}.

\section{Scale structures}\label{sec:scale-structures}

\subsubsection*{Scales of sets}

\begin{definition}[Scales]
A \textbf{\Index{scale} of sets} or a\index{scale!structure}
\textbf{scale structure on a set \boldmath$A$}
is a nested sequence of subsets
\[
     A=A_0\supset A_1\supset A_2\supset \dots 
\]
The subset $A_m$ is called the\index{level!of a scale}
\textbf{level \boldmath$m$ of the scale}
and its elements\index{points!of regularity $m$}
\textbf{points of regularity \boldmath$m$}.
The elements of the intersection
\[
     A_\infty:=\bigcap_{m\in\N_0} A_m
\]
are called the\index{smooth!points}
\textbf{smooth points} of the scale.
Given a level $A_m$, the enclosing levels $A_0,\dots,A_{m-1}\supset A_m$
are called \textbf{\Index{superlevels}}, the enclosed levels
$A_m\supset A_{m+1},A_{m+2},\dots$ \textbf{\Index{sublevels}},
\textbf{of \boldmath$A_m$}.
\end{definition}

\begin{definition}[Subscale]
A \textbf{\Index{subscale}} of a scale of sets $A$
is a scale of sets $B$ whose levels are subsets of the
corresponding levels of $A$, that is
\begin{equation*}
\begin{tikzcd} [row sep=small, column sep=tiny]
\text{\footnotesize scale}
  & A
  &=& A_0
  &\supset& A_1
  &\supset& A_2
  &\supset& \dots
\\
\text{\footnotesize subscale}
  & B
  &=& B_0
     \arrow[draw=none]{u}[sloped,auto=false]{\subset}
  &\supset& B_1
     \arrow[draw=none]{u}[sloped,auto=false]{\subset}
  &\supset& B_2
     \arrow[draw=none]{u}[sloped,auto=false]{\subset}
  &\supset& \dots
\end{tikzcd} 
\end{equation*}
\end{definition}

\begin{definition}[Constant scale]
The \Index{constant scale} structure on a set $A$
is the one whose levels $A_m:=A$ are all given by $A$ itself.
\end{definition}

\begin{definition}[Induced scale $B^{\cap A}\subset A$]
A scale structure on a set $A$ induces a scale structure on any subset
$B\subset A$, called the\index{induced!scale}
\textbf{\boldmath induced scale}
or\index{$B^{\cap A}$ subscale of $A$ generated by $B\subset A$}
\textbf{\boldmath subscale generated by $B$},
denoted\index{subscale!generated by $B$}
by $B^{\cap A}$. By definition the $m^{\rm th}$ level
\begin{equation}\label{eq:induced-scale}
     (B^{\cap A})_m:=B_m:=B\CAP A_m,\qquad m\in\N_0
\end{equation}
is the \textbf{\boldmath\Index{part of $B$ in level $A_m$}}.
Observe that $B_\infty:=\CAP_{m\ge 0} B_m=B\CAP A_\infty$.
\end{definition}

Note that for an induced scale emptiness $B_\infty=\emptyset$
is possible, even if $A_\infty\not=\emptyset$.

\begin{example}
[Not every subscale is an induced scale]
\label{ex:subscale-not=-subsetscale}
\begin{equation*}
\begin{tikzcd} [row sep=small, column sep=tiny]
\text{\footnotesize subscale $B$ of $A$}
  & \{1,2\}
  &\supset& \{1\}
  &\supset& \emptyset
  &\supset& \emptyset
  & \dots
\\
\text{\footnotesize scale $A$}
  & \{0,1,2\}
     \arrow[draw=none]{u}[sloped,auto=false]{\supset}
  &\supset& \{1,2\}
     \arrow[draw=none]{u}[sloped,auto=false]{\supset}
  &\supset& \{2\}
     \arrow[draw=none]{u}[sloped,auto=false]{\supset}
  &\supset& \emptyset
  & \dots
\\
\text{\footnotesize induced scale $\{1,2\}^{\cap A}$}
  & \{1,2\}
     \arrow[draw=none]{u}[sloped,auto=false]{\subset}
  &=& \{1,2\}
     \arrow[draw=none]{u}[sloped,auto=false]{\subset}
  &\supset& \{2\}
     \arrow[draw=none]{u}[sloped,auto=false]{\subset}
  &\supset& \emptyset
  & \dots
\end{tikzcd} 
\end{equation*}
\end{example}

\begin{definition}[Shifted scale $A^k$]
\label{def:shifted-scale}
Forget the first $k$ levels of a scale $A$
and use $A_k$ as the new level zero to obtain
the\index{$A^k$ shifted scale with levels $(A^k)_m:=A_{k+m}$}
\textbf{\Index{shifted scale}} $A^k$ with levels
\[
     (A^k)_m:=A_{k+m}, \qquad m\in\N_0.
\]
We sometimes abbreviate $A^k_m:=(A^k)_m$.
\end{definition}

\subsection*{Banach scales (\boldmath$\SC$-Banach spaces)}

\begin{definition}[Scale Banach space]\label{def:sc-B-sp}
A \textbf{scale structure} or an \textbf{\boldmath\Index{sc-structure}}
on a Banach space $E$, is a nested sequence of linear spaces
\begin{equation*}
\begin{tikzcd} [column sep=tiny]   
E=E_0
  & \supset & E_1
  & \supset & E_2
  & \supset & \dots
\end{tikzcd} 
\end{equation*}
called \Index{level}s such that the following axioms are satisfied.
\begin{labeling}{\texttt{(Banach levels)}}
\item[\Index{\texttt{(Banach levels)}}]
  Each level $E_m$ is a Banach space (coming with a
  norm\index{$(\abs{\cdot}$@$\abs{\cdot}_m$ norm in Banach scale level $E_m$}
  $\abs{\cdot}_m:=\abs{\cdot}_{E_m}$).
\item[\Index{\texttt{(compactness)}}]
  The inclusions $E_{m}\stackrel{I_m}{\INTO} E_{m-1}$
  are\index{operator!compact linear --}
  compact linear operators for all $m$.
\item[\Index{\texttt{(density)}}]
  The set of smooth points $E_\infty:=\bigcap_{m\in\N_0} E_m$ is dense
  in each level $E_m$.
\end{labeling}
An \textbf{\boldmath\Index{sc-Banach space}},
also called a \textbf{scale Banach space}\index{scale!Banach space}
or\index{scale!structure!on a Banach space}
a \textbf{\Index{Banach scale}},
is a Banach space $E$ endowed with a scale structure.
\end{definition}

\begin{exercise}[\Index{sc-direct sum}]
The Banach space direct sum $E\oplus F$ of two $\SC$-Banach spaces $E$
and\index{$F\oplus G$ direct sum Banach scale}
$F$ is a Banach scale with respect to the natural levels
\begin{equation}\label{eq:DS-levels}
     (E\oplus F)_m:=E_m\oplus F_m,\qquad m\in\N_0.
\end{equation}
\end{exercise}

\begin{exercise}[Finite dimensional Banach scales are constant]
\label{exc:constant-B-scale}
A finite dimensional Banach space $E$
has the unique $\SC$-structure
$E_0=E_1=\dots$.
\end{exercise}

\begin{exercise}[Infinite dimensional Banach scales $E$]
\label{exc:dense-levels}
Note that any inclusion operator $E_{m+\ell}\INTO E_m$
is compact, hence continuous.
Show that
\begin{itemize}
\item[(i)]
every level $E_m$ is a dense subset of each of its superlevel Banach spaces;

\item[(ii)]
no level $E_m$ is a closed subset of any of its superlevel Banach spaces.
Equivalently, every level $E_m$ has a non-empty set complement 
in each of its superlevel Banach spaces, in symbols
\[
     E_{m-\ell}\setminus E_{m}\not=\emptyset 
\]
whenever $m\in\N$ und $\ell\in\{1,\dots, m\}$.
\end{itemize}
\end{exercise}

\begin{definition}
A Banach scale is called \textbf{reflexive}
(resp. \textbf{separable})
if\index{Banach scale!reflexive --}
every\index{reflexive!Banach scale}
level\index{separable!Banach scale}
is\index{Banach scale! separable --}
a reflexive (resp. separable) Banach space.
\end{definition}

\begin{lemma}[Induced nested sequences]
\label{le:open-subset-scale}
Any subset of an $\SC$-Banach space $E$
induces via level-wise intersection a scale of sets;
see~(\ref{eq:induced-scale}).
\begin{labeling}{\rm (closed)}
\item[\rm (closed)]
  A closed subset $A\subset E$ meets any level $E_m$
  in a closed set $A_m=A\CAP E_m$.
  If $A\subset E$ is a closed sub\underline{space},
  then the inclusion $i_m\colon A_m\INTO A_{m-1}$ is a compact
  linear operator between Banach spaces.
\item[\rm (open)]
  If $U\subset E$ is an open subset, then $U_m=U\CAP E_m$ is open
  in $E_m$ and the set $U_\infty=U\CAP E_\infty$ of smooth points
  is dense in every $U_m$.
\end{labeling}  
\end{lemma}

\begin{proof}
The intersection $A\CAP E_m=( I_1\circ\dots\circ I_m)^{-1}(A)$ is the
pre-image under a continuous map; analogous for $U$.
\texttt{(Compactness):}
Pick a bounded subset $B$ of $A_m$. Then $B$ is a subset
of all four spaces in the diagram
\begin{equation*}
\begin{tikzcd} 
A_{m-1}
\arrow[rr, hook, "{\iota_{m-1}}"]   
  &&
  E_{m-1}
\\
A_m
\arrow[rr, hook, "{\iota_m}"]   
\arrow[u, hook, dashed, "i_m"]
  &&
  E_m
  \arrow[u, hook, "I_m"']
\end{tikzcd} 
\end{equation*}
The closure of $B$ in $E_{m-1}$ is compact since $ I_m\circ \iota_m$
is a compact linear operator. But $A_{m-1}$ is a closed
subspace of $E_{m-1}$ which contains $B$. Thus the closure
of $B$ is contained in $A_{m-1}$ as well.
\texttt{(Density):}
Pick $p\in U_\infty=\bigcap_{k\in\N_0} (U\CAP E_k)\subset (U_m\CAP E_\infty)$.
By density of $E_\infty$ in $E_m$ there is a sequence $E_m\ni p_\nu\to
p$ in $E_m$.
But $p\in U_m$ and $U_m\subset E_m$ is open.
\end{proof}

\subsection*{\boldmath$\mathrm{Sc}$-subspaces I}

As a closed linear subspace of a Banach space
is a Banach space itself under the restricted norm,
it is natural to call it a Banach subspace.
In view of this the following definition seems natural
in the setting of Banach scales.

\begin{definition}[Banach subscale]
\label{def:Banach-subscale}
A\index{scale!Banach sub--}\index{Banach!subscale}
\textbf{Banach subscale}\index{subscale!Banach --}
of a Banach scale $E$ is a Banach scale $B$
whose levels $B_m$ are Banach subspaces of the
corresponding levels $E_m$ of $E$.
\end{definition}

On the other hand, we just saw in Lemma~\ref{le:open-subset-scale}
that a closed linear subspace $A$ in an $\SC$-Banach space $E$
generates a nested sequence $A^{\cap E}=(A\CAP E_m\subset E_m)_{m=0}^\infty$
of Banach subspaces. So it is natural to ask
\begin{itemize}
\item[1)]
  Does the intersection sequence $A\CAP E_m$ always form a Banach scale?\\
  Answer: \emph{No}. (Even if $\dim A<\infty$; see
  Lemma~\ref{le:fin-diml-sc-subspaces}.)
\item[2)]
  Is a Banach subscale $B\subset E$ generated by its top level $B$?
  In symbols, is every level $B_m$ given by intersection $B\CAP E_m$?\\
  Answer: \emph{Yes}. (See Lemma~\ref{le:sc-subspaces-and-Banach-subscales}.)
\end{itemize}

\begin{definition}[Scale subspaces]
\label{def:subspace-scale}
An \textbf{\boldmath$\SC$-subspace} of an $\SC$-Banach space $E$
is\index{sc-subspace}
a closed subspace $A$ of $E$ whose intersections with the levels of
$E$ form the levels of a Banach subscale of $E$.\footnote{
    The axioms \texttt{(Banach levels)} and \texttt{(compactness)}
    are automatically satisfied for \emph{any closed} subspace $B$ of $E$; see
    Lemma~\ref{le:open-subset-scale}.
    The problematic axiom is \texttt{(density)}.
    }
Speaking of an $\SC$-subspace $A$ of $E$ implicitly carries the
information that $A$ is the Banach subscale of $E$ whose levels are
given by
\[
     A_m:=A\CAP E_m .
\]
Alternatively $A^{\cap E}$ denotes the Banach scale generated by an
$\SC$-subspace $A$.\index{$A^{\cap E}$ Banach scale generated $A$}
\end{definition}

\begin{lemma}\label{le:sc-subspaces-and-Banach-subscales}
a) The top level $B_0$ of a Banach subscale $B$
of a Banach scale $E$ is, firstly, an $\SC$-subspace of $E$
and, secondly, generates $B$ ($={B_0}^{\cap E}$).
b) Every $\SC$-subspace of $E$ arises this way.
\end{lemma}

\begin{proof}
a) By \texttt{(density)} of the set $B_m$ in the Banach space $B_0$,
the closure $\widebar{B_m}^{\,0}$ with respect to the $B_0$ norm is
the whole space $B_0$. Hence
\[
     E_m\CAP B_0=E_m\CAP \widebar{B_m}^{\,0}=E_m\CAP B_m=B_m
\]
where identity two, also three, holds since $B_m$ itself is a closed
subspace of the Banach space $E_m$ by axiom \texttt{(Banach levels)}.
b) By definition an $\SC$-subspace generates a Banach subscale.
\end{proof}

\begin{lemma}[Finite dimensional $\SC$-subspaces]
\label{le:fin-diml-sc-subspaces}
Given a scale Banach space $E$ and a finite dimensional linear
subspace $B\subset E$. Then
  \[
     \text{$B$ is an $\SC$-subspace of $E$}
     \qquad\Leftrightarrow\qquad
     \text{$B\subset E_\infty$.}
  \]
The $\SC$-subspace $B$ generates the constant Banach scale with levels
$B_m=B$.
\end{lemma}

\begin{proof}
'$\Rightarrow$' The finite dimensional linear subspace
$B\CAP E_\infty=\bigcap_m(B\CAP E_m)=B_\infty$ of $B$
is dense by the \texttt{(density)} axiom for the subspace scale
generated by $B$. Thus by finite dimension
it is even equal to $B=B\CAP E_\infty\subset E_\infty$.
'$\Leftarrow$'
By assumption $B\subset E_\infty\subset E_m$, thus
$B_m:=B\CAP E_m=B$. So $B$ generates the constant
scale which by Exercise~\ref{exc:constant-B-scale}
is a Banach scale since $\dim B<\infty$.
\end{proof}

\begin{example}[Closed but not $\SC$]
Let $E=L^2([0,1])$ with the, even reflexive, 
Banach scale structure $E_m:=W^{m,2}([0,1])$.
Then the characteristic function $\chi=\chi_{[0,1/2]}$
generates a 1-dimensional, thus closed, subspace $A$ of $E$.
Since $\chi$ lies in $L^2$, but not in $W^{m,2}$ for $m\ge 1$,
the levels $A_m:=A\cap E_m=\{0\}$ are trivial for $m\ge 1$,
hence $A_\infty=\{0\}$ is not dense in $E_0=L^2$.
\end{example}

\begin{exercise}
Infinite dimensional $\SC$-subspaces cannot lie inside~$E_\infty$.

\vspace{.1cm}\noindent 
[Hint: Given an $\SC$-subspace $A\subset E$,
show $A\subset E_\infty$ $\Rightarrow$ $A_\infty=A$, so $A_1=A$.
But $A_1\subset E_1\INTO E$ embeds compactly in $E$, whereas $A$ is
closed in $E$.]
\end{exercise}

The finding that for finite dimensional linear subspaces
``being located in the set of smooth points'' is equivalent
to ``generating a (constant) Banach subscale'' is 
extremely useful. For instance, this enters the proofs of
\begin{itemize}
\item 
  Prop.~\ref{prop:fin-diml-sc-complemented}:
  Finite dimensional $\SC$-subspaces are $\SC$-complemented;
\item 
  Prop.~\ref{prop:fin-codiml-sc-subspaces}:
  Finite codimension $\SC$-subspaces are $\SC$-complemented;
\item 
  Le.~\ref{le:Crit-fin-codiml-subspace-SC}:
  Characterization of finite codimensional $\SC$-subspaces.
\end{itemize}
This list shows that certain classes of scale subspaces
have properties analogous to the corresponding
class of Banach subspaces.
\\
Suppose $A$ and $B$ are $\SC$-subspaces
of an $\SC$-Banach space $E$.
How about the sum $A+B$ and the intersection $A\CAP B$?
\\
Is it possible, in general, to endow
the sum $A+B$ and the intersection $A\CAP B$
with the structure of Banach scales?
So it is natural to ask the following.
\begin{itemize}
\item[3)]
  Is the sum $A+B$ of $\SC$-subspaces always an $\SC$-subspace?\\
  Answer: \emph{No.} The sum of two closed subspaces,
  even in Hilbert space, is not even
  closed in general;\footnote{
    The Hilbert space $l^2$ of square summarizable real sequences
    contains the closed subspaces $A:=\{a\in l^2\mid a_{2n}=0\;\forall n\}$
    and $B:=\{b\in l^2\mid b_{2n}=b_{2n-1}/2n\,\forall n\}$.
    The sum $A+B$ cannot be closed, because it is dense in $l^2$
    (since it contains all sequences of compact support)
    and $A+B$ is not all of $l^2$: Write $(1/n)_n\in l^2$ in the
    form $a+b$ with $a\in A$ and $b\in B$. Then
    $1/2n=a_{2n}+b_{2n}=b_{2n}=b_{2n-1}/2n$.
    So $b_{2n-1}=1$ for all $n$, hence $b\notin l^2$.
    }
  cf.~\citet{Schochetman:2001a}.\\
  Answer: \emph{Yes}, if $A$ and $B$ are finite dimensional.\\
  Answer: \emph{Yes}, if $A$ or $B$ is of finite codimension;
  see Exercise~\ref{exc:sc-subspaces-sum-intersection}.

\item[4)]
  Is the intersection $A\CAP B$ of $\SC$-subspaces an $\SC$-subspace?
  \\
  Answer: 
  \emph{Yes}, if $A$ or $B$ is finite dimensional.\\
  Answer: \emph{Yes}, if $A$ and $B$ are of finite codimension;
  see Exercise~\ref{exc:sc-subspaces-sum-intersection}.
  \\
  (General case: In each level $E_m$ the intersection $(A\CAP B)\CAP
  E_m=A_m\CAP B_m$ is closed. How about density of $(A\CAP B)_\infty$
  in $A\CAP B$?)
\end{itemize}

\section{Examples}\label{sec:examples}

Throughout $\SS^1$ denotes the unit circle
in\index{$\SS^1$ unit circle}
$\R^2$ or, likewise, the quotient space $\R/\Z$.
It is convenient to think of functions $f\colon \SS^1\to \R$
as $1$-periodic functions on the real line,
that is $f\colon \R\to\R$ such that $f(t+1)=f(t)$ for every $t$.

By definition a counter-example is an example with negative sign.

\begin{example}[Not a Banach scale]
\label{ex:C0_bd}\index{counter-examples:!not a Banach scale}
The\index{$C^k_{\rm bd}(\R)$ is not a Banach scale}
vector space $C^k_{\rm bd}(\R)$ of $k$ times continuously
differentiable functions $f\colon \R\to\R$ which,
together with their derivatives up to order $k$,
are bounded is a Banach space with respect to the
$C^k$ norm.
However, the scale whose \texttt{(Banach levels)}
are $E_m:=C^m_{\rm bd}(\R)$ satisfies \texttt{(density)}
since $E_\infty$ is equal to $C^\infty_{\rm bd}(\R)$, but it does not
satisfy \texttt{(compactness)}. A counter-example
is\index{counter-examples:! bump running to infinity}
provided by a \textbf{bump running to infinity}:
Pick a bump, that is a compactly supported function $\chi\ge 0$ on $\R$,
and set $\chi_\nu:=\chi(\cdot-\nu)$.
Then the set $C:=\{\chi_\nu|\nu\in\N\}$ is bounded in $E_1$, indeed
$\norm{\chi_\nu}_{C^1}=\norm{\chi}_{C^1}=:c_\chi<\infty$,
but there is no convergent subsequence with respect to the $C^0$ norm,
i.e. in $E_0$.
\end{example}

So non-compactness of the domain
obstructs the \texttt{(compactness)} axiom.
There are two ways to fix this.
The obvious one is to use a compact domain; below we illustrate
this by choosing the simplest one $\SS^1$.
Another way is to impose a decay condition
when approaching infinity. This works well
for domains which are a product of a compact manifold with $\R$.
Concerning targets, replacing $\R$ by $\R^n$ makes
no difference in the arguments.

\begin{exercise}[The non-reflexive Banach scale $C^k(\SS^1)$]
\label{ex:C0_S1}\index{$C^k(\SS^1)$ is a Banach scale}
Show that the Banach space $C^k(\SS^1)$
endowed with the scale structure whose levels are
the Banach spaces $E_m:=C^{k+m}(\SS^1)$
is a separable non-reflexive Banach scale.

\vspace{.1cm}\noindent
[Hint: Concerning \texttt{(compactness)}
use the Arzel\`{a}--Ascoli Theorem~\ref{thm:AA}.
For separability see e.g. discussion in~\citet[App.\ A]{Weber:2017b}.]
\end{exercise}

\begin{example}[Sobolev scales -- compact domain]
\label{ex:Wkp_S1}\index{$W^{k,p}(\SS^1)$ is a Banach scale}
Fix an integer $k\in\N_0$ and a real $p\in[1,\infty)$.
The Sobolev space $W^{k,p}(\SS^1,\R^n)$
endowed with the scale structure whose levels are
the Banach spaces $E_m:=W^{k+m,p}(\SS^1,\R^n)$ is a Banach
scale. These Sobolev scales are separable ($1\le p<\infty$) and
reflexive ($1<p<\infty$) by Theorem~\ref{thm:Lp-spaces}.

\vspace{.1cm}\noindent
[Hints: Sobolev embedding theorems and $E_\infty=C^\infty(\SS^1,\R^n)$.]
\end{example}

\begin{exercise}[Weighted Sobolev scales -- non-compact domain $\R$]
\label{exc:Wkp_R}\index{$W^{k,p}_{\delta_k} (\R)$ is a Banach scale}
Fix a monotone cutoff function $\beta\in C^\infty(\R,[-1,1])$
with $\beta(s)-1$ for $s\le-1$ and $\beta(s)=1$ for $s\ge 1$,
as illustrated by Figure~\ref{fig:fig-Morse-exp-weight}.
\begin{figure}
  \centering
  \includegraphics
                             [height=4cm]
                             {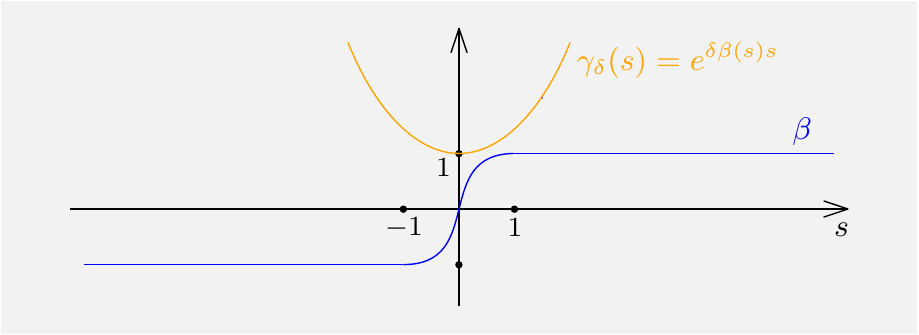}
  \caption{Exponential weight function $\gamma_\delta$ and
                 monotone cutoff function $\beta$}
  \label{fig:fig-Morse-exp-weight}
\end{figure}
Given a constant $\delta\ge0$,
define an exponential weight function by
\[
     \gamma_\delta(s):=e^{\delta s\beta(s)}.
\]
Let $k\in\N_0$ and pick a constant $p\in(1,\infty)$.
Check that the set defined by
\begin{equation}\label{eq:W-kp-delta}
     W^{k,p}_\delta(\R,\R^n)
     :=\{f\in W^{k,p}(\R,\R^n)\mid
     \gamma_\delta f\in W^{k,p}(\R,\R^n)\}
\end{equation}
is a real vector space on which 
\[
     \Norm{f}_{W^{k,p}_\delta}
     :=\Norm{\gamma_\delta f}_{W^{k,p}}.
\]
defines a complete norm.
Consider a strictly increasing sequence
\begin{equation}\label{eq:weights}
     0=\delta_0<\delta_1<\dots
\end{equation}
of reals. Prove that the levels defined by
\begin{equation*}
     E_m:=W^{m,p}_{\delta_m}(\R,\R^n) ,\quad m\in\N_0
\end{equation*}
form a Banach scale structure on the Banach space $L^p(\R,\R^n)$.
\end{exercise}

\begin{exercise}[Strictly increasing is necessary]
Show that if two weights $\delta_{m-1}=\delta_m$ are equal
in~(\ref{eq:weights}), then the \texttt{(compactness)} axiom fails.
\end{exercise}

\begin{exercise}[Reflexivity and separability]
Show that the weighted Sobolev space $W^{m,p}_{\delta}(\R,\R^n)$
is a closed subspace of $W^{m,p}(\R,\R^n)$.
Conclude that the weighted Sobolev scales in the previous example are
separable ($1\le p<\infty$) and reflexive ($1<p<\infty$).
\end{exercise}

\begin{example}[Completion scale -- H\"older spaces are not Banach scales]
\label{ex:Hoelder}\index{counter-examples:!not a Banach scale}
Fix a constant $\mu\in(0,1)$. The sequence of H\"older spaces
$E_m:=C^{m,\mu}(\SS^1)$ for $m\in\N_0$ satisfies the
\texttt{(compactness)} axiom by the Arzel\`{a}--Ascoli
Theorem~\ref{thm:AA}, but the set of smooth points
$E_\infty=C^\infty(\SS^1)$ is not dense in any level $E_m$.
However, taking the closure of $E_\infty$ in each level
produces a Banach scale $\bar E_m:=\overline{E_\infty}^{k,\mu}$ called
the \textbf{\Index{completion scale}}.\index{Banach scale!completion --}
This works for every nested sequence of Banach spaces
that satisfy \texttt{(compactness)} as shown by
\citet[Le.\ 4.11]{Fabert:2016a};
they also solve Exercise~\ref{exc:Wkp_R}.
\end{example}

\begin{exercise}
For which $p\in[1,\infty]$, if any,
is $L^p(\SS^1)$ endowed with the levels $E_m:=L^{p+m}(\SS^1)$ a
Banach scale?
\end{exercise}

\begin{definition}[Weighted Hilbert space valued Sobolev spaces]
\label{def:weighted-spaces}
Let $k\in\N_0$, $p\in(1,\infty)$, and $\delta\ge 0$.
Suppose $H$ is a separable Hilbert space
and define the space $W^{k,p}_\delta(\R,H)$
by~(\ref{eq:W-kp-delta}) with $\R^n$ replaced by $H$.
This\index{Sobolev spaces!Hilbert space valued --}
is again a Banach space; see \citet[Appendix]{Frauenfelder:2018a}.
\end{definition}

\begin{example}[Path spaces for Floer homology]
\label{ex:HS-valued-Sobolevscales}
A a monotone unbounded function $f\colon\N\to(0,\infty)$
is called a\index{growth function!of Floer PDE}
\textbf{growth function}.
Common types of Floer homologies provide such $f$,
order refers to spatial order:

\vspace{.5cm}
\begin{tabular}{llll}
\toprule
     Floer homology
  & Order
  & Mapping space
  & Growth type
\\
\midrule
     Periodic
  & $1^{\rm st}$ 
  & loop space
  & $f(\nu)=\nu^2$
\\
     Lagrangian
  & $1^{\rm st}$ 
  & path space
  & $f(\nu)=\nu^2$
\\
     Hyperk\"ahler
  & $1^{\rm st}$ 
  & $\Map(M^3,\R^{2n})$
  & $f(\nu)=\nu^{2/3}$
\\
     Heat flow
  & $2^{\rm nd}$ 
  & loop space
  & $f(\nu)=\nu^4$
\\
\bottomrule
\end{tabular}
\vspace{.5cm}

\noindent
Here Periodic and Lagrangian Floer homology
refer, respectively, to the elliptic PDEs studied by
\citet{floer:1988c,floer:1989a} on the cylinder $\R\times\SS^1$ and by
\citet{floer:1988a} imposing Lagrangian
boundary conditions along the strip $\R\times[0,1]$.
Hyperk\"ahler and Heat flow Floer homology refer to the theories
established by~\citet*{hohloch:2009a}, respectively, by
\citet{weber:2013a,weber:2013b,weber:2014c}.
The heat flow is described by a parabolic PDE
that relates to Floer's elliptic PDE; see \citet*{salamon:2006a}.

Given a constant $p\in(1,\infty)$, let $\delta_m$ for $m\in\N_0$
be a sequence as in~(\ref{eq:weights}). Given a growth function $f$,
let $H_m=\ell^2_{f^m}$ be the fractal Hilbert scale
on $H=\ell^2$ introduced by \citet[Ex.\ 3.8]{Frauenfelder:2018a}.
Then the Banach space $E_m$ is defined as intersection
of $m+1$ Banach spaces, namely
\begin{equation*}
     E_m:=\bigcap_{i=0}^m W^{i,p}_{\delta_m}(\R,H_{m-i}) ,\quad m\in\N_0.
\end{equation*}
The norm on $E_m$ is the maximum of the $m+1$ individual norms. 
This is a complete norm. This endows $E=L^p(\R,H)$ with the
structure of a Banach scale; see \citet[Thm. 8.6]{Frauenfelder:2018a}.
\end{example}

\section{Scale linear theory}\label{sec:scale-linear-theory}

\subsection{Scale linear operators}\label{sec:scale-linear-operators}

\begin{definition}[Scale linear operators $T$ and their level
  operators $T_m$]\mbox{ }
\begin{itemize}
\item[(i)]
  A \textbf{scale linear operator}\index{scale!linear operator}
  is a linear operator $T\colon E\to F$ between Banach scales
  which\index{operator!sc-linear}\index{level!preserving}
  is \textbf{level preserving}, that is $T(E_m)\subset
  F_m$ for every $m\in\N_0$.
\item[(ii)]
  The restriction of a scale linear operator $T\colon E\TO F$ to a level
  of $E$ takes values in the corresponding level of $F$. Hence
  $T$ viewed as a map between corresponding levels
  is\index{$T_m\colon E_m\to F_m$ level operator}
  a linear operator 
  \[
     T_m:=T|_{E_m}\colon E_m\to F_m ,\qquad  m\in\N_0
  \]
  between Banach spaces, called the\index{level!operator}
  \textbf{\boldmath$m^{\rm th}$ level operator}.
\end{itemize}
\end{definition}

If a scale linear operator $T\colon E\to F$ is, in addition, a bijective map,
then each level operator
$
     T\colon E_m\rightarrowtail F_m
$
is injective -- but not necessarily surjective.
It will be surjective if the inverse linear map $T^{-1}\colon F\to E$
is level preserving: In this case each level operator
$(T^{-1})_m\colon F_m\rightarrowtail E_m$ is injective.
This proves

\begin{lemma}
Suppose a scale linear operator $T\colon E\to F$
is bijective and its inverse is level preserving.
Then every level operator
\begin{equation*}
\begin{tikzcd} 
T_m:=T|_{E_m}\colon E_m
\arrow[r, tail, two heads]
  & F_m
\end{tikzcd} 
\end{equation*}
is a bijective linear map between Banach spaces.
\end{lemma}

\subsubsection*{Scale continuous operators}

\begin{definition}[$\SC$-operators]
A scale linear operator $T\colon E\to F$ is called
\textbf{scale continuous} or\index{scale!continuous}
\textbf{scale bounded} or of\index{scale!bounded}
\textbf{\boldmath class $\SCz$}, if each level operator
$
     T_m\in\Ll(E_m,F_m)
$
is a continuous linear operator between Banach spaces.
  \begin{equation*}\label{eq:sc-operators}
  \begin{tikzcd} [row sep=small] 
     E=E_0\arrow[rrrr, "{T_0:=T}", "{\text{\tiny continuous}}"']
     &&&&
       F_0=F
     \\
     \vdots
     \arrow[u, hook]
     &&&& 
       \vdots
       \arrow[u, hook]
     \\
     E_m\arrow[rrrr, "{T_m:=T|_{E_m}}", "{\text{\tiny continuous}}"']
     \arrow[u, hook]
     &&&&
       F_m
       \arrow[u, hook]
     \\
     E_{m+1}
     \arrow[u, hook, "{\text{\tiny\; compact}}"', "I_{m+1}"]
     \arrow[rrrr, "{T_{m+1}:=T|_{E_{m+1}}}", "{\text{\tiny continuous}}"']
     &&&&
     F_{m+1}
     \arrow[u, hook, "J_{m+1}"', "{\text{\tiny compact\;\,}}"]
     \\
     \vdots
     \arrow[u, hook]
     &&&& 
       \vdots
       \arrow[u, hook]
       \\
     \cap_m E_m=:E_\infty \arrow[rrrr, "{T_\infty:=T|_{E_\infty}}"]
     \arrow[u, "{\text{\tiny\; dense}}"']
     &&&&
       F_\infty:=\cap_m F_m
       \arrow[u, "{\text{\tiny dense\;\,}}"]
  \end{tikzcd} 
  \end{equation*}
Such $T$ is called an \textbf{\boldmath\Index{sc-operator}}
between Banach scales.
In the realm of scale linear operators $\SC$ does not abbreviate
\emph{scale}, but \emph{\underline{s}cale \underline{c}ontinuous}. Let
$
     \Llsc(E,F)
$
be\index{$\Llsc(E,F)$ $\SC$-operators}
the \textbf{set of \boldmath$\SC$-operators} between the Banach
scales $E$ and $F$.
\end{definition}

\begin{exercise}
Check that $\Llsc(E,F)$ is a linear space.
\end{exercise}

\begin{exercise}
Given $\Llsc(E,F)$, consider the sequence $\left(\Ll_m\right)_{m\in\Nn_0}$
of Banach spaces $\Ll_m:=\Ll(E_m,F_m)$ under the operator norm. 
Characterize the case
in which one has inclusions $\Ll_{m+1}\subset\Ll_m$ as
a) sets and 
b) continuous maps between Banach spaces.
In b) characterize the case in which 
c) the set $\Ll_\infty:=\CAP_m\Ll_m$ is dense in each level $\Ll_m$ and
d) every inclusion operator $\Ll_{m+1}\INTO\Ll_m$ is compact.
\end{exercise}

\begin{definition}[$\SC$-projections]
An \textbf{\Index{sc-projection}} is
a scale continuous operator $P$
whose level operators $P_m$ are all \textbf{\Index{projection}s},
i.e. $P_m\circ P_m=P_m$.
Equivalently, the $\SC$-projections are those $\SC$-operators $P$
with $P^2=P$.
\end{definition}

\begin{lemma}[Image and kernel of $\SC$-projections are $\SC$-subspaces]
\label{le:sc-projections-ker-im}
The image, hence the kernel, of any $P=P^2\in\Llsc(E)$
are $\SC$-subspaces.\index{sc-projection!induced sc-splitting}
\end{lemma}

\begin{proof}
As $Q:=\1-P$ is an $\SC$-projection whose image is the kernel of $P$,
it suffices to show that the images $R_m:=\im P_m=\Fix\, P_m$ form
a Banach subscale of $E$. The inclusion $R_{m+1}\subset R_m$ holds
by $E_{m+1}\subset E_m$. And $R_m=\Fix\, P_m$ is a closed (linear) subspace of
$E_m$ by continuity and linearity of~$P_m$.
\texttt{(Compactness)}
of the inclusion $I_m\colon E_m\INTO E_{m-1}$,
together with $R_m\subset E_m$ being closed,
tells that each inclusion $i_m\colon R_m\INTO R_{m-1}$
takes bounded sets into pre-compact ones.
\\
It remains to check \texttt{(density)}
of $R_\infty=\CAP_\ell R_\ell$ in $R_m$.
To see this pick $r_m\in R_m\subset E_m$
and, by density of $E_\infty$ in $E_m$, pick some in $E_m$ convergent
sequence $E_\infty\ni e_\nu\to r_m$.
Since $R_m=\Fix\, P_m$ and by continuity of $P_m$ we get
\[
     R_m\ni r_m=P_mr_m=\lim_{\nu\to\infty} P_me_\nu.
\]
For each $e_\nu$ it holds that
\[
     P_me_\nu=P e_\nu \in E_\ell\CAP R_\ell =R_\ell\quad\forall \ell.
\]
Here the first equality holds since $e_\nu\in
E_\infty\subset E_m$ and $P_m$ is the restriction of $P$ to $E_m$,
so $P_me_\nu=P e_\nu$. But $P$ preserves levels and $e_\nu$
lies in every $E_\ell$, so $Pe_\nu$ lies in every $E_\ell$
and $P e_\nu=P_\ell e_\nu\in \im P_\ell=R_\ell$.
Thus $P e_\nu\in R_\infty$.
\end{proof}

\begin{definition}[$\SC$-isomorphisms]
A (linear) \textbf{\Index{sc-isomorphism}} is
a bijective $\SC$-operator whose inverse\footnote{
  Inverses of bijective $\SC$-operators are not
  automatically level preserving: Consider the
  identity
  operator from the forgetful Banach scale
  $E_0\supset E_2\supset E_3\dots$
  to $E_0\supset E_1\supset E_2\dots$
  }
is level preserving.
\end{definition}

\begin{exercise}
For an $\SC$-isomorphism
$T\colon E\to F$ all level operators
\[
     T_m\in\Ll(E_m,F_m),\qquad
     (T^{-1})_m\in\Ll(F_m,E_m)
\]
are continuous bijections with continuous inverses.

\vspace{.1cm}\noindent 
[Hint: Bounded inverse theorem, equivalently, open mapping theorem.]
\end{exercise}

\subsubsection*{Scale compact operators include \boldmath$\SC^+$-operators}

\begin{definition}[Scale compact operators]
An \textbf{\Index{sc-compact operator}}
is a scale linear operator $S\colon E\to F$ whose
level operators $S_m\colon E_m\to F_m$
are all compact (hence bounded) linear operators between Banach spaces.
\end{definition}

Scale compact operators are $\SC$-operators,
i.e. elements of $\Llsc$.

\begin{definition}[$\SC^+$-operators]
Suppose $E$ and $F$ are scale Banach spaces.
Recall that $F^1$ denotes the Banach scale that arises
from $F$ by forgetting the $1^{\rm st}$ level $F_0$ and taking
$F_1$ as the new level $0$.
The elements $S\in\Llsc(E,F^1)$ are called\index{sc$^+$-operator}
\textbf{\boldmath$\SC^+$-operators}
and we use the notation\index{$\Llscp$ $\SC^+$-operators}
\[
     \Llscp(E,F):=\Llsc(E,F^1).
\]
\end{definition}

\begin{remark}[$\SC^+$-operators are scale compact]
\label{rem:sc-plus}
The \texttt{(compactness)} axiom not only shows 
$
     \Llscp(E,F)\subset\Llsc(E,F)
$,
but also that any $\SC^+$-operator $S\colon E\to F$
is $\SC$-compact: This follows from the commutative diagram
\begin{equation*}
\begin{tikzcd} 
E_m
\arrow[rrrr, dashed, "S"]
\arrow[drrrr, "bounded", "{S\colon E_m\to(F^1)_m=F_{m+1}}"']
  &&&&
  F_m
\\
  &&&&
  F_{m+1}
  \arrow[u, hook, "\text{compact}"']
\end{tikzcd} 
\end{equation*}
since the composition of a bounded and a compact linear operator is compact.
\end{remark}

\begin{remark}[Are scale compact operators always $\SC^+$-operators?] \

\noindent
No: Let $E$ be an infinite dimensional Banach scale.
The inclusion $\iota\colon E^1\to E$
has compact level operators $\iota_m\colon E_{m+1}\to E_m$
and it is an $\SC^+$-operator, indeed
$\iota\in\Llsc(E^1,E^1)=:\Llscp(E^1,E)$.
Now forget level one in $E^1$ and in $E$, 
denote the resulting Banach scales by $E^1_{\times 1}$ and
$E_{\times 1}$, respectively.
All level operators of the
inclusion $\iota\colon E^1_{\times 1} \to E_{\times 1}$
are still compact, but $\iota$ does not even map
level zero $(E^1_{\times 1})_0=E_1$ to level one 
$(E_{\times 1})_1=E_2$, let alone be continuous.
\end{remark}

\subsection*{\boldmath$\mathrm{Sc}$-subspaces II}

\subsubsection*{Direct sum and \boldmath$\SC$-complements of $\SC$-subspaces}

\begin{definition}\label{def:sc-direct-sum}
An $\SC$-subspace $F$ of a Banach scale $E$
is called \textbf{\boldmath\Index{sc-complement}ed} if there
is an $\SC$-subspace $G\subset E$
such that every Banach space direct sum of corresponding levels
\[
     F_m\oplus G_m=E_m
\]
is equal to the ambient level $E_m$. Such $G$ is called an
\textbf{\boldmath\Index{sc-complement}} of $F$. 
So the Banach space $F\oplus G$ carries
the natural Banach scale structure~(\ref{eq:DS-levels}). Such a pair
$(F,G)$ or such direct sum $F\oplus G$
is called an \textbf{\boldmath\Index{sc-splitting}} of $E$.
\end{definition}

\begin{exercise}
Let $G$ be an $\SC$-complement of $F\subset E$. Check that the Banach
space $F\oplus G$ together with the natural level structure~(\ref{eq:DS-levels})
indeed satisfies the axioms of a Banach scale.
\end{exercise}

\begin{exercise}[Sc-projections $\SC$-split]
\label{exc:sc-projections-split}
There is\index{sc-projection!induced sc-splitting}
an $\SC$-splitting 
\[
     E=\ker P\oplus \im P
\]
associated to any $\SC$-projection, i.e. any idempotent
$P=P^2\in\Llsc(E)$.

\vspace{.1cm}\noindent 
[Hint: $P_m=(P_m)^2\in\Ll(E_m)$ means
$E_m=\ker P_m\oplus \Fix\,P_m$.
Lemma~\ref{le:sc-projections-ker-im}.]
\end{exercise}

\begin{proposition}\label{prop:fin-diml-sc-complemented}
Finite dimensional $\SC$-subspaces are
$\SC$-complemented.
\end{proposition}

\begin{proof}
We recall the proof given in~\citet[Prop.\,1.1]{Hofer:2017a}.
Suppose $F$ is a finite dimensional $\SC$-subspace
of a Banach scale $E$. Then $F\subset E_\infty$ by
Lemma~\ref{le:fin-diml-sc-subspaces}
and $F$ generates the constant Banach scale with levels
$F_m:=F\CAP E_m=F$ by Exercise~\ref{exc:constant-B-scale}.
Pick a basis $e_1,\dots,e_k$ of $F\subset E_\infty$
and let $e_1^*,\dots,e_k^*\in F^*$ be the dual basis.
By the Hahn--Banach Theorem~\ref{thm:HB} any $e_i^*$ extends to a
continuous linear functional $\lambda_i$ on $E$. 
The linear operator $P\colon E\to E$ defined by
$
     P(x):=\sum_{i=1}^k \lambda_i(x) e_i
$
is continuous and satisfies $P\circ P=P$ by straightforward calculation.
Note that the image of $P$ is $F$, that is $P(E)=F$, and that $F$ is
contained in $E_\infty$, hence in every level $E_m$.
This shows that $P$ is level preserving and admits
level operators $P_m\colon E_m\to E_m$.
By the continuous inclusion $E_m\INTO E$ the restrictions,
still denoted by $\lambda_i$, are continuous linear functionals
$\lambda_i\colon E_m\INTO E\to\R$ on every level $E_m$.
By the same arguments as for $P$ every level
operator $P_m$ is a continuous linear projection $P_m\colon E_m\to E_m$
with image $P(E_m)=F$.
Hence $P\in\Llsc(E)$ is an $\SC$-projection.

\vspace{.1cm}\noindent
\textsc{Goal.}
Given the finite dimensional $\SC$-subspace $F\subset E$
(generating the constant Banach scale $F_m=F$),
find a closed subspace $G\subset E$ such that
\begin{itemize}
\item[a)]
  $G_m:=G\CAP E_m$ are the levels of a Banach
  subscale ($G$ is an $\SC$-subspace);
\item[b)]
  $F_m\oplus G_m=E_m$ for every $m$.
\end{itemize}
\noindent
\textsc{Solution.}
The subspace of $E$ defined by $G:=(\1-P)E$ is closed
since $G=\ker P$ and $P$ is continuous.
a) 
By Lemma~\ref{le:open-subset-scale}
only \texttt{(density)} remains to be checked.
To see that $G_\infty:=\bigcap_m G_m=G\CAP E_\infty$ is dense in
any level $G_m$ pick $g\in G_m\subset E_m$.
By density of $E_\infty$ in $E_m$
choose a sequence $e_\nu\in E_\infty$ that converges in $E_m$ to
$g\in E_m$. The sequence $g_\nu:=(\1-P_m)e_\nu$
lies in $G\CAP E_m=:G_m$ and converges in $G_m$ to $g$: Indeed
${{\color{cyan}e_\nu-g_\nu}=P_me_\nu=P_m(e_\nu-g)}$, since 
$g\in G_\infty\subset G_m=\ker P_m$, 
so together with $\norm{P_m}\le 1$ we get 
\[
     \Abs{g-g_\nu}_m
     \le\Abs{g-e_\nu}_m+\Abs{{\color{cyan}e_\nu-g_\nu}}_m
     \le 2\Abs{g-e_\nu}_m\to 0,\quad
     \text{as $\nu\to\infty$.}
\]
b) For $x\in F\CAP G=\im P\CAP\ker P$ one has $x=Py$ for a $y\in
E$, hence $0=Px=PPy=Py=x$. So the intersection of subspaces
$F_m\CAP G_m=\{0\}$ is trivial, too.
It remains to show the equality $\im P_m +\ker P_m= E_m$, $\forall m$.
'$\subset$' Obvious.
'$\supset$' Pick $e\in E_m$ and set $f:=P_me$ and $g:=e-P_me$.
\end{proof}

\begin{exercise}
Give an example of a finite dimensional subspace $F$
of a Banach scale that is not $\SC$-complemented. 
%
[Hint: Pick $f\in C^0(\SS^1)\setminus C^1(\SS^1)$.]
\end{exercise}

\subsection*{Quotient Banach scales}

If you are not familiar with the quotient construction
for Banach spaces, have a look
at the neighborhood of Proposition~\ref{prop:Banach-quotient}
and its proof for definitions and explanations. Understanding that
proof helps to prove

\begin{proposition}[Quotient Banach scales]
\label{prop:quotient-scale}
Let $E$ be a Banach scale and $A$ an $\SC$-subspace.
Then the quotient Banach space $E/A$ with levels
\[
     (E/A)_m:=E_m/A_m:=\{x+A_m\mid x\in E_m\}, \quad m\in\N_0
\]
and inclusions
\begin{equation}\label{eq:quot-incl}
     E_{m+1}/A_{m+1}\INTO E_m/A_m,\quad
     x+A_{m+1}\mapsto x+A_m
\end{equation}
is a Banach scale. 
\end{proposition}

By Proposition~\ref{prop:Banach-quotient}
the norm on the coset space $E_m/A_m$ defined by
\[
     \Norm{x+A_m}_m
     :=d(x,A_m)
     :=\inf_{a\in A_m} \Abs{x-a}_m
\]
is complete. It is called the\index{quotient!norm}
\textbf{quotient norm} and measures the distance between the coset
$x+A$ and the zero coset, the subspace $A$ itself.

\begin{proof}
By the $\SC$-subspace assumption on $A$
every level $A_m:=A\CAP E_m$ is a Banach subspace of $E_m$, hence the
quotient spaces $E_m/A_m$ endowed with the norms $\norm{\cdot}_m$
are \texttt{(Banach levels)}.
To prove that the natural inclusions~(\ref{eq:quot-incl})
are compact linear operators pick a sequence $x_\nu+A_{m+1}$ in the
unit ball of $E_{m+1}/A_{m+1}$. (Note that $x_\nu\in E_{m+1}$.)
This means that the distance of each $x_\nu$ to the zero coset
$A_{m+1}$ of $E_{m+1}/A_{m+1}$ is not larger than $1$.
Hence for each $x_\nu$ there is a point $a_\nu$ in the zero coset $A_{m+1}$
at $E_{m+1}$ distance less than $2$, that is $\abs{x_\nu-a_\nu}_{m+1}<2$.
What we did is to choose for the given bounded sequence
of cosets $x_\nu+A_{m+1}=x_\nu-a_\nu+A_{m+1}$ a sequence of new
representatives $x_\nu-a_\nu$ which, most importantly, is bounded in
$E_{m+1}$. By compactness of the inclusion $E_{m+1}\INTO E_m$
there is a subsequence, still denoted by $x_\nu-a_\nu$, which
converges to some element $x\in E_m$. By continuity of the quotient
projection $\pi_m\colon E_m\to E_m/A_m$, $x\mapsto x+A_m$, see
Proposition~\ref{prop:Banach-quotient}, we obtain that
\begin{equation*}
\begin{split}
     \lim_{\nu\to\infty} (x_\nu+A_m)
   &=\lim_{\nu\to\infty} (x_\nu-a_\nu+A_m)\\
   &=\lim_{\nu\to\infty}\pi_m(x_\nu-a_\nu)\\
   &=\pi_m\left(\lim_{\nu\to\infty} x_\nu-a_\nu\right)\\
   &=\pi_m(x)\\
   &=x+A_m.
\end{split}
\end{equation*}
This proves the \texttt{(compactness)} axiom.
The set of smooth points
\[
     (E/A)_\infty:=\bigcap_j E_j/A_j=\{x+A_\infty\mid x\in E_\infty\}
\]
is dense in every level $E_m/A_m$, because
the image of a dense subset $E_\infty\subset E_m$ under the
continuous surjection $\pi_m\colon E_m\to E_m/A_m$ is dense in the
target space by Lemma~\ref{le:cont-surj-dense}.
This proves the \texttt{(density)} axiom
and Proposition~\ref{prop:quotient-scale}.
\end{proof}

It seems that so far the literature misses out
on the analogues for finite \underline{co}dimensional $\SC$-subspaces of
Proposition~\ref{prop:fin-diml-sc-complemented}
(existence of $\SC$-complement for finite dimensional $\SC$-subspaces)
and Lemma~\ref{le:fin-diml-sc-subspaces}
(characterization of finite dimensional $\SC$-subspaces). Let's change this.

\begin{proposition}\label{prop:fin-codiml-sc-subspaces}
Finite codimension 
$\SC$-subspaces are $\SC$-complemented.
\end{proposition}

Of course, the asserted $\SC$-complement $C$ in 
Proposition~\ref{prop:fin-codiml-sc-subspaces} has as
dimension the mentioned finite codimension. Hence $C$
carries the constant Banach scale structure and consists of
smooth points only; see Lemma~\ref{le:fin-diml-sc-subspaces}.

\begin{proof}
Let $A$ be an $\SC$-subspace of a scale Banach space $E$
of finite \textbf{\Index{co\-dimen\-sion}} $r=\codim A:=\dim E/A$.
By closedness and finite codimension the subspace $A$ of $E$
has a topological complement $C$; cf.~\citet[Prop.\,11.6]{brezis:2011a}.
Since $A$ is the kernel of the quotient projection
$\pi\colon E=A\oplus C\twoheadrightarrow E/A$
defined in~(\ref{eq:quot-proj}) we get $\dim C=\dim E/A=r$
for any topological complement.

Recall that a finite dimensional $\SC$-complement of $A$ is an
$\SC$-subspace $C$, endowed with constant levels $C_m=C$, such that
\[
     A_m\oplus C=E_m,\quad m\in\N_0.
\]
Constructing such $C$ inside the vector space $E_\infty$ of smooth points, see
Lemma~\ref{le:fin-diml-sc-subspaces}, is equivalent to $C$ being an $\SC$-subspace.

To define $C$ observe that
the Banach scale $E$ and the $\SC$-subspace $A$
give rise to the quotient Banach scale in
Proposition~\ref{prop:quotient-scale}.
Because the top level $(E/A)_0=E/A$  is of finite
dimension $r$, all sublevels are finite dimensional
and therefore the quotient Banach scale is actually constant.
Note that
\[
     E_\infty/A_\infty=E_m/A_m
\]
because both sides are of \emph{the same} finite dimension $r$ and
there is the natural inclusion
$E_\infty/A_\infty\to E_m/A_m$, $\varphi+A_\infty\mapsto \varphi+A_m$.
Pick a basis of 
\[
     (E/A)_\infty:=\bigcap_m E_m/A_m=E_\infty/A_\infty
\]
say $\varphi_1+A_\infty,\dots,\varphi_r+A_\infty$.
Observe that each $\varphi_j\in E_\infty$
and define
\[
     C:=\SPAN\{\varphi_1,\dots,\varphi_r\}\subset E_\infty.
\]
We show that $C$ is a topological complement~of~$A_m$.
To prove $A_m\CAP C=\{0\}$, pick $c\in A_m\CAP C$.
The quotient projection $\pi_m\colon E_m\to E_m/A_m$, $e\mapsto e+A_m$,
whose kernel is $A_m$ maps $c$ to the zero coset $0+A_m$.
On the other hand, the only element of $C$ that gets
mapped to the zero coset under $\pi_m$ is $c=0$.
We prove that $A_m+C=E_m$: ``$\subset$'' Obvious since
$A_m\subset E_m$ and $C\subset E_\infty\subset E_m$.
``$\supset$'' Pick $e\in E_m$ and express
$\pi_m(e)=e+A_m\in E_m/A_m$ in terms of the basis
$\{\varphi_j+A_m\}_{j=1}^r$,
let $c^1,\dots,c^r\in\R$ be the coefficients.
Set $c:=\sum_{j=1}^r c^j\varphi_j\in C$.
Then $\pi_m(c)=\pi_m(e)$, hence $a:=e-c$
lies in the kernel of $\pi_m$ which is $A_m$.
Hence $e=a+c$ is of the desired form.
\end{proof}

  A finite codimensional closed subspace of an ordinary Banach space $X$
  not only admits a topological complement, but there is even one in each
  dense subspace $X_\infty$ of $X$; see e.g.~\citet[Le.\,2.12]{Hofer:2007a}
  or~\citet[Prop.\,11.6]{brezis:2011a}. This enters the proof of

\begin{lemma}[Finite codimensional $\SC$-subspaces]
\label{le:Crit-fin-codiml-subspace-SC} 
Suppose $E$ is a scale Banach space
and $A$ is a linear subspace, then
  \[
     \text{$A$ is an $\SC$-subspace of $E$}
     \qquad\Leftrightarrow\qquad
     \text{$A$ is closed in $E$}
  \]
whenever $A\subset E$ is of finite codimension $r$.\,\footnote{
  A finite codimension subspace
  in Banach space need not be closed; see e.g.
  \citet[Prop.\,11.5]{brezis:2011a}.
  }
\end{lemma}

\begin{proof}
An $\SC$-subspace is closed by definition.
To prove the reverse implication, let $A$ be a closed subspace of $E$
of finite codimension, say $r$.
By the result mentioned above
the subspace $A$ of $E$ admits a topological complement $C$ 
contained in the dense subset $E_\infty$.
In the proof of Proposition~\ref{prop:fin-codiml-sc-subspaces}
we saw that topological complements satisfy $\dim C=\codim A=:r$.

We need to show that the levels defined by
$A_m:=A\CAP E_m\subset E_m$ satisfy the three axioms
of a Banach scale. As $A$ is closed in $E$,
by Lemma~\ref{le:open-subset-scale}
only the \texttt{(density)} axiom remains to be shown:
density of $A_\infty$ in each $A_m$.
\\
The inclusion $\underline{C\subset E_\infty}$ means that
$C$ is a constant Banach scale by
Lemma~\ref{le:fin-diml-sc-subspaces}
and Exercise~\ref{exc:constant-B-scale}.
Before proving density we show that
\begin{equation}\label{eq:vbhbhj}
     A_m\oplus C=E_m ,\quad m\in\N_0
\end{equation}
is a direct sum of closed subspaces of the Banach space $E_m$: Firstly,
closedness of $A_m$ we already know and $C$ is closed due to its
finite dimension.
Secondly, trivial intersection $A_m\CAP C=\{0\}$ holds true since
it even holds for the larger space $A\supset A_m$.
Thirdly, we prove $A_m+C=E_m$.
``$\subset$'' Obvious.
``$\supset$'' Any $e\in E_m\subset E=A+C$
is of the form $e=a+c$ for some $a\in A$ and $c\in C$.
But $a=e-c$ is also in $E_m$ (so we are done), because
both $e$ and $c\in \underline{C\subset E_\infty}\subset E_m$ are
and $E_m$ is a vector space.

We prove density of $A_\infty$ in $A_m$. Given $a\in A_m\subset E_m$,
by density of $E_\infty$ in $E_m$ there is some in the $E_m$ norm
convergent sequence $E_\infty\ni e_\nu\to a\in E_m$.
On the other hand, by~(\ref{eq:vbhbhj}) there is the direct sum of
Banach spaces $E_m=A_m\oplus C$, so $e_\nu$ is of the form
$e_\nu=a_\nu+c_\nu$ with $(a_\nu,c_\nu)\in A_m\times C$.
Clearly $a_\nu-a+c_\nu=e_\nu-a\to 0$ in $E_m$. But most
importantly $a_\nu=e_\nu-c_\nu\in E_\infty$
since the linear space $E_\infty$ contains $e_\nu$ and
$c_\nu\in \underline{C\subset E_\infty}$.
So
\[
     a_\nu\in \left(A_m\cap E_\infty\right)\subset\left( A\CAP E_\infty\right)
     =A\CAP\,\bigcap_m E_m=\bigcap_m A\CAP E_m
     =\bigcap_m A_m=A_\infty.
\]
Since $A_m$ and $C$ are topological complements
of one another, see~(\ref{eq:vbhbhj}), the norm in $E_m$
splits in the following sense.
By~\citet[Thm.\,2.10]{brezis:2011a}
there is a constant $\mu\ge 0$ such that
for any element $e$ of $E_m$
the norms of its parts in $A_m$ and in $C$
are bounded above by $\mu\abs{e}_m$.
For $e:=e_\nu-a=(a_\nu-a)+c_\nu\in A_m+ C$ we get that
\[
     \Abs{a_\nu-a}_m+\Abs{c_\nu}_m
     \le 2\mu\Abs{e_\nu-a}_m
     \to 0,\quad \text{as $\nu\to\infty$.}
\]
Hence $A_\infty\ni a_\nu\to a\in A_m$ in the $E_m$ norm.
This proves Lemma~\ref{le:Crit-fin-codiml-subspace-SC}.
\end{proof}

\begin{corollary}[Closed finite codimensional subspaces $\SC$-split $E$]
\label{cor:Crit-fin-codiml-subspace-SC} 
Given a scale Banach space $E$ and a 
finite codimension $r$ subspace $A$, then
  \[
     \text{$A$ is closed in $E$}
     \qquad\Leftrightarrow\qquad
     \text{$E=A\oplus C$ $\SC$-splits for some $C\subset E_\infty$.}
  \]
The $\SC$-splitting $E=A\oplus C$ has levels $E_m=(A\CAP E_m)\oplus C$
and $\dim C=r$.
\end{corollary}

\begin{proof}
``$\Rightarrow$'' Lemma~\ref{le:Crit-fin-codiml-subspace-SC} 
and Proposition~\ref{prop:fin-codiml-sc-subspaces}.
``$\Leftarrow$'' An $\SC$-subspace is closed by definition.
\end{proof}

\begin{exercise}[Intersection and sum of $\SC$-subspaces]
\label{exc:sc-subspaces-sum-intersection}
a) If $A,B\subset E$ are finite dimensional $\SC$-subspaces,
so are $A\CAP B$ and $A+B$.
\\
b) If $A,B\subset E$ are finite codimensional $\SC$-subspaces,
so are $A\CAP B$ and $A+B$.

\vspace{.1cm}\noindent 
[Hints: a) Lemma~\ref{le:fin-diml-sc-subspaces}.
b) By Lemma~\ref{le:Crit-fin-codiml-subspace-SC}
it suffices to show for $A\CAP B$ and for $A+B$
closedness\footnote{
  \citet[Prop.\,11.5]{brezis:2011a}: A subspace
  containing a closed one of finite codimension is closed.
  }
and finite codimension.\footnote{
  $\codim (A\CAP B)\le \codim A +\codim B$ and
  $\codim (A+B)\le\min\{\codim A,\codim B\}$.
  }
\end{exercise}

\subsection{Scale Fredholm operators}\label{sec:scale-Fredholm-operators}

\begin{definition}[$\SC$-Fredholm operators]
\label{def:sc-Fredholm}
An \textbf{\boldmath\Index{sc-Fredholm operator}}
is an $\SC$-operator $T\colon E\to F$ that satisfies the following axioms, namely
\begin{labeling}{\texttt{($\SC$-isomorphism)}}
\item[\texttt{($\SC$-splittings)}]
  there\index{\texttt{(}\unboldmath$\SC$\texttt{-splittings)}}
  are $\SC$-splittings $E=K\oplus X$,
  $F=Y\oplus C$ such that
\item[\Index{\texttt{(Ker)}}]
  $K$ is the kernel of $T$ and of finite dimension,
\item[\Index{\texttt{(Coker)}}]
  $Y$ is the image of $T$ and $C$ is of finite dimension,
\item[\texttt{($\SC$-isomorphism)}]
  the\index{\texttt{(}\unboldmath$\SC$\texttt{-isomorphism)}}
  operator $T$ viewed as a map $T\colon X\to Y$ is an 
  $\SC$-isomorphism.
\end{labeling}
The \textbf{Fredholm index} of $T$\index{Fredholm!index}
is the integer
\[
     \INDEX T:=\dim K-\dim C
     =\dim\ker T-\codim\im T.
\]
\end{definition}

By finite dimension the Banach subscales
generated by $K$ and $C$ are constant.
So trivially one gets the identities $K=K_\infty$ and $C=C_\infty$.
Combined with the equally trivial inclusions $K_\infty\subset E_\infty$
and $C_\infty\subset F_\infty$ they provide the precious
information that $K\subset E_\infty$ and $C\subset F_\infty$ consist
of smooth points.

\begin{proposition}\label{prop:sc-Freds-regularize}
$\SC$-Fredholm operators $T\colon E\to F$ are
\textbf{\Index{regularizing}}:
If\index{level!regularity}\index{regularity!level --}
$T$ maps $e\in E$ to level $m$, then already $e$ was
in level $m$; cf~(\ref{eq:level-reg}).
\end{proposition}

\begin{proof}
Let $e\in E$ and $Te\in F_m$. But $F_m=T(X_m)\oplus C$,
so $Te=Tx+c$ for some $x\in X_m\subset E_m\subset E$ and $c\in C$.
As $T(E)\CAP C=Y\CAP C=\{0\}$ and $e-x\in E$,
the identity $T(e-x)=c$ shows that both sides are zero. So
$e-x\in\ker T=K \mathop{{\color{magenta}\mbf{=}}} 
K_\infty\subset E_\infty\subset E_m$.
Therefore $e=(e-x)+x\in E_m$.
\end{proof}

\begin{exercise}[Intersection level $Y_m=Y\CAP F_m$
is image of level operator~$T_m$]\label{exc:sc-Fred-image-scale}
Consider an $\SC$-Fredholm operator 
$T\colon E=K\oplus X\to F=Y\oplus C$
where the $\SC$-subspace $Y:=\im T$
is the image $T(E)=T(X)$.
Recall that an $\SC$-subspace $Y$
generates a Banach subscale 
whose levels are given by intersection $Y\CAP F_m$.
Show that for $\SC$-Fredholm operators each intersection
level is equal to the image of the corresponding level operator
$T_m\colon E_m\to F_m$,~i.e.
\[
     Y_m:=Y\CAP F_m=T(E)\CAP F_m
     \mathop{{\color{magenta}\mbf{=}}} 
     T(E_m)=:\im T_m.
\]
[Hint: ``$\mathop{{\color{magenta}\mbf{\subset}}}$''
Suppose $Te=:y\in F_m$ where $e\in E$.]
\end{exercise}

\begin{exercise}[Isn't the axiom \texttt{($\SC$-isomorphism)} superfluous?]
\label{exc:sc-isom-axiom-superfluous?}
In view of Exercise~\ref{exc:sc-Fred-image-scale}
the fourth axiom in Definition~\ref{def:sc-Fredholm}
seems to be a consequence of the previous three axioms.
Is it?
\end{exercise}

\begin{exercise}
The composition $T\circ T'$ of two $\SC$-Fredholm operators
is an $\SC$-Fredholm operator
and $\INDEX(T\circ T')=\INDEX T + \INDEX T'$.
\end{exercise}

\begin{proposition}[Stability of $\SC$-Fredholm property]
\label{prop:stability-sc-Fredholm}
Consider an $\SC$-Fredholm operator $T\colon E\to F$
and an $\SC^+$-operator $S\colon E\to F$, then their sum
$T+S$ is also an $\SC$-Fredholm operator
of the same Fredholm index.
\end{proposition}

\begin{proof}
The sum $T+S\colon E\to F$ is an $\SC$-operator.
How about $\SC$-splittings?

\vspace{.1cm}\noindent 
\textsc{Domain splitting.}
Both $T$ and $S$ provide level operators $E_m\to F_m$
that are Fredholm and compact, respectively. Hence
the sum level operators $(T+S)_m=T_m+S_m\colon  E_m\to F_m$ are Fredholm
for each level $m$. Note that the kernel $K_m$
of $(T+S)_m$ contains $K_{m+1}$.
To see the reverse inclusion pick $x\in K_m$.
Then $Tx=-Sx \in F_{m\underline{+1}}$ by the $\SC^+$ nature of $S$.
Thus $x\in E_{m+1}$ by the regularity
Proposition~\ref{prop:sc-Freds-regularize}.
Hence $x\in K_{m+1}$. Thus $K_m=K_{m+1}$.
So $K:=K_0=K_m=K_\infty$
is finite dimensional and $K=K_\infty\subset E_\infty$.
Hence $K$ is an $\SC$-subspace
by Lemma~\ref{le:fin-diml-sc-subspaces}
and generates a constant Banach scale.
By Proposition~\ref{prop:fin-diml-sc-complemented}
the kernel scale $K$ admits an $\SC$-complement $X$ in $E$.
Summarizing, we have
\[
     E_m=K\oplus X_m,\qquad
     K=\ker(T_m+S_m)\subset E_\infty 
\]
for every $m\in\N_0$ and where $K$ does not depend on $m$.

\vspace{.1cm}\noindent
\textsc{Target splitting.}
Consider the image $Y:=(T+S)(E)=(T+S)(X)$ of the level zero Fredholm
operator $T+S\colon E\to F$. But the image of a Fredholm operator is
closed and of finite codimension, say $r$.
Hence $Y$ is an $\SC$-subspace of $F$
by Lemma~\ref{le:Crit-fin-codiml-subspace-SC}
and admits an $r$-dimensional $\SC$-complement $C\subset F_\infty$ 
by Corollary~\ref{cor:Crit-fin-codiml-subspace-SC}.
Summarizing, we have
\[
     F_m=Y_m\oplus C,\qquad
     Y_m=\im(T_m+S_m),\quad C\subset F_\infty 
\]
for every $m\in\N_0$ and where $C$ does not depend on $m$.

\vspace{.1cm}\noindent
\textsc{$\SC$-isomorphism.}
It is clear that $T$ as a map $T:X\to Y$ is bijective and
level preserving with continuous level operators $T_m:X_m\to Y_m$,
still injective. But why are these surjective? Exercise. 
Continuity of the inverse of $T_m$ then follows from the bounded inverse theorem.

\vspace{.1cm}\noindent
\textsc{Fredholm-index.}
Adding a compact operator, say $S:E\to F$, to a Fredholm operator, say
$T:E\to F$, does not change the Fredholm index.
This concludes the proof that $T+S$ is an $\SC$-Fredholm operator.
\end{proof}

\subsubsection*{Scale Fredholm operators -- 
na\text{\"i}ve approach through level operators}

Intuitively, if not {\bf naively}, an $\SC$-Fredholm operator should be a
level preserving linear operator $T\colon E\to F$ between Banach scales
whose level operators $T_m\colon E_m\to F_m$ are
\textbf{Fredholm operators}:\index{Fredholm!operator}
Each $T_m$ is linear and
continuous, has a finite dimensional kernel $K_m:=\ker T_m$, a closed image
$Y_m:=\im T_m$, and a finite dimensional \textbf{\Index{cokernel}}
$\coker T_m:=F_m/Y_m$. One calls the integer
\[
     \INDEX T_m:=\dim\ker T_m-\dim\coker T_m
\]
the\index{$index$@$\INDEX T:=\dim\ker T-\dim\coker T$ Fredholm index}
\textbf{Fredholm index} of $T_m$.\index{Fredholm!index}
(As we'll find out, one more condition to come.)

$\bullet$ Firstly, note that the kernels already form a nested sequence
$K:=K_0\supset K_1\supset\dots$ of (by continuity of $T_m$)
closed subspaces $K_m\subset E_m$. Note that $K_m=K\CAP E_m$
since $\ker T|_{E_m}=\ker T\CAP E_m$.
For a Banach scale it only misses the \texttt{(density)} axiom saying that
$K_\infty$ is dense in every level $K_m$.
\newline
Before adding a density requirement to the intuitive definition
of an $\SC$-Fredholm operator let us investigate
its consequences and see if a simpler condition
could do the same job.
If density holds, then $K$ is an $\SC$-subspace
and, by finite dimension, generates the constant Banach scale
(see Lemma~\ref{le:fin-diml-sc-subspaces}),
still denoted by $K$ and called the \textbf{\Index{kernel Banach scale}}.
\newline
Therefore  we add to the intuitive definition of $\SC$-Fredholm
the requirement
\begin{equation}\label{eq:kernel-scale-cond}
     \textsf{all level operators $T_m$ have the same kernel $K$}
\end{equation}
in symbols $K:=\ker T=\ker T_m\subset E_m$ $\forall m$.\footnote{
  Constant dimension
  $\dim\ker T_m=\dim\ker T$ suffices by the
  inclusions $K_{m-1}\supset K_m$.
  }
By Proposition~\ref{prop:fin-diml-sc-complemented}
the kernel $\SC$-subspace $K\subset E_\infty$ admits an $\SC$-complement
in $E$, say $X$.

$\bullet$ Secondly, the images $Y_m:=\im T_m=T(X_m)\supset T(X_{m+1})$ form a
nested sequence $Y:=Y_0\supset Y_1\supset\dots$ of closed subspaces
$Y_m\subset F_m$ of finite codimensions~$r_m$.
What is missing that the \textbf{image scale} $\im T$ with levels
$\im T_m$ is a Banach scale is I) \texttt{(density)} again, just as in
case of the kernel scale.
However, this time there is one more thing missing
that was automatic for the kernel scale. Namely,
we would like to have that II) the image scale is in fact generated by its top
level $Y=\im T$, that is we wish that
\[
     (\im T)\CAP F_m=\im T_m\quad \forall m.
\]
Suppose I) and II) hold. Namely, the image scale with levels $Y_m:=\im T_m$ is
a Banach scale and arises by intersection with its top level $Y=\im T$.
In other words, the closed finite codimensional
subspace $Y=\im T$ is an $\SC$-subspace and the generated
Banach subscale has intersection levels $Y\CAP F_m=\im T_m$
which are equal to the images of the level operators.
Let $r=\dim F/Y$ be the codimension of $Y$.
Then $Y$ admits by Proposition~\ref{prop:fin-codiml-sc-subspaces}
an $r$ dimensional $\SC$-complement $C$
which necessarily generates the constant Banach scale $C_m=C$.
Lemma~\ref{le:fin-diml-sc-subspaces}
tells that $C\subset F_\infty$.

By Corollary~\ref{cor:Crit-fin-codiml-subspace-SC}
and Lemma~\ref{le:Crit-fin-codiml-subspace-SC} 
a sufficient condition that $Y=\im T$
is an $\SC$-subspace, thus generating a Banach subscale,
is the following which we add as a requirement
to the intuitive definition of $\SC$-Fredholm:
\begin{equation}\label{eq:image-scale-cond}
     \textsf{existence of a topological complement $C\subset F_\infty$
                 of the image of $T$.}
\end{equation}

$\bullet$ Thirdly, to enforce that the intersection levels of the
Banach scale generated by $Y=\im T$
coincide with the images of the level operators, i.e.
\begin{equation}\label{eq:imag=intersection}
     T(X)\CAP F_m=T(X_m),\quad m\in\N_0
\end{equation}
we add to the intuitive definition of $\SC$-Fredholm the requirement
\begin{equation}\label{eq:level-reg}
     \boxed{
          T(E\setminus E_m)\CAP F_m=\emptyset,\quad m\in\N_0
     }
\end{equation}
of\index{level!regularity}\index{regularity!level --}\index{regularizing}
\textbf{level regularity}. So $e\in E$ and $Te\in F_m$
together imply $e\in E_m$.
The next exercise shows that (\ref{eq:level-reg}) implies all three
conditions~(\ref{eq:kernel-scale-cond}--\ref{eq:imag=intersection}).

\begin{exercise}\label{exc:naive-Fredholm}
To the na\text{\"i}ve notion of  $\SC$-Fredholm
add~(\ref{eq:level-reg}) to prove 
\begin{itemize}
\item[a)]
  Constancy of kernel scale~(\ref{eq:kernel-scale-cond}) holds true.
  (Thus $K$ is an $\SC$-subspace of finite dimension and therefore
  $K$ admits an $\SC$-complement $X$.)
\item[b)] 
  Equality of scales~(\ref{eq:imag=intersection}) holds true.
  (That is the image scale $Y$ with levels $T(X_m)$ equals the intersection
  scale with levels $Y\CAP F_m$.)
\item[c)]
  The image scale is a Banach subscale of $F$ generated by its top
  level $Y=\im T$.
  (That is $Y$ is an $\SC$-subspace. So~(\ref{eq:image-scale-cond}) is
  satisfied by Proposition~\ref{prop:fin-codiml-sc-subspaces}
  and $C\subset F_\infty$ by Lemma~\ref{le:fin-diml-sc-subspaces}.)
\item[d)]
  Each level operator as a map $T_m\colon X_m\to Y_m$
  is an isomorphism.
\end{itemize}
[Hints: 
a) trivial.
b) ``$\subset$" easy, ``$\supset$" trivial. 
c) It only remains to show density of $Y_\infty$ in $Y$.
    By a) $X$ generates a Banach subscale, so $X_\infty$ is dense in $X$.
    Show that $Y_\infty=T(X_\infty)$, then apply
    Lemma~\ref{le:cont-surj-dense}.
d) Equality~(\ref{eq:imag=intersection}).]
\end{exercise}

\begin{definition}[$\SC$-Fredholm operator -- via level operators]
\label{def}
An \textbf{\boldmath\Index{sc-Fredholm operator}} 
is a level preserving linear operator $T\colon E\to F$ between Banach scales
all of whose level operators $T_m\colon E_m\to F_m$ are Fredholm
and which satisfies the level regularity condition~(\ref{eq:level-reg}).
\end{definition}

\begin{exercise}\label{def:sc-Fredholm-II}
Show that Definitions~\ref{def:sc-Fredholm}
and~\ref{def:sc-Fredholm-II} are equivalent.
\end{exercise}

\section{Scale differentiability}\label{sec:scale-differentiability}

Motivated by properties of the shift map,
see our discussion in the introduction around~(\ref{eq:shift-map-0}),
the notion of scale differentiability was introduced 
by~\citet*{Hofer:2007a}; see also~\citet{Hofer:2010b,Hofer:2017a}.

\subsection*{Scale continuous maps -- class \boldmath$\SC^0$}

An open subset $U$ of an $\SC$-Banach space $E$ induces 
via level-wise intersection a nested
sequence $U^{\cap E}$ of open subsets $U_m=U\cap E_m$ of the corresponding
Banach spaces $E_m$; cf. Lemma~\ref{le:open-subset-scale}.

\begin{definition}\label{def_partial-quad}
A \textbf{\Index{partial quadrant}} in a Banach scale $E$ is a closed
convex subset $C$ of $E$ such that there is an $\SC$-isomorphism
$T\colon E\to\R^n\oplus W$, for some $n$ and some $\SC$-Banach space $W$,
satisfying $T(C)=[0,\infty)^n\oplus W$.
Note that $C$ necessarily contains the origin $0$ of $E$.
\\
An\index{sc-triple}\index{$(U,C,E)$ $\SC$-triple}
\textbf{\boldmath$\SC$-triple} $(U,C,E)$ consists of a Banach scale
$E$, a partial quadrant $C\subset E$, and a relatively open subset
$U\subset E$.
Observe that both $U$ and $C$ inherit
from $E$ nested sequences of subsets whose levels
are the closed subsets $C_m:=C\cap E_m\subset E_m$
and the relatively open subsets $U_m:=U\cap C_m\subset C_m$.
\end{definition}

The notion of partial quadrant is introduced to describe
boundaries and corners. At first reading think of $C=E$, so $U$ is an
open subset of $E$.

\begin{definition}[Scale continuity]
\label{def:sc-continuity}
Let $(U,C,E)$ and $(V,D,F)$ be $\SC$-triples.
A\index{sc-continuous map}
map\index{$sc^0$@$\SC^0$ scale continuous}\index{scale!continuous}
$f\colon U\to V$ is called \textbf{scale continuous} or of
\textbf{\boldmath class $\SC^0$}~if
\begin{itemize}
\item[(i)]
  $f$ is\Index{level!preserving}
  level preserving, that is $f(U_m)\subset V_m$ for
  every $m$, and
\item[(ii)]
  each restriction viewed as a map $f_m:=f|_{U_m}\colon U_m\to V_m$ to
  level $m$ is continuous. The maps $f_m$ are called \textbf{level maps}.
\end{itemize}
\end{definition}

 Let us abbreviate terminology as follows.

\begin{convention}\label{conv:sc-convention}
If we say \textbf{{\boldmath``suppose $f\colon U\to V$ is of class ${\SC}^k$''}}
it\index{$f\colon U\to V$ is of class $\SC^k$}
means that $f$ is an $\SC^k$ map between $\SC$-triples $(U,C,E)$ and
$(V,D,F)$ -- suppose at first reading between $(U,E,E)$ and $(V,F,F)$ ;-)

Given Banach scales $E$ and $F$, an operator $T\colon E\to F$
can have the property of being $\SC$-linear between the
Banach scales $E$ and $F$, i.e. $T\in\Llsc(E,F)$, or it can be
\textbf{continuous and linear in the usual sense} between the Banach spaces $E$ and
$F$, i.e. $T\in\Ll(E,F)$. In the latter case, for extra emphasis, we
often write $E_0$ and $F_0$, instead of $E$ and $F$,
and $T\in\boldmath\Ll(E_0,F_0)$.
\end{convention}

\begin{definition}[Diagonal maps of height $\ell$]
\label{def:sc-minus-maps}
Let $f\colon U\to V$ be an $\SC^0$-map. Pick $\ell\in\N$.
View a level map $f_{m+\ell}$ as a map into the higher level $V_m$
  \begin{equation*}
  \begin{tikzcd}
     &&
       V_m
     \\
     (U^\ell)_m=U_{m+\ell}
     \arrow[rr, "f_{m+\ell}"']
     \arrow[rru, dashed, "f_{m+\ell}^{m}=f|"]
     &&
       V_{m+\ell}
       \arrow[u, hook]
  \end{tikzcd}
  \end{equation*}
to obtain a continuous map $f_{m+\ell}^{m}=f|\colon U_{m+\ell}\to V_m$
given by restriction of $f$ and called a
\textbf{\boldmath diagonal map of height $\ell$}.
For simplicity one usually writes $f\colon U_{m+\ell}\to V_m$
and calls it an \textbf{induced map}.\index{induced!map}
The collection of all diagonal maps of $f$ of height $\ell$
is denoted by 
\[
     f^{-\ell}=f|\colon U^\ell\to V^0
\]
with level maps $(f^{-\ell})_m=f_{m+\ell}^m$. It is of class $\SC^0$,
called the \index{induced!$\SC$-map of height~$\ell$}
\textbf{\boldmath induced $\SC$-map of height~$\ell$}.
If we just say\index{diagonal!map}
\textbf{diagonal map} we mean one of height $1$.
\end{definition}

\subsection*{Continuously scale differentiable maps -- class \boldmath$\SC^1$}


To define scale differentiability let us
introduce the notion of tangent
bundle.\index{sc-Banach space!tangent bundle of --}
The \textbf{\Index{tangent bundle}} of a Banach scale
$E$ is defined as the Banach scale
\[
     TE:=E^1\oplus E^0.
\]
If $A\subset E$ is a subset we denote by
$A^k\subset E^k$, as in Definition~\ref{def:shifted-scale},
the shifted scale of subsets whose levels
are given by $(A^k)_m=A_{k+m}$ where $m\in\N_0$.

\begin{definition}\label{def:sc1}
The\index{tangent bundle!of $\SC$-triple}
\textbf{\boldmath tangent bundle of an $\SC$-triple}
$(U,C,E)$ is the $\SC$-triple \Index{$T(U,C,E):=(TU,TC,TE)$}
where\footnote{
  The symbol \Index{$U^1\oplus E^0$} actually denotes the subset
  $U^1\times E^0$ of the $\SC$-Banach space $E^1\oplus E^0$
  and is just meant to remind us that the ambient Banach space
  is a direct sum.
  }
\[
     TU:=U^1\oplus E^0,\qquad
     TC:=C^1\oplus E^0,\qquad
     TE:=E^1\oplus E^0.
\]
Note that the levels, for instance of $TU$, are given by
\[
     (TU)_m=U_{m+1}\oplus E_m.
\]
\end{definition}

\begin{definition}[Scale differentiability]
\label{def:sc-differentiability}
Suppose $f\colon U\to V$ is of class $\SC^0$.
Then $f$ is called \textbf{\Index{continuously scale differentiable}}
or of \textbf{\boldmath class $\SC^1$} if for every point $x$
in the first sublevel $U_1\subset U$ there is a bounded linear operator
\begin{equation}\label{eq:sc-differential}
     Df(x)\in\Ll(E_0,F_0),\quad x\in U_1
\end{equation}
between the top level \emph{Banach spaces},  called
the\index{sc-derivative!of $f$ at $x$}
\textbf{\boldmath$\SC$-derivative of $f$ at $x$}
or\index{sc-derivative!$Df(x)$}
the \textbf{\boldmath$\SC$-linearization}, such
that\index{$Df(x)$ $\SC$-derivative}
the
following three conditions hold.
\begin{labeling}{\texttt{(ptw diff'able)}}
\item[\texttt{(ptw diff'able)}]
  The upmost diagonal map $f\colon U_1\to V_0$
  is\index{\texttt{(ptw diff)}}
  \emph{pointwise}\index{\texttt{(extension)}}
  differentiable\index{\texttt{(}\unboldmath$Tf$ \texttt{is} ${\SC}^0$\texttt{)}}
  in the usual sense, notation
  $df(x)\in\Ll(E_1,F_0)$; see Definition~\ref{def:Frechet-differential}.
\item[\texttt{(extension)}]
  The $\SC$-derivative $Df(x)$ extends $df(x)$
  from $E_1$ to $E_0$, i.e. the diagram
  \begin{equation}\label{eq:sc-deriv-extends}
  \begin{tikzcd} 
     E_0\arrow[rr, dashed, "Df(x)"]
     &&
       F_0
     \\
     E_1
     \arrow[u, hook, "I_1"]
     \arrow[rru, "{df(x)\,\in\,\Ll(E_1,F_0),\;x\in U_{1}}"']
     &&
  \end{tikzcd} 
  \end{equation}
  commutes.\footnote{
    So $df(x)\colon E_1\to F_0$ is compact. This implies $f\in C^1(U_1,V_0)$;
    see Lemma~\ref{le:f_1-is-C1} (ii).
    }
  Motivated by the diagram let us call $df(x)$ a\index{derivative!diagonal --}
  \textbf{diagonal derivative}\index{diagonal!derivative}
  if the level index between domain and target drops by 1.
\item[{\texttt{($Tf$ is $\text{sc}^0$)}}]
  The \textbf{\Index{tangent map}} $Tf\colon TU\to TV$ defined by
  \[
     Tf(x,\xi):=\left(f(x),Df(x)\xi\right)
  \]
  for $(x,\xi)\in U^1\oplus E^0=TU$
  is\index{$Tf\colon TU\to TV$ tangent map}
  of class $\SC^0$.
\end{labeling}
\end{definition}

\begin{remark}[A continuity property of $Df$]\label{rem:axiom-Tf}
Suppose $f\in\SC^1(U,V)$.\index{$sc^1(U,V)$@$\SC^1(U,V)$}
By $\SC^0$ there are continuous level maps $f_m=F|\colon U_m\to V_m$,
whereas the axiom \texttt{($Tf$ is $\text{sc}^0$)} requires continuous level maps
\begin{equation}\label{eq:axiom-TF-sc0}
\begin{split}
     (Tf)_m\colon U_{m+1}\oplus E_m
   &\to
     V_{m+1}\oplus F_m
     \\
     (x,\xi)
   &\mapsto
     \left( f(x), Df(x)\xi\right)
\end{split}
\end{equation}
In particular, for each $m\in\N_0$ the second component map
\begin{equation}\label{eq:comp-op-top}
     \Phi\colon U_{m+1}\oplus E_m\to F_m,\quad
     (x,\xi)\mapsto Df(x)\xi
\end{equation}
still denoted by $Df$, is continuous whenever $f\in\SC^1(U,V)$.
It is linear in $\xi$.
\end{remark}

\begin{remark}[Continuity in compact-open, but \emph{not in norm}, topology]
\label{rem:noncont-norm-top}
The compact-open and the norm topologies
are reviewed in great detail in Appendix~\ref{sec:Ana-TVS}.
Continuity of the map $\Phi=Df$ in~(\ref{eq:comp-op-top})
means that
\[
     Df\in C^0\left(U_{m+1},\Llco(E_{{\color{cyan}m}},F_{{\color{cyan}m}})\right)
\]
is continuous whenever the target carries the compact-open topology.
Let's refer to this\index{horizontal continuity in compact-open topology}
as\index{continuity!horizontal -- in compact-open topology}
\textbf{{\color{cyan}horizontal} continuity in the compact-open topology},
because both $E_{{\color{cyan}m}}$ and $F_{{\color{cyan}m}}$ are of
the {\color{cyan} same level $m$}.
It is crucial that the domain has better regularity $m+1$,
see Lemma~\ref{le:sc-deriv-reg-preserving}.
\\
In general, continuity is not true in the norm topology,
that is with respect to $\Ll(E_{{\color{cyan}m}},F_{{\color{cyan}m}})$.
The map which prompted the discovery of scale calculus,
the shift map~(\ref{eq:shift-map-0}), provides a counterexample to
continuity of
\[
     Df:U_{m+1} \to\Ll(E_{{\color{cyan}m}},F_{{\color{cyan}m}})
\]
for details see e.g.~\citet[\S 2]{Frauenfelder:2018a}.

Things improve drastically if instead of $E_m$ one starts at better regularity
$E_{m{\color{brown} +1}}$, see Lemma~\ref{le:sc-deriv-reg-preserving}.
Now the linear map $Df(x): E_{m{\color{brown} +1}}\to F_m$
changes level, we say ``is diagonal'', and
one has $Df=df$ and norm continuity,~that~is
\[
     Df=df\in C^0\left(U_{m+1},\Ll(E_{m{\color{brown} +1}},F_m)\right)
\]
referred to\index{diagonal!continuity in norm}
as\index{continuity!diagonal -- in norm}
\textbf{{\color{brown}diagonal} continuity in the norm topology}.
\end{remark}

\begin{remark}[Uniqueness of extension]\label{rem:uniqueness-BLT}
Since $E_1$ is dense in the Banach space $E_0$
the scale derivative $Df(x)$ is uniquely determined by the
requirement~(\ref{eq:sc-deriv-extends}) to restrict along $E_1$ to $df(x)$.
However, observe that the mere requirement that $f\colon U_1\to F_0$ is
pointwise differentiable does not guarantee that a bounded extension
of $df(x)\in\Ll(E_1,F_0)$ from $E_1$ to $E_0$ exists.
Here the B.L.T. Theorem~\ref{thm:BLT} does not help,
because the completion of $E_1$ is $E_1$ itself..
Existence of such an extension
is part of the definition of $\SCo$.
\end{remark}

\begin{exercise}\label{exc:const-sc1-C1}
Show that for constant Banach scales $E$ and $F$, in other words,
for finite dimensional normed spaces equipped with the
constant scale structure, a map $f\colon U\to V$ is of class $\SC^1$
iff it is of class $C^1$.
\end{exercise}

\begin{exercise}
What changes in Exercise~\ref{exc:const-sc1-C1}
if $E$ or $F$ are constant?
\end{exercise}

\subsubsection*{Scale derivative \boldmath$Df(x)$ induces only
some level operators}

\begin{lemma}[Level preservation and continuity properties of $Df(x)$]
\label{le:sc-deriv-reg-preserving}
Let $f\colon U\to V$ be of class $\SC^1$ and $m\in\N_0$. Then the following
is true for every point $x\in U$ of \textbf{\Index{regularity}} $m+1$,
that is $x\in U_{m+1}$.
\begin{itemize}
\item[\rm (a)]
  \emph{Existence of level operators down to one level above $x$:}
  That the $\SC$-derivative $Df(x)\in\Ll(E_0,F_0)$ is level
  preserving is guaranteed only for levels $0,\dots,m$.
\item[\rm (b)]
  \emph{Continuity of these level operators:}
  The\index{{$D_\ell f(x):=Df(x)\mid_{E_\ell}\colon E_\ell\to F_\ell$}}
  induced\index{level!operator!of $\SC$-derivative $Df(x)$}
  \textbf{level operators} are bounded linear operators, in symbols
  \[
     D_\ell f(x):=Df(x)|_{E_\ell}\in\Ll(E_\ell,F_\ell),\qquad
     x\in U_{m+1},\quad \ell=0,\dots, m.
  \]
\item[\rm (c)]
  \emph{{\color{cyan}Horizontal} continuity in compact-open topology:}
  By continuity of the
  map $\Phi\colon U_{m+1}\oplus E_m\to F_m$ in~(\ref{eq:comp-op-top}),
  still denoted by $D_mf$ or even $Df$, it holds that
  $Df\in C^0(U_{m+1},\Llco(E_{{\color{cyan}m}},F_{{\color{cyan}m}}))$.
  By\index{continuity!w.r.t. compact-open topology}
  linearity\index{topology!compact-open --}\index{compact-open topology}
  of $Df(x)$ this simply means that along any convergent sequence
  $x_\nu\to x$ in $U_{k+1}$ the scale derivative applied
  to any individual $\xi\in E_m$ converges, that is
  \begin{equation}\label{eq:cont-comp-open}
     \lim_{\nu\to \infty}
     \Norm{Df(x_\nu)\xi-Df(x)\xi}_{F_{{\color{cyan}m}}}=0,\qquad\xi\in E_{{\color{cyan}m}}.
  \end{equation}
\item[\rm (d)]
 \emph{{\color{brown}Diagonal} continuity in norm topology:}
  The $\SC$-derivative as a map 
  \[
     Df\colon U_{m+1}\to\Ll(E_{m{\color{brown}+1}},F_m),\quad
     x\mapsto Df(x)
  \]
  is continuous. Actually $Df=df\colon  U_{m+1}\to\Ll(E_{m{\color{brown}+1}},F_m)$;
  see~(\ref{eq:diag-m-deriv-extends}).
\end{itemize}
\end{lemma}

\begin{corollary}\label{cor:sc-deriv_is_sc-op_smooth-pts}
At smooth points $\SC$-derivatives
are $\SC$-operators, that is
\[
     Df(x)\in\Llsc(E,F),\qquad x\in U_\infty.
\]
\end{corollary}

\begin{proof}[Proof of Lemma~\ref{le:sc-deriv-reg-preserving}]
Let $f\colon U\to V$ be of class $\SC^1$.
Pick $x\in U_{m+1}$ and $\ell\in\{0,\dots,m\}$.
Hence $x\in U_{\ell+1}$ and so $(x,\xi)\in U_{\ell+1}\oplus
E_\ell=(TU)_\ell$ for $\xi\in E_\ell$.
The axiom \texttt{($Tf$ is $\text{sc}^0$)}
means by definition of $\SC^0$ that every level map
$(Tf)_\ell$, see~(\ref{eq:axiom-TF-sc0}), is continuous.
In particular, for fixed $x\in U_{m+1}\subset U_{\ell+1}$ the map
between second components
$E_\ell\to F_\ell$, $\xi\mapsto Df(x)\xi$, is continuous. This proves~(a--b).
Since $\Phi$ in~(\ref{eq:comp-op-top}) is continuous
so is $\Phi(\cdot,\xi)\colon U_{m+1}\to F_m$
for each fixed $\xi\in E_m$. This proves~(c).
Part~(d) holds true by Proposition~\ref {prop:cont-BS}~c)
for the above map $\Phi$ and the compact operator
$S:=I_{m+1}\colon E_{m+1}\INTO E_m$.
\end{proof}

\subsubsection*{\boldmath Characterization of $\SC^1$ in terms of the
scale derivative $Df(x)$}

The next lemma and proof are taken from~\citet{Frauenfelder:2018a}.

\begin{lemma}[{Characterization of $\SC^1$ in terms of the
$\SC$-derivative}]
\label{le:charact_sc-1_by_Df} \mbox{ } \\
Let $f\colon U\to V$ be $\SC^0$.
Then $f$ is $\SC^1$ iff the following conditions hold:
\begin{labeling}{\rm \texttt{(level operators)}}
\item[\rm \texttt{(ptw diff'able)}]
  (i)  The restriction $f\colon U_1\to F_0$, that is the top diagonal map, is pointwise
  differentiable in the usual sense.\index{\texttt{(ptw diff'able)}}
\item[\rm \texttt{(extension)}]
  (ii)  Its derivative $df(x)\in\Ll(E_1, F_0)$ at any $x\in U_1$
  has\index{\texttt{(extension)}}
  a continuous extension $Df(x)\colon E_0\to F_0$.
\item[\rm \texttt{(level operators)}]
  (iii)   The continuous extension $Df(x)\colon E_0\to F_0$ restricts,
  for\index{\texttt{(level operators)}}
  all levels $m\in\N_0$ and base points $x\in U_{m+1}$, to continuous
  linear operators
  (called \textbf{\boldmath level operators})\index{level!operator}
  \[
     D_m f(x):=Df(x)|_{E_m}\colon E_m\to F_m
  \]
  such that the corresponding maps
  \[
     Df|_{U_{m+1}\oplus E_m}\colon  U_{m+1}\oplus E_m\to F_m
  \]
  are continuous.
\end{labeling}
\end{lemma}

\begin{proof}
'$\Rightarrow$'
Suppose $f$ is $\SC^1$. Then statements (i) and (ii) are obvious
and in statement (iii) the restriction assertion 
holds by Lemma~\ref{le:sc-deriv-reg-preserving} part~(b),
the continuity assertion by part~(c).

'$\Leftarrow$'
Suppose $f$ is $\SC^0$ and satisfies (i--iii).
It remains to show that the tangent map is $\SC^0$, namely,
a) level preserving and b) admitting continuous level maps.
a) To see that $Tf$ maps $(TU)_m$ to $(TV)_m$ for every $m\in\N_0$,
pick $(x,\xi)\in(TU)_m=U_{m+1}\oplus E_m$.
Since $f$ is $\SC^0$ we have that $f(x)\in V_{m+1}$.
By (iii) we have that $Df(x)\xi\in F_m$. Hence
\[
     Tf(x,\xi)
     =\left( f(x), Df(x)\xi \right)
     \in V_{m+1}\oplus F_m=(TV)_m .
\]
b) To see that $TF$ as a map $Tf|_{(TU)_m}\colon (TU)_m\to(TV)_m$
is continuous assume $(x_\nu,\xi_\nu)\in (TU)_m=U_{m+1}\oplus E_m$
is a sequence which converges to $(x,\xi)\in (TU)_m$.
Because $f$ is $\SC^0$, it follows that
\[
     \lim_{\nu\to\infty} f(x_\nu) =f(x).
\]
Continuity of $Df$ provided by (iii) guarantees that
\[
     \lim_{\nu\to\infty} Df(x_\nu)\xi_\nu =Df(x)\xi.
\]
Therefore
\[
    \lim_{\nu\to\infty} Tf(x_\nu,\xi_\nu)
     =\lim_{\nu\to\infty} \left( f(x_\nu), Df(x_\nu)\xi_\nu\right)
     =\left( f(x), Df(x)\xi\right)
     =Tf(x,h).
\]
This proves continuity b) and hence the lemma holds.
\end{proof}

\subsection*{Higher scale differentiability -- class \boldmath$\SC^k$}

For $k\ge 2$ one defines higher continuous scale differentiability
$\SC^k$ recursively as follows.
In the definition of $\SC^1$ one requires a map $f\colon U\to V$ between
open subsets of Banach scales $E$ and $F$ to be $\SC^0$ and then
defines a tangent map $F:=Tf\colon TU\to TV$, again between open subsets
of Banach scales $TE$ and $TF$, which among other things is required to
be $\SC^0$, too.
If the map $F$ itself is of class $\SC^1$, that is if among other things
$TF=TTf\colon TTU\to TTV$ is of class $\SC^0$, one says that
$f$ is of class $\SC^2$, and so on.

\begin{definition}[Higher scale differentiability]\label{def:sc^k}
An $\SC^1$-map $f\colon U\to V$ is \textbf{of class \boldmath$\SC^k$}
if\index{$sc^k$@$\SC^k$ scale differentiability}
and only if its tangent map $Tf\colon TU\to TV$ is $\SC^{k-1}$.
It is called \textbf{\Index{sc-smooth}}, or of
\textbf{class \boldmath$\SC^\infty$}, if
it is of class $\SC^k$ for every $k\in\N$.
\end{definition}

An $\SC^k$-map has iterated tangent maps as follows.
Recursively one defines the\index{iterated!tangent bundle}
\textbf{iterated tangent bundle}
as\index{tangent bundle!iterated --}
\[
     T^{k+1}U:=T(T^kU).
\]
Let us consider the example $T^2 U$.
Recall that for an open subset $U\subset E$ of a Banach scale
we set $TU:=U^1\oplus E^0$. Now consider the open
subset $TU$ of the Banach scale $TE:=E^1\oplus E^0$ to obtain that
\begin{equation*}
\begin{split}
     T^2U:=T(TU):
   &=(TU)^1\oplus (TE)^0\\
   &=\left(U^1\oplus E^0\right)^1\oplus \left(E^1\oplus E^0\right)^0\\
   &=U^2\oplus E^1\oplus E^1\oplus E^0.
\end{split}
\end{equation*}
For $f$ of class $\SC^k$ define\index{iterated!tangent map}
its\index{tangent map!iterated --}
\textbf{iterated tangent map} $T^kf\colon T^kU\to T^kV$
recursively as
\[
     T^kf:=T(T^{k-1} f).
\]
For example
\[
     T^2f\colon  U^2\oplus E^1\oplus E^1\oplus E^0
     \to V^2\oplus F^1\oplus F^1\oplus F^0
\]
is (as shown in the proof of Lemma~\ref{le:charact_sc-2_by_Df} below) given by
\begin{equation}\label{eq:T^2f}
\begin{split}
     T^2f(x,\xi,\hat x,\hat\xi)
   &=\left(Tf(x,\xi),D(Tf)|_{(x,\xi)}(\hat x,\hat \xi)\right)\\
   &
     =\Bigl(
     \underbrace{f(x),Df(x)\xi}_{=:Tf(x,\xi)}
     \, ,
     \underbrace{Df(x)\hat x, D^2f(x)(\xi,\hat x)+Df(x)\hat \xi}
     _{=: D(Tf)_{(x,\xi)}(\hat{x},\hat{\xi})}
     \Bigr).
\end{split}
\end{equation}
Here $D^2f$ is the $\SC$-Hessian of $f$ which we introduce next.
The following lemma and proof are taken from~\citet{Frauenfelder:2018a}.

\begin{lemma}[Characterization of $\SC^2$ in terms of
the $\SC$-derivative]
\label{le:charact_sc-2_by_Df} \mbox{} \\
Let $f\colon U\to V$ be $\SC^1$.
Then $f$ is $\SC^2$ iff the following conditions hold:
\begin{itemize}
\item[\rm (a)]
  The restriction $f\colon U_2\to V_0$, that is the top diagonal map of height two,  
  is pointwise twice differentiable in the usual sense.
\item[\rm (b)]
  Its second derivative $d^2f(x)\in\Ll(E_2\oplus E_2,F_0)$ at any $x\in U_2$
  has a continuous extension $D^2f(x)\colon E_1\oplus E_1\to F_0$,
  the \textbf{\boldmath\Index{sc-Hessian}~of~$f$~at~$x$}.
\item[\rm (c)]
  The continuous extension $D^2f(x)\colon E_1\oplus E_1\to F_0$ restricts,
  for all $m\in\N_0$ and $x\in U_{m+2}$, to continuous bilinear maps
  \[
     D^2_mf(x):=D^2f(x)|_{E_{m+1}\oplus E_{m+1}}\colon E_{m+1}\oplus E_{m+1}\to F_{m}
  \] 
  such that the corresponding maps
  \[
     D^2_mf\colon  U_{m+2}\oplus E_{m+1}\oplus E_{m+1}\to F_{m},\quad
     (x,\xi_1,\xi_2)\mapsto D^2f(x) (\xi_1,\xi_2)
  \]
  are continuous. 
\end{itemize}
\end{lemma}

\begin{proof}
'$\Leftarrow$'
Suppose $f \colon U \to F$ is $\mathrm{sc}^1$ and satisfies the
three conditions~(a-c) of the Lemma. We need to show that $f$ is
$\mathrm{sc}^2$ (meaning by definition that $Tf\in\mathrm{sc}^1$).
Since $f$ is $\mathrm{sc}^1$ we have a well defined
tangent map
$$
     Tf \colon TU=U^1 \oplus E^0 \to TF=F^1 \oplus F^0,\quad
     (x,\xi )\mapsto \left( f(x),Df(x)\xi \right),
$$
of class $\mathrm{sc}^0$. Suppose that
$$
     (x,\xi ) \in (TU)_1=U_2 \oplus E_1.
$$
Hypotheses~(a) and~(b) guarantee that the linear map
$$
     D(Tf)(x,\xi ) \colon (TE)_0=E_1 \oplus E_0 \to (TF)_0=F_1 \oplus F_0
$$
defined for $(\hat{x},\hat{\xi }) \in E_1 \oplus E_0=(TE)_0$ by
$$
     D(Tf)_{(x,\xi )}(\hat{x},\hat{\xi })
     :=\big(Df(x)\hat{x},D^2f(x)(\xi ,\hat{x})+Df(x){\color{cyan}\hat{\xi }}\big).
$$
is well defined and bounded. To see that this map is the
$\SC$-derivative of $Tf$, see~(\ref{eq:sc-differential}), we need to
check the three axioms in the definition of scale differentiability for $Tf$.
Concerning the first two axioms we need to investigate
differentiability of the 'diagonal map', i.e. the restriction of
$Tf\colon (TU)_0\to (TF)_0$ to $(TU)_1$. It suffices to show that
$$ 
     \lim_{\norm{(\hat{x},\hat{\xi })}_{(TE)_1} \to 0}
     \frac{\norm{Tf(x+\hat{x},\xi +\hat{\xi })-Tf(x,\xi )-D(Tf)_{(x,\xi )}(\hat{x},\hat{\xi })}_{(TF)_0}}
     {\norm{(\hat{x}, \hat{\xi })}_{(TE)_1}}
     =0.
$$
Since we already know that the first component $f$
of $Tf$ is $\mathrm{sc}^1$ it suffices to check the second component
and show that
\begin{equation}\label{eq:lim}
\begin{split}
     \lim_{\norm{\hat{x}}_2+\norm{\hat{\xi }}_1\to 0}
     \frac{\norm{Df(x+\hat{x})(\xi +\hat{\xi })-Df(x)(\xi +{\color{cyan} \hat{\xi }})
     -D^2 f(x)(\xi ,\hat{x})}_0}
     {\norm{\hat{x}}_2+\norm{\hat{\xi }}_1}
     =0.
\end{split}
\end{equation}
We estimate
\begin{eqnarray}\label{est}
   && \nonumber
     \frac{\norm{
        Df(x+\hat{x})(\xi +\hat{\xi })-Df(x)(\xi +{\color{cyan} \hat{\xi }})
        -D^2 f(x)(\xi ,\hat{x})
     }_0}
     {\norm{\hat{x}}_2+\norm{\hat{\xi }}_1}
     \\
   &\le&
      \frac{\norm{Df(x+\hat{x})\hat{\xi }-Df(x) {\color{cyan} \hat{\xi }}}_0}
     {\norm{\hat{x}}_2}
     \\\nonumber
   &&
     +\frac{\norm{Df(x+\hat{x})\xi -Df(x)\xi -D^2 f(x)(\xi ,\hat{x})}_0}
     {\norm{\hat{x}}_2} .
\end{eqnarray}
Because $D^2 f \colon U_2 \oplus E_1 \oplus E_1 \to F_0$ is continuous
by hypothesis~(b) there exists an open neighborhood $V$ of
$x$ in $U_2$ and $\delta>0$ such that for every $y \in V$ and every
$v$ and $w$ in $B_\delta$, namely the $\delta$-ball around the origin
of $E_1$, it holds 
$$
     \norm{D^2 f(y)(v,w)}_0 \leq 1.
$$
By bilinearity of $D^2 f(y)$ for any $v,w \in E_1$ we get the estimate
\begin{equation}\label{uni}
     \norm{D^2 f(y)(v,w)}_0 \leq \frac{\norm{w}_1\norm{v}_1}{\delta^2}
\end{equation}
at each $y\in V$.
We can assume without loss of generality that $V$ is convex. We
rewrite the first term in (\ref{est}) as follows
\begin{equation}\label{eq:lim22}
     \frac{1}{\norm{\hat{x}}_2}
     \norm{Df(x+\hat{x})\hat{\xi }-Df(x)\hat{\xi }}_0
     =\biggl\|\int_0^1 D^2
        f(x+t\hat{x})\Bigl(\hat{\xi },\frac{\hat{x}}{||\hat{x}||_2}\Bigr)
     dt\biggr\|_0.
\end{equation}
From uniform boundedness (\ref{uni}) we conclude that
$$
     \lim_{\norm{\hat{x}}_2+\norm{\hat{\xi }}_1\to 0}
     \frac{1}{\norm{\hat{x}}_2}
     \norm{Df(x+\hat{x})\hat{\xi }-Df(x)\hat{\xi }}_0
     \le\lim_{\norm{\hat{\xi }}_1\to 0}
     \frac{c \norm{\hat{\xi }}_1}{\delta^2}
     =0
$$
where $c\ge 1$ is a bound for the linear inclusion
$E_2\INTO E_1$, so $\norm{\frac{\hat x}{\norm{\hat x}_2}}_1\le c$.
Hence in view of (\ref{est}) in order to show (\ref{eq:lim}) we are
left with showing
\begin{equation}\label{eq:lim2}
     \lim_{\norm{\hat{x}}_2\to 0}
    \frac{1}{\norm{\hat{x}}_2}
     \norm{Df(x+\hat{x})\xi -Df(x)\xi -D^2f(x)(\xi ,\hat{x})}_0=0.
\end{equation}
Fix a constant $\kappa\ge 1/\delta^2$ where $\delta$ is the constant
in~(\ref{uni}). Now choose $\epsilon>0$. By taking advantage of the
fact that $E_2$ is dense in $E_1$ we can choose
$$
     \xi ' \in E_2,\qquad
     \norm{\xi -\xi '}_1 \leq \frac{\epsilon}{3 \kappa c}.
 $$
Choose $W \subset V$ a convex open neighborhood of $x$ with the
property that for every $x+\hat{x} \in W$ it holds that
$$
   \frac{1}{\norm{\hat{x}}_2}
     \norm{df(x+\hat{x})\xi '-df(x)\xi '-d^2f(x)(\xi ',\hat{x})}_0
     \leq \frac{\epsilon}{3}.
$$
Suppose that $x+\hat{x} \in W$.
We are now ready to estimate
\begin{eqnarray*}
   &&
   \frac{1}{\norm{\hat{x}}_2}
     \norm{Df(x+\hat{x})\xi -Df(x)\xi -D^2f(x)(\xi ,\hat{x})}_0\\
   &\leq&
   \frac{1}{\norm{\hat{x}}_2}
     \norm{df(x+\hat{x})\xi '-df(x)\xi '-d^2f(x)(\xi ',\hat{x})}_0\\
   &&
     +\biggl\|\int_0^1 D^2
      f(x+t\hat{x}) \Bigl(\xi -\xi ',\frac{\hat{x}}{||\hat{x}||_2}\Bigr)dt
     \biggr\|_0
     +\Bigl\| D^2 f(x) \Bigl(\xi -\xi ',\frac{\hat{x}}{||\hat{x}||_2}\Bigr) \Bigr\|_0\\
     &\leq&
     \epsilon.
\end{eqnarray*}
To obtain the first inequality we wrote each of the three terms $\xi $
in line one in the form $\xi =\xi ' +(\xi -\xi ')$, we used that $df=Df$
for diagonal restrictions of $f$, and we used formula~(\ref{eq:lim22})
for $\hat \xi =\xi -\xi '$. The second inequality uses, in particular,
the estimate~(\ref{uni}) on both $D^2f$ terms.
This proves (\ref{eq:lim2}) and therefore the first two axioms of
scale differentiability of $Tf$.

It remains to prove axiom three, namely that the tangent map of
$Tf$, i.e.
$$
     T^2 f=(Tf, D(Tf)) \colon T^2 U=(TU)^1\oplus TE
     \to T^2 F=(TF)^1\oplus TF
$$
is $\mathrm{sc}^0$: the map $T^2 f$ must be level preserving
and the corresponding level maps
$$
     (T^2U)_m
     =U_{m+2}\oplus E_{m+1}\oplus E_{m+1}\oplus E_m
     \to 
     F_{m+2}\oplus F_{m+1}\oplus F_{m+1}\oplus F_m
$$
given by formula~(\ref{eq:T^2f}) must be continuous for all $m\in\N_0$.
For  the $Tf$ part both assertions are true, because $Tf\in\SCo$.
Concerning the $D(Tf)$ part there are three terms to be checked.
Since $Tf\in\SCo$ part~(iii) of Lemma~\ref{le:charact_sc-1_by_Df}
applies and asserts that term one exists as a map $Df\colon U_{m+2}\oplus
E_{m+1}\to F_{m+1}$ and is continuous, similarly for the map
$Df\circ(\iota,\Id)\colon U_{m+2}\oplus E_m\to U_{m+1}\oplus E_m\to F_m$ in
term three. Concerning term two use hypothesis~(c)
to see that $D^2f\colon U_{m+2}\oplus E_{m+1}\oplus E_{m+1}\to F_m$
is well defined and continuous.
This finishes the proof of the implication
that under the assumptions~(a-c) of the Lemma $f$ is $\mathrm{sc}^2$.

'$\Rightarrow$'
For the other implication, namely that if $f$ is $\mathrm{sc}^2$
it satisfies the conditions (a-c) of the Lemma, we point
out that by a result of Hofer, Wysocki, and Zehnder
\citet[Prop.\,2.3]{Hofer:2010b} it follows that $f$ is actually of
class $C^2$ as a map $f \colon U_{m+2} \to F_m$ for every $m \in
\mathbb{N}_0$. This in particular implies properties~(a) and~(b).
Property~(c) is straightforward; cf. proof of
Lemma~\ref{le:charact_sc-1_by_Df}~(iii)
based on Lemma~\ref{le:sc-deriv-reg-preserving} parts~(b) and~(c).
This concludes the proof of Lemma~\ref{le:charact_sc-2_by_Df}.
\end{proof}

\begin{exercise}[Symmetry of scale Hessian]\index{scale!Hessian}
Show that the \textbf{scale Hessian}\index{Hessian!scale --}
$\Hess_x f:=D^2f(x)\colon E_1\oplus E_1\to F_0$
is symmetric, that is $\Hess_x f(\xi,\eta)=\Hess_x f(\eta,\xi)$
for all $\xi,\eta\in E_1$.

\vspace{.1cm}\noindent
[Hint: The usual second derivative $d^2f(x)\colon E_2\oplus E_2\to F_0$
is symmetric and $E_2$ is a dense subset of the Banach space $E_1$.]
\end{exercise}

Applying the arguments in the proof of Lemma~\ref{le:charact_sc-2_by_Df}
inductively -- Lemma~\ref{le:charact_sc-1_by_Df} playing the role of
the induction hypothesis -- we obtain

\begin{lemma}[Characterizing $\SC^k$ by higher
$\SC$-derivatives $D^kf(x)$]
\label{le:charact_sc-k_by_Df} 
Let $k\in\N$ and $f\colon U\to V$ be $\SC^{k-1}$.
Then $f$ is $\SC^k$ iff the following conditions~hold:
\begin{itemize}
\item[\rm (i)]
  The restriction $f\colon U_k\to V_0$, that is the top diagonal map of
  height $k$, is pointwise $k$ times differentiable in the usual sense.
\item[\rm (ii)]
  Its $k^{\rm th}$ derivative $d^kf(x)\in\Ll(E_k\oplus\dots\oplus E_k,F_0)$
  at any $x\in U_k$ has a continuous extension
  \[
     D^kf(x)\colon \underbrace{E_{k-1}\oplus\dots\oplus E_{k-1}}_{\text{$k$ times}}
     \to F_0.
  \]
\item[\rm (iii)]
  The continuous extension $D^kf(x)\colon E_{k-1}\oplus\dots\oplus E_{k-1}\to F_0$
  restricts, for all $m\in\N_0$ and $x\in U_{m+k}$, to continuous
  $k$-fold multilinear maps
  \[
     D^k_m f(x):=D^kf(x)\colon 
     \underbrace{E_{k-1+m}\oplus\dots\oplus E_{k-1+m}}_{\text{$k$ times}}
     \to F_{m}
  \] 
  such that the corresponding maps
  \[
     D^kf|_A 
     \colon  A:=U_{k+m}\oplus E_{k-1+m}\oplus\dots\oplus E_{k-1+m}\to F_{m}
  \]
  are continuous.
\end{itemize}
\end{lemma}

\section{Differentiability --  Scale vs Fr\'{e}chet}
\label{sec:props-sc-differentiability}
\sectionmark{Scale versus Fr\'{e}chet}

First we investigate how the new class $\SC^1$ of continuously scale
differentiable maps $f\colon U\to V$ relates to $C^1$ continuous
differentiability in the usual Fr\'{e}chet sense
of all diagonal maps $f\colon U_{m+1}\to V_m$ of height $1$.
Then we investigate how the class $\SC^k$ of higher scale
differentiable maps $f\colon U\to V$  relates 
to $C^\ell$ differentiability of all diagonal maps $f\colon U_{m+\ell}\to V_m$
of height $\ell\in\{0,\dots,k\}$, hence up to height of at most $k$.
For further details see~\citet{Hofer:2010b}.

\subsection*{Maps of class \boldmath$\SC^1$}


\begin{convention}[Topologies]
Given Banach spaces $E_0$ and $F_0$, then $\Ll(E_0,F_0)$ denotes
the vector space of bounded linear maps $T\colon E_0\to F_0$
equipped with the (complete) operator norm
(Section~\ref{sec:lin-op}). By
$
     \Llco(E_0,F_0)
$
we denote\index{$\Llco(E,F)$ compact-open topology}
the same vector space equipped
with the compact-open topology.
\end{convention}

\begin{lemma}[Continuity properties of $Df$ and the diagonal differential~$df$]
\label{le:f_1-is-C1}
Let $f\colon U\to V$ be of class $\SC^1$. Then the following is true.
\begin{itemize}
\item[\rm (i)]
  The map $U_1\oplus E_0\to F_0$, $(x,\xi)\mapsto Df(x)\xi$, is continuous.
\item[\rm (ii)]
  The usual differential $df\colon U_1\to\Ll(E_1,F_0)$  of the diagonal map
  $f\colon U_1\to V_0$ is continuous, in symbols $f\in C^1(U_1,V_0)$.
\item[\rm (iii)]
  Every diagonal map $f\colon U_{m+1}\to V_m$ is of class $C^1$. In
  other words, its differential, the
  so-called\index{differential!diagonal --}\index{diagonal!differential}
  \textbf{diagonal differential} 
  \[
     df\colon U_{m+1}\to\Ll(E_{m+1},F_m)
  \]
  is a continuous map.
\item[\rm (iv)]
  At $x\in U_{m+1}$ the diagonal derivative $df(x)\colon E_{m+1}\to F_m$
  in (iii) extends to $E_m$ and the extension is the restriction
  $Df(x)|_{E_m}\in\Ll(E_m,F_m)$ of the $\SC$-derivative~(\ref{eq:sc-differential});
  cf. Lemma~\ref{le:sc-deriv-reg-preserving}~(b). That is, the diagram
  \begin{equation}\label{eq:diag-m-deriv-extends}
  \begin{tikzcd} 
     E_m\arrow[rrrr, dashed, "{Df(x) |_{E_m}=:D_mf(x)}"]
     &&&&
       F_m
     \\
     E_{m+1}
     \arrow[u, hook, "I_{m+1}"]
     \arrow[rrrru, "{df(x)\,\in\,\Ll(E_{m+1},F_m),\;\;x\in U_{m+1}}"']
     &&&&
  \end{tikzcd} 
  \end{equation}
  commutes. As a map $Df:U_{m+1}\oplus E_m\to F_m$ the level $m$ scale
  derivative is continuous; cf.~(\ref{eq:cont-comp-open}).
\end{itemize}
\end{lemma}

Of course, the lemma could be stated more economically, but
we enlist the assertions in their order of proof.

\begin{proof}
We follow essentially~\citet{Cieliebak:2018a}.
(i)~By assumption $f$ is $\SC^1$, so the induced map $f\colon U_1\to V_0$ is
pointwise differentiable and for every $x\in U_1$ the usual
derivative $df(x)\in\Ll(E_1,F_0)$ extends from $E_1$ to a map
$Df(x)\in\Ll(E_0,F_0)$.
Moreover, by axiom \texttt{($Tf$ is $\text{sc}^0$)} the map
\[
     \varphi\colon U_1\oplus E_0\to F_0,\quad
     (x,\xi)\mapsto Df(x)\xi
\]
is continuous, cf.~(\ref{eq:axiom-TF-sc0}), which is assertion (i).

(ii)~As the inclusion $S:=I_1\colon E_1\INTO E_0$ is compact, continuity
of the map
\[
     U_1\mapsto\Ll(E_1,F_0),\quad
     x\mapsto \varphi(x,S\,\cdot)=Df(x)\,\cdot=df(x)\,\cdot
\]
holds by Proposition~\ref {prop:cont-BS}~c).
We used that $Df(x)=df(x)$ along $E_1$.

(iii+iv) For $m=0$ the assertions are true by~(i) and~(ii) and~(ii)
will be a key input for the present proof, see Step 1 below, that
\begin{itemize}
\item[a)]
  as a map $f\colon U_{m+1}\to V_m$ is of class $C^1$, thereby proving (iii), and 
\item[b)]
  its derivative $df(x)=Df(x)|_{E_{m+1}}\colon E_{m+1}\to F_m$ is
  the $\SC$-derivative $Df(x)\colon E_0\to V_0$ applied to the
  elements of $E_{m+1}$ or, equivalently, the restriction to the dense
  subset $E_{m+1}$ of the
  level operator $D_mf(x)=Df(x)|_{E_m}\in\Ll(E_m,F_m)$ which
  exists by Lemma~\ref{le:sc-deriv-reg-preserving}~(b).
\end{itemize}
By density $E_{m+1}\subset E_m$ part b) shows that
the continuous extension of $df(x)\colon E_{m+1}\to F_m$
to $E_m$ is the level operator $D_mf(x)$.
The yet missing continuity assertion in (iv) holds true
by~(\ref{eq:comp-op-top}).
Step 2 below will prove a) and b) which then completes the proof of (iii+iv).
Step 1 is just a preliminary.

\vspace{.1cm}\noindent
\textbf{Step 1.} Given $x\in U_1$, let $\xi\in E_1$
be sufficiently small such that the image of the map
$\gamma\colon [0,1]\to U_1\subset E_1$, $t\mapsto x+t\xi$, is contained in
$U_1$. Then
\[
     f(x+\xi)-f(x)
     =\int_0^1 \frac{d}{dt} f(x+t\xi)\, dt
     =\int_0^1 Df(x+t\xi)\xi \, dt .
\]

\noindent
\textit{Proof of Step 1.} As $f\in C^1(U_1,V_0)$ by~(ii),
identity one is the integral form of the mean value theorem;
see e.g.~\citet[XIII Thm.\,4.2]{Lang:1993a}.
Identity two holds since
$\frac{d}{dt} f(x+t\xi)=df(x+t\xi)\xi=Df(x+t\xi)\xi$.
Equality two is by~(\ref{eq:sc-deriv-extends}) -- by definition of
$\SC^1$ the scale derivative restricted to $E_1$ is $df(x)$.

\vspace{.1cm}\noindent
In the proof of Step 2 we will use Step 1 for the elements
of the subset $U_{m+1}\subset U_1$ and for small $\xi\in E_{m+1}\subset E_1$.
For such $x$ and $\xi$ the term $f(x+\xi)-f(x)$ even
lies in $F_{m+1}$, since $f$ is level preserving (it is of class
$\SC^0$ by assumption).
However, we shall only estimate the $F_m$ norm, as this
gives us the opportunity to bring in compactness of the inclusion
$S:=I_{m+1}\colon  E_{m+1}\INTO E_m$ on the domain side~of~$f$.

\vspace{.1cm}\noindent
\textbf{Step 2.} $\forall m$ the map $f\colon U_{m+1}\to V_m$ is of class
$C^1$~with~derivative~${\color{brown} df=Df}$.

\vspace{.1cm}\noindent
\textit{Proof of Step 2.} 
Pick $x\in U_{m+1}$ and a non-zero short vector $\xi\in E_{m+1}$ to get
\begin{equation*}
\begin{split}
   &\frac{1}{\Abs{\xi}_{E_{m+1}}}\Abs{f(x+\xi)-f(x)-{\color{brown} Df}(x)\xi}_{F_m}\\
   &=\frac{1}{\Abs{\xi}_{E_{m+1}}}
     \Abs{\int_0^1 \left( Df(x+t\xi)\xi-Df(x)\xi\right) dt}_{F_m}\\
   &\le c\int_0^1\Abs{Df(x+t\xi)\frac{\xi}{\Abs{\xi}_{E_{m+1}}}
     -Df(x)\frac{\xi}{\Abs{\xi}_{E_{m+1}}}}_{F_m} dt\\
   &\le c\int_0^1 \Norm{Df(x+t\xi)-Df(x)}_{\Ll(E_{m+1},F_m)} dt\;\;
     \longrightarrow\;\; 0 \quad\text{, as $\Abs{\xi}_{E_{m+1}}\to 0$.}
\end{split}
\end{equation*}
Here the equality holds by Step 1.
Concerning inequality one note that the path $[0,1]\to F_m$,
$t\mapsto Df(x+t\xi)\xi-Df(x)\xi$, is continuous by
Lemma~\ref{le:sc-deriv-reg-preserving}~(c) since $x+t\xi\in
U_{m+1}$ and $\xi\in E_{m+1}\subset E_m$.
So the map is in $\L^1([0,1],F_m)$, hence
the norm of the integral is less or equal than the integral along the norm;
see e.g.~\citet[VI \S4 (4)]{Lang:1993a}.
Inequality two holds by definition of the operator norm.

We prove convergence to zero. This will
follow from continuity of the map $Df\colon U_{m+1}\to \Ll(E_{m+1},F_m)$,
see Lemma~\ref{le:sc-deriv-reg-preserving}~(d).
However, due to infinite dimension it is not
just some compactness argument: Let $\xi_\nu\to 0$
be any in $E_{m+1}$ convergent sequence. Then the family of
bounded linear operators
\[
     \Ff:=\{Df(x+t\xi_\nu)\mid
     \nu\in\N,t\in[0,1]\}\subset\Ll(E_{m+1},F_m)
\]
generates, for each element $\zeta\in E_{m+1}$, a bounded orbit
\[
     \Ff\zeta:=\{Df(x+t\xi_\nu)\zeta\mid
     \nu\in\N,t\in[0,1]\}\subset B_{R(\zeta)}\subset F_m .
\]
Indeed by continuity of the map $Df\colon U_{m+1}\to \Ll(E_{m+1},F_m)$, as
guaranteed by Lemma~\ref{le:sc-deriv-reg-preserving}~(d), and
convergence $\xi_\nu\to 0$ there is a radius $R=R(\zeta)$ such that
the whole sequence of elements $Df(x+\xi_\nu)\zeta$ of $F_m$ lies in the ball
$B_R\subset F_m$ of radius $R$ and centered at $Df(x)\zeta$.
But by convexity of $B_R$ all segments from the center $Df(x)\zeta$
to $Df(x+\xi_\nu)\zeta$ also lie in $B_R$.
The Banach--Steinhaus Theorem~\ref{thm:Banach-Steinhaus}
then provides a uniform upper bound $c_\Ff$ for the operator norms of
all members of $\Ff$.
Now the constant function $g\equiv 2c_\Ff\colon [0,1]\to [0,\infty)$ is integrable
and dominates ($g\ge \abs{F_\nu}$) each function
\[
     F_\nu(t):=\Norm{Df(x+t\xi_\nu)-Df(x)}_{\Ll(E_{m+1},F_m)}\le c_\Ff + c_\Ff
     ,\quad t\in[0,1].
\]
The pointwise limit $F_\nu(t)\to 0$, as $\nu\to\infty$, is the
constant function $0$ on $[0,1]$, again by continuity of $Df$ and by
continuity of the norm function.
Thus the dominated convergence theorem applies,
see e.g.~\citet[VI Thm.\,5.8]{Lang:1993a},
and yields $\lim_\nu\int F_\nu=\int \lim_\nu F_\nu=\int 0 =0$.
This proves convergence to zero.

It remains to prove continuity of the map
\[
     df\colon U_{m+1}\to\Ll(E_{m+1},F_m),\quad
     x\mapsto df(x)=\Phi(x)S .
\]
Continuity holds by Proposition~\ref {prop:cont-BS}~c)
for the by~(\ref{eq:axiom-TF-sc0}), cf. Lemma~\ref{le:charact_sc-1_by_Df},
continuous map $\Phi\colon U_{m+1}\oplus E_m\to F_m$, $(x,\xi)\mapsto Df(x)\xi$,
and the compact inclusion $S:=I_{m+1}\colon E_{m+1}\to E_m$.
As any $\xi\in E_{m+1}$ lies in $E_1$, one has
\[
     \Phi(x) S\xi:=Df(x) I_{m+1}\xi=df(x)\xi
\]
since the diagram~(\ref{eq:sc-deriv-extends}) commutes.
This proves Lemma~\ref{le:f_1-is-C1}.
\end{proof}

\begin{lemma}[{Characterization of $\SC^1$ via
diagonal maps being of class $C^1$}]
\label{le:charact_sc-1_by_df} 
An $\SC^0$-map $f\colon U\to V$ is of class $\SC^1$ iff
\begin{itemize}
\item[\rm (i)]
  all diagonal maps $f\colon U_{m+1}\to V_m$ are of class $C^1$ and for each of them
\item[\rm (ii)]
  the derivative $df(x)\in\Ll(E_{m+1},F_m)$, at any $x\in U_{m+1}$,
  extends to a continuous linear operator on $E_m$,
  notation $D_mf(x):E_m\to F_m$, and
\item[\rm (iii)]
  the extension as a map
  $D_mf:U_{m+1}\oplus E_m\to F_m$, $(x,\xi)\mapsto D_mf(x)\xi$,
  is continuous; cf.~(\ref{eq:cont-comp-open}).
\end{itemize}
\end{lemma}

\begin{proof}
'$\Rightarrow$' Lemma~\ref{le:f_1-is-C1}.
'$\Leftarrow$' By (i-ii) for $m=0$ the first two axioms of
$\SC^1$ are satisfied. By (ii-iii) for all $m$ the axiom
\texttt{($Tf$ is $\text{sc}^0$)} is also satisfied.
\end{proof}

\begin{remark}
For any map $f\colon U\to V$ of class $\SC^1$
the induced map $f\colon E_\infty\supset U_\infty\to V_\infty\subset F_\infty$
between Fr\'{e}chet spaces is of class $C^1$;
cf.~\citet[Probl.\,5.5]{Cieliebak:2018a}.
\end{remark}

\subsection*{Maps of class \boldmath$\SC^k$}

It is an immediate consequence of 
Lemma~\ref{le:charact_sc-1_by_df},
together with the identity $(TU)^1=T(U^1)$,
that for an $\SC^k$ map $f\colon U\to V$
one can lift both indices equally
and still have an $\SC^k$ map, say $f\colon U^\ell\to V^\ell$.

\begin{lemma}[Lifting indices, {\citet[Prop.\,2.2]{Hofer:2010b}}]
\label{le:lifting-indices}
If\index{lifting indices}\index{indices!lifting --}
$f\colon U\to V$ is an $\SC^k$-map, then the induced map
$f\colon U^1\to V^1$ is also of class $\SC^k$.
\end{lemma}

\begin{proof}
Induction over $k\in\N$. Case $k=1$: This holds true by 
Lemma~\ref{le:charact_sc-1_by_df} which characterizes $\SC^1$ by some
conditions on all\,\footnote{
  saying ``\emph{all} diagonal maps $f\colon U_{m+1}\to V_m$'' refers to the set
  $\{f\colon U_{m+1}\to V_m\}_{m\in\N_0}$
  }
diagonal maps $f\colon U_{m+1}\to V_m$ and
their extensions $D_mf(x)$. Replacing $U,V$ by
$U^1,V^1$ means to simply forgetting the two maps for $m=0$.
\\
Induction step $k\Rightarrow k+1$:
Let $f\colon U\to V$ be $\SC^{k+1}$.
By definition this means that $f$ is $\SC^1$ and $Tf$ is $\SC^{k}$.
So by induction hypothesis applied to $Tf\in\SC^k(TU,TV)$ that same
map just between shifted spaces, namely $Tf\colon (TU)^1\to (TV)^1$,
is as well of class $\SC^k$.
But $(TU)^1=T(U^1)$, 
hence $Tf\colon T(U^1)\to T(V^1)$
is $\SC^k$. 
Note that $f\colon U^1\to V^1$ is also of class $\SC^1$
as a consequence of the case $k=1$
applied to $f\in\SC^1(U,V)$.
But an $\SC^1$ map, say $f\colon U^1\to V^1$,  whose tangent map is $\SC^k$
is of class $\SC^{k+1}$ by Definition~\ref{def:sc^k}.
\end{proof}

\begin{lemma}[Necessary and sufficient conditions for $\SC^k$-smoothness]
\label{le:nec_suff_cond_for_sc_by_Ck} 
Let $U,V$ be relatively open subsets of partial quadrants
in $\SC$-Banach spaces $E,F$.
\begin{labeling}{\rm (Necessary)}
\item[\rm (Necessary)]
  If $f\colon U\to V$ is $\SC^k$, then all diagonal maps $f\colon U_{m+\ell}\to V_m$
  of height $\ell$ are of class $C^\ell$ for all
  heights from $0$ up to $k$.
\item[\rm (Sufficient)]
  Assume that a map $f\colon U\to V$ induces for every level $m\in\N_0$
  and every height $\ell$ between $0$ and $k$
  a diagonal map $f\colon U_{m+\ell}\to V_m$ which, moreover, is of class
  $C^{\ell+1}$.
  Such a map $f\colon U\to V$ is of class $\SC^{k+1}$.
\end{labeling}
\end{lemma}

\begin{proof}[Sketch of proof]
\textbf{Necessary.}
Suppose $f\in\SC^k(U,V)$.
Firstly, it suffices to prove the case $\ell=k$, because an $\SC^k$ map is also
an $\SC^\ell$ map for $\ell\in\{0,\dots,k\}$.
Secondly, it suffices to prove the case $m=0$, namely the

\vspace{.1cm}\noindent
\textbf{Claim.} The map $f\colon U_k\to V_0$ is of class $C^k$.

\vspace{.1cm}\noindent
Given $m\in\N_0$, the claim implies that $f\colon U_{m+k}\to V_m$ is of
class $C^k$ and we are done. Indeed by Lemma~\ref{le:lifting-indices}
the map $f\colon U^m\to V^m$ is also of class $\SC^k$ and for this map the
claim asserts that $f\colon (U^m)_k\to (V^m)_0$ is of class $C^k$.

One proves the claim by induction over $k$. For $k=0$ the map
$f\colon U_0\to V_0$ is $C^0$ as $f\in\SC^0$, for $k=1$ the map
$f\colon U_1\to V_0$ is $C^1$ by Lemma~\ref{le:f_1-is-C1} (ii).

The induction step $k\Rightarrow k+1$ is very similar
in character to the proof of Lemma~\ref{le:f_1-is-C1}~(iii+iv)
just more technical as one is looking at $k$-fold
derivatives, thus $k$-multilinear maps.
For details see~\citet[Prop.\,2.3]{Hofer:2010b}.

\textbf{Sufficient.}
Proof by induction over $k$. Case $k=0$. By assumption there
is for each $m\in\N_0$ a $C^1$ level map $f_m:=f|_{U_m}\colon U_m\to V_m$.
Together with the restriction $i_{m+1}$ of the linear, hence smooth,
embedding $I_{m+1}\colon E_{m+1}\INTO E_m$ one has a commutative diagram
  \begin{equation*}
  \begin{tikzcd} 
     U_m\arrow[rr, "f_m\in C^1"]
     &&
       V_m
     \\
     U_{m+1}
     \arrow[u, hook, "C^\infty\ni i_{m+1}"]
     \arrow[rru, dashed, "{f=f_m\circ i_{m+1}\in C^1}"']
     &&
  \end{tikzcd} 
  \end{equation*}
of $C^1$ maps, in particular, all diagonal maps $f\colon  U_{m+1}\to
V_m$ are $C^1$. By the chain rule one gets the identity
$$
     df(x)=df_m(x)\circ I_{m+1}\in\Ll(E_{m+1},F_m)
$$
for $x\in U_{m+1}$.
The identity also shows that $df_m(x)\in\Ll(E_m,F_m)$ extends $df(x)$ from
$E_{m+1}$ to $E_m$. Thus (i) and (ii) in Lemma~\ref{le:charact_sc-1_by_df}
are satisfied and it remains to check (iii).
But this follows by pre-composing the first variable of the (by the
$C^1$-assumption on $f_m$) continuous
map $U_m\oplus E_m\to F_m$, $(x,\xi)\mapsto df_m(x)\xi$, with the
continuous embedding $U_{m+1}\INTO U_m$.
The step $k\Rightarrow k+1$ is very technical,
see~\citet[Prop.\,2.4]{Hofer:2010b}.
\end{proof}

\section{Chain rule}\label{sec:chain-rule}

A key element of calculus, the chain rule, is also available
in $\SC$-calculus. This is rather surprising given the
fact that the $\SC$-derivative arises by
differentiating the diagonal map $f\colon U_1\to V_0$
thereby loosing one level, so for a composition
one would expect the loss of two levels.
However, using {\color{magenta}\texttt{(compactness)}} of the embeddings
$E_{m+1}\INTO E_m$ one can avoid loosing two levels.

\begin{theorem}[Chain rule~{\citet[Thm.\,2.16]{Hofer:2007a}}]
\label{thm:chain-rule}
Suppose $f\colon U\to V$ and $g\colon V\to W$ are $\SC^1$-maps.
Then the composition $g\circ f\colon U\to W$ is also $\SC^1$ and
\begin{equation*}\label{eq:chain-rule-T}
     T(g\circ f)=Tg \circ Tf.
\end{equation*}
Equivalently, in terms of $\SC$-derivatives it holds that
\begin{equation}\label{eq:chain-rule-D}
     D(g\circ f)|_x\,\xi
     =Dg|_{f(x)}\, Df|_x\,\xi,\quad
     (x,\xi)\in U^1\oplus E^0=TU.
\end{equation}
\end{theorem}

\begin{proof}
The main principles and tools of the proof have been detailed
and referenced in the slightly simpler setting of proving Step~2 in
the proof of Lemma~\ref{le:f_1-is-C1} (iii+iv).
Fix $x\in U_1$. Because $V_1$ is an open neighborhood
of $f(x)$ in the cone $D_1\subset F_1$ and because
the level map $f\colon U_1\to V_1$ is continuous,
there is a radius $\delta>0$ open ball $B_\delta$ in $E_1$
centered at $0$ such that $x+B_\delta$ is contained in $U_1$
and such that the map
\[
     \phi(t,\xi):=tf(x+\xi)+(1-t) f(x) \in V_1
\]
takes values in $V_1$ for all $t\in[0,1]$ and $\xi\in B_\delta$.
Because \underline{$g$ is $\SC^1$}, as a map $g\colon V_1\to W_0$ it is of
class $C^1$ by Lemma~\ref{le:charact_sc-1_by_df}~(i).
Apply the mean value theorem, observing that $\p_t\phi(t,h)=f(x+\xi)-f(x)$,
and {\color{brown} add zero} to obtain
\begin{equation*}
\begin{split}
   &g(f(x+\xi))-g(f(x))-{\color{cyan} Dg|_{f(x)}}Df|_x\xi\\
   &=\int_0^1 Dg|_{\phi(t,\xi)}\left(\p_t\phi(t,\xi)
     -{\color{brown} Df|_x\xi}\right) dt\\
   &\quad+\int_0^1\left({\color{brown} Dg|_{\phi(t,\xi)}}
     -{\color{cyan} Dg|_{f(x)}}\right) Df|_x \xi\, dt.
\end{split}
\end{equation*}

Divide by $\abs{\xi}_{E_1}$, so the first integral becomes
\begin{equation}\label{eq:gyhg687}
     \int_0^1 Dg|_{\phi(t,\xi)}\, h(\xi)\, dt,\qquad
     h(\xi):=\frac{f(x+\xi)-f(x)-Df|_x \xi}{\abs{\xi}_{E_1}}.
\end{equation}
Since \underline{$f$ is $\SC^1$} the restriction
of $Df|_x\colon E_0\to F_0$ to $E_1$ is $df|_x$ whenever $x\in U_1$,
see~(\ref{eq:sc-deriv-extends}), hence $h(\xi)\to 0$ as
$\abs{\xi}_{E_1}\to 0$ by Definition~\ref{def:Frechet-differential} of
the Fr\'{e}chet derivative $df|_x:=df(x)$.
Now $\phi\colon [0,1]\times B_\delta\to V_1$ is continuous
and $\phi(t,\xi)\to f(x)$, as $\abs{\xi}_{E_1}\to 0$, uniformly in $t\in[0,1]$.
Since \underline{$g$ is $\SC^1$} Lemma~\ref{le:charact_sc-1_by_Df}~(iii)
guarantees that the map $V_1\oplus F_0\to G_0$,
$(y,\eta)\mapsto Dg|_y\eta$, is continuous. Thus
$Dg|_{\phi(t,\xi)} h(\xi)\to 0$, as $\abs{\xi}_{E_1}\to 0$, uniformly in $t\in[0,1]$.
So the integral~(\ref{eq:gyhg687}) vanishes in the limit as
$\abs{\xi}_{E_1}\to 0$.

The second integral divided by $\abs{\xi}_{E_1}$ becomes
\begin{equation}\label{eq:gyhg66787}
     \int_0^1 \left( Dg|_{\phi(t,\xi)}-Dg|_{f(x)}\right)
     \frac{Df|_x\xi}{\abs{\xi}_{E_1}} \, dt .
\end{equation}
By {\color{magenta}\texttt{(compactness)}} of the inclusion $E_1\INTO E_0$
and continuity of the $\SC$-derivative $Df|_x\colon E_0\to F_0$ 
the set of all $Df|_x\xi/\abs{\xi}_{E_1}$ with $0\not=\xi\in B_\delta$
has {\color{magenta} compact closure} in $F_0$.\footnote{
  images of compact sets under continuous maps are compact
  }
Since the map $V_1\oplus F_0\to G_0$, $(y,\eta)\mapsto Dg|_y\eta$, is
continuous by Lemma~\ref{le:charact_sc-1_by_Df}~(iii) --
due to \underline{$g$ being $\SC^1$} -- it follows as above that
the integrand in~(\ref{eq:gyhg66787}) converges in $G_0$ to $0$ uniformly in
$t\in[0,1]$, so the integral~(\ref{eq:gyhg66787}) converges in $G_0$ to $0$,
both as $\abs{\xi}_{E_1}\to 0$.
 
This shows that the $\SC^0$ map given by the composition
$g \circ f\colon U\to W$ satisfies the first two axioms
in Definition~\ref{def:sc-differentiability} of $\SC^1$.
Indeed as a map $g\circ f\colon U_1\to V_1\to W_0$ is pointwise differentiable
and at $x\in U_1$ the derivative $d(g \circ f)|_x\colon E_1\to G_0$ has a
continuous extension, namely the composition of bounded linear
operators $Dg|_{f(x)} Df|_x\colon E_0\to G_0$. So by definition
this composition is the $\SC$-derivative $D(g\circ f)|_x$ associated
to $g\circ f$.
Thus $T(g\circ f)=Tg\circ Tf\colon TU\to TW$.
Because both $Tf$ and $Tg$ are $\SC^0$, so is $T(g\circ f)$.
Thus $g\circ f$ satisfies axiom three in Definition~\ref{def:sc-differentiability}.
So $g\circ f$ is $\SC^1$.
\end{proof}

\section{Boundary recognition}\label{sec:boundary-recognition}

Let $C$ be a partial quadrant in an $\SC$-Banach space $E$.
Pick a linear $\SC$-isomorphism $T\colon E\to \R^n\oplus W$
with $T(C)=[0,\infty)^n\oplus W$. For $x\in C$ write
$Tx=(a_1,\dots,a_n,w)\in [0,\infty)^n\oplus W$ and define its
\textbf{\Index{degeneracy index}}\index{index!degeneracy --}
by
\begin{equation}\label{eq:deg-index}
     d_C(x):=\#\{i\in\{1,\dots,n\}\mid a_i=0\} \in\N_0.
\end{equation}
A point $x\in C$ satisfying $d_C(x)=0$ is an interior point of $C$,
a boundary point if $d_C(x)=1$, and a corner point if $d_C(x)\ge 2$.
See Figure~\ref{fig:fig-deg-index-quadrant}.
\begin{figure}
  \centering
  \includegraphics
                             [height=4cm]
                             {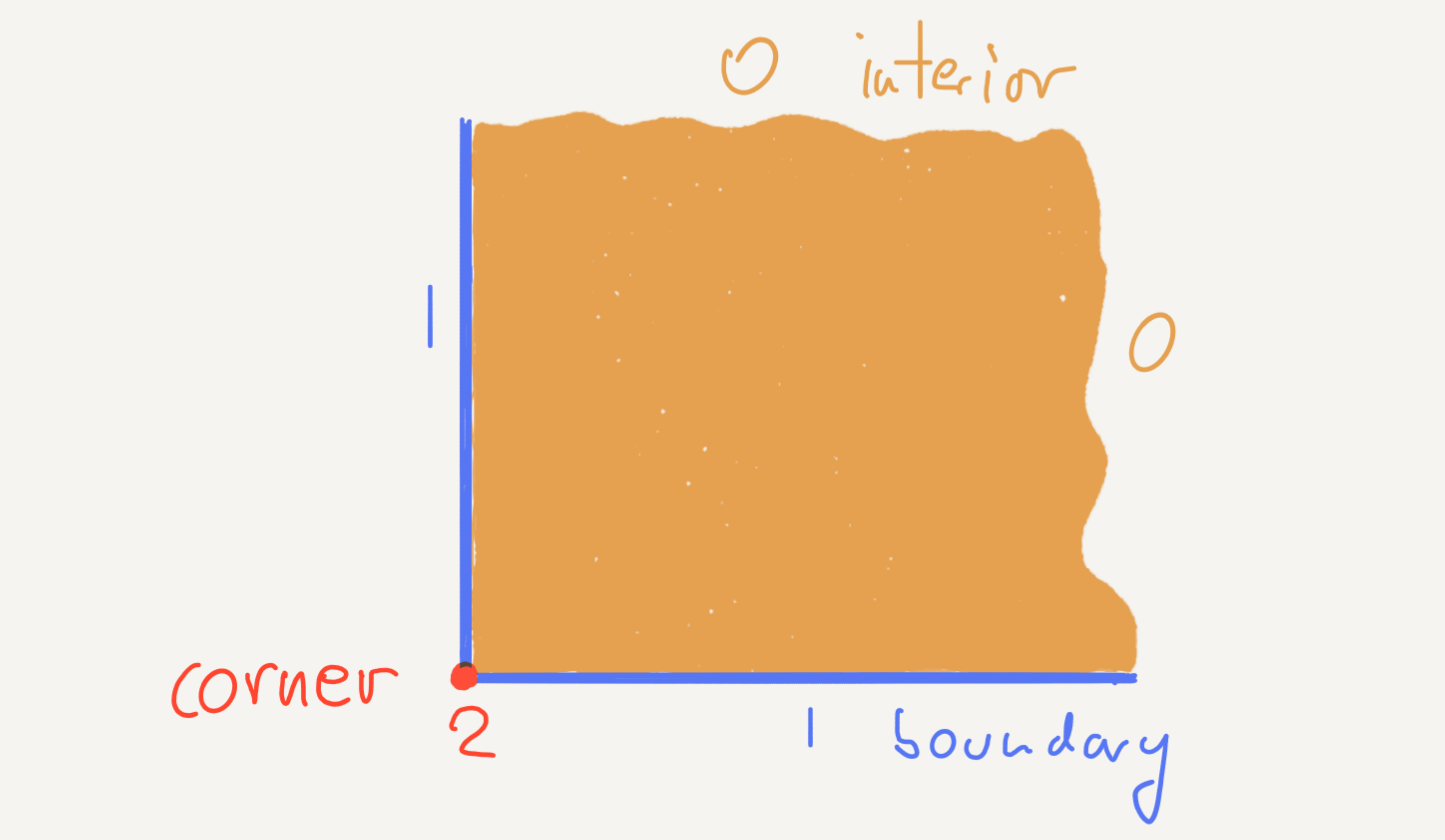}
  \caption{Quadrant $C\subset\R^2$ and points of degeneracy index
                  {\color{red} two}, {\color{cyan} one}, {\color{brown} zero}}
  \label{fig:fig-deg-index-quadrant}
\end{figure}

\begin{exercise}
The degeneracy index $d_C$ does not depend on the choice
of linear $\SC$-isomorphism $T\colon E\to \R^n\oplus W$.
\end{exercise}

\begin{theorem}[Invariance under $\SC^1$-diffeomorphisms]
\label{thm:inv-deg-ind}
Let $(U,C,E)$ and $(V,D,F)$ be partial quadrants. Let
$f\colon U\to V$ be an \textbf{\boldmath$\SC^1$-diffeomorphism},
that\index{sc-diffeomorphism}
is an $\SC^1$-map with an $\SC^1$-inverse, then for every $x\in U$ one
gets equality
\[
     d_C(x)=d_D(f(x)).
\]
\end{theorem}

\begin{proof}
\citet[Thm.\,1.19]{Hofer:2007a}
\end{proof}

\section{Sc-manifolds}\label{sec:sc-manifolds}

The new notion of differentiability of maps between the new linear spaces 
-- $\SC^k$ differentiability of maps between $\SC$-Banach spaces -- 
allows to carry over the new calculus to topological spaces modeled locally
on $\SC$-Banach spaces. This results in a new class of manifolds,
called $\SC^k$-manifolds. Their construction parallels the definition
of $C^k$ Banach manifolds; see Section~\ref{sec:B-mfs}.

To complement Section~\ref{sec:B-mfs} (case $C^k$)
we spell out here the smooth case (case $\SC^\infty$).
Suppose $X$ is a topological space.\index{sc-charts}
An \textbf{\boldmath$\SC$-chart} $(V,\phi,(U,C,E))$ for $X$
consists of an $\SC$-triple $(U,C,E)$ and a homeomorphism
$\phi\colon X\supset V\to U\subset C$ between open subsets.
Two\index{sc-charts!sc-compatible --}
$\SC$-charts are called \textbf{\boldmath$\SC$-smoothly compatible}
if the transition map (cf. Figure~\ref{fig:fig-sc-transition})
\[
     \phi\circ\widetilde{\phi}^{-1}\colon 
     \widetilde{E}\supset
     \widetilde{\phi}(V\CAP \widetilde{V})\to\phi(V\CAP\widetilde{V})
     \subset E
\]
is\index{sc-diffeomorphism}
\begin{figure}[b]
  \centering
  \includegraphics
                             [height=4cm]
                             {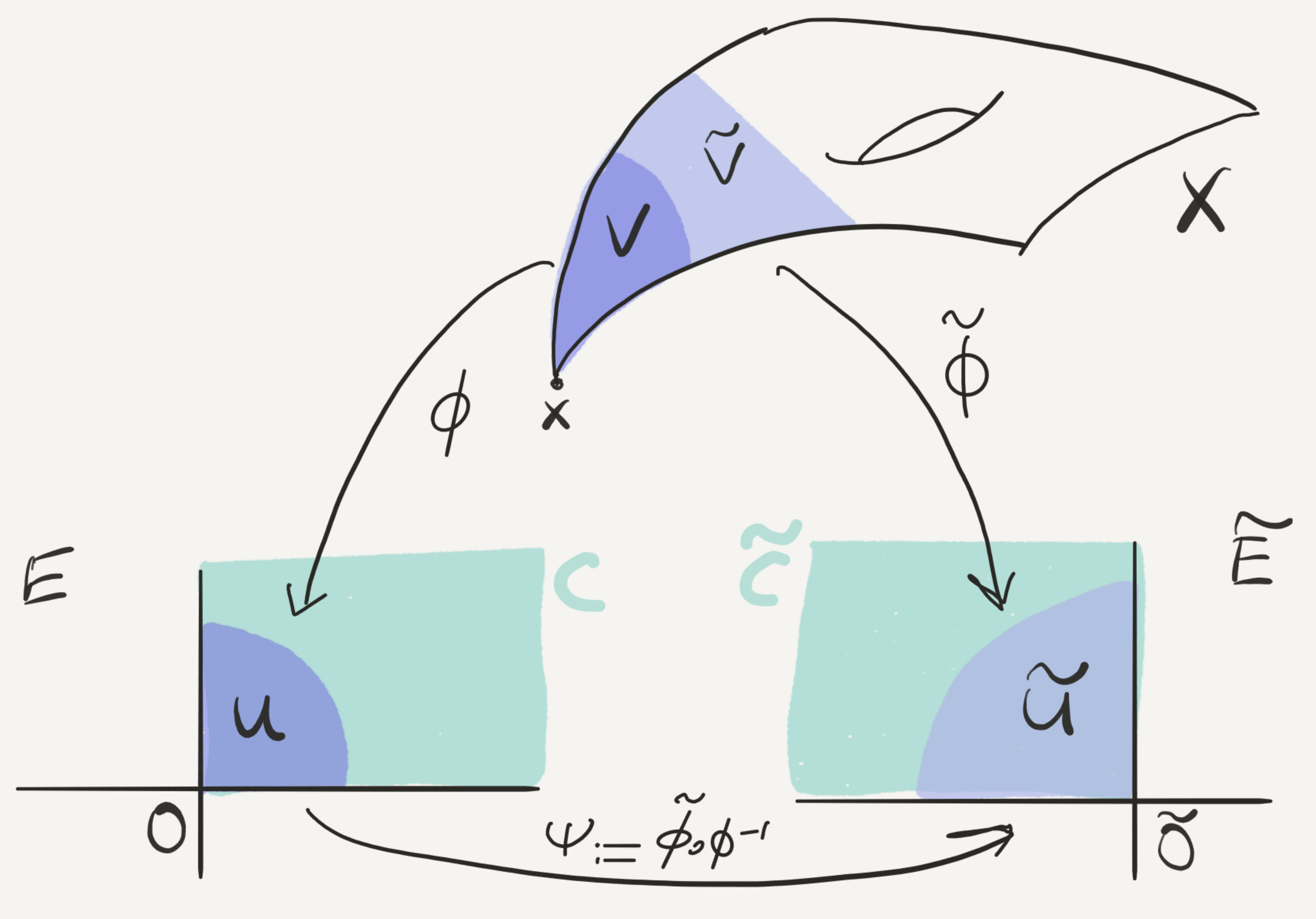}
  \caption{Transition map between $\SC$-charts of $\SC$-manifold $X$}
  \label{fig:fig-sc-transition}
\end{figure}
an \textbf{\boldmath$\SC$-smooth diffeomorphism}
(invertible $\SC$-smooth map with $\SC$-smooth inverse).
An \textbf{\boldmath$\SC$-smooth atlas for $X$} is a collection
$\Aa$ of pairwise $\SC$-smooth compatible Banach $\SC$-charts for $X$
such that the chart domains form a cover $\{V_i\}_i$ of $X$.
Two atlases are called \textbf{equivalent} if their union forms an atlas.

\begin{exercise}
Let $X$ be a topological space endowed with an $\SC$-smooth atlas $\Aa$.
a)~Is it true that $X$ is connected iff it is path connected? 
b)~Show that if $X$ is connected, then all model $\SC$-Banach spaces
$\widetilde{E}$ appearing in the charts of $\Aa$ are (linearly) $\SC$-isomorphic
to one and the same $\SC$-Banach space, say $E$.
In this case one says that $(X,\Aa)$ is
\textbf{modeled on \boldmath$E$}.

\vspace{.1cm}\noindent
[Hint: b) Given a transition map
$\psi\colon E\supset U\to \widetilde{U}\subset\widetilde{E}$
between two $\SC$-charts, observe that $U_\infty$ is a dense subset of $U$ and
that $\SC$-derivatives taken at smooth points are $\SC$-operators by 
Corollary~\ref{cor:sc-deriv_is_sc-op_smooth-pts}.]
\end{exercise}

\begin{definition}\label{def:sc-mfs}
An \textbf{\boldmath$\SC$-manifold} is a paracompact Hausdorff space
$X$, see Definition~\ref{def:HD-paracompact},
endowed with an equivalence class of
$\SC$-smooth atlases.\index{sc-manifold}
If all model spaces are $\SC$-Hilbert spaces one speaks of an
\textbf{Hilbert \boldmath$\SC$-manifold}.\index{sc-manifold!Hilbert --}
\end{definition}

\begin{definition}[Sc-smooth maps between $\SC$-manifolds]
a)~A continuous map $f\colon X\to Y$ between $\SC$-manifolds
is\index{sc-smooth!map}
called \textbf{\boldmath$\SC$-smooth} if for all $\SC$-charts
$\phi\colon X\supset V\to C\subset E$ and $\psi\colon Y\subset W\to
D\subset F$ the chart representative
\[
     \psi\circ f\circ \phi^{-1}\colon E\supset C\supset \phi(V\CAP f^{-1}(W))
     \to D\subset F
\]
is of class $\SC^\infty$ as a map from an open subset of
the partial quadrant $C$ in the $\SC$-Banach space $E$
\begin{figure}
  \centering
  \includegraphics
                             [height=4cm]
                             {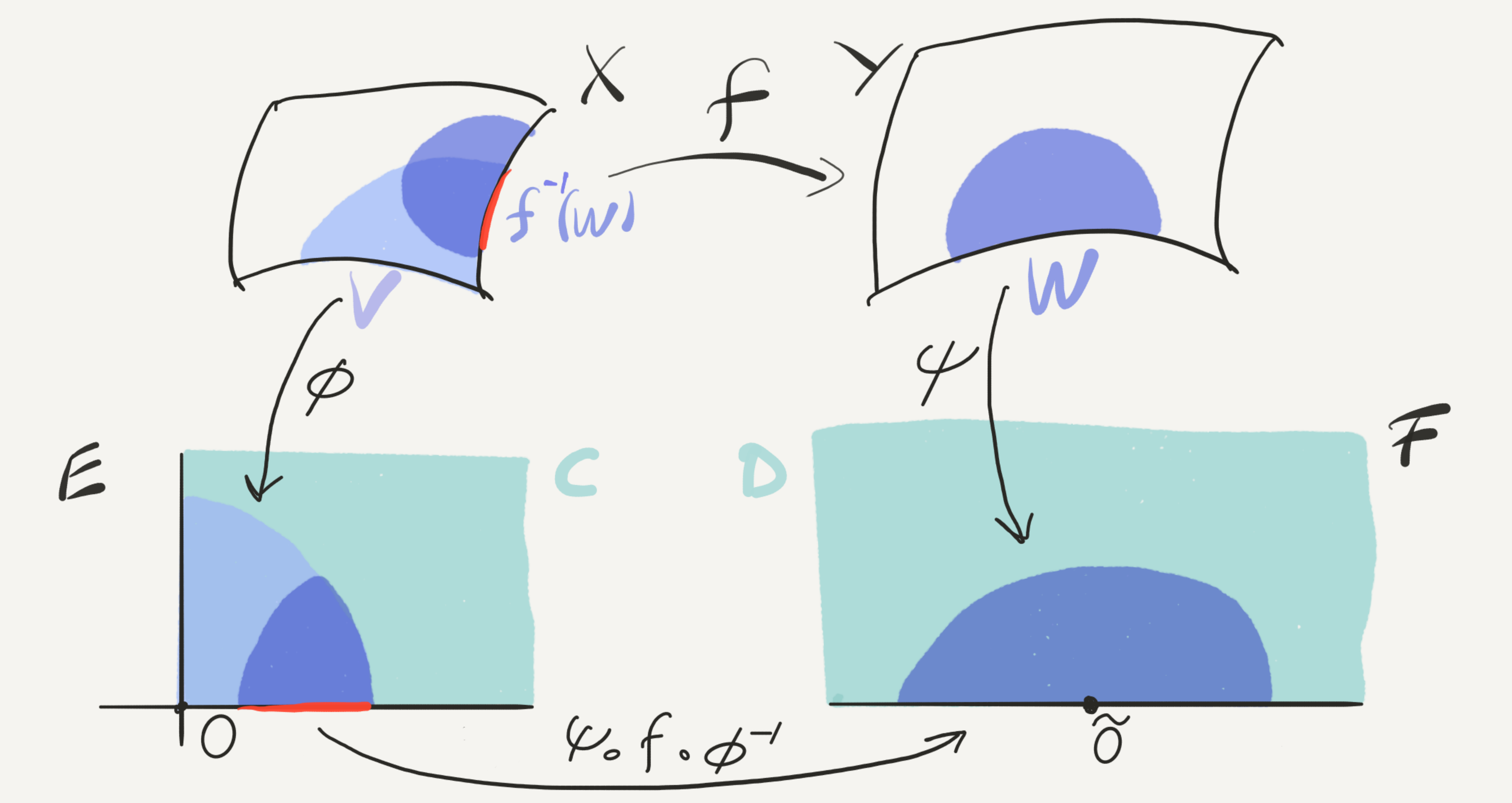}
  \caption{Local representative of $\SC$-smooth map between
                 $\SC$-manifolds}
  \label{fig:fig-sc-representative}
\end{figure}
into the $\SC$-Banach space $F$. See Figure~\ref{fig:fig-sc-representative}.

b)~An\index{sc-diffeomorphism!between sc-manifolds}
\textbf{\boldmath$\SC$-diffeomorphism} between $\SC$-manifolds
is an invertible $\SC$-smooth map whose inverse is $\SC$-smooth.
\end{definition}

\subsubsection*{Detecting boundaries and corners}

Suppose $X$ is an $\SC$-manifold.
To define the \textbf{degeneracy index} of a point
$x\in X$, pick an $\SC$-chart $\phi\colon X\supset V\to C\subset E$ about $x$
and set
\[
     d_X(x):=d_C(\phi(x))\in\N_0.
\]
By Theorem~\ref{thm:inv-deg-ind}
the definition \emph{does not depend}\footnote{
  For M-polyfolds the definition might depend on the choice of chart.
  The way out will be to take the minimum over all charts.
  }
on the choice of $\SC$-chart.
One calls a point $x$ of degeneracy index $d_X(x)=k$
an \textbf{\Index{interior point}} if $k=0$,
a \textbf{\Index{boundary point}} if $k=1$,
and a \textbf{corner point of complexity \boldmath$k$} in case
$k\ge 2$.\index{corner point of complexity $k\ge 2$}
This is illustrated by Figure~\ref{fig:fig-deg-index-quadrant}
for $X=C=[0,\infty)^2$.

\subsubsection*{Levels of \boldmath$\SC$-manifolds are
topological Banach manifolds}

A point $x$ of an $\SC$-manifold $X$ is said to
be\index{levels of sc-manifolds}
\textbf{on level \boldmath$m$}\index{sc-manifold!levels of --}
if $\phi(x)\in E_m$ lies on level $m$ for some (thus every)
$\SC$-chart $\phi\colon X\supset V\to C\subset E$ about $x$.
Indeed the definition does not depend on the choice of chart,
even for topological $\SC$-manifolds (those of class $\SC^0$),
since any transition map is of class $\SC^0$, hence level preserving
(with continuous level maps).
\textbf{Level \boldmath$m$ of the $\SC$-manifold}
is the set
\[
     X_m:=\{\text{all points of $X$ on level $m$}\}.
\]
By levelwise\index{$X_m$ level $m$ of sc-manifold}
continuity of transition maps each level $X_m$ of an $\SC$-manifold is
a topological Banach manifold (in general not $C^1$).

To summarize, an $\SC$-manifold $X$ decomposes
into a nested sequence of topological Banach manifolds
\[
    X=X_0\supset X_1\supset X_2\supset\dots\supset
    X_\infty:=\bigcap_{m\ge 0} X_m
\]
whose intersection $X_\infty$ carries the structure of a smooth
Fr\'{e}chet manifold with boundaries and corners;
cf.~\citet[\S 5.3]{Cieliebak:2018a}.

Furthermore, each level $X_k$ of an $\SC$-manifold $X$ inherits
the structure of an $\SC$-manifold denoted by $X^k$
and called the \textbf{shifted \boldmath$\SC$-manifold $X^k$}.
By\index{sc-manifold!shifted --}\index{$X^k$ shifted sc-manifold structure}
definition level $m$ of $X^k$ is level $X_{k+m}$ of $X$.

\subsubsection*{Levels of strong \boldmath$\SC$-manifolds
are smooth Banach manifolds}

Suppose $(U,C,E)$ and $(V,D,F)$ are $\SC$-triples.
The notion of scale differentiability $\SC^1$ is based
on usual $C^1$ differentiability of all diagonal maps of height one.
A natural way to strengthen this is to ask
all level maps (height zero) to be $C^1$ (or $C^k$).
Given $k\in\N$ or $k=\infty$, an $\SC^0$ map $f\colon U\to V$ between
$\SC$-triples is called \textbf{strongly \boldmath$\SC^k$} or of
\textbf{class \boldmath$\SSC^k$} if all level maps
$
    f_m\colon U_m\to V_m
$
are of class $C^k$.
This means that on each level one works with the usual calculus on
Banach spaces. Now one calls a paracompact Hausdorff space $X$
an\index{strongly scale differentiable}\index{sc-differentiable!strongly --}
\textbf{\boldmath$\SSC^k$-manifold}
if\index{$ssc^k$@$\SSC^k$-manifold}\index{ssc-manifold}
all transition maps are of class $\SSC^k$,
that is if they are level-wise $C^k$.

Important classes of function spaces
fit into the framework of strong scale differentiability,
for instance loop spaces of finite dimensional manifolds.

\begin{example}[Loop spaces are $\SSC^\infty$-manifolds]
\label{ex:loop-space-ssc}
Let $M$ be a manifold of finite dimension.
Then the \textbf{\Index{loop space}} 
\[
     X:=W^{1,2}(\SS^1,M)
\]
that consists of all absolutely continuous maps
$x\colon \R\to M$ of period one, that is $x(t+1)=x(t)$ for every $t$,
is a strongly $\SC$-smooth manifold.
\end{example}

\begin{example}
The previous example generalizes to $X:=W^{k,p}(N,M)$ where
$N$ can be any compact manifold-with-boundary of finite dimension $n$
and where the numbers $k\in\N$ and $p\in[1,\infty)$ must satisfy the
condition $k>n/p$ (assuring continuity of the functions that are
the elements of $X$).
\end{example}

\subsubsection*{Tangent bundle of \boldmath$\SC$-manifolds}

Let $X$ be an $\SC$-manifold. For an $\SC$-chart $(V,\phi,(U,C,E))$ we
shall use the short notation $(V,\phi)$ with the understanding that
$U=\phi(V)$ is an open subset of a partial quadrant $C$ in an
$\SC$-Banach space $E$.
Recall that $X^1$ denotes the $\SC$-manifold that arises from $X$ by
forgetting level zero. Let $V^1\subset X^1$ denote the corresponding
scale of levelwise open subsets generated by $V_1:=V\CAP X_1$.
Now consider tuples $(V,\phi,x,\xi)$ where $(V,\phi)$ is an $\SC$-chart
of $X$, the point $x\in V^1$ lies on level one, and $\xi\in E^0$ is a
vector in level zero of the $\SC$-Banach space $E$.
Two tuples are called equivalent if the two points $x,\widetilde{x}\in X^1$
are equal and the two vectors correspond to one another
through the $\SC$-derivative, in symbols
\[
     x=\widetilde{x},\qquad
     D(\widetilde{\phi}\circ\phi^{-1})|_{\phi(x)}\xi=\widetilde{\xi}.
\]
An\index{tangent vector to sc-manifold}
equivalence\index{sc-manifold!tangent vector to --}
class $[V,\phi,x,\xi]$ is called a \textbf{tangent vector to the
\boldmath$\SC$-manifold} at a point $x$ on level one. There is a
canonical projection defined on $TX$, the \textbf{set of all tangent
vectors at all points of \boldmath$X_1$}, namely
\[
     p\colon TX\to X^1,\quad
     [V,\phi,x,\xi]\mapsto x.
\]

\begin{exercise}[Tangent bundle as $\SC$-manifold]
\label{exc:tangent-bdl-to-sc-mf}
Naturally\index{sc-manifold!tangent bundle of --}
endow\index{tangent bundle!of sc-manifold}
the set $TX$ with the structure of an $\SC$-manifold
such that the projection $p\colon TX\to X^1$ becomes $\SC$-smooth
as a map between $\SC$-manifolds.

\vspace{.1cm}\noindent
[Hint: See Remark~\ref{rem:B-mf-set}.
Each chart $\phi\colon X\supset V\to E$ of $X$ gives a bijection
\[
     \Phi:=T\phi\,\colon \,TV:=p^{-1}(V\CAP X^1)\to TE=E^1\oplus E^0,\quad
     [V,\phi,x,\xi]\mapsto\left(\phi(x),\xi\right)
\]
onto the open subset $U^1\oplus E^0$ of $E^1\oplus E^0$
where $U=\phi(V)$.]
\end{exercise}


\cleardoublepage
\phantomsection

\chapter[Sc-retracts -- the local models for M-polyfolds]
{Sc-retracts -- local models}
\label{sec:local-models}
\chaptermark{Sc-retracts}


Let us indoctrinate you right away to the intuition
behind the key players $O$ and maps between them.
Think of an $\SC$-retract $O$ as a \emph{compressed open set}
-- the image of some idempotent map $r=r\circ r\colon U\to U$,
called a projection or retraction.
Vice versa, think of the open set $U$, likewise $r$, as a 
\textbf{decompression of \boldmath$O$}.
The great variety of possible properties of such $O$
--~there can be corners and even jumping dimension~-- 
are desirable in applications, because
solution spaces to PDEs often exhibit such behavior.
In contrast, to do analysis it is desirable that domains of maps are
open, so difference quotients, hence derivatives, can be defined.
Idempotents $r=r\circ r\colon U\to U$ combine and provide both of these,
somehow contradictory, properties. One uses such $O$ as geometric
model space and when it comes to analysis one just decompresses
$O=r(U)$ and uses the open set $U$ as domain. For instance, to define
differentiability of a function $f\colon O\to \R$ one
\textbf{\Index{decompress}es} the domain $O$
and calls $f$ differentiable if the pre-composition
$f\circ r\colon U\to \R$ is.\index{decompression of $f$}
In such context we often call $f\circ r$ or $r$ itself a
\textbf{decompression~of~\boldmath$f$}.
A second highly useful property of images $O=r(U)$ of projections
$r\colon U\to U$ is that any such is precisely the fixed point set $\Fix\, r=O$ of $r$.

In Chapter~\ref{sec:local-models}
our main source is again~\citet{Hofer:2017a},
together with~\citet{Cieliebak:2018a} and \citet{Fabert:2016a}.
Concerning terminology our convention is and was to assign the adjective
$\SC$-smooth (or the equivalent symbol $\SC^\infty$) to maps
that are $k=\infty$ many times continuously scale differentiable.
In case of sets, e.g. $\SC$-manifolds or $\SC$-retracts,
the ``$\SC$'' itself already indicates $\SC$-smooth.

\bigskip

\textit{Outline of Chapter~\ref{sec:local-models}.}
In the present chapter 
M-polyfolds\footnote{
  \underline{M}-polyfolds are defined analogous to
  \underline{m}anifolds, just based on $\SC$-differentiability and
  more general model spaces. 
  In contrast, polyfolds correspond classically to orbifolds.
  }
are constructed based on the
new notion of scale differentiability and locally modeled on
rather general topological spaces $O$ which might have corners,
even jumping dimension along components, but they will still be
accessible to the new weaker form of calculus -- $\SC$-calculus.
The class of spaces are $\SC$-retracts, generalizing
smooth retracts in Banach manifolds.
Section~\ref{sec:Cartan} ``Cartan's last theorem'' deals with smooth retracts
and is the motivation for the generalizations in the following sections.
Section~\ref{sec:sc-retracts} ``Sc-smooth retractions and their images $O$''
provides the local model spaces $O$ for M-polyfolds.
A key step is to extend $\SC$-calculus from $\SC$-Banach spaces $E$
to $\SC$-smooth retracts $O$.
Section~\ref{sec:M-polyfolds} ``M-polyfolds and their tangent bundles''
defines M-polyfolds, in analogy to manifolds, by patching together
local models and asking transition maps to be $\SC$-smooth
(in the sense of the extended $\SC$-calculus).
Section~\ref{sec:strong-bundles} ``Strong bundles over M-polyfolds''
provides the environment to implement $\SC$-Fredholm sections $f$.
The need for $\SC^+$-sections requires fibers be shiftable
in scale by $+1$ leading to double scale structures.
In practice $f$ arises as a differential operator of order $\ell$
leading to asymmetry in base and fiber levels.

\subsubsection{Detailed summary of Chapter~\ref{sec:local-models}}

Section~\ref{sec:Cartan} ``Cartan's last theorem''
recalls and proves the surprising result that the image $O=r(U)$
of a smooth idempotent map $r=r\circ r\colon U\to U$ on a Banach
manifold, called a \textbf{smooth retraction}, is a smooth submanifold.

\bigskip

Section~\ref{sec:sc-retracts} ``Sc-smooth retractions and their images $O$''
is at the heart of the whole theory.
It introduces the local model spaces for M-polyfolds, called
\textbf{\boldmath$\SC$-retracts} and denoted by $(O,C,E)$, or simply $O$.
These are images $O$ of $\SC$-smooth idempotents $r=r^2\colon U\to U$,
called \textbf{\boldmath$\SC$-retractions}, defined on $\SC$-triples $(U,C,E)$.
It is useful to observe that image and fixed point set of $r$ coincide
and to think of $r$ as a projection onto its fixed point set, in symbols
\[
     r=r\circ r\colon U\to O:=\im r=\Fix\, r.
\]
While the domain $U$ is a (relatively) open subset of a
partial quadrant $C$ in an $\SC$-Banach space $E$,
its image $O=r(U)$ is a projected or compressed version of $U$.
Motivated by continuous retractions
one might expect that the compressed set, the $\SC$-retract $O=r(U)$
has non-smooth properties, e.g. jumping dimension or
having corners, as illustrated by Figure~\ref{fig:fig-sc-retract-schippe}.
In contrast, the images of in the usual sense smooth retractions on
Banach manifolds are smooth Banach submanifolds
by Theorem~\ref{thm:Cartan}.

How can one do analysis and define a derivative on a
possibly non-open set $O=r(U)$? The key idea is to
decompress $O$ and use the open subset $U$ of $E$ as domain.
(We assume $C=E$ for illustration).
Let us call $U$, likewise $r$, a \textbf{decompression of \boldmath$O$}.
Of course, if one defines a property of $O$ using a decompression
one needs to check independence of the chosen decompression of $O$.
For instance, one defines \textbf{\boldmath{$\SC$-smoothness}} of a map
between\index{sc-smooth!retract map}\index{retract map!sc-smooth --}
$\SC$-retracts
\[
     f\colon O\to O^\prime
\]
if some, hence by Lemma~\ref{le:for-some-thus-every-123} any, decompression
\[
     f\circ r\colon U\to U^\prime,\qquad O=r(U)
\]
of $f$ is an $\SC$-smooth map in the ordinary sense; see Definition~\ref{def:sc^k}.
Such $f$ is called an \textbf{\boldmath$\SC$-smooth retract map}
-- the future M-polyfold transition maps.
Given an $\SC$-retract $(O,C,E)$, the tangent map
of a decompression $r=r\circ r\colon U\to U$ of $O=r(U)$ is an
$\SC$-smooth retraction itself
\[
     Tr=(Tr)\circ (Tr)\colon TU=U^1\oplus E^0\to TU,\quad
     (x,\xi)\mapsto \left(r(x),Dr(x)\xi\right).
\]
Hence the image
\[
     TO:=Tr(TU)=\Fix\, Tr\subset O^1\oplus E^0
\]
is an $\SC$-retract $(TO,TC,TE)$ in the tangent $\SC$-triple
$(TU,TC,TE)$. Here $TO$ is independent of the choice
of\index{$p\colon TO\to O^1$ tangent bundle of $\SC$-retract}
the decompression $r$ of $O$ by Lemma~\ref{le:TO}.
The\index{sc-retract!tangent bundle of --}
\textbf{tangent bundle of the \boldmath$\SC$-retract $O$}
is the natural surjection
\[
     p\colon TO\to O^1,\quad (x,\xi)\mapsto x.
\]
It is an $\SC$-smooth map between $\SC$-retracts.
The tangent space at $x\in O^1$
\[
     T_x O:=\Fix[Dr(x)\colon E\to E]\subset E
\]
is a Banach subspace, even an $\SC$-subspace for $x\in O_\infty$,
by Corollary~\ref{cor:sc-deriv_is_sc-op_smooth-pts}.

The \textbf{tangent map of an \boldmath$\SC$-smooth retract map}
$f\colon O\to O^\prime$ is defined as the restriction to $TO$ of the tangent map
\begin{equation*}
\begin{split}
     Tf:=T(f\circ r)|_{TO}\colon Tr(TU)=TO&\to T O^\prime\\
     (x,\xi)&\mapsto \left(f(x),Df(x)\xi\right)
\end{split}
\end{equation*}
of some, by Lemma~\ref{le:gyu6768} any, decompression $f\circ r$.
Here $f\circ r(x)=f(x)$ since 
$O^1\subset O=\Fix\, r$ and $D(f\circ r)|_x=Df(x)$
on $T_x O=\Fix\, Dr(x)\subset E$.
Section~\ref{sec:sc-retracts} on $\SC$-retracts is rounded off by the
chain rule for compositions of $\SC$-smooth retract maps.

\bigskip

Section~\ref{sec:M-polyfolds} ``M-polyfolds and their tangent bundles''
defines M-polyfolds in analogy to Banach manifolds just using the
rather general class of $\SC$-retracts as local models and requiring
only scale smoothness of the transition maps. In particular, to define
an \textbf{M-polyfold \boldmath$X$} one starts with a paracompact
Hausdorff space $X$. E.g. $\SC$-manifolds are M-polyfolds ($r=id$ and $O=U$)
and so are open subsets of M-polyfolds.
Sc-smoothness of maps
\[
     f\colon X\to Y
\]
between M-polyfolds is defined
in terms of local coordinate representatives of $f$ which are
required to be $\SC$-smooth retract maps.
An M-polyfold $X$ inherits a set scale structure from the local
model spaces. Let $X_m$, called \textbf{level \boldmath$m$ of $X$},
consist of all points of $X$ which
are mapped in some, hence any, coordinate chart into level
$m$ of model space. Each level $X_m$ is a topological Banach manifold
and inherits the structure of an M-polyfold denoted by $X^m$.

To construct the \textbf{tangent bundle \boldmath$p\colon TX\to X^1$}
one first defines $TX$ as a set and then a natural map $p$,
using the local coordinate charts $\phi\colon V\to O$ of $X$ to define
bijections denoted by $T\phi\colon TX\supset TV\to TO$. 
Given an atlas $\Aa$ of $X$, these bijections induce the collection
$$
     \Bb=\{(T\phi)^{-1}(W)\mid 
     \text{$\phi\in \Aa$ and $W\subset TO$ open}\}\subset 2^{TX}
$$
of subsets of $TX$. It forms a basis of a paracompact
Hausdorff topology. Endowing $X$ with that topology the bijections $T\phi$
become homeomorphisms and one gets a natural M-polyfold atlas $T\Aa$
for $TX$.

\vspace{0.1cm} 
\textit{Sub-M-polyfolds.} A subset $A\subset X$ of an M-polyfold
is a \textbf{\Index{sub-M-polyfold}}
if around any point $a\in A$ there is an open neighborhood
$V\subset X$ and an $\SC$-smooth retraction
$r=r^2\colon V\to V$ such that $A\CAP V=r(V)=\Fix\, r$.
Such $r$ is called a \textbf{local generator} for the sub-M-polyfold $A$. 
Viewed as a map $r\colon V\to A$ a local generator is $\SC$-smooth and
$T_ar(T_aX)=T_a A$ at any point $a\in A\CAP V$. At smooth points
the tangent space $T_aA$ is $\SC$-complemented~in~$T_aX$.

\vspace{0.1cm} 
\textit{Boundaries and corners -- tameness.}
Recall from~(\ref{eq:deg-index}) that the degeneracy index
$k=d_C(p)$ of a point $p$ of a partial quadrant $C$ tells whether $p$
is an interior point ($k=0$), a boundary point ($k=1$), or a corner
point of complexity $k\ge 2$. Unfortunately, for points $x$ of
M-polyfolds $X$ the degeneracy index $d_X(x):=d_C(\phi(x))$ defined
in terms of an M-polyfold chart $\phi$ may depend on the chart; see
Figure~\ref{fig:fig-deg-index-M-pfs-L_a}.
Thus one introduces a new class, the so-called \textbf{tame}
M-polyfolds, for which there is no dependence on $\phi$.

\bigskip

Section~\ref{sec:strong-bundles} ``Strong bundles over M-polyfolds''
provides the environment to implement partial differential operators
whose zero sets will represent the moduli spaces
which are under investigation in many different geometric analytic
situations. Often moduli spaces, hence zero sets, are of finite dimension
and are modeled on the kernels of surjective Fredholm operators.
To achieve surjectivity in a given geometric PDE scenario
one usually perturbs some already present, but inessential, quantity.
These perturbations should be related to bounded operators,
so the overall Fredholm property is preserved.
\\
Recall from Proposition~\ref{prop:stability-sc-Fredholm} that the
$\SC$-Fredholm property of a linear map $T\colon E\to F$ is preserved under
addition of $\SC^+$-operators $S\colon E\to F$. The latter operators are
characterized by the property of improving their output regularity by one level, that is
$S(E_m)\subset F_{m+1}$. As a consequence all level operators
$S_m\colon E_m\to F_{m+1}\INTO F_m$ are compact.

\vspace{0.1cm} 
\textit{Motivation.}
Replacing now the linear domain $E$ by an M-polyfold $X$ as domain of
a partial differential operator $f$ of order, say $\ell$, the task at
hand\footnote{
  freely borrowed from one of my favorite authors
  }
is to construct vector bundles $P\colon Y\to X$ with fibers modeled on an
$\SC$-Banach space $F$, so that the differential operator becomes a
section $f\colon X\to Y$.
Concerning the implementation of Fredholm properties one has to allow
for fiber level shifts by $+1$, that is all fibers $Y_x:=P^{-1}(0)$ should
be identifiable with the $\SC$-Banach space $F^0=F$, as well
as with the shifted one $F^1$; cf. Remark~\ref{rem:sc-plus}.
In practice, the level indices $m$ correspond to the degree of
differentiability of the level elements. So the domain of $f$ should be $X^{\ell+m}$
in which case $f$ takes values in level $m$, sometimes even $m+1$.
Then one can exploit composition with \emph{compact} embeddings up to
level $0$; see Remark~\ref{rem:motivation-asym-prod}.
This motivates the following asymmetric double scale
structure which must be subsequently reduced to two versions
of individual scales, in order to be accessible to scale calculus
(there is no double scale calculus).

\vspace{0.1cm} 
\textit{Trivial-strong-bundle retracts $K$ -- the local models.}
Let $E,F$ be Banach scales and $U\subset E$ be open.
The \textbf{non-symmetric product} $U\triangleright F$ is the
subset $U\times F$ of the Banach space $E\oplus F$ endowed with the
\textbf{double scale}, also called double filtration,
defined by
\[
     (U\triangleright F)_{m,k}:=U_m\oplus F_k,\qquad
     m\in\N_0,\quad k\in\{0,\dots m+1\}.
\]
Projection onto the first component
\[
     U\triangleright F\to E,\quad
     (u,\xi)\mapsto u
\]
is called the \textbf{trivial-strong-bundle projection}.
However, for $\SC$-calculus one needs one scale structure, not a
double scale. Consider the $\SC$-manifolds
\[
     (U\triangleright F)^{[0]}:=U\oplus F,\qquad
     (U\triangleright F)^{[1]}:=U\oplus F^1.
\]
For $i\in\{0,1\}$ projection on component one is an $\SC$-smooth map
$$
     p=p^{[i]}\colon (U\triangleright F)^{[i]}\to U
$$
between $\SC$-manifolds
called a \textbf{trivial strong sc-bundle}.
A \textbf{trivial-strong-bundle retraction}
is an idempotent \textbf{trivial-strong-bundle map}\footnote{
  i.e. double scale preserving and with $\rho_u\xi:=\rho(u,\xi)$ being
  linear in $\xi$
  }
\begin{equation*}
\begin{split}
     R=R\circ R\colon  U\triangleright F&\to U\triangleright F\\
     (u,\xi)&\mapsto\left(r(u),\rho_u\xi\right)
\end{split}
\end{equation*}
The first component $r$ of $R$ is necessarily an $\SC$-smooth retraction on
$U$, called \textbf{associated base retraction}. Its image, the $\SC$-retract
$O=r(U)$, is called the \textbf{associated base retract}.
A \textbf{trivial-strong-bundle retract}\footnote{
  'strong' indicates 'doubly scaled' and the
  retraction acts on a 'trivial bundle'
  }
$(K, C\triangleright F, E\triangleright F)$ is the image
\[
     K:=R(U\triangleright F)=\left(\Fix\, R\right)
     \subset \left(O\triangleright F\right)
\]
of a trivial-strong-bundle retraction
$R=R\circ R$ on $U\triangleright F$
where $O=r(U)$ is the associated base retract.
One likewise calls the natural surjection
\[
     p\colon K\to O,\quad (x,\xi)\mapsto x
\]
a \textbf{trivial-strong-bundle retract}.
Call $K:=R(U\triangleright F)$ \textbf{tame} if $R$ is tame.
As a subset of the doubly scaled space $U\triangleright F$
there is an induced double scale
\begin{equation*}
\begin{split}
     K_{m,k}:&=K\CAP \left(U_m\oplus F_k\right)\\
   &=\bigcup_{x\in O_m}\left(\{x\}\oplus \Fix\,[\rho_x\colon F_k\to F_k]\right)
\end{split}
\end{equation*}
for $m\in\N_0$ and $k\in\{0,\dots m+1\}$. The spaces
\begin{equation*}
     K^{[i]}:=K\CAP\left(E^0\oplus F^i\right)
     =\im R^{[i]}=R(U\triangleright F)^{[i]}
     ,\qquad i=0,1
\end{equation*}
with levels $K^{[i]}_m=K_{m,m+i}$ are $\SC$-retracts,
hence M-polyfolds. The surjections
\begin{equation*}
\begin{split}
     p=p^{[i]}\colon K^{[i]}&\to O,\qquad\quad i=0,1\\
     (x,\xi)&\mapsto x
\end{split}
\end{equation*}
are $\SC$-smooth maps between $\SC$-retracts.

A \textbf{section} of a trivial-strong-bundle retract $p\colon K\to O$
is a map $s\colon O\to K$ that satisfies $p\circ s=\id_O$.
If $s$ is $\SC$-smooth as an $\SC$-retract map
\[
     s^{[i]}\colon O\to K^{[i]},\quad
     x\mapsto \bigl(x,\bs^{[i]}(x)\bigr),\qquad
     \bs^{[i]}\colon O\to F^i 
\]
it is called an \textbf{\boldmath$\SC$-section} (case $i=0$) or
an \textbf{\boldmath$\SC^+$-section} (case $i=1$). The map
$\bs^{[i]}\colon O\to F^i$ is called the \textbf{principal part} of the section.

\vspace{0.1cm} 
\textit{Strong bundles.}  
A \textbf{strong bundle over an M-polyfold} $X$ is a continuous
surjection $P\colon Y\to X$ defined on a paracompact Hausdorff space $Y$
such that each pre-image $Y_x:=P^{-1}(x)$ is a Banachable space,
together with an equivalence class of strong bundle atlases.
\\
As usual, one patches together local model bundles which in our
case are the trivial-strong-bundle retracts
$K=R(U\triangleright F)\to O$ outlined above.
A strong bundle atlas for $P\colon Y\to X$ consists of suitably compatible
\textbf{strong bundle charts}
\[
     \left(\Phi,P^{-1}(V),(K, C\triangleright F, E\triangleright F)\right).
\]
Such tuple consists of
\begin{itemize}
\item 
  a trivial-strong-bundle retract $K$, that is
  $p\colon K=R(U\triangleright F)\to O$
  where $O=r(U)$ is the associated base retract;
\item 
  a homeomorphism $\varphi\colon X\supset V\to O$ between an open
  subset of the base M-polyfold $X$ of $Y$ and the base retract $O$ of $K$;
\item 
  a homeomorphism $\Phi\colon P^{-1}(V)\to K$ which covers $\varphi$ in the sense
  that the diagram
  \begin{equation*}
  \begin{tikzcd} [column sep=tiny] 
    Y
    \arrow[d, "P"']
    & \supset
      &
      P^{-1}(V)
      \arrow[d, "P"']
      \arrow{rrrrr}[name=U]{\Phi\;\,}
        &&&&&
        K=R(U\triangleright F)
        \arrow[d, "p", shift right=10.8]
    \\
     X
     & \supset
      &
      V
      \arrow[swap]{rrrrr}[name=D]{\,\;\varphi}
      \arrow[to path={(U) node[midway,scale=1.2] {\;\;\,\,$\circlearrowleft$}  (D)}]
        &&&&&
        O=r(U)\hphantom{hhh\,} 
  \end{tikzcd} 
  \end{equation*}
  commutes. Consequently, for every point $v\in V$ the restriction
  of $\Phi$ to $P^{-1}(v)$ takes values in $p^{-1}(\varphi(v))$.
  It is also required that $\Phi$ as a map
  \[
     \Phi\colon Y_v=P^{-1}(v)\stackrel{\simeq}{\longrightarrow} p^{-1}(\varphi(v))
     =\rho_{\varphi(v)}(F),\qquad \forall v\in V
  \]
 is a continuous linear bijection between the Banach/able space fibers.
\end{itemize}

A strong bundle atlas $\Aa^Y_X$ for $P\colon Y\to X$ provides
a double scale structure on $X$ induced by local charts.
As earlier, one extracts two individual scale structures
and obtains two \textbf{induced \boldmath$\SC$-bundle atlases}
$\Aa^{Y^{[0]}}_X$ and $\Aa^{Y^{[1]}}_X$ for $\SC$-bundles\footnote{
  The definition of $\SC$-bundles is indicated around~(\ref{eq:sc-bundles}).
  }
\[
     P^{[0]}\colon Y^{[0]}\to X,\qquad P^{[1]}\colon Y^{[1]}\to X.
\]
A \textbf{section} of a strong bundle $P\colon Y\to X$
is a map $s\colon X\to Y$ that satisfies $P\circ s=\id_X$.
If $s$ is $\SC$-smooth as a map between M-polyfolds
\[
     s^{[i]}\colon X\to Y^{[i]}
\]
then $s$ is called in case $i=0$ an \textbf{\boldmath$\SC$-section} of
$P\colon Y\to X$ and in case $i=1$ an \textbf{\boldmath$\SC^+$-section} of
$P\colon Y\to X$.

\section{Cartan's last theorem}\label{sec:Cartan}
\sectionmark{Cartan's last theorem}

In the realm of continuous linear operators $R$ on a Banach space $E$
an idempotent $R=R^2$ is called a \textbf{\Index{projection}}.
Note that the image 
$
     \im R=\Fix\, R
$
is equal to the fixed point set of
$R$. 
  (Both inclusions are immediate, only '$\subset$' uses idempotency.)
But the image of a linear operator is a linear subspace and the
fixed point set of a continuous map is a closed subset. So the image
of a projection is a closed linear subspace which, furthermore, is
complemented by the (again due to continuity) closed linear subspace
$\ker R$. To summarize
\[
     R^2=R\in\Ll(E)\quad\Rightarrow\quad
     E=\ker R\oplus \im R=\ker R\oplus \Fix\, R.
\]

More generally, given a topological space $X$, 
a continuous idempotent map $r=r\circ r\colon X\to X$ is
called a \textbf{\Index{retraction} on \boldmath$X$}
and the closed subset
\[
     \im r=\Fix\, r\subset X
\]
is called a \textbf{\Index{retract} of \boldmath$X$}.

\begin{theorem}[{\citet{Cartan:1986a}}]
\label{thm:Cartan}
The image of a smooth retraction $r\colon X\to X$ on a Banach manifold
is a topologically closed smooth submanifold of~$X$.
\end{theorem}

\begin{proof}
We follow~\citet{Cieliebak:2018a}.
Closedness of the set $\im r=\Fix\, r$ holds by continuity of $r$.
To be a submanifold is a local property. Pick $x\in \Fix\, r$
and a Banach chart $(V,\phi,E)$ about $x$ with $\phi(x)=0$;
cf. Section~\ref{sec:B-mfs}.
It suffices to show that $\Fix\, r$ is locally near $x$
the image under a diffeomorphism, say $\alpha^{-1}$, of an open
subset of a linear subspace, say $\Fix\,[R\colon E\to E]$
for some $R=R^2\in\Ll(E)$, of the Banach space $E$.
This takes three steps.

\vspace{.1cm}\noindent
\textbf{Step 1. (Localize)}
The retraction $r$ on $X$ descends to a smooth retraction on an
open subset $U\subset E$ of the local model Banach space, still denoted by
\[
     r=r^2\colon E\supset U\to U,\quad
     r(0)=0,\quad
     U:=\phi(V\CAP r^{-1}(V)).
\]
The derivative $R:=dr(0)=R^2\in\Ll(E)$ is a projection
and the maps
\[
     \alpha,\beta\colon U\to E,\quad
     \alpha:=\beta+R\circ r,\quad
     \beta:=(\1-R)\circ(\1-r)
\]
take on the same value $\alpha(0)=\1-R=\beta(0)$ at the origin.

\vspace{.1cm}\noindent
\textit{Proof of Step 1.}
Observe that $V\CAP r^{-1}(V)$ is not only an open neighborhood of the
fixed point $r(x)=x\in V$, but it is also invariant under $r$: Indeed
\begin{equation}\label{eq:guy6576}
     r\left(V\CAP r^{-1}(V)\right)
     \subset\left( r(V)\CAP V \right)
     \subset\left( V\CAP r^{-1}(V) \right)
\end{equation}
where both inclusions are immediate, only the second one uses $r\circ r=r$.
Hence the local representative $\phi\circ r\circ\phi$, hereafter still
denoted by $r$, is a smooth retraction on $U$ and it maps $0=\phi(x)$
to itself. The latter fixed point property enters the identity
$R=dr|_0=d(r\circ r)|0= dr_{r(0)}\circ dr|_0=R\circ R$.

\vspace{.1cm}\noindent
\textbf{Step 2. (Local diffeomorphism)}
The map $\alpha$ conjugates $r$ and $R$
\[
     \alpha\circ r=R\circ \alpha\colon E\supset U\to E
\]
and it holds that $\alpha(0)=0$ and $d\alpha(0)=\1$.

\vspace{.1cm}\noindent
\textit{Proof of Step 2.}
The retraction properties of $r$ and $R$ imply the identities
\[
     \beta\circ r=(\1-R)\circ(\1-r)\circ r
     =(\1-R)\circ(r-r^2)=0
\]
and
\[
     R\circ \beta=R\circ (\1-R)\circ(\1-r)
     =(R-R^2)\circ(\1-r).
\]
These two identities imply, respectively, the identities
\[
     \alpha\circ r=\beta\circ r+R\circ r\circ r=R\circ r
\]
and
\[
     R\circ \alpha=R\circ\beta+R\circ R\circ r=R\circ r.
\]
Thus $\alpha\circ r=R\circ \alpha$. Hence
$\alpha(0)=\alpha(r(0)=R(\alpha(0))=R(\1-R)=0$ and
\begin{equation*}
\begin{split}
     d\alpha(0)
   &=d\left((\1-R)\circ(\1-r)+R\circ r\right)|_0\\
   &=(\1-R)\circ(\1-R)+R\circ R\\
   &=\1-2R+R^2+R^2\\
   &=\1.
\end{split}
\end{equation*}

\vspace{.1cm}\noindent
\textbf{Step 3. (Conjugation to linearization)}
There is an open subset $W\subset U$ of $E$ such that $\alpha\colon W\to E$
is a diffeomorphism onto its image $\alpha(W)$ and $r(W)\subset W$.
Moreover, the linear retraction $R=dr(0)\colon E\to E$ restricts
to a smooth retraction on $W$ and coincides with the composition
\[
     R=\alpha\circ r\circ \alpha^{-1}\colon 
     \alpha(W)\to W\to r(W)\subset W\to\alpha(W).
\]

\vspace{.1cm}\noindent
\textit{Proof of Step 3.} 
Since $d\alpha(0)=\1$ is invertible there is by the inverse function
theorem an open neighborhood $W^\prime\subset U$ of the fixed point
$0\in E$ of $\alpha$ and $r$ such that the restriction
$\alpha\colon W^\prime\to E$ is a diffeomorphism onto its image.
To obtain, in addition, invariance under $r$ replace
$W^\prime$ by $W:=W^\prime\CAP r^{-1}(W^\prime)$.
To see this repeat the arguments that led to~(\ref{eq:guy6576}).

\vspace{.1cm}\noindent
\textbf{Step 4. (Diffeomorphism to open set in Banach space)} 
Step~3 shows
\begin{equation*}
\begin{split}
     \Fix\,[r\colon W\to W]
   &=\alpha^{-1}\left(\Fix\, [R\colon \alpha(W)\to\alpha(W)]\right)\\
   &=\alpha^{-1}\left(\alpha(W)\CAP \Fix\, [R\colon E\to E]\right).
\end{split}
\end{equation*}
Step~4 proves Theorem~\ref{thm:Cartan}:
Indeed $\alpha(W)$ is an open neighborhood in $E$ of the fixed point $0$
of $r$ and $\Fix\, R$ is a (closed) linear subspace of $E$.
So the intersection is an open neighborhood of $0$ in the Banach
space $\Fix\, R$. But that intersection is diffeomorphic, under $\alpha^{-1}$,
to the part of $\Fix\, r$ in the open set $W$.
\end{proof}

\boldmath
\section{Sc-smooth retractions and their images~$O$}\unboldmath
\label{sec:sc-retracts}
\sectionmark{Sc-smooth retractions}

In this section the local model spaces for M-polyfolds
are constructed and the maps between them are endowed
with an adequate notion of $\SC$-smoothness, namely,
$\SC$-smoothness when viewed as maps between decompressed domains.
The model spaces are images $O$ of
$\SC$-smooth retractions $r=r^2\colon U\to U$
on $\SC$-triples $(U,C,E)$. It is useful to observe that
image and fixed point set of $r$ coincide and
to think of $r$ as a projection onto its image
\[
     r=r\circ r\colon U\to O:=\im r=\Fix\, r.
\]

\subsubsection*{Sc-retracts and sc-smoothness of maps between them}

\begin{definition}[Sc-retracts $O$]\index{sc-smooth!retraction}
An \textbf{sc-smooth retraction} on an $\SC$-triple
$(U,C,E)$ is an $\SC$-smooth idempotent map $r=r\circ r\colon U\to U$.
Note that
\[
     r\circ r=r\qquad
     \Leftrightarrow\qquad
     \im r=\Fix\, r .
\]
$\Fix\, r\subset U$ is (relatively) closed by continuity of $r$.
An \textbf{\Index{sc-retract}} $O\subset C\subset E$
in\index{$O$ $\SC$-smooth retract}
a partial quadrant $C$ in a Banach scale $E$
is the image (fixed point set)
\[
     O=r(U)=\Fix\, r,\qquad r\circ r=r\colon U\to U
\]
of \underline{some} $\SC$-smooth retraction $r$ whose domain
$U\subset C$ is (relatively) open. Usually we abbreviate
the notation $(O,C,E)$ of an \textbf{\boldmath$\SC$-retract} 
by simply writing $O$.
As pointed out in~\citet[before Prop.\,2.3]{Hofer:2017a}, the ambient partial
quadrant $C\subset E$ matters, because it is possible that $O$ is an
$\SC$-retract with respect to some non-trivial $C$, but not for $C=E\colon$
Think of $(O,E,E)$ as local models for M-polyfolds in regions
without boundary and $(O,C,E)$ as such near boundaries with corners;
cf.~\citet[after Def.\,1.13]{Hofer:2010b}.
\end{definition}

\begin{lemma}\label{le:O^k}
If $r\colon U\to U$ is an $\SC$-smooth retraction, then all level maps
are \emph{continuous} retractions
\[
     r_m=r_m\circ r_m\colon U_m\to U_m,\qquad U_m:=U\CAP E_m
\]
and $O_m:=O\cap E_m$ is equal to the image $r(U_m)$.
In terms of shifted scales
\begin{equation}\label{eq:O^k=r(U^k}
     O^k=r(U^k)=\Fix\,[r\colon U^k\to U^k],\qquad k\in\N_0.
\end{equation}
\end{lemma}

\begin{proof}
To be shown is the equality of sets $O_m=r(U_m)$.
'$\subset$' Pick $x\in O\CAP E_m\subset U\CAP E_m$, then
$x=r(x)\in r(U_m)$.
'$\supset$' Pick $x\in U_m$, then $r(x)\in r(U)\CAP E_m=O\CAP E_m$
since $U_m\subset U$ and $r$ is level preserving, respectively.
\end{proof}

Whereas the image of a \emph{smooth} retraction on a Banach manifold
is a smooth submanifold by Cartan's last theorem, Theorem~\ref{thm:Cartan},
an $\SC$-retract can be connected and nevertheless have
pieces of various dimensions; see Figure~\ref{fig:fig-sc-retract-schippe}.
\begin{figure}
  \centering
  \includegraphics
                             [height=4cm]
                             {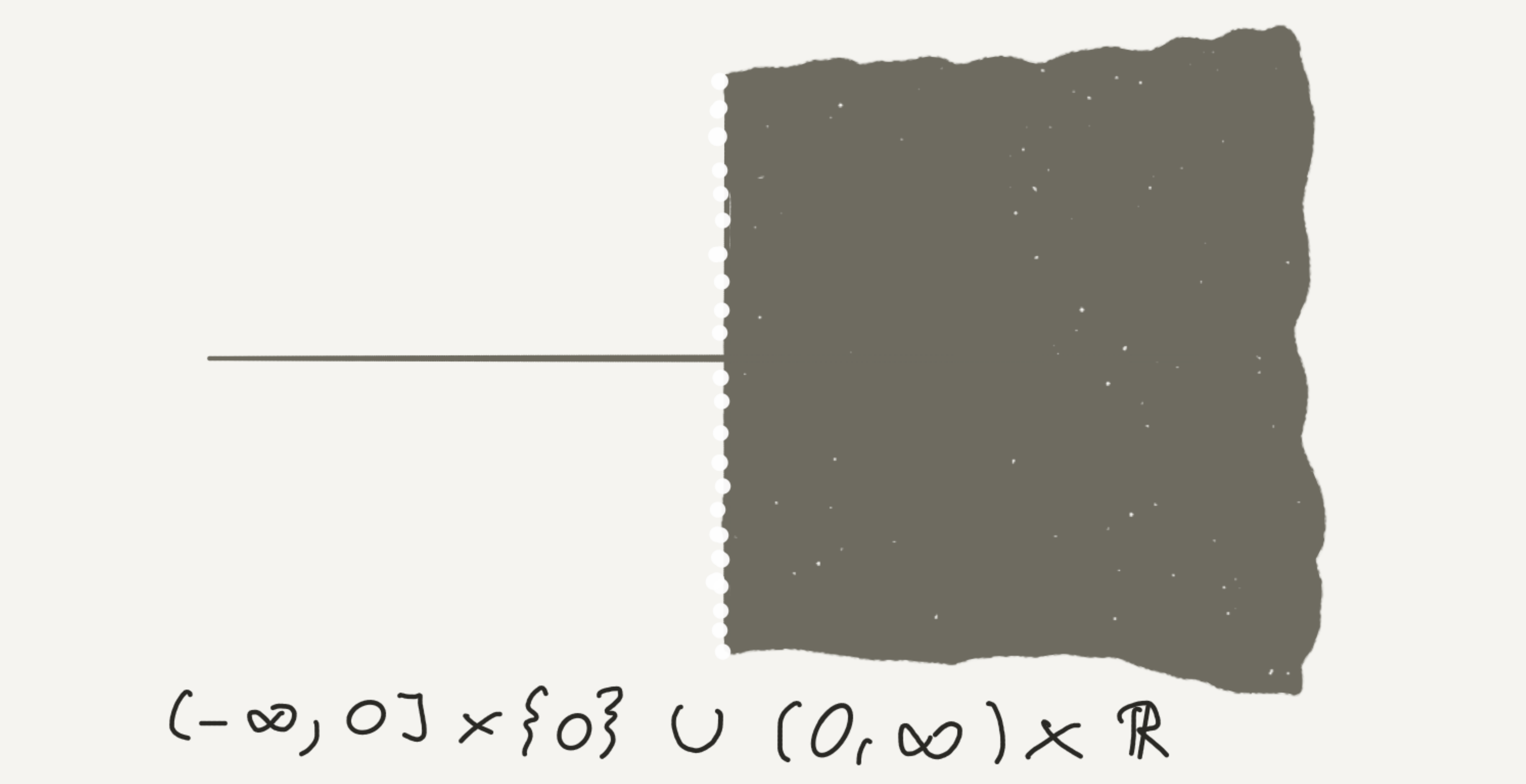}
  \caption{Jumping dimension along $\SC$-retract
                $O=\im r_\pi$ in Section~\ref{sec:splicing-example}}
  \label{fig:fig-sc-retract-schippe}
\end{figure}
How can one ever do analysis on such spaces? Let's see:

\vspace{.1cm}
\noindent
\textbf{Decompression.} 
To start with, given a map $f\colon O\to O^\prime$ between $\SC$-retracts,
one can ``decompress'' or ``unpack'' the, possibly ``cornered'',
domain $O$ of the map $f$ into an open set by pre-composing with an
$\SC$-smooth retraction $r\colon U\to U$ whose image is $O=r(U)$.
Indeed the map $f\circ r\colon U\to O^\prime\subset U^\prime$ 
has the same image as $f$, but lies within the reach of
$\SC$-calculus since domain and target are (relatively) open subsets
of partial quadrants $C$ and $C^\prime$ in $\SC$-Banach spaces.

\begin{definition}[Sc-smooth maps among $\SC$-retracts --
decompress domain]
\label{def:sc-map-between-sc-retracts}
A map $f\colon O\to O^\prime$ between $\SC$-retracts is called an
\textbf{\boldmath$\SC$-smooth retract map}
if\index{sc-smooth!retract map}\index{retract map!decompressing a --}
the composition $f\circ r\colon U\to U^\prime$
is $\SC$-smooth\footnote{
  Sc-smoothness of $f\circ r\colon U\to U^\prime$ implies
  continuity of $f\colon O\to O^\prime$.
  }
for some, thus by Lemma~\ref{le:for-some-thus-every-123} for every,
$\SC$-smooth retraction $r$ whose image is
$O=r(U)$.\index{$f\colon O\to O^\prime$ $\SC$-smooth retract map}
Let us refer to such pre-composition
process\index{decompressing domain}
as \textbf{decompressing (the domain of) \boldmath$f$}.
${\text{Sc-smooth}}$ retract maps $O\to O^\prime$ are continuous.
\end{definition}

\begin{lemma}\label{le:for-some-thus-every-123}
Given $\SC$-smooth retractions with equal image $O=r(U)=s(V)$ and a
map $f\colon O\to U^\prime$, then if one of the maps
\[
     f\circ r\colon U\to U^\prime,\qquad
     f\circ s\colon V\to U^\prime
\]
is $\SC$-smooth, so is the other one.
\end{lemma}

\begin{proof}
By assumption $\im r=O=\Fix\, s$ and $\im s=O=\Fix\, r$, hence
\[
     s\circ r=r\colon U\to O,\qquad
     r\circ s=s\colon V\to O.
\]
By hypothesis $s$ is $\SC$-smooth. If also $f\circ r$ is $\SC$-smooth,
so is by the chain rule their composition $(f\circ r)\circ s$.
But $f\circ (r\circ s)=f\circ s$.
\end{proof}

\subsubsection*{Tangent bundle of sc-retracts and tangent map
of sc-retract maps}

\begin{lemma}[Tangent map of retraction is itself a retraction]
\label{le:Tr}
Let $r\colon U\to U$ be an $\SC$-smooth retraction
on an $\SC$-triple $(U,C,E)$.
Then its tangent map
\begin{equation}\label{eq:Tr}
     Tr=(Tr)\circ (Tr)\colon TU\to TU,\quad
     (x,\xi)\mapsto\left(r(x),Dr(x)\xi\right)
\end{equation}
is an $\SC$-smooth retraction
on\index{sc-triple!tangent --}
the \textbf{tangent \boldmath$\SC$-triple}
\[
     T \left(U,C,E\right)
     :=\left(TU,TC,TE\right)
     :=\left(U^1\oplus E^0,C^1\oplus E^0,E^1\oplus E^0\right).
\]
\end{lemma}

\begin{proof}
The tangent map of an $\SC$-smooth map is
$\SC$-smooth by the iterative definition of $\SC^k$-smoothness;
see Definition~\ref{def:sc^k}.
It remains to show that
\[
     \im Tr=\Fix\, Tr.
\]
'$\supset$' A fixed point $x=f(x)$ of a map lies in its image.
'$\subset$' An element of $Tr(TU)$ is of the form
$(y,{\color{brown}\eta})=(r(x), {\color{brown} Dr|_x\,\xi})$ for some
$(x,\xi)\in U_1\oplus E^0$.
Hence $r(y)=r\circ r(x)=r(x)=y$ and
\[
     Dr|_y\, {\color{brown}\eta}=Dr|_{r(x)}\,{\color{brown}Dr|_x\,\xi}
     =D(r\circ r)|_x\,\xi=Dr|_x\,\xi=\eta
\]
where we used the chain rule~(\ref{eq:chain-rule-D}).
Hence $Tr(y,\eta)=(y,\eta)\in\Fix\, Tr$.
\end{proof}

\begin{lemma}[Tangent maps of two decompressions have same image]
\label{le:TO}
Assume an $\SC$-smooth retract $(O,C,E)$
is the image of two $\SC$-smooth retractions
\[
     r(U)=O=s(V),\qquad
     r=r\circ r\colon U\to U,\quad
     s=s\circ s\colon V\to V.
\]
Then both tangent maps have equal image
$Tr(TU)=Ts(TV)\subset(O^1\oplus E^0)$.
\end{lemma}

\begin{proof}
We need to show $\Fix\, Tr=\Fix\, Ts$.
'$\subset$' Pick $(x,\xi)\in \Fix\, Tr=Tr(TU)\subset TU\subset E^1\oplus E^0$.
Then ${\color{brown}\xi\in\Fix\, Dr|_x}\subset E$
and with~(\ref{eq:O^k=r(U^k}) for $r$ and $s$ we conclude
\[
     x\in(\Fix\, r)\CAP E^1=\Fix\,[r\colon U^1\to U^1]=O^1
     =\Fix\,[s\colon V^1\to V^1]\subset V^1.
\]
So $(x,\xi)\in V^1\oplus E^0=TV$ lies in the domain of $Ts$ and
we get that
\begin{equation*}
\begin{split}
     Ts(x,\xi)=\left(s(x),Ds|_x\, {\color{brown}\xi}\right)
   &=\left(s(r(x)),Ds|_{r(x)}\,{\color{brown}Dr|_x\,\xi}\right)\\
   &=\left(r(x),D(s\circ r)|_x\,\xi\right)\\
  &=\left(x,{\color{brown}Dr|_x\,\xi}\right)\\
   &=\left(x,{\color{brown}\xi}\right).
\end{split}
\end{equation*}
Here we used twice the identity $s\circ r=r$
which holds since $r(U)=O=\Fix\, s$ by hypothesis.
'$\supset$' Same argument.
\end{proof}

\begin{definition}[Tangent of $\SC$-retract]\label{def:TO}
If $O=r(U)$ is an $\SC$-retract, then
\[
     TO:=Tr(TU)=\Fix\, Tr\subset \left( O^1\oplus E^0\right)
\]
is an $\SC$-retract, too. Notation $T(O,C,E):=(TO,TC,TE)$.
The definition of $TO$ does not depend
on the choice of $(r,U)$ by Lemma~\ref{le:TO}.
\end{definition}

\begin{lemma}[Tangent bundle of $\SC$-retract $O$]
The natural projection
\[
     p\colon TO\to O^1,\quad (x,\xi)\mapsto x
\]
is\index{tangent bundle! of $\SC$-smooth retract}
an\index{sc-smooth!retract!tangent bundle of --}
open surjective $\SC$-smooth retract map,
cf.\index{$TO\to O^1$ tangent bundle of $\SC$-smooth retract $O$}
Definition~\ref{def:sc-map-between-sc-retracts},
called \textbf{tangent bundle of the
\boldmath$\SC$-retract \unboldmath$\mbf{O}$}.
The pre-image of a point, denoted~by
\[
     T_x O:=p^{-1}(x)\subset E,   
     \qquad x\in O^1
\]
is a Banach subspace of $E$, an $\SC$-subspace whenever $x\in O_\infty$.
\end{lemma}

\begin{proof}
Let $O=r(U)$. Then $TO:=Tr(TU)$ where $TU:=U^1\oplus E^0$. Hence the
first component of $Tr$ is the map $r$ on the domain $U^1$, see~(\ref{eq:Tr}).
But $r(U^1)=O^1$ by Lemma~\ref{le:O^k} which proves surjectivity of $p$.
The decompression
\[
     p\circ Tr\colon TU=U^1\oplus E^0\to TO\to U^1,\quad
     \left(x,\xi\right)\mapsto \left(r(x),Dr|_x\,\xi\right)\mapsto r(x)
\]
of $p$ is constant in $\xi$, and in $x$ it is the map $r\colon U^1\to U^1$
which is $\SC$-smooth by Lemma~\ref{le:lifting-indices},
since $r\colon U\to U$ is $\SC$-smooth by assumption.
The pre-image
\begin{equation*}\label{eq:T_xO-full}
     T_xO:=p^{-1}(x)=\{x\}\oplus \Fix\,[Dr(x)\colon E\to E]
     \subset O^1\oplus E
\end{equation*}
is the fixed point set of a linear operator on the Banach space $E$
and therefore it is a linear subspace. It is a closed linear subspace,
because the linear operator is continuous.
For simplicity we shall simply write
\begin{equation}\label{eq:T_xO}
     T_xO:=p^{-1}(x)=\Fix\,[Dr(x)\colon E\to E]\subset E.
\end{equation}
The $\SC$-derivative $Dr(x)\colon E\to E$ at any $x\in O_\infty$
restricts to a continuous linear operator on every level $E_m$
by Corollary~\ref{cor:sc-deriv_is_sc-op_smooth-pts}.
Hence $(p^{-1}(x))_m:=p^{-1}(x)\CAP E_m$ 
are the levels of a Banach scale by Exercise~\ref{exc:jhjghj8}.
\end{proof}

\begin{exercise}\label{exc:jhjghj8}
a) Show that $T_xO$ is an $\SC$-subspace of $E$ whenever $x\in O_\infty$.
b)~Show that the projection $p\colon TO\to O^1$ is an open map.

\vspace{.1cm}\noindent
[Hint: a) Let $O=r(U)$. Show that
$F_m:=(p^{-1}(x))_m:=p^{-1}(x)\CAP E_m$ equals
$F_m=\Fix\,[Dr(x)\colon E_m\to E_m]$ and
$F=F_0\supset F_1\supset\dots$ satisfies the three axioms
\texttt{(Banach levels)}, \texttt{(compactness)}, and
\texttt{(density)} of a Banach scale.
b)~First consider the case $C=E$, decompress $p$.]
\end{exercise}

\begin{lemma}\label{le:gyu6768}
Let $f\colon O\to O^\prime$ be an $\SC$-smooth retract map.
If $r\colon U\to U$ and $s\colon V\to V$ are 
$\SC$-smooth retractions with image  $O$,
then the restrictions
\[
     T(f\circ r)|_{TO}=T(f\circ s)|_{TO}\colon TO\to TO^\prime
\]
a) coincide and b) take values in $TO^\prime$
and c) are $\SC$-smooth retract maps.
\end{lemma}

\begin{proof}
a) For $(x,\xi)\in TO={\color{brown}\Fix\, Ts}$,
as $r\circ s=s$ ($\Fix\, r=O=\im s$), we get
\[
     T(f\circ r)\,{\color{brown}(x,\xi)}
     =T(f\circ r)\, {\color{brown} Ts(x,\xi)}
     =T(f\circ r\circ s)\, (x,\xi)
     =T(f\circ s)\, (x,\xi).
\]
b) Let $O^\prime=t(W)$, then it suffices to show
$\im T(f\circ r)\subset\Fix\, Tt$. Observe that $t\circ f=f$ since
$\im f\subset O^\prime=\Fix\,t$. Hence $(x,\xi)\in TU$ provides 
a {\color{cyan}fixed~point}
\[
     Tt\left({\color{cyan}T(f\circ r)\,(x,\xi)}\right)
     =T(t\circ f\circ r) \,(x,\xi)
     ={\color{cyan}T(f\circ r) \,(x,\xi)}.
\]
c) The decompression of $T(f\circ r)|_{TO}$ given by
\[
     T(f\circ r) \circ Tr
     =T(f\circ r\circ r)=T(f\circ r)
     \colon TU\to TU^\prime
\]
is $\SC$-smooth, because $f\circ r\colon U\to U^\prime$ is $\SC$-smooth due
to the assumption that $f\colon O\to O^\prime$ is an $\SC$-smooth retract map,
see Definition~\ref{def:sc-map-between-sc-retracts}.
\end{proof}

\begin{definition}[Tangent of retract maps via domain
decompression]
\label{def:Tf-sc-retracts}
The \textbf{tangent map} of an\index{$Tf\colon TO\to TO^\prime$}
$\SC$-smooth retract map $f\colon O\to O^\prime$ is the restriction
\[
     Tf:=T(f\circ r)|_{TO}\colon Tr(TU)=TO\to TO^\prime,\quad
     (x,\xi)\mapsto \left(f(x),Df(x)\xi\right)
\]
of the tangent map $T(f\circ r)\colon TU\to TU^\prime$
for a decompression $r$ of $O=r(U)$.
\end{definition}

Some remarks are in order.
Firstly, by Lemma~\ref{le:gyu6768} the definition of $Tf$
does not depend on the $\SC$-smooth retraction $r\colon U\to U$
with image $O$.
Secondly, concerning component one $f\circ r(x)=f(x)$ since
$x\in O^1\subset O=\Fix\, r$.
Thirdly, concerning component two
\[
     D(f\circ r)|_x
     =Df|_{r(x)}\circ Dr(x)
     =Df(x)\colon T_xO
     \to T_{f(x)}O^\prime
\]
since $T_x O$ is the fixed point set of $Dr(x)$.

\begin{theorem}[Chain rule for $\SC$-smooth retract maps]
\label{thm:sc-retract-chain rule}
Let $f\colon O\to O^\prime$ and $g\colon O^\prime\to O^{\prime\prime}$
be $\SC$-smooth retract maps. Then the composition
$g\circ f\colon O\to O^{\prime\prime}$ is also a $\SC$-smooth 
retract map and the tangent maps satisfy
\[
     T(g\circ f)=Tg\circ Tf\colon TO\to TO^{\prime\prime}.
\]
\end{theorem}

\begin{proof}
Sc-smoothness of the retract maps $f$ and $g$ by definition
means $\SC$-smoothness of $f\circ r$ and $g\circ r^\prime$
where $r\colon U\to U$ and $r^\prime\colon U^\prime\to U^\prime$ are $\SC$-smooth
retractions with images $O$ and $O^\prime$, respectively. The inclusion
$\im f\subset O^\prime=Fix\, r^\prime$ provides the identity $f=r^\prime\circ f$.
Hence $(g \circ f)\circ r=(g\circ r^\prime)\circ(f\circ r)$ is a composition of two
$\SC$-smooth maps, so it is $\SC$-smooth itself by the chain
rule for $\SC$-smooth maps, Theorem~\ref{thm:chain-rule}.
By definition of $Tf$, the chain rule, and $f=r^\prime\circ f$ we get
$
     Tg \circ Tf
     =T(g\circ r^\prime)\circ T(f\circ r) |_{TO}
     =T(g\circ f\circ r) |_{TO}
     =:T(g\circ f)
$.
\end{proof}

In Section~\ref{sec:M-polyfolds}
the next exercise will be useful a) to show
that open subsets of M-polyfolds are M-polyfolds
and b) to construct sub-M-polyfold charts.

\begin{exercise}\label{exc:open-sets-in-sc-retracts}
Given an $\SC$-retract $(O,C,E)$, prove the following.
\begin{itemize}
\item[a)]
  Open subsets $O^\prime$ of the $\SC$-retract $O$ are $\SC$-retracts in $C$.
\item[b)]
  Suppose $V$ is an open subset of $O$ and $s=s\circ s\colon V\to V$
  is an idempotent $\SC$-smooth retract map.
  The image of such $s$ is an $\SC$-retract $(o,C,E)$.
\end{itemize}

\vspace{.1cm}\noindent
[Hints: Let $O=\im[r\colon U\to U]$. a) How about $U^\prime:=r^{-1}(O^\prime)$ and
$r^\prime:=r|_{U^\prime}$\,?
b) Let $o:=\im s=\Fix\, s$, then $o\subset s^{-1}(o)=V\subset O=\Fix\, r$.
How about $U^\prime:=r^{-1}(V)=r^{-1}(s^{-1}(o))\subset U$
and the $\SC$-smooth map $s\circ r\colon  U^\prime\to U^\prime$\,?]
\end{exercise}

\subsection{Special case: Splicings and splicing cores }

Following~\citet[Def.\,2.18]{Hofer:2017a},
an\index{splicing!sc-smooth --}\index{sc-smooth!splicing}
\textbf{\boldmath$\SC$-smooth splicing on an $\SC$-Banach space $E$}
consists of the following data. A relatively open neighborhood $V$
of $0$ in a partial quadrant $[0,\infty)^\ell\times\R^{d-\ell}$ in
$\R^d$ and a family $\{\pi_v\}_{v\in V}$ of $\SC$-projections
$\pi_v=\pi_v\circ\pi_v\in\Llsc(E)$ such that the map
\[
     \pi\colon \R^d\oplus E\supset V\oplus E\to E,\quad
     (v,f)\mapsto \pi_v f
\]
is $\SC$-smooth. Note that in this case each projection
$\pi_v$ restricts to a continuous linear operator
$\pi_v|_{E_m}\in\Ll(E_m)$ on every level. But in the operator norm
these operators do not, in general, depend continuously on $v$.
The subset of $\R^d\oplus E$ composed of the images (fixed points) of each
projection, i.e.
\[
     K^{\pi}:=\bigcup_{v\in V}\{v\}\times\im\pi_v
     =\{(v,f)\in V\oplus E\mid \pi_vf=f\}
\]
is called the\index{splicing!core}
\textbf{splicing core} of the splicing.

\begin{exercise}[Induced sc-smooth retraction]
Given an $\SC$-smooth splicing $\{\pi_v\}_{v\in V\subset\R^d}$
on an $\SC$-Banach space $E$, consider the map given by
\[
     r_\pi\colon V\oplus E\to V\oplus E,\quad
     (v,f)\mapsto\left(v,\pi_v f\right).
\]
Show that the map $r_\pi$ defines an $\SC$-smooth retraction on the $\SC$-triple
$(V\oplus E, ([0,\infty)^\ell\times\R^{d-\ell})\oplus E,\R^d\oplus E)$
and that its image is the splicing core $K^\pi$.
\end{exercise}

As remarked in~\citet[previous to Def.\,5.6]{Fabert:2016a}, this setup of
splicing with finitely many ``gluing'' parameters covers the $\SC$-retractions
relevant for Morse theory and holomorphic curve moduli spaces.

\subsection{Splicing core with jumping finite dimension}
\label{sec:splicing-example}

Fix a smooth  bump function $\beta\ge 0$ on $\R$ supported in $[-1,1]$ of
unit~$L^2$~norm. For $t>0$ consider the family
$\beta_t(s):=\beta(s+e^{1/t})$ of left translates of $\beta$ by
$e^{1/t}$ -- huge left translations for $t$ near $0$
and almost no translation for $t\sim\infty$.
Fix a strictly increasing sequence of reals $\delta_m$
starting at $\delta_0=0$ and let $E=L^2(\R)$ be the $\SC$-Hilbert
space whose levels are given by the weighted Sobolev spaces
$E_m:=W^{m,2}_{\delta_m}(\R)$ introduced in Exercise~\ref{exc:Wkp_R}.
Consider the family 
\begin{equation*}
\begin{split}
     \{\pi_t\colon E\to E\}_{t\in\R},\qquad
     \pi_t f:=\begin{cases}
       0&\text{, $t\le 0$}\\
       \INNER{f}{\beta_t}\beta_t&\text{, $t>0$}
     \end{cases}
\end{split}
\end{equation*}
of linear operators on $E$. Note that the image of $\pi_t$ is $\{0\}$
whenever $t\le 0$, whereas for each $t>0$ the image of $\pi_t$ is
$\R \beta_t$, hence one dimensional.

\begin{exercise}
Check that each linear operator $\pi_t\colon E\to E$ is continuous
and a projection, that is $\pi_t\circ\pi_t=\pi_t$.
%
\end{exercise}

\begin{proposition}\label{prop:spiked-half-plane}
The map $\pi\colon \R\oplus E\to E$, $(t,f)\mapsto \pi_t f$, is $\SC$-smooth.
\end{proposition}

\begin{proof}
The result and details of the (hard) proof of $\SC$-smoothness are given
in \citet[Ex.\,1.22 and Le.\,1.23]{Hofer:2010b}; see
also~\citet[Prop.\,6.8]{Cieliebak:2018a}.
\end{proof}

To summarize, the family $\{\pi_t\}_{t\in\R}$ of projections defines
an $\SC$-smooth splicing on $E=L^2$. The corresponding splicing core
$K^\pi\subset \R\oplus E$ is represented in
Figure~\ref{fig:fig-sc-retract-schippe} as a subset of $\R^2$
homeomorphic to $K^\pi=\im r_\pi=r_\pi(\R\oplus E)$.
Although connected, there are parts of dimension one and two.

\section{M-polyfolds and their tangent bundles}
\label{sec:M-polyfolds}
\sectionmark{M-polyfolds}

M-polyfolds are defined analogous to $\SC$-manifolds, just
use as local models instead of $\SC$-triples $(U,C,E)$
$\SC$-retracts $O=\im[r=r^2\colon U\to U]$ in $C\subset E$.

Recall two standard methods to define
manifolds.\index{manifolds!methods to define --}
\textbf{Method 1} starts with a topological space $X$,
then one defines a collection of homeomorphisms to open sets in
model Banach spaces, whose domains are open subsets of $X$ which
together cover $X$. The collection must be suitably compatible on overlaps.
\textbf{Method 2} starts with \emph{only a set} $S$, then one defines a
collection of bijections between subsets of $S$ onto open subsets
of local model Banach spaces, again the domains together must cover
$S$. Now one uses the bijections to define a topology on the set $S$,
essentially by declaring pre-images of open sets in model space to be
open sets in $S$.

In practice one often employs Method 1 to define a manifold $X$.
Then one employs Method 2 in order to define the tangent bundle $TX$.
Namely, \emph{as a  set} called $TX$ of equivalence classes
whose definition utilizes the manifold charts of $X$ and their tangent
maps. The latter are used to define the required bijections that endow
the set $TX$ of equivalence classes with a topology.

\subsubsection*{M-polyfolds and maps between them}

\begin{figure}
  \centering
  \includegraphics
                             [height=4.5cm]
                             {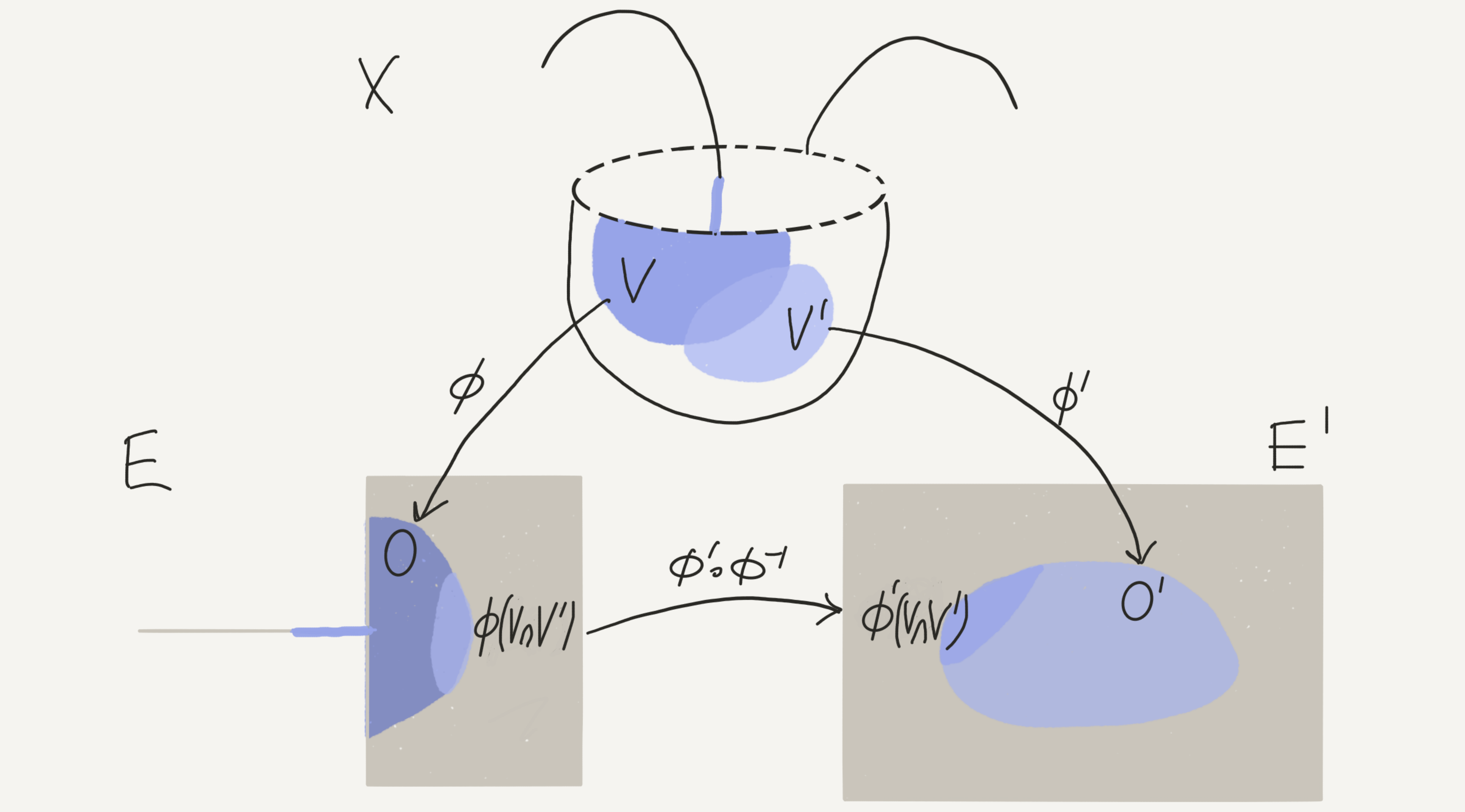}
  \caption{Transition map between M-polyfold charts of M-polyfold $X$}
  \label{fig:fig-M-polyfold-transition}
\end{figure}

\begin{definition}\label{def:M-polyfold-charts}
Let $X$ be a topological space.
An\index{M-polyfold!charts}
\textbf{M-polyfold chart} $(V,\phi,(O,C,E))$, often abbreviated
$(V,\phi,O)$, consists of
\begin{itemize}
\item
  an open set $V$ in $X$;
\item
  an $\SC$-retract $O=r(U)=\Fix\, r$ in a partial quadrant $C\subset E$;
\item
  a homeomorphism $\phi\colon V\to O$
  (open sets in $\SC$-retracts are $\SC$-retracts).
\end{itemize}
Two M-polyfold charts\index{M-polyfold charts!compatible --}
are \textbf{\boldmath$\SC$-smoothly compatible}
if\index{M-polyfold charts!transition map}
the \textbf{transition map}
\[
     \psi:=\phi^\prime\circ\phi^{-1}\colon O\supset\phi(V\CAP V^\prime)
     \to\phi^\prime(V\CAP V^\prime)\subset O^\prime
\]
and its inverse are both $\SC$-smooth retract maps\footnote{
  Indeed open subsets of $\SC$-retracts are $\SC$-retracts
  by Exercise~\ref{exc:open-sets-in-sc-retracts}.
  }
(i.e. $\SC$-smooth after domain decompression).
An\index{M-polyfold!atlas}
\textbf{M-polyfold atlas for \boldmath$X$} is a collection
$\Aa$ of pairwise $\SC$-smoothly compatible
M-polyfold charts $\phi\colon V\to O$ whose domains cover~$X$.
Two atlases are called \textbf{equivalent}
if\index{M-polyfold!atlases!equivalent}
their union is again an M-polyfold atlas.
\end{definition}

\begin{definition}\label{def:M-polyfold}
An \textbf{\Index{M-polyfold}} is a paracompact
Hausdorff space $X$ endowed with an equivalence class
of M-polyfold atlases.
\end{definition}

\begin{definition}\label{def:M-polyfold-map}
A map $f\colon X\to X^\prime$ between M-polyfolds is
called\index{M-polyfold map!sc-smooth --}
an \textbf{\boldmath$\SC$-smooth M-polyfold map}
if every local M-polyfold chart representative
\[
     \phi^\prime\circ f\circ \phi^{-1}\colon O\supset\phi(V\CAP f^{-1}V^\prime)
     \to\phi^\prime(V^\prime)\subset O^\prime
\]
of $f$ is an $\SC$-smooth retract map.
An \textbf{\boldmath$\SC$-smooth diffeomorphism}
between M-polyfolds\index{M-polyfold!diffeomorphism}
\begin{figure}
  \centering
  \includegraphics
                             [height=4.5cm]
                             {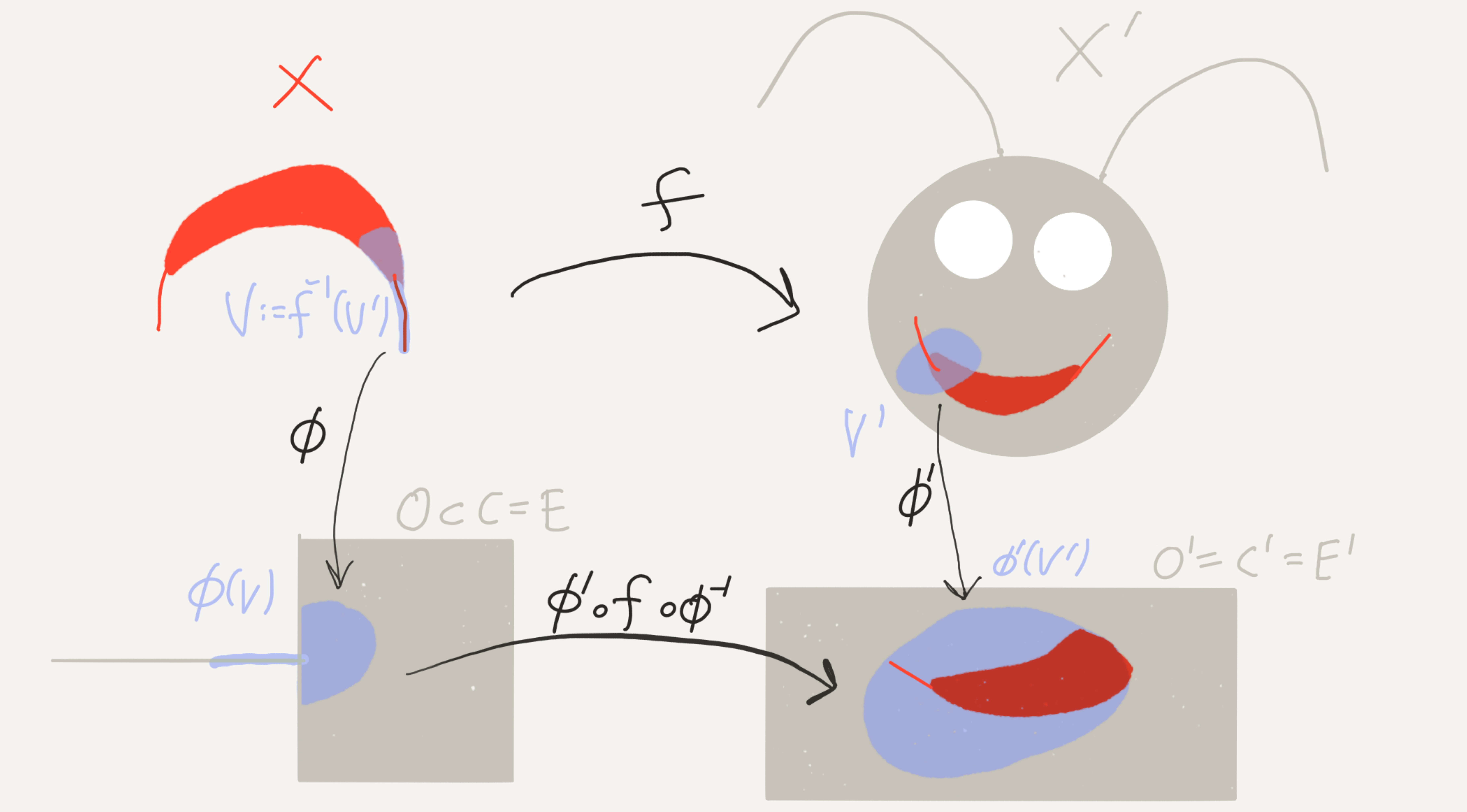}
  \caption{Freedom of speech among M-polyfolds --  local representative ;-)}
  \index{freedom of speech}\index{M-polyfold map!``freedom of speech''}
  \label{fig:fig-M-polyfold-map-representative}
\end{figure}
is a bijective $\SC$-smooth map between M-polyfolds
whose inverse is also $\SC$-smooth.
\end{definition}

\begin{exercise}\label{exc:open-subsets-M-pfs}
a) Sc-manifolds are M-polyfolds.
\\
b) Open subsets of M-polyfolds are M-polyfolds.
\\
c) Check that $X$ and $X^\prime$ in
Figure~\ref{fig:fig-M-polyfold-map-representative}
are M-polyfolds.
\\
d) Use the $\SC$-retract $O=\im r_\pi$ in
Figure~\ref{fig:fig-sc-retract-schippe} to built
further fun M-polyfolds.
\\
e) It
is\index{open problem@{{\color{red}open problem}}}
an {\color{red} open problem},
\citet[Quest.\,4.1]{Hofer:2017a},
whether there is an $\SC$-smooth retract
$(O,E,E)$ so that $O$ is homeomorphic to the letter
\textsf{T} in $\R^2$.
\end{exercise}

\begin{definition}\label{def:M-polyfold-levels}
One defines \textbf{level \boldmath$m$} of an M-polyfold $X$
to be the set $X_m$\index{M-polyfold!levels}
that consists of all points $x\in X$ which are mapped
to level $m$ in some, hence any,\footnote{
  transition maps are $\SC$-smooth, thus level preserving
  }
M-polyfold chart. 
\end{definition}

Thus for an M-polyfold $X$ there is the nested sequence of levels
\[
     X=X_0\supset X_1\supset\dots\supset X_\infty:=\bigcap_{m\ge 0} X_m.
\]
Each level $X_m$ inherits the structure of an M-polyfold,
notation $X^m$, see~\citet[p.\,21]{Hofer:2017a}
for charts, and each inclusion $X_{m+1}\INTO X_m$ is
continuous (as a map between topological spaces),
see~\citet[Le.\,2.1]{Hofer:2017a}.

\subsubsection*{Construction of the M-polyfold tangent bundle}

\textsc{The base M-polyfold.}
Let $X$ be an M-polyfold, in particular, a paracompact Hausdorff space,
with M-polyfold atlas $\Aa=\{(V_i,\phi_i,(O_i,C_i,E_i))\}_{i\in I}$.

\smallskip\noindent
\textsc{The tangent bundle as a set.}
By definition $TX$ is the set of equivalence classes of tuples
$(x,V,\phi,(O,C,E),\xi)$, abbreviated $(x,\phi,\xi)$, that consist of
\begin{itemize}
\item
  a point $x\in X_1$ on level 1;
\item
  an M-polyfold chart $\phi\colon X\supset V\to O=r(U)\subset C\subset E$
  for $X$ about $x$;
\item
  a tangent vector $\xi\in T_{\phi(x)}O=\Fix\,[Dr|_{\phi(x)}\in\Ll(E)]$,
  see~(\ref{eq:T_xO}).
\end{itemize}
Two tuples $(x,V,\phi,O,\xi)\sim
(x^\prime,V^\prime,\phi^\prime,O^\prime,\xi^\prime)$
are said \textbf{equivalent} if
\[
     x=x^\prime,\qquad
     Dr(\phi^\prime\circ \phi^{-1})|_{\phi(x)}\xi=\xi^\prime.
\]

\smallskip\noindent
\textsc{The natural projection.}
There is a \textbf{natural projection}
\begin{equation}\label{eq:p:TX->X^1}
     p\colon TX\to X^1,\quad
     [x,\phi,\xi]\mapsto x.
\end{equation}
The pre-image of any point $x\in X^1$, denoted by
\[
     T_xX:=p^{-1}(x)
\]
and called the \textbf{tangent space of
\boldmath$X$ at $x$},\index{$T_xX$ M-polyfold tangent space}
is a linear space over the reals:
\[
     \lambda[x,\phi,\xi]+\mu[x,\phi,\eta]:=[x,\phi,\lambda\xi+\mu\eta]
     ,\qquad\lambda,\mu\in\R.
\]
To represent the two input equivalence classes choose the same
M-polyfold chart $\phi$ about $x$ for both of them (choose any two
representatives and restrict to the intersection of their
domains). This way $\xi$ and $\eta$ are both in the same vector space,
here $T_{\phi(x)}O$, and so adding them makes sense.

\smallskip\noindent
\textsc{The induced bijections.}
For every M-polyfold chart $\phi\colon V\to O$, say where $O=\im r=\Fix\, r$
for some $\SC$-smooth retraction $r$ on an $\SC$-triple $(U,C,E)$,
the map named and defined by
\begin{equation}\label{eq:T-phi}
    T\phi\colon TV:=p^{-1}(V\CAP X_1)\to TO=\Fix\, Tr,\quad
     [x,\phi,\xi]\mapsto \left(\phi(x),\xi\right)
\end{equation}
is a bijection.
For a given level 1 point $x\in V_1:=V\CAP X_1$
the map
\[
     T_x\phi:=T\phi(x,\cdot)\colon T_x X=T_x V=p^{-1}(x)\to T_{\phi(x)}O,\quad
     [x,\phi,\xi]\mapsto\xi
\]
is a bijection (the identity) on the Banach subspace
$T_{\phi(x)}O=\Fix\, Dr|_{\phi(x)}$ of~$E$; cf.~(\ref{eq:T_xO}).
So $T_x X$ inherits the Banach space structure of $T_{\phi(x)}O$.
At smooth points $T_{\phi(x)}O$ is an $\SC$-subspace of $E$ by
Exercise~\ref{exc:jhjghj8}, so the linear bijection
$T_x\phi$ endows $T_x X$ with the structure of an $\SC$-Banach space.

\smallskip\noindent
\textsc{The induced topology.}
Consider the collection $\Bb$ that consists of all
subsets of $TX$ that are pre-images under $T\phi$
of all open subsets in the target space $TO$,
for all M-polyfold charts $\phi\colon V\to O$ of $X$, in symbols
\[
     \Bb:=\left\{
     (T\phi)^{-1}W\mid
     \text{$(V,\phi,O)\in\Aa$, $W\subset TO$ open}
     \right\}\subset 2^{TX}.
\]

\begin{exercise}
Show that $\Bb$ is a basis for a topology; cf.
Theorem~\ref{thm:basis-some-topology}.
\end{exercise}

By definition the \textbf{topology on \boldmath$TX$}
is the topology generated by the basis~$\Bb$: The open sets in $TX$
are arbitrary unions of members of $\Bb$.\index{topology!on $TX$}

\begin{proposition}
The topology on $TX$ is Hausdorff and paracompact.
\end{proposition}

\begin{proof}
\citet[\S\,2.6.3]{Hofer:2017a}
\end{proof}

\begin{exercise}
The map $p\colon TX\to X^1$ in~(\ref{eq:p:TX->X^1}) is continuous and open.
\end{exercise}

\smallskip\noindent
\textsc{The M-polyfold charts.}
For any M-polyfold chart $\phi\colon V\to O$ of $X$, where $O=r(U)$ say,
the bijection $T\phi\colon TV\to TO$ defined by~(\ref{eq:T-phi})
is an M-polyfold chart for $TX$:
\begin{itemize}
\item
  $TV\subset TX$ is open by definition of $\Bb$;
\item
  $TO=Tr(TU)$ is an $\SC$-retract by Definition~\ref{def:TO};
\item
  $T\phi\colon TV\to TO$ is a homeomorphism by definition of $\Bb$.
\end{itemize}
Furthermore, if $\phi,\phi^\prime\in\Aa$ are compatible for $X$,
then $T\phi,T\phi^\prime$ are compatible for $TX$:
We need to show that the map
$(T\phi^\prime)\circ (T\phi)^{-1}\colon TO\to TO^\prime$ given by
\begin{equation*}
\begin{split}
     (\phi(x),\xi)
   &\mapsto[x,\phi,\xi]
     =\left[x,\phi^\prime,D(\phi^\prime\circ\phi^{-1})_{\phi(x)}\,\xi\right]\\
   &\mapsto
     \left(\phi^\prime(x),D(\phi^\prime\circ\phi^{-1})_{\phi(x)}\,\xi\right)
     =T(\phi^\prime\circ\phi^{-1})\,(x,\xi)
\end{split}
\end{equation*}
is an $\SC$-smooth retract map. But
$\phi^\prime\circ\phi^{-1}$ is an $\SC$-smooth retract map by the
chain rule, Theorem~\ref{thm:sc-retract-chain rule}, and so is the
tangent map. 
This\index{$T\Aa$ M-polyfold atlas for $TX$}
shows that an M-polyfold atlas $\Aa$ for $X$ induces an
\textbf{M-polyfold atlas for \boldmath$TX$}, namely
\[
     T\Aa:=\left\{\left(TV,T\phi,(TO,TC,TE)\right)\mid
     \left(V,\phi,(O,C,E)\right)\in\Aa\right\}.
\]

Let us then summarize the previous constructions and findings in form of

\begin{theorem}
Let $X$ be an M-polyfold. Then $TX$ is an M-polyfold and
\[
     p\colon TX\to X^1
\]
is an $\SC$-smooth
map between M-polyfolds.
\end{theorem}

\subsubsection*{M-polyfold tangent maps}

\begin{definition}
The \textbf{tangent map} of an $\SC$-smooth M-polyfold map
$f\colon X\to Y$\index{tangent map!of M-polyfold map}
is the $\SC$-smooth M-polyfold map defined by
\[
     Tf\colon TX\to TY,\quad
     [x,\phi,\xi]\mapsto
     \left[f(x),\psi,D(\psi\circ\phi^{-1})|_{\phi(x)}\,\xi\right]
\]
where $\psi$ is any M-polyfold chart about $f(x)$.
\end{definition}

\begin{exercise}
Show that $Tf\colon TX\to TY$ is $\SC$-smooth as a map
between M-polyfolds.
Show that for $x\in X_1$ the map
\[
     T_xf\colon T_x X\to T_{f(x)}Y,\quad
     v:=[x,\phi,\xi]\mapsto
     \left[f(x),\psi,D(\psi\circ\phi^{-1})|_{\phi(x)}\,\xi\right]
\]
is a continuous linear operator and 
$T_xf$ is an $\SC$-operator whenever $x\in X_\infty$.
\end{exercise}

\subsection{Sub-M-polyfolds}

An M-polyfold $X$ is locally modeled on the images $O$
of $\SC$-smooth retractions $r=r^2\colon E\supset U\to U$
in an $\SC$-Banach space $E$.
Thus it is natural to define a \emph{sub-M-polyfold}
of $X$ as a subset $A\subset X$
that is locally the image of an $\SC$-smooth retraction $r=r^2\colon V\to V$
acting on an open subset $V$ of $X$.

\begin{definition}
A subset $A\subset X$ of an M-polyfold\index{M-polyfold!sub- --}
is called a \textbf{\Index{sub-M-polyfold}}
if around any point $a\in A$ there is an open neighborhood
$V\subset X$ and an $\SC$-smooth retraction
$r=r^2\colon V\to V$ such that $A\CAP V=r(V)=\Fix\, r$.
Such $r$ is called a \textbf{local generator}
for\index{local generator!of sub-M-polyfold}
the\index{M-polyfold!sub --!local generator}
sub-M-polyfold $A$.
\end{definition}

\begin{proposition}\label{prop:sub-M-pf}
Suppose $A\subset X$ is a sub-M-polyfold.
\begin{itemize}
\item[\rm (i)]
  A sub-M-polyfold $A$ inherits an M-polyfold structure from the
  ambient $X$.
\item[\rm (ii)]
  The inclusion $\iota\colon A\INTO X$ is an $\SC$-smooth map between
  M-polyfolds and a homeomorphism onto its image.
\item[\rm (iii)]
  A local generator $r$ for $A$, viewed as a map
  $r\colon V\to A$, is $\SC$-smooth and
  $
     T_ar(T_aX)=T_a A
  $
  at any point
  $a\in A\CAP V$.
\item[\rm (iv)]
  At points $a\in A_\infty$ the
  tangent space $T_aA$ is $\SC$-complemented~in~$T_aX$.
\end{itemize}
\end{proposition}

\begin{proof}
\citet[Prop.\,2.6]{Hofer:2017a}.
\end{proof}

\subsection{Boundary and corners -- tameness}

Unfortunately, on M-polyfolds the degeneracy index of a point, defined
through an M-polyfold chart, might depend on the choice
of chart, as this example shows:
For a real parameter $a\ge 0$ define the
orthogonal projection
\[
     r_a=r_a\circ r_a\colon C\to C=[0,\infty)^2,\quad
     (x,y)\mapsto\frac{x+ay}{1+a^2}\, \left(1,a\right).
\]
The image of the retraction $r_a$ is the half line $L_a=\{(x,ax)\mid
x\ge 0\}$ in the quadrant $C$.
On the M-polyfold $X=[0,\infty)$ we choose the global chart $\phi_a\colon X\to
O=r_a(C)=L_a\subset C$ shown in Figure~\ref{fig:fig-deg-index-M-pfs-L_a}. 
In this chart the degeneracy index, see
Section~\ref{sec:boundary-recognition}, of each point $x\in X$
(also depending on whether {\color{cyan} $a=0$} or $a>0$) is given by
\[
     d_C\left(\phi_a(x)\right)=
     \begin{cases}
       {\color{red}      2}&\text{, $x=0$,}\\
       {\color{cyan}    1}&\text{, $x>0$ and {\color{cyan} $a=0$ ($L_0$)},}\\
       {\color{brown}  0}&\text{, $x>0$ and $a>0$ ($L_a$).}
     \end{cases}
\]
On the other hand, representing $X$ in the obvious global M-polyfold chart
\[
     \phi^\prime=\id\colon X\to [0,\infty)
     =O^\prime=\im r^\prime,\quad
     r^\prime=\id,\quad
     U^\prime=C^\prime=[0,\infty)\subset \R
\]
the degeneracy indices of points are the rather different, but
\begin{figure}
  \centering
  \includegraphics
                             [height=4cm]
                             {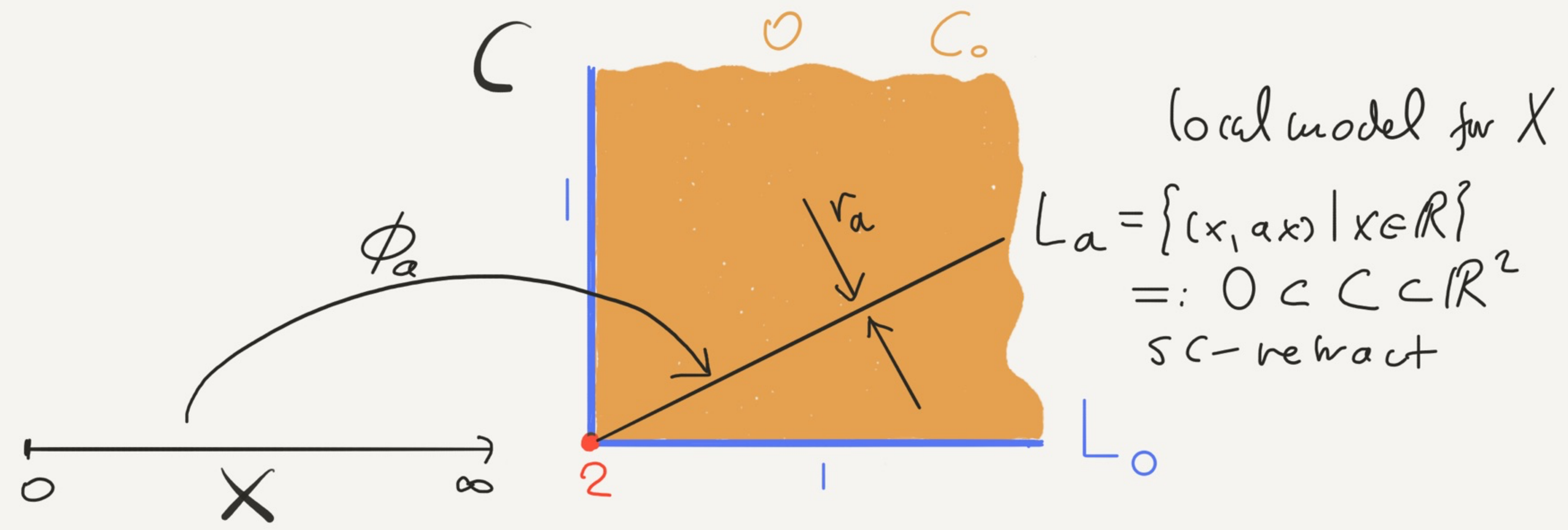}
  \caption{Global M-polyfold chart $\phi_a\colon X=[0,\infty)\to O=L_a\subset C$}
  \label{fig:fig-deg-index-M-pfs-L_a}
\end{figure}
expected, values
\[
     d_{C^\prime}\left(\phi^\prime(x)\right)=
     \begin{cases}
       1&\text{, $x=0$,}\\
       0&\text{, $x>0$.}
     \end{cases}
\]
Of course, the discrepancy between $d_C$ and $d_{C^\prime}$
could be caused by incompatibility of charts. However, this is not the
case, both transition retract maps are $\SC$-smoothly compatible.
Indeed the decompression of $\psi:=\phi^\prime\circ {\phi_a}^{-1}$
\[
     \psi\circ r_a=\phi^\prime\circ {\phi_a}^{-1}\circ r_a\colon 
     C\to r_a(C)=L_a\to O^\prime=[0,\infty),\quad
     (x,y)\mapsto\frac{x+ay}{1+a^2}
\]
is even $C^\infty$ smooth and so is
$\psi^{-1}\circ r^\prime\colon [0,\infty)\to O=L_a\subset C$,
$x\mapsto (x,ax)$.

\begin{definition}[Degeneracy index on M-polyfolds $X$]
Given a point $x\in X$, just take the minimum
\[
     d_X(x):=\min_\phi d_C\left(\phi(x)\right)
\]
over\index{$d_X$ degeneracy index on M-polyfold $X$}
all M-polyfold charts $\phi\colon V\to O\subset C$ about the point $x$.
\end{definition}

\subsubsection*{Degeneracy index stratification of quadrant -- Tameness}

To see what went wrong for the chart $\phi_a$ in the example above
note that the quadrant
$C={\color{brown} C_0}\CUP {\color{cyan} C_1}\CUP {\color{red} C_2}$
decomposes into disjoint subsets $C_i:={d_C}^{-1}(i)$,
the\index{degeneracy index!stratification}
strata\index{stratification!degeneracy index --}
of the \textbf{degeneracy index stratification}.
Now one identifies two problems:
\begin{itemize}
\item[\rm a)]
  The $\SC$-retraction $r_a$ does not preserve
the degeneracy index strata.
\item[\rm b)]
  The $\SC$-retract $r_a(C)=L_a$
  is in a certain sense not transverse to the degeneracy index
  stratification of the quadrant $C$.
\end{itemize}

One avoids the problem by giving a name to retractions that do not
have the defects a) and b) and then considers only such in theorems.

\begin{definition}\label{def:tame-retraction}
An $\SC$-smooth retraction $r\colon U\to U$ on an $\SC$-triple
$(U,C,E)$\index{tame!retraction}
is\index{sc-smooth!retraction!tame --}
called \textbf{tame} if
\begin{itemize}
\item[\rm a)]
  the map $r$ preserves the $d_C$-stratification:
  $d_C(r(x))=d_C(x)$ $\forall x\in U$;
\item[\rm b)]
  the image of $r$ is transverse to the $d_C$-stratification:
  For every smooth point $x$ in the image $r(U_\infty)=O_\infty$
  there must be an $\SC$-complement $A$ of the
  $\SC$-subspace $T_x O=Dr(x)E$ of $E$,
  cf.~(\ref{eq:T_xO}) and Exercise~\ref{exc:jhjghj8}, with $A\subset
  E_x:=T_x C(x)$ where $C(x):={d_C}^{-1}(d_C(x))$ is the stratum~of~$x$.
\end{itemize}
\end{definition}

If in b) above such $A$ exists, then one can choose $A=(\1-Dr(x))E$
by \citet[Prop.\,2.9]{Hofer:2017a}. So for tame $\SC$-smooth retractions
$r\colon U\to U$ with image $O$ one has the $\SC$-splittings
\[
     E=H_x\oplus V_x,\quad
     H_x:=T_x O=Dr(x)E,\quad
     V_x:= (\1-Dr(x))E,\quad
     x\in O_\infty
\]
one for each smooth point of $O=r(U)$.

\begin{remark}[Fixed origin]
Consider the quadrant $C:=[0,\infty)^n$ in $\R^n$
and suppose $r\colon C\to C$ is a \emph{tame} smooth retraction.
Then the origin $0=r(0)$ is fixed by $r$. Indeed $x=0$ is the only point
in $C$ with $d_C(x)=n$.
Moreover, the image $r(C)$ is an open neighborhood of $0$ in $C$;
cf.~\citet[Problem\,6.5]{Cieliebak:2018a}.
\end{remark}

\begin{definition}
An $\SC$-retract $(O,C,E)$ is called
\textbf{tame}\index{tame!retract}\index{sc-retract!tame --}
if $O=r(U)$ is the image of a tame $\SC$-smooth retraction $r$.
An M-polyfold is called
\textbf{tame}\index{M-polyfold!tame --}\index{tame!M-polyfold}
if $X$ admits an equivalent M-polyfold atlas modeled on
tame $\SC$-smooth retracts.
\end{definition}

For tame M-polyfolds $X$ the degeneracy
index $d_X(x):=d_C(\phi(x))$ of a point $x\in X$
defined via an M-polyfold chart $\phi\colon V\to O\subset C$
about $x$ does not depend on the choice of the chart;
see~\citet[Eq. (2.12)]{Hofer:2017a}.

\section{Strong bundles over M-polyfolds}
\label{sec:strong-bundles}
\sectionmark{Strong bundles}

We recall the notion of a vector bundle over a manifold and sketch
how to generalize the base to M-polyfolds (bringing in scale and retracts) and
accommodate Fredholm sections via $\SC^+$-sections
(bringing in double scales).

\subsubsection*{Motivation and comparison of old and new concepts}

The following overview is not meant to be, and is not, rigorous.

\vspace{.2cm}\noindent
\textit{Classical vector bundles over manifolds
-- trivial bundles $U\times F\to U$.}

\noindent 
A classical vector bundle over a manifold is locally
modeled by trivial bundles $U\times F\to U$. Here $U$ is an open
subset of a linear space $E$, the model space of the manifold, and
$F$ is a linear space, the model of the fibers of the vector bundle.
Any two local models must be related by a diffeomorphism
\[
     \Psi\colon U\times F\to \tilde U\times F,\quad
     (u,\xi)\mapsto \left(\psi(u),\Tt_u\xi\right)
\]
called a \emph{vector bundle transition map}, whose
second\index{transition map!vector bundle --}
component $(u,\xi)\mapsto \Tt (u,\xi)$
restricts at every point $u$ to a vector space
isomorphism~$\Tt_u:=\Tt(u,\cdot)\colon F\to F$.
So the building blocks for classical vector
bundles are \textbf{trivial bundles}
\[
     U\times F\to U.
\]

\vspace{.2cm}\noindent
\textit{Sc-bundles over M-polyfolds
-- trivial-bundle sc-retracts~$R(U\oplus F)\to r(U)$.}

\noindent
To define $\SC$-bundles over M-polyfolds one needs to generalize
trivial bundles taking into account that now one deals with $\SC$-triples
$(U,C,E)$ and $\SC$-Banach spaces $F$.
A useful notation for trivial bundles is $U\oplus F\to U$ which
indicates that the set $U\times F$ sits inside the $\SC$-direct sum
$E\oplus F$ and thereby inherits the scale structure
$(U\oplus F)_m=U_m\oplus F_m$.
\\
Because local models for the base M-polyfold are $\SC$-retracts $O$,
one should replace $U$ by its image $O=r(U)$ under an $\SC$-retraction.
It is suggesting to replace the whole space $U\oplus F$ by its
image under an $\SC$-retraction
\[
     R=R\circ R\colon U\oplus F\to U\oplus F,\quad
     (u,\xi)\mapsto\left(r(u),\rho_u\xi\right)
\]
for which $F\to F$, $\xi\mapsto \rho_u\xi:=\rho(u,\xi)$ is linear,
at any $u\in U$. Such retraction
\begin{itemize}
\item 
  produces a local M-polyfold model $O=r(U)$ in the component $U$ and
\item 
also respects the linear structure of the second component $F$.
\end{itemize}
A crucial observation is that along $O=r(U)=\Fix\, r\subset U$ idempotency
of $R$ implies that the linear map $\rho_x=(\rho_x)^2\colon F\to F$
is also idempotent, hence a projection.
Choose the identity retraction $R(u,\xi):=(u,\xi)$ and forget scale structures
to recover trivial bundles $U\oplus F\to U$, hence classical vector bundles.
The building blocks
\begin{equation}\label{eq:sc-bundles}
     K=R(U\oplus F)\subset U\oplus F
\end{equation}
for $\SC$-bundles over M-polyfolds
are called \emph{trivial-bundle sc-retracts}.
Projection onto the second component provides an $\SC$-smooth surjection
\[
     p\colon K\to O,\quad (x,\xi)\mapsto x
\]
onto an $\SC$-retract,
the local model of an M-polyfold. Each pre-image $p^{-1}(x)$ is a
closed linear subspace $\rho_x(F)=\Fix\,\rho_x$ of $F$.
\\
To summarize, the local models for $\SC$-bundles over M-polyfolds,
called \textbf{trivial-bundle sc-retracts},
are $\SC$-retracts $K$ in $U\oplus F$ that are families
\[
     K=R(U\oplus F)=\bigcup_{x\in O}\left(\{x\}\oplus \Fix\,\rho_x\right)
     \;\to\; O
\]
of projection images $\Fix\,\rho_x$ in $F$ parametrized
by local M-polyfold models $O$.
To\index{sc-bundles over M-polyfolds}
construct \textbf{\boldmath$\SC$-bundles over M-polyfolds}
one defines $\SC$-bundle charts as in
Definition~\ref{def:strong-bundle-charts}
disregarding double scales -- just replace the double scale symbol
$\triangleright$ by the $\SC$-direct sum $\oplus$.
Compatibility of charts and $\SC$-bundle atlases are defined as usual.

\begin{definition}[Sc-bundles over M-polyfolds]
An \textbf{\Index{sc-bundle} over an M-polyfold} $X$
is an $\SC$-smooth surjection $\pi\colon Y\to X$
between\index{$(pi$@$\pi\colon Y\to X$ sc-bundle}
M-polyfolds endowed with an equivalence class of
$\SC$-bundle atlases.
\end{definition}

\vspace{.2cm}\noindent
\textit{Accommodating Fredholm sections: double scale gives two
  scales $\mbox{}^{[0]}$ and $\mbox{}^{[1]}$.}

\noindent
The local model building blocks for strong bundles over M-polyfolds
are trivial-strong-bundle retracts $K=R(U\triangleright F)$. These
come with a double scale structure $K_{m,k}$ by definition
and a natural projection $p:K\to O$ where the $\SC$-retract
$O=r(U)$ is called the \emph{associated base retract} and the
$\SC$-retraction $r:U\to U$ is the first component of $R$.
Reducing the double scale in two ways to a scale
one obtains two $\SC$-bundles
\[
     p ^{[i]}\colon K ^{[i]}=R(U\triangleright F) ^{[i]}\to O=r(U),\qquad
     i=0,1
\]
over the M-polyfold $O$. The sections of $p^{[1]}$
generalize $\SC^+$ operators.

\subsection{Trivial-strong-bundle retracts - the local models}

Throughout $F$ is an $\SC$-Banach space and $(U,C,E)$ an
$\SC$-triple, that is $U$ is a relatively open subset of the partial
quadrant $C$ in the $\SC$-Banach space $E$.

\begin{remark}[Motivation for non-symmetric product and shift by $1$]
\label{rem:motivation-asym-prod}
At first sight the introduction of a double scale/filtration in
Definition~\ref{def:non-symm-product}, even an asymmetric one, and its
immediate reduction to a single scale in
Definition~\ref{def:str-triv-bdls}, in two versions though, might be
confusing and even appear superfluous given that the two versions
inherit their scale structure as subsets of the simple and well known
$\SC$-direct sums $U\oplus F$ and $U\oplus F^1$.

a) To perceive the need to shift the vector space part of $U\oplus F$ by one,
recall stability of the $\SC$-Fredholm property under addition of
$\SC^+$-operators; see Proposition~\ref{prop:stability-sc-Fredholm}.

b) In practice, when implementing a differential operator $f$ of order $\ell$
the level indices $m$ of $U$ indicate differentiability, simply speaking.
So one needs to forget the first $\ell$ levels and choose $U_\ell$, or
any sublevel of it, as domain for $f$. More precisely, one chooses the
shifted scale $U^\ell$. Then $f\colon (U^\ell)_m\to F_m$ and one can
subsequently exploit composition with the compact embeddings $F_m\INTO
F_{m-1}\dots\INTO F_0$.
\end{remark}

\begin{definition}[Non-symmetric product -- double scale]
\label{def:non-symm-product}
The\index{$(U\triangleright F)_{m,k}$ non-symmetric product}
\textbf{\Index{non-symmetric product}} $U\triangleright F$ is the subset
$U\times F$ of the Banach space $E\oplus F$
endowed with the\index{scale!double --}
\textbf{double scale}, also called\index{double!scale}\index{double!filtration}
\textbf{double filtration}, defined by\footnote{
  We use the symbol $U\triangleright F$, as opposed to $U\triangleleft F$,
  since the levels of $U$ are unlimited, any $m\in\N_0$ is allowed,
  whereas the ones of $F$ depend on $m$ and are restricted to
  $0,\dots,m+1$.
  }
\[
     (U\triangleright F)_{m,k}:=U_m\oplus F_k,\qquad
     m\in\N_0,\quad k\in\{0,\dots m+1\}.
\]
\end{definition}

Non-symmetric products $U\triangleright F$
serve as \textbf{total spaces} of strong trivial bundles.
Projection\index{strong bundle!trivial!total space}
onto\index{trivial-strong-bundle!total space}
the first component
\[
     U\triangleright F\to E,\quad
     (u,\xi)\mapsto u
\]
is called the \textbf{trivial-strong-bundle projection}.
However, for $\SC$-calculus one needs one scale structure, not a
double scale. To achieve this substitute $k$ by a useful function of $m$,
say $k=m$ or $k=m+1$.

\begin{definition}[Trivial strong bundles: Two relevant scale structures]
\label{def:str-triv-bdls}
Motivated by~\citet[\S 2.5]{Hofer:2017a} we denote the $\SC$-manifolds
$U\oplus F$ and $U\oplus F^1$ by the symbols
\[
     (U\triangleright F)^{[0]}:=U\oplus F,\qquad
     (U\triangleright F)^{[1]}:=U\oplus F^1.
\]
By definition
of\index{strong bundle!trivial}
shifted scales the levels are\footnote{
  cf. Definition~\ref{def:shifted-scale}
  }
\[
     (U\triangleright F)^{[0]}_m=U_m\oplus F_m,\qquad
     (U\triangleright F)^{[1]}_m=U_m\oplus F_{m+1}.
\]
The projections onto the first component
\[
     p=p^{[i]}\colon (U\triangleright F)^{[i]}\to U
     ,\quad (u,\xi)\mapsto u
     ,\qquad i=0,1
\]
are $\SC$-smooth maps between $\SC$-manifolds
called \textbf{\Index{trivial strong bundles}}.
We often write $p$ for simplicity and because the values do not depend
on the choice of shift $i$ for the second component $F$.
If the domain matters we shall write $p^{[i]}$.
\end{definition}

\begin{definition}[Morphisms of trivial strong bundles]
a) A\index{trivial-strong-bundle!map}
\textbf{trivial-strong-bundle map}
$\Psi\colon  U\triangleright F\to \tilde U\triangleright \tilde F$ is a map
that \emph{preserves the double scale} and is of the form
\[
     \Psi(u,\xi)=\left(\varphi(u),\Gamma(u,\xi)\right)
\]
where $\Gamma_u\xi:=\Gamma(u,\xi)$ is linear in $\xi$.
Moreover, it is required that both induced maps between $\SC$-manifolds
\[
     \Psi=\Psi^{[i]}\colon  (U\triangleright F)^{[i]}\to (\tilde U\triangleright \tilde F)^{[i]}
     ,\quad i=0,1
\]
are $\SC$-smooth.
b) A\index{trivial-strong-bundle!isomorphism}
\textbf{trivial-strong-bundle isomorphism}
is\index{isomorphism!of trivial-strong-bundles}
an invertible strong trivial bundle map whose
inverse is also a strong trivial bundle map.
\end{definition}

It is the previous definition where \underline{double scale preservation}
is required.

\begin{exercise}
Check that the second component $\Gamma$ of a strong trivial bundle
map $\Psi$ gives rise to an $\SC$-operator $\Gamma_u\in\Llsc(F,\tilde F)$
along the smooth points $u\in U_\infty$.
\end{exercise}

\begin{definition}[Trivial-strong-bundle retraction]
A\index{trivial-strong-bundle!retraction}
\textbf{trivial-strong-bundle retraction}
is an idempotent trivial-strong-bundle map
\[
     R=R\circ R\colon  U\triangleright F\to U\triangleright F,\quad
     (u,\xi)\mapsto\left(r(u),\rho_u\xi\right).
\]
The first component $r$ of $R$ is necessarily an $\SC$-smooth retraction on
$U$, called the \textbf{associated base retraction},
whose image\index{associated!base retraction}
$\SC$-retract $O=r(U)$ is called the\index{associated!base retract}
\textbf{associated base retract}.
One\index{base retraction!associated --}
calls\index{base retract!associated --}
$R$ \textbf{tame} in case the associated base retraction is tame.
\end{definition}

\begin{exercise}
Let $R(u,\xi)=\left(r(u),\rho_u\xi\right)$
be a strong trivial bundle retraction. Check that $r$ is an
$\SC$-smooth retraction on $U$ and that $\rho_x\in\Ll(F)$ is a
projection for $x\in \Fix\, r =r(U)=:O$, even an $\SC$-projection
at smooth points, i.e. $x\in O_\infty$.
\end{exercise}

\begin{definition}[Trivial-strong-bundle retracts $K$ -- the local models]
A\index{trivial-strong-bundle!retract}
\textbf{trivial-strong-bundle retract}\footnote{
  'strong' indicates 'doubly scaled' and the
  retraction acts on a 'trivial bundle'
  }
$(K, C\triangleright F, E\triangleright F)$, or simply $K$,
is\index{$p\colon K\to O$ trivial-strong-bundle retract}
given by the image
\[
     K=R(U\triangleright F)=\left(\Fix\, R\right)
     \subset \left(O\triangleright F\right)
\]
of a trivial-strong-bundle retraction
$R=R\circ R$ on $U\triangleright F$
where $O=r(U)$ is the associated base retract.
One\index{strong bundle! trivial--}
likewise calls the natural surjection
\[
     p\colon K\to O,\quad (x,\xi)\mapsto x
\]
\textbf{trivial-strong-bundle retract}.
For\index{trivial-strong-bundle!retract}
simplicity we identify point pre-images
\[
     p^{-1}(x)=\{x\}\times K_x,
     \qquad K_x:=\rho_x(F)
\]
with the Banach subspace $K_x:=\rho_x(F)$ of $F$,
an $\SC$-subspace for smooth points,
called the \textbf{fiber of \boldmath$K$} over $x$.
Call $K=R(U\triangleright F)$ \textbf{tame} if $R$ is tame.
\end{definition}

Being a subset of the doubly scaled space $U\triangleright F$
a trivial-strong-bundle retract $K$ inherits the double scale
\[
     K_{m,k}:=K\CAP \left(U_m\oplus F_k\right)
     =\bigcup_{x\in O_m}\left(\{x\}\oplus \Fix\,[\rho_x\colon F_k\to F_k]\right)
\]
for $m\in\N_0$ and $k\in\{0,\dots m+1\}$.
Note that the spaces\footnote{
  the symbol $R(U\triangleright F)^{[i]}$ abbreviates
  $R((U\triangleright F)^{[i]})$
  }
\begin{equation}\label{eq:sc-ret-K}
     K^{[i]}:=K\CAP\left(E^0\oplus F^i\right)
     =\im R^{[i]}=R(U\triangleright F)^{[i]}
     ,\qquad i=0,1
\end{equation}
with levels $K^{[i]}_m=K_{m,m+i}$ are $\SC$-retracts,
so M-polyfolds. The surjections
\[
     p=p^{[i]}\colon K^{[i]}\to O
     ,\quad(x,\xi)\mapsto x, \qquad i=0,1
\]
are both $\SC$-smooth maps between $\SC$-retracts.
Indeed by Definition~\ref{def:sc-map-between-sc-retracts}
this requires that some, hence any, decompression, say
\[
     p\circ R\colon  (U\triangleright F)^{[i]} \to K^{[i]}\to O
     ,\quad(u,\xi)\mapsto r(u),\qquad i=0,1
\]
be $\SC$-smooth. But the associated base retraction $r$ is
$\SC$-smooth by assumption.

\begin{definition}[Strong retract maps]
A map $\Ff$ of the form
\[
     \Ff\colon K\to \tilde K,\quad (x,\xi)\mapsto\left(f(x),\phi_x\xi\right),
     \qquad f\colon O\to \tilde O
\]
between trivial-strong-bundle retracts
is\index{retract map!strong --}
called a \textbf{\Index{strong retract map}}
if $\Ff$ is linear in the fibers, that is $\phi_x\colon K_x\to
\tilde K_{f(x)}$ is linear, if $\Ff$ preserves the double filtrations,
and if both induced maps between $\SC$-retracts
\[
     \Ff^{[i]}\colon  K^{[i]}\to \tilde K^{[i]},\qquad i=0,1
\]
are $\SC$-smooth (meaning $\SC$-smoothness after decompression).
\end{definition}

\begin{definition}[$\SC$- and $\SC^+$-sections of trivial-strong-bundle
retracts]
A\index{section!of trivial-strong-bundle}
\textbf{section}\index{trivial-strong-bundle!section of --}
of a trivial-strong-bundle retract $p\colon K\to O$
is a map $s\colon O\to K$ that satisfies $p\circ s=\id_O$.
If $s$ is $\SC$-smooth as an $\SC$-retract map
\[
     s^{[i]}\colon O\to K^{[i]},\quad
     x\mapsto \bigl(x,\bs^{[i]}(x)\bigr),\qquad
     \bs^{[i]}\colon O\to F^i
\]
it\index{sc-section!of trivial-strong-bundle}
is\index{trivial-strong-bundle!sc-section of --}
called in case $i=0$
an \textbf{\boldmath$\SC$-section}
and in case $i=1$\index{trivial-strong-bundle!$\SC^+$-section}
an\index{sc$^+$-section!of trivial-strong-bundle}
\textbf{\boldmath$\SC^+$-section}.
The map $\bs^{[i]}\colon O\to F^i$ is called the \textbf{principal part}
of\index{sc-section!principal part of --}
the\index{section!principal part of --}\index{principal part of section}
section.
\end{definition}

Note that a section is $\SC$-smooth iff its principal part is. For
simplicity we sometimes omit the superscript$^{[i]}$ if the
level shift is clear from the context.

\subsection{Strong bundles}

Throughout $F$ is an $\SC$-Banach space and $(U,C,E)$ an
$\SC$-triple, that is $U$ is a relatively open subset of the partial
quadrant $C$ in the $\SC$-Banach space $E$.

\begin{definition}[Strong bundle charts]\label{def:strong-bundle-charts}
Let $P\colon Y\to X$ be a continuous surjection from a paracompact Hausdorff
space $Y$ onto an M-polyfold $X$ such that every pre-image
$Y_x:=P^{-1}(x)$ has the structure of a Banachable space.\footnote{
  A \textbf{\Index{Banachable space}} is an equivalence class
  that consists of all Banach spaces with pairwise equivalent norms.
  }
A\index{strong bundle!chart}
\textbf{strong bundle chart}\index{bundle chart!strong --}
for $P\colon Y\to X$ is a tuple
\[
     \left(\Phi,P^{-1}(V),(K, C\triangleright F, E\triangleright F)\right)
\]
that consists of
\begin{itemize}
\item
  a trivial-strong-bundle retract $p\colon K=R(U\triangleright F)\to O$
  where $O=r(U)$ is the associated base retract;
\item
  a homeomorphism $\varphi\colon V\to O$ between an open subset of the
  base M-polyfold $X$ of $Y$ and the base retract $O$ of $K$;
\item
  a homeomorphism $\Phi\colon P^{-1}(V)\to K$ that covers $\varphi$, 
  that is the diagram
  \begin{equation*}\label{eq:strong-bundle-chart}
  \begin{tikzcd} [column sep=tiny] 
    Y
    \arrow[d, "P"']
    & \supset
      &
      P^{-1}(V)
      \arrow[d, "P"']
      \arrow{rrrrr}[name=U]{\Phi\;\,}
        &&&&&
        K=R(U\triangleright F)
        \arrow[d, "p", shift right=10.8]
    \\
     X
     & \supset
      &
      V
      \arrow[swap]{rrrrr}[name=D]{\,\;\varphi}
      \arrow[to path={(U) node[midway,scale=1.2] {\;\;\,\,$\circlearrowleft$}  (D)}]
        &&&&&
        O=r(U)\hphantom{hhh\,} 
  \end{tikzcd} 
  \end{equation*}
  commutes. As a consequence, for every point $v\in V$ the restriction
  of $\Phi$ to $P^{-1}(v)$ takes values in $p^{-1}(\varphi(v))$.
  One also requires that $\Phi$ viewed as a map
  \[
     \Phi\colon Y_v=P^{-1}(v)\stackrel{\simeq}{\longrightarrow} p^{-1}(\varphi(v))
     =\rho_{\varphi(v)}(F),\qquad \forall v\in V
  \]
 is a continuous linear bijection\footnote{
     making sense although the domain is just Banachable
     }
  between fibers.\footnote{
      strictly speaking, the target is $\{\varphi(v)\}\times\rho_{\varphi(v)}(F)$
      }
\end{itemize}
\end{definition}

\begin{definition}[Strong bundle atlases]
Two\index{strong bundle!chart!compatible --}
strong bundle charts are called \textbf{compatible} if, firstly,
the transition map
\[
     \Psi:=\tilde \Phi\circ \Phi^{-1}\colon 
     K\supset\Phi(P^{-1}(V\CAP\tilde V))
     \to\tilde\Phi(P^{-1}(V\CAP\tilde V))
     \subset\tilde K
\]
is a strong retract map, thus preserves the double scales,
and, secondly,  the two induced maps $\Psi^{[0]}$ and $\Psi^{[1]}$
between open subsets of $\SC$-retracts (cf.~(\ref{eq:sc-ret-K})),
hence M-polyfolds, are $\SC$-smooth diffeomorphisms.
A\index{atlas!strong bundle --}\index{strong bundle!atlas}
\textbf{strong bundle atlas} $\Aa^Y_X$ consists of
pairwise\index{$\Aa^Y_X$ strong bundle atlas}
compatible strong bundle charts covering $Y$.
Two such atlases are called equivalent if their union
is again a strong bundle atlas.
\end{definition}

\begin{definition}[Strong bundles over M-polyfolds]
A \textbf{\Index{strong bundle} over an M-polyfold} $X$
is a continuous surjection $P\colon Y\to X$
from\index{$P\colon Y\to X$ strong bundle}
a paracompact Hausdorff space
equipped with an equivalence class of strong bundle atlases.
\end{definition}

\begin{exercise}[A strong bundle provides two M-polyfolds]
Check that a strong bundle atlas $\Aa^Y_X$ for $Y\to X$ naturally
provides two M-polyfold atlases $\Aa_{Y^{[0]}}$ and $\Aa_{Y^{[1]}}$
for M-polyfolds $Y^{[0]}$ and $Y^{[1]}$, respectively.
\end{exercise}

\begin{exercise}[A strong bundle provides two $\SC$-bundles]
A strong bundle atlas $\Aa^Y_X$ for $P\colon Y\to X$ naturally
provides\index{$\Aa^{Y^{[0]}}_X$ induced sc-bundle atlas}
two \textbf{induced \boldmath$\SC$-bundle atlases}
$\Aa^{Y^{[0]}}_X$\index{sc-bundle!atlases!induced --}
and $\Aa^{Y^{[1]}}_X$ for $\SC$-bundles $P^{[0]} \colon Y^{[0]}\to X$
and $P^{[1]}\colon Y^{[1]}\to X$, respectively.
\end{exercise}

\subsubsection*{Induced double scale and section types}

A strong bundle $P\colon Y\to X$ carries an asymmetric double scale
structure $Y_{m,k}$, where $m\in\N_0$ and $k=0,\dots,m+1$,
transmitted from the local models $K$ by the strong bundle charts.
Here it enters that the transition maps $\Psi$ are strong retract
maps, thus preserve the double scale of the local models~$K$.

\begin{definition}[$\SC$- and $\SC^+$-sections of strong bundles]
A\index{section!of strong bundle}
\textbf{section}\index{strong bundle!section of --}
of a strong bundle $P\colon Y\to X$
is a map $s\colon X\to Y$ that satisfies $P\circ s=\id_X$.
If a section $s$ of $P\colon Y\to X$
is $\SC$-smooth as a map between M-polyfolds
\[
     s^{[i]}\colon X\to Y^{[i]}
\]
then\index{sc-section!of strong bundle}
$s$ is\index{strong bundle!sc-section of --}
called an \textbf{\boldmath$\SC$-section} (case $i=0$)
or\index{strong bundle!$\SC^+$-section}
an\index{sc$^+$-section!of strong bundle}
\textbf{\boldmath$\SC^+$-section} (case $i=1$).
\end{definition}

\subsubsection*{Pull-back bundle}

Suppose $f\colon Z\to X$ is an $\SC$-smooth map between M-polyfolds
and $P\colon Y\to X$ is a strong bundle over $X$.
The pull-back bundle $P_f\colon f^*Y\to X$
consists of the subset of $Z\times Y$ defined by
\[
     f^*Y:=\{(z,y)\in Z\times Y\mid P(y)=f(z)\}
\]
and projection $P_f(z,y)=P_1(z,y)=z$ onto the first component.
Together with projection onto the second component,
denoted by $P_2$, the diagram
  \begin{equation*}
  \begin{tikzcd} [column sep=tiny]
    Z\times Y
    & \supset
      &
      f^*Y
      \arrow[d, dashed, "P_f"', "P_1"]
      \arrow[rrrrr, dashed, "P_2"]
        &&&&&
        Y
        \arrow[d, "P"]
        \arrow[d, phantom, "\circlearrowleft", shift right=10.5]
    \\
     &
      &
      Z
      \arrow[rrrrr, "f"']
        &&&&&
        X
  \end{tikzcd} 
  \end{equation*}
commutes.


\begin{exercise}[Induced strong bundle structure]
Given an $\SC$-smooth map $f\colon Z\to X$ between M-polyfolds,
show that a strong bundle structure on $P\colon Y\to X$ induces naturally
a strong bundle structure on the pull-back bundle $f^*P\colon f^*Y\to Z$.
\end{exercise}


\cleardoublepage
\phantomsection

\appendix
\chapter{Background from Topology and
Functional Analysis}\label{sec:app-FA}
\chaptermark{Topology and Functional Analysis}

\section{Analysis on topological vector spaces}
\label{sec:Ana-TVS}
All vector spaces will be over the real numbers $\R$.
Let us first repeat

\subsubsection{Some basics about sets}
The elements of a set $S$ are often called \textbf{\Index{points}}.
If a set $S$ contains only finitely many elements it is called \textbf{finite}.
The number of elements of a \Index{finite set} is denoted
by\index{set!finite --}
$\abs{S}$.\index{$(\abs$@$\abs{S}$ number of elements of finite set $S$}
The set with no element is called the \textbf{\Index{empty set}},
denoted by $\emptyset$ or, in order to indicate the ambient
universe $S$, by
$\emptyset_S$.\index{$(empty$@$\emptyset_S$ empty set in ambient universe $S$}
We avoid terminology like \emph{a set of sets}, instead
we shall speak of a \textbf{\Index{family} of sets} or of a
\textbf{\Index{collection} of sets}.
Let $2^S$ be the collection of all subsets of
$S$.\index{$(2^S$@$2^S$ collection of all subsets of set $S$}
The empty set $\emptyset$ is a subset of any set $S$,
in symbols $\emptyset\subset S$ or $\emptyset\in 2^S$.
Our use of $\subset$ allows for equality, otherwise we write
$\subsetneq$.\index{$(subs$@$\subset$, $\subsetneq$}
For more basics on set theory and logic see
e.g.~\citet[Ch.\,I]{Munkres:2000a}. See also
Ch.\,I, in particular I.9 on axiomatics, in~\citet{Dugundji:1966a}.

\begin{definition}
Let $S$ be a set.
Given a family $\Aa\subset 2^S$ of subsets $A$~of~$S$,
union and intersection of the members of $\Aa$ are the subsets
of $S$ defined by
\[
     \bigcup\Aa
     =\bigcup\{A\mid A\in\Aa\}
     :=\bigcup_{A\in\Aa} A
     :=\{x\in S\mid \exists A\in\Aa\colon  x\in A\}\subset S
\]
and
\[
     \bigcap\Aa
     =\bigcap\{A\mid A\in\Aa\}
     :=\bigcap_{A\in\Aa} A
     :=\{x\in S\mid \forall A\in\Aa\colon  x\in A\}\subset S.
\]
\end{definition}

\begin{exercise}
For $\Aa=\emptyset,\{\emptyset_S\}\subset 2^S$ show 
$
     \bigcup\{\emptyset_S\}=\emptyset_S
     =\bigcap\{\emptyset_S\}
$,
but\index{$(cup$@$\bigcup\emptyset=\emptyset_S$, $\bigcap\emptyset=S$, for $\emptyset\subset 2^S$}
\[
     \bigcup\emptyset=\emptyset_S,\qquad
     \bigcap\emptyset=S,\qquad
     \text{where $\emptyset\subset 2^S$.}
\]

\vspace{.1cm}\noindent 
[Hint: Final assertion -- empty truth.]
\end{exercise}

\subsubsection{Maps and exponential law}

Suppose $A,B,C$ are sets. A
\textbf{map \boldmath$f$ from $A$ to $B$}, in symbols $f\colon A\to B$,
is determined by a subset\index{$f\colon A\to B$ map}\index{map!between sets}
$G(f)\subset A\times B$ such that for each domain element $a\in A$
the set $\{b\in B\mid(a,b)\in G(f) \}$ has precisely 1 element.
The unique $b\in B$ such that $(a,b)\in G(f)$ is denoted by $f(a)$
and called the \textbf{image} of $a$ under $f$.
The set $A$ is the \textbf{\Index{domain}} of $f$ and $B$ the
\textbf{\Index{codomain}} or the \textbf{\Index{target}}.
The subset $G(f)\subset A\times B$ is called the \textbf{graph} of
$f$.\index{graph of a map}
A \textbf{\Index{function}} is a map $f\colon A\to \R$ that takes values
in the set of real numbers $\R$.

Let $\Map(A,B)$, or $B^A$, denote the \textbf{set of all maps} from $A$ to
$B$.\index{$B^A=\Map(A,B)$ set of all maps $f\colon A\to B$}
Motivated by the exponential notation the bijection
\[
     \Lambda\colon C^{A\times B} \to \left(C^B\right)^A,
     \quad f\mapsto F,\qquad
     F(a)(b):=\left(F(a)\right) (b):=f(a,b)
\]
is called the\index{exponential!map}\index{exponential! law}
\textbf{exponential map} or the \textbf{exponential law}.

\subsection{Topological spaces}\label{sec:top-spaces}
For an elementary overview see e.g.~\citet[Ch.\,II]{Munkres:2000a},
for an exhaustive treatment~\citet{Dugundji:1966a},
we also found extremely useful~\citet{Muger:2016a}.

\begin{definition}[Topology] 
A \textbf{\Index{topology}} on a set $M$
is a family $\Tt$ of subsets $U\subset M$,
called the \textbf{\Index{open!sets}},
such that the following axioms hold.
\begin{itemize}
\item[(i)]
  Both the empty set $\emptyset$ and $M$ itself are open.
\item[(ii)]
  Arbitrary unions of open sets are open.
\item[(iii)]
  Finite intersections of open sets are open.
\end{itemize}
Such pair $(M,\Tt)$ is called a\index{topological!space}
\textbf{topological space}.
The\index{$U^{\rm C}:=X\setminus U$ complement}
complements $\comp{U}:=X\setminus U$ of the open sets form the
family of \textbf{\Index{closed set}s}.
\end{definition}

\begin{exercise}\label{exc:int-top-is-top}
The intersection of a collection of topologies is a topology.
\end{exercise}

A topology $\Tt$ on a set $M$ induces on any subset $A\subset M$
a topology $\Tt^{\cap A}$ which consists of the intersections
of\index{$\Tt^{\cap A}:=\Tt\CAP A$ subset topology}
$A$
with all the members of the family $\Tt$ of subsets of $M$.
The topology $\Tt^{\cap A}$ is called the\index{subset!topology}
\textbf{subset topology}
or\index{topology!induced --}\index{topology!subset --}
the\index{induced!topology}
\textbf{induced topology} on a subset $A$.
A \textbf{\Index{subspace}}\index{topological!space!subspace of --}
is a subset of a topological space endowed with the subset topology.

Properties of topological spaces that are inherited by
subspaces are called \textbf{\Index{hereditary properties}}.

A topological space is called \textbf{compact}
if every open cover admits a finite sub-cover.
A\index{compact!set}
subset $K$ of a topological space $(M,\Tt)$
is called \textbf{compact}
if the topology on $K$ induced by $\Tt$ is compact.
A subset is called
\textbf{\Index{pre-compact}}\index{compact!pre- --}
if its closure is compact.

One often writes, instead of the pair $(M,\Tt)$,
simply $M$ and calls it a topological space.
An \textbf{open \Index{neighborhood}} of a subset $P\subset M$
is an open set $U$ that contains $P$, in symbols $P\subset U\in\Tt$.
Any subset $A\subset M$ that contains an open neighborhood
of $P$ is called a \textbf{neighborhood} of $P$. If $P=\{x\}$ is a
point set we speak of a neighborhood of a point $x\in M$.
It is convenient to write $U_x$ to indicate that a set $U$ contains
the point $x$. With this convention ``for any open
neighborhood\index{$U_x\subset X$ means $x\in U\subset X$}
$U$ of $x$'' becomes ``\textbf{for any open \boldmath$U_x$}''.

\subsubsection{Basis of a given topology}

\begin{definition}[Basis]\label{def:basis-topology}
Given a sub-collection $\Cc\subset\Tt$ of a topology,
let
\[
     \qquad\qquad\qquad\quad\quad
     \boxed{\Tt_\Cc
      :=\left\{{\textstyle \bigcup\sigma}\mid\sigma\subset\Cc\right\}
     \subset\Tt\subset 2^M}
     \qquad\qquad\;
     {\color{gray}\footnotesize(\text{recall 
     $\textstyle\bigcup\sigma\subset M$})}
\]
be the collection of all unions of elements $C\subset M$ of $\Cc$.
If a sub-collection $\Bb\subset\Tt$ satisfies
$\Tt_\Bb=\Tt$, i.e. if all open sets are unions of elements of $\Bb$,
one calls $\Bb$ a\index{topology!basis of the --}
\textbf{\Index{basis of the topology} \boldmath$\Tt$} and
says that \textbf{the topology is generated by \boldmath$\Bb$}.
\end{definition}

The elements of a basis $\Bb$ are called \textbf{basic open sets}.
Any\index{basic!open sets}
open set is a union of basic ones.
Uniqueness of a basis fails as badly, as existence is trivial: Given $\Tt$,
pick $\Bb:=\Tt$.
Often in practice, the smaller a basis, the better. So a criterion
for being a basis is desirable.

\begin{lemma}\label{le:basis-characterisation}
For a subset $\Cc\subset\Tt$ of a topology
the following are equivalent:
\begin{itemize}
\item[(i)]
  The collection $\Cc$ is a basis of $\Tt$, in symbols $\Tt_\Cc=\Tt$.
\item[(ii)]
  The collection $\Cc$ is \emph{dominated by $\Tt$}
  in the following sense:
  Each point $x\in U\in\Tt$ of an open set
  also lies in a collection member $C\in\Cc$ that is
  contained in $U$, in symbols $x\in C\subset U$.
\end{itemize}
\end{lemma}

\begin{proof}
See e.g.~\citet[III.2]{Dugundji:1966a}.
\end{proof}

\begin{definition}[Sub-basis]\label{def:sub-basis-topology}
For a sub-collection $\Ss\subset\Tt$ of a topology, let
\[
     \boxed{\Bb_\Ss
     :=\left\{{\textstyle \bigcap\sigma}\mid\text{$\sigma\subset\Ss$,
     $\abs{\sigma}<\infty$}\right\}
     \subset\Tt\subset 2^M}
\]
be the collection of all finite intersections of elements of $\Ss$.
If $\Bb_\Ss$ is a basis~of~$\Tt$, i.e. if all open sets are
arbitrary unions of finite intersections of elements~of~$\Ss$,
one calls $\Ss$\index{topology!sub-basis of the --}
a \textbf{\Index{sub-basis of the topology} \boldmath$\Tt$} and
$\Bb_\Ss$ the \textbf{basis generated~by~\boldmath$\Ss$}.
\end{definition}

\begin{definition}
A topological space is called \textbf{\Index{second countable}}
if it admits a countable basis. This property is hereditary.
\end{definition}

\begin{definition}
A subset of a topological space is called \textbf{\Index{dense}}
if\index{dense subset}
it meets (has non-empty intersection with)
every non-empty open set or, equivalently,
if its closure is equal to the whole space.
A topological space is called \textbf{\Index{separable}} if it admits
a dense sequence (countable subset).
Separability is \emph{not} hereditary.
\end{definition}

\begin{exercise}
Show that second countability is hereditary, whereas separability is
not, that second countable implies separable
and that in metric spaces (endowed with the metric topology $\Tt_d$)
the converse is true, too.
\end{exercise}

\begin{definition}[Local basis]\label{def:local-basis}
Let $(M,\Tt)$ be a topological space and $x\in M$.
A collection $\Bb(x)$ of open neighborhoods $B_x$ of $x$
is\index{topology!local basis of --}
called\index{local basis of the topology at $x$}
a \textbf{local basis of the topology at \boldmath$x$}
if every open neighborhood $U_x$ of $x$
contains a member of $\Bb(x)$, in symbols $U_x\supset B_x\in \Bb(x)$.
\end{definition}

\begin{exercise}\label{exc:neighbourhood-base}
Let $(M,\Tt)$ be a topological space.
\begin{itemize}
\item[(i)]
  Given a basis $\Bb$ of $\Tt$, for every $x\in M$ the family
  $\Uu(x) =\{U\in \Tt\mid x\in U\}$ of all open neighborhoods
  of $x$ is a local basis of $\Tt$ at $x$.
\item[(ii)]
  Vice versa, given for every point $x$ of $M$ a local basis $\Bb(x)$
  for $\Tt$ at $x$, show that their union
$
     \Bb
     :=\bigcup_{x\in M}\Bb(x)
     =\{B\mid B\in\Bb(x),\; x\in M\}\subset\Tt\subset 2^M
$
forms a basis of $\Tt$.
\end{itemize}
\end{exercise}

\subsubsection{From sets to topologies}

Starting with just a set $S$, let $\Cc$ be any collection of
subsets of $S$. The definitions above still provide collections
$\Tt_\Cc, \Bb_\Cc\subset 2^S$. Note that always
$\emptyset\in\Tt_\Cc$ and $M\in\Bb_\Cc$ (pick
$\sigma:=\emptyset\subset\Cc$).
It is a natural question to ask under what conditions on $\Cc$
the collections $\Tt_\Cc$ or $\Tt_{\Bb_\Cc}$ are topologies on $S$.

\begin{exercise}[Any collection is a sub-basis of \underline{some} topology]
\label{exc:sub-basis-top}
Let $S$ be a set and $\Ss\subset 2^S$ \emph{any} collection of subsets.
Then $\Tt_{\Bb_\Ss}$ is a topology on $S$, the smallest
topology that contains $\Ss$, and $\Bb_\Ss$ is a basis.

\vspace{.1cm}\noindent 
[Hints: Let $\Tt^\Ss$ be the intersection of all topologies $\Tt$
containing $\Ss$ (for example $\Tt=2^S$). Show $\Tt^\Ss=\Tt_{\Bb_\Ss}$.
See e.g.~\citet[III.3]{Dugundji:1966a}.]
\end{exercise}

While any collection of subsets of $S$ is a sub-basis of some
topology on $S$, a sufficient condition for
being a basis of some topology is the following.

\begin{theorem}[Being a basis of \underline{some} topology]
\label{thm:basis-some-topology}
Given a set $S$, let $\Bb\subset 2^S$ be a collection of subsets $V$
of\index{topology!basis for some --}
$S$ such that
\begin{itemize}
\item[(i)]
  $\Bb$ is a cover of $S$ (the union of all members of $\Bb$ is $S$) and
\item[(ii)]
  every point $p\in V_1\CAP V_2$ in an intersection of two $\Bb$ members
  simultaneously belongs to a $\Bb$ member $V_3\subset V_1\CAP V_2$
  contained in the intersection.
\end{itemize}
Under these conditions $\Tt_\Bb$ is a topology on $S$, the smallest
topology containing~$\Bb$, and~$\Bb$ is a basis.
\end{theorem}

\begin{proof}
See e.g.~\citet[III Thm.\,3.2]{Dugundji:1966a}.
\end{proof}

\begin{exercise}
Let $S$ be a set. The three collections
$\Ss=\emptyset$, $\{\emptyset_S\}$, $\{S\}\subset 2^S$
lead, respectively, to the three bases 
$\Bb_\Ss=\{S\}$, $\{\emptyset_S,S\}$, $\{S\}\subset 2^S$
each\index{topology!trivial --}\index{topology!indiscrete --}
of which generates the trivial, also called indiscrete, topology
$\Tt_{\Bb_\Ss}=\{\emptyset_S,S\}$.
\end{exercise}

Here is another method to topologize a set $S$ starting with a family
of candidates for local bases, one candidate at each point $x$ of the set.
It is a two step process.
Firstly, at every point $x\in S$ we wish to specify a collection
$\Bb(x)$ of subsets $V_x\subset S$ in such a way that, secondly, we can
construct a unique topology $\Tt(\Bb)$ on $S$ for which the collection
$\Bb(x)$ will be a local basis at $x$ and this is true for every $x\in S$.
Since prior to step two there is no topology, hence no notion of local
basis, we call $\Bb(x)$ a local pre-basis at $x$.

\begin{definition}\label{def:local-pre-basis}
Let $S$ be a set and $x\in S$. Suppose $\Bb\subset 2^S$ is the union
of a collection of non-empty families $\emptyset\not=\Bb(x)$ of
subsets of $S$, one family associated to each point $x$ of $S$,
such that the following is satisfied at all points $x,y\in S$.
\begin{itemize}
\item[(1)]
  Every member of $\Bb(x)$ contains $x$.
  \hfill {\footnotesize\color{gray} ($\emptyset \notin \Bb(x)$)}
\item[(2)]
  The intersection $V_1\CAP V_2\supset V_3$ of any two members
  of $\Bb(x)$ contains a $\Bb(x)$-member $V_3$.\footnote{
    Note that (2) makes sense since any intersection $V_1\CAP V_2\ni x$
    is non-empty.
    }
  \hfill {\footnotesize\color{gray} ($\Bb(x)$ downward directed)}
\item[(3)]
  For any $\Bb(x)$-member $V_x$ each of
  its points $y$ belongs to a $\Bb(y)$-member
  $Y_y$ contained in $V_x$, i.e. any $V\in\Bb$ is a
  union of $\Bb$-members.
  \hfill {\footnotesize\color{gray} ($\Bb\subset\Tt(\Bb)$)}
\end{itemize}
The\index{pre-basis on a set}\index{topology!pre-basis for a --}
family $\Bb(x)$ is called a \textbf{local pre-basis at \boldmath$x$},
the union $\Bb :=\bigcup_{x\in M}\Bb(x)$ of all of them
is called a \textbf{pre-basis on the set \boldmath$S$}.
The family of subsets
\[
     \Tt(\Bb)
     :=\{U\subset S\mid \text{for every $y\in U$ there is
     a $\Bb(y)$ member $Y_y\subset U$}\}
\]
is called the \textbf{topology generated by the pre-basis\boldmath}
$\Bb$ on the set $S$.
\end{definition}

\begin{exercise}\label{exc:jhj45667}
a)~Under conditions (1) and (2) show that $\Tt(\Bb)$ is a topology on $S$.\footnote{
  While $\Tt(\Bb)$ under conditions (1) and (2) is already a topology,
  only in combination with (3) every member of $\Bb(x)$ will be an
  open set -- a necessary condition for a local basis.
  }
From now on suppose in addition condition (3).
b)~Show $\Bb\subset\Tt(\Bb)$.
c)~For each point $x\in S$ show that $\Bb(x)$ is a local basis of $\Tt(\Bb)$ at $x$. 
(Hence $\Tt_\Bb=\Tt(\Bb)$ by Exercise~\ref{exc:neighbourhood-base}~(ii),
i.e. $\Bb$ is a basis of $\Tt(\Bb)$.)
\end{exercise}

The conditions in Definition~\ref{def:local-pre-basis} are related to
the theory of filters; see e.g.~\citet[\S 1.1.2]{Narici:2011a}.
See also~\citet[Thm.\,2.3.1]{Narici:2011a}.

\subsubsection{Convergence and continuity}

\begin{definition}[Convergence]\label{def:top-conv}
A subset sequence $(x_n)\subset M$ in a topological space
is\index{convergence in topological space}
said\index{$x_n\to z$ in topological space}
to \textbf{converge} to a point $z\in M$,
in symbols $x_n\to z$, if any open neighborhood $U_z$
of $z$ contains all but finitely many of the sequence members.\footnote{
  In symbols, there is $N\in\N$ such that $x_n\subset U_z$
  whenever $n\ge N$.
  }
\end{definition}

\begin{definition}[Continuity]
A map $f\colon M\to N$ between topological spaces is
called\index{continuous!at a point}
\textbf{continuous at a point} $\mbf{x}$ if the pre-image of any open
neighborhood $V_{f(x)}$ of the image point $f(x)$ contains an open
neighborhood $U_x$ of $x$.
A\index{map!continuous --}
\textbf{continuous map} is one that is continuous at every point
of its domain. Let $C(M,N)$ denote the set of continuous maps from $M$
to\index{$C(M,N)$ continuous maps from $M$ to $N$}
$N$.
\end{definition}

\begin{exercise}
a) A map $f\colon M\to N$ between topological spaces is continuous at $x$
iff
the pre-image of any open neighborhood $V_{f(x)}$ is open.
\\
b) A map $f$ is continuous iff pre-images of open sets are open.
\end{exercise}

\subsubsection{Hausdorffness and paracompactness}

A \textbf{\Index{cover}} of a topological space $(M,\Tt)$ is
a family of subsets of $M$ whose union is $M$.
The members (elements) of such family are called
the sets of the cover or simply the\index{cover!sets}
\textbf{cover sets}.
A cover is called\index{locally!finite}
\textbf{locally finite} if every point of $M$
admits an open neighborhood which meets (intersects) only finitely
many cover sets.
A cover is called a \textbf{\Index{refinement}} of another cover
if every member of the former is a subset of some member of the latter.
A cover $\Uu$ is called \textbf{open} if every cover set is open,
in symbols $\Uu\subset \Tt$.\index{open!cover}

\begin{definition}\label{def:HD-paracompact}
A topological space $M$ is called\index{Hausdorff (or $T_2$)}
\textbf{Hausdorff}\index{topological!space!Hausdorff (or $T_2$)}
or {\boldmath$T_2$}\index{$T_2$ -- Hausdorff topological space}
whenever \textbf{the topology separates points}:\index{separating points}
Any two points admit disjoint open neighborhoods.
Such a topology is called a\index{Hausdorff (or $T_2$)!topology}
\textbf{Hausdorff topology}.
If the topology separates any two closed sets, then $M$
is\index{$T_4$ -- normal topological space}
called\index{topological!space!normal (or $T_4$)}
\textbf{normal} or {\boldmath$T_4$}.

A topological space is called\index{topological!space!paracompact}
\textbf{\Index{paracompact}} if every open cover $\Uu$
admits a locally finite open refinement $\Vv$.
\end{definition}

\begin{exercise}[Hausdorff property]\label{exc:hausd-comp-closed}
Show the following.
\begin{itemize}
\item[a)]
  The Hausdorff property ($T_2$) is hereditary, normality ($T_4$) is not.\footnote{
    However, \emph{closed} subspaces of normal spaces are
    normal; cf.~\citet[Exc.\,8.1.25]{Muger:2016a}.
    }
\item[b)]
  In Hausdorff spaces points and, more generally, compact sets are closed.
  Thus normal implies Hausdorff.
  \hfill {\footnotesize\color{gray} ($T_4$ $\Rightarrow$ $T_2$)}
\item[c)]
  In Hausdorff spaces limits are unique:
  \[
     \text{$x_n\to y$\, and\, $x_n\to z$}\qquad\Rightarrow\qquad y=z.
  \]
\item[d)]
  Metric spaces are normal ($T_4$). (With respect to the metric topology.)

\end{itemize}
[Hints: a) Counter-example $T_4$~\citet[Cor.\,8.1.47]{Muger:2016a}.
b) Show the complement of a point is open. c) By contradiction
$y\not= z$. d)~\citet[Le.\,8.1.11]{Muger:2016a}.]
\end{exercise}

Whereas already Hausdorff by itself is useful to avoid pathological spaces
like a real line with two origins, for a Hausdorff space
paracompactness is equivalent to existence of a continuous partition
of unity subordinate to any given open cover.
For a concise presentation including proofs
we recommend~\citet[\S 2.2]{Cieliebak:2018a}.

\subsubsection{Surjections}

\begin{lemma}\label{le:cont-surj-dense}
Let $M_\infty$ be a dense subset of a topological space $M$.
Then the image of $M_\infty$ under any continuous surjection
$f\colon M\twoheadrightarrow Y$ is a dense subset $f(M_\infty)$ of the
target topological space~$Y$.
\end{lemma}

\begin{proof}
Suppose by contradiction that there is a non-empty open set
$V\subset Y$ disjoint to $f(M_\infty)$.
Then\index{$f^{-1} V$ pre-image}
the \textbf{\Index{pre-image}}
\[
     f^{-1} V:=\{x\in M\mid f(x)\in V\}\subset 2^M
\]
is an open subset of $M$ by continuity of $f$
and non-empty as $f$ is surjective. But
\[
     f^{-1} V\cap M_\infty
     =f^{-1}(V\cap f(M_\infty))
     =f^{-1}\emptyset
     =\emptyset
\]
which contradicts density of $M_\infty$ in $M$.
\end{proof}

\subsubsection{Compact-open topology}

Let $C(M,N)$ be the set of continuous functions between topological spaces
$M$ and $N$. Any pair given by a compact subset $K\subset M$ and an
open subset $U\subset N$ determines a collection of continuous functions
\begin{equation}\label{eq:co-top-basic-coll-general}
     \Ff_{K,U}:=\{f\in C(M,N)\mid f(K)\subset U\}\in 2^{C(M,N)}.
\end{equation}
Let $\Ff=\{\Ff_{K,U}\}_{K,U}\subset 2^{C(M,N)}$ be the family of all
such collections and denote by $\Ttco:=\Tt_{\Bb_\Ff}$ the associated
topology on the set $C(M,N)$; cf. Exercise~\ref{exc:sub-basis-top}.
It consists of arbitrary unions of finite
intersections of elements of $\Ff$.
One calls $\Ttco$ the 
\textbf{\Index{compact-open topology} on \boldmath$C(M,N)$},
cf.~\citet[Ex.\,2.6.9]{Narici:2011a}, 
notation\index{$C(M,N)$ continuous z@$\Cco(M,N)$ is $C(M,N)$ endowed with compact-open topology}
\begin{equation}\label{eq:Cco_c}
     \Cco(M,N):=\left(C(M,N),\Ttco\right).
\end{equation}

\begin{exercise}\label{exc:c-topology}
a) Show that $\Cco(M,N)$ is Hausdorff if the target $N$ is.
\\
b) For metric spaces $(N,d)$ convergence in $\Ttco$ is
equivalent to uniform convergence on compact sets: Show that
$f_n\to f$ in $\Ttco$ if and only if
\[
     d_\infty^K\left(f_n,f\right)
     :=\sup_{x\in K} d\left(f_n(x),f(x)\right)
     \to 0
\]
for every compact subset $K\subset M$.

\vspace{.1cm}\noindent 
[Hints: a) \citet[Ch.\,XII]{Dugundji:1966a}
or \citet[Le.\,7.9.1]{Muger:2016a}.
b) Cf. Proposition~\ref{prop:norm-top=bo-top}.]
\end{exercise}

\begin{remark}[Only sub-basis]\label{rmk:only-sub-basis}
In general, the collections $\Ff_{K,U}$ do not form a basis
for the compact-open topology 
$$
     \Ttco:=\Tt_{\Bb_\Ff}
$$
in symbols $\Ff\subsetneq\Bb_\Ff$, in general.
Indeed it is not necessarily true that any non-empty intersection
\[
     \emptyset\not=\left(\Ff_{K_1,U_1}\CAP\Ff_{K_2,U_2}\right)
     \supset\Ff_{K,U}\not=\emptyset.
\]
contains a non-empty family member $\Ff_{K,U}\in\Ff$ (let alone one that
contains a given point; cf. Theorem~\ref{thm:basis-some-topology}).
Hence $\Ff$ cannot be a basis: Indeed if $\Ff$ was a basis, then the
non-empty LHS was open, hence a union of members of $\Ff$ -- at least
one of which non-empty. We encountered two basis counter-examples on
math.stackexchange.com:
  \\
  \textbf{Counter-example A.}\index{counter-examples:!sub-basis only}
  Let $M=N=\{a,b\}$ with the discrete topology $\Tt=2^M$
  and\index{topology!discrete --}
  let $K_1=U_1=\{a\}$ and $K_2=U_2=\{b\}$.
  Then $\Ff_{K_1,U_1}\CAP\Ff_{K_2,U_2}=\{\id_M\}$
  contains only one element, the identity map.
  The inclusion $\Ff_{K,U}\subset\Ff_{K_1,U_1}\CAP\Ff_{K_2,U_2}$
  implies $K\supset K_1\CUP K_2=M\not=\emptyset$, hence $K=M$.
  Thus non-emptiness of $\Ff_{K,U}$ requires $U\not=\emptyset$.
  But $\Ff_{M,U}$ is not a subset of, equivalently equal to, the singleton
  $\{\id_M\}$ in any of the three possibilities $U=\{a\},\{b\},\{a,b\}$.
  \\
  \textbf{Counter-example B.} $M=N=\R$ with the standard topology.
  One can show that there are no subsets $K\subset\R$ compact and
  $U\subset\R$ open such that
  \[
     \emptyset\not=\left(\Ff_{\{0,1\},(0,1)}\CAP\Ff_{\{1,2\},(0,2)}\right)
     \supset\Ff_{K,U}\not=\emptyset
  \]
  by constructing certain continuous functions subject to (non-linear)
  pointwise constraints.
\end{remark}

\subsection{Topological vector spaces}\label{sec:TVS}

For topological vector spaces and, most importantly, topologies
on the vector space of continuous linear maps between them we
recommend the books by~\citet{Rudin:1991b},~\citet[III \S 3]{Schaefer:1999a},~\citet[\S 2.6]{Narici:2011a} (here the additive topological
group is investigated first and scalar multiplication is superimposed
only from Ch.\,4 onward), and~\citet{Treves:1967a}.
There is a book of
counter-examples by~\citet[CH.\,2]{Khaleelulla:1982a}.
The present section was originally inspired by the excellent Lecture
Notes by Kai~\citet{Cieliebak:2018a}.

\begin{definition}\label{def:TVS}
A\index{topological!vector space}
\textbf{topological vector space} (TVS) is a vector space $X$
endowed\index{TVS!topological vector space}
with\index{topology!compatible with vector space operations}
a\index{compatible topology}
topology \textbf{compatible with the vector space operations}
in the sense that both scalar multiplication $\R\times X\to X$ and
addition $X\times X\to X$, are continuous maps.
Also it is required that points are
closed.\footnote{Many books on topological
vector spaces do not require closedness of points.}
\end{definition}

\begin{lemma}\label{le:inv-trans-dil}
For a TVS $X$ (without using closedness of points) it holds:
\begin{itemize}
\item[\rm (i)]
  The closure of a linear subspace is again a linear subspace.
\item[\rm (ii)]
  Given a vector $y\in X$ and a scalar $\alpha\in \R$,
  translation $y+\cdot\colon X\to X$ and dilation $\alpha\cdot\colon X\to X$
  are linear homeomorphisms. Consequence:

  \textrm{\bf Invariance under translation and dilation.}
  If $U$ is an open subset of $X$, then so are $x+U$ and $tU$ for all $x\in X$
  and $t\in\R\setminus \{0\}$.\footnote{
    Consequently the open sets containing $0$ determine
    all open sets, hence the topology.
    }
\item[\rm  (iii)]
  Any\index{neighborhood!symmetric --}\index{symmetric!neighborhood}
  open neighborhood $V$ of $0$ contains an open neighborhood
  $U$ of $0$ which is \textbf{symmetric} $(U=-U)$ and fits into $V$
  ``twice'' $(U+U\subset V)$.
\item[\rm  (iv)]
  \textrm{\bf Closed and compact subsets are separated in a strong sense.}
  For any closed set $C$ and any disjoint compact set $K$
  there is an open neighborhood $U_0$ of $0$ such that the
  open neighborhoods $C+U_0$ of $C$   and $K+U_0$ of $K$
  are still disjoint,\footnote{
    Disjointness remains true if one takes
    the closure of either $C+U_0$ or of $K+U_0$.
    }
  in symbols $(C+U_0)\CAP (K+U_0)=\emptyset$.
\end{itemize}
\end{lemma}

\begin{proof}
(i) \citet[Thm.\,4.4.1]{Narici:2011a}.
(ii) \citet[Thm.\,4.3.1]{Narici:2011a}.
(iii) By continuity of addition and as $0+0=0\in V$
there are open sets $W\ni 0$ and $\tilde W\ni 0$
with $W+\tilde W\subset V$. The open set $\tilde U:=W\CAP\tilde W$
satisfies $\tilde U+\tilde U\subset V$.
The open~set $U:=\tilde U\CAP -\tilde U$ is symmetric
and $U+U\subset \tilde U+\tilde U\subset V$.
(iv) \citet[Thm.\,1.10]{Rudin:1991b}.
\end{proof}

Because of the requirement that points of a TVS are closed,
part (iv) of the previous lemma applies to
$C=\{x\}$ and $K=\{y\}$ and yields disjoint
open neighborhoods of any two points $x\not= y$ of $X$.
This proves

\begin{corollary}\label{cor:TVS-Hausdorff}
A TVS is Hausdorff.
\end{corollary}

\begin{definition}\label{def:TVS-linear-maps}
(i) A subset $A$ of a TVS\index{bounded!subset of TVS}
is called a \textbf{bounded set}\index{subset!bounded --}
if for each open neighborhood $U\subset X$ of $0$
there is a constant $s>0$ such that $A\subset t U$
is contained in the rescaled neighborhood \emph{for all}\,\footnote{
  If $A\subset t U$ for some $t$, isn't the inclusion automatically
  true for all larger values of $t$?
  }
parameters $t>s$.

(ii) A linear map $T\colon X\to Y$\index{TVS!bounded linear map between --}
between\index{bounded!linear map between TVS's}
topological vector spaces is called \textbf{bounded}
if it takes bounded sets to bounded sets
and it is called \textbf{compact}\index{compact!linear operator}
if it takes bounded sets to\index{pre-compact!set}
\Index{pre-compact set}s (compact closure).
\end{definition}

\begin{exercise}[Bounded sets]
Subsets of a bounded set are clearly bounded.
If $A$ and $B$ are bounded sets, so are
$A\CUP B$, $A+B$ and $\alpha A$ whenever $\alpha\in\R$.

\vspace{.1cm}\noindent 
[Hint: If you get stuck consult~\citet[I\,\S\,5.1]{Schaefer:1999a}.]
\end{exercise}

\begin{lemma}
In a TVS $X$ compact subsets are closed and bounded,
whereas the reverse holds iff $\dim X<\infty$.
\end{lemma}

\begin{proof}
Exercise~\ref{exc:hausd-comp-closed}~b)
and \citet[Thm.\,1.15~b)]{Rudin:1991b}.
\end{proof}

\subsubsection{Spaces of linear maps as topological vector spaces
 -- $\SSfrak$-topologies}
Given\index{$\Ll(X,Y$ cont.lin.ops.}
topological vector spaces $X$ and $Y$, the set
\[
     \Ll(X,Y)
\]
of all continuous linear operators $T\colon X\to Y$
is a vector space under addition of two operators $T,S\in\Ll(X,Y)$,
defined by $(T+S)x:=Tx+Sx$, and scalar multiplication
with real numbers $\alpha\in\R$, defined by
$(\alpha T)x:= \alpha Tx$, both whenever $x\in X$.

We will review the standard abstract machinery that produces various
topologies on $\Ll(X,Y)$ for which both operations are continuous,
see e.g.~\citet[\S 11.2]{Narici:2011a} or~\citet[III.3]{Schaefer:1999a}.
For some of them points $T$ are closed,
so the operator space $\Ll(X,Y)$ endowed with such topology is a TVS:
An example is one of the most popular topologies, namely,
the\index{topology!c- --}\index{c-topology!is the compact-open topology}
compact-open topology or \textbf{c-topology} on $\Ll(X,Y)$.
Replacing the family of compact sets by any non-empty
family of bounded sets closed under finite unions
still guarantees that the generated topology is compatible
with addition and scalar multiplication.
Hausdorffness might be lost if the sets in the family are not any more
compact, but it can be recovered by assumptions on $Y$,
e.g. being normed.
\\
Actually all one needs are topological spaces $M$ and $N$;
cf. Exercise~\ref{exc:c-topology}.
How one arrives at the c-topology by generalizing a natural
construction which provides the point-open topology, or \textbf{p-topology},
is\index{topology!p- --}\index{p-topology!is the point-open topology}
nicely explained in~\citet[\S\,7.9.1]{Muger:2016a}.

\begin{exercise}\label{exc:cont-at-0}
Let $T\colon X\to Y$ be a linear map between topological vector spaces.
(i) Show that $T$ is continuous iff it is\index{continuous!at $0$}
\textbf{continuous at \boldmath$0$},
meaning that the pre-image of every open neighborhood of $0$ is open.
(ii) Show that continuity implies boundedness of $T$. (The reverse
holds if the domain $X$ is a Fr\'{e}chet space.)
\end{exercise}

Let $X$\index{$\SSfrak$-family}
and $Y$ be topological vector spaces.
Let $\mbf{\SSfrak}\subset 2^X$ be a non-empty family of
subsets $A$ of $X$, closed under finite unions, that is 
\[
     A_1,\dots,A_k\in \SSfrak\qquad\Rightarrow\qquad
     A_1\CUP\dots\CUP A_k\in\SSfrak.
\]
Examples are the families
\[
     \text{$\SSfrak_{\rm p}$ / $\SSfrak_{\rm c}$ / $\SSfrak_{\rm b}$
     $=\{$all finite-\underline{p}oint / \underline{c}ompact /
     \underline{b}ounded subsets of $X\}$.
     }
\]

\begin{definition}[Basic collections]\label{def:basic-collections}
For $A\in \SSfrak\subset 2^X$ and any element $U$ of the \textbf{family
\boldmath$\Uu_0$ of open neighborhoods of $0$ in $Y$} consider the collection
$\Bb_{A,U}$ of all continuous linear operators which map $A$ into $U$, in symbols
\begin{equation}\label{eq:basic-collection}
     \Bb_{A,U}:=\{T\in\Ll(X,Y)\mid T(A)\subset U\}\in 2^{\Ll(X,Y},\quad
     A\in\SSfrak,\,U\in\Uu_0
\end{equation}
Collections of the form $\Bb_{A,U}$ are called\index{basic!collection}
\textbf{basic collections}.
\end{definition}

\begin{lemma}\label{le:BbSU-family}
a) Any basic collection $\Bb_{A,U}\ni 0$ contains the zero operator.\\
b) Any intersection $\Bb_{12}:=\Bb_1\CAP \Bb_2$ of two basic collections
contains one, i.e.
\begin{equation}\label{eq:Bb-not}
     \Bb_3\subset\left(\Bb_1\CAP \Bb_2\right)\subset\Ll(X,Y),
     \quad \Bb_i:=\Bb_{A_i,U_i}
\end{equation}
for some $A_3\in\SSfrak\subset 2^X$ and some open origin neighborhood
$U_3\in\Uu_0\subset 2^Y$.
\\
c) If $U+U\subset V$, then $\Bb_{A,U}+\Bb_{A,U}\subset \Bb_{A,V}$.
\\
d) If $r\in\R\setminus\{0\}$, then $r\Bb_{A,U}=\Bb_{r^{-1}A,U}=\Bb_{A,rU}$.
\end{lemma}

\begin{proof}
Let $T,S\in\Bb_{A,U}$.
a) Obvious. b) $\Bb_{A_1\CUP A_2, U_1\CAP U_2}$.
c) $\forall a\in A\colon$ $(T+S)a=Ta+Sa\in U+U\subset V$.
d) $(rT)(r^{-1} A)\subset U$ and $(rT)(A)\subset rU$.
\end{proof}

We\index{$\Bb(0)$ family of all basic collections $\Bb_{A,V}$}
denote by $\Bb(0)$ the\index{family!of all basic collections}
\textbf{family of all basic collections},\footnote{
  A collection of non-empty sets such that the intersection
  of any two of them contains another one is called a 
  \textbf{\Index{filter base}}.
  So $\Bb(0)$ is a filter base and so is each translate $\Bb(x)$.
  }
in symbols
\[
     \Bb(0)
     :=\{\Bb_{A,U}\mid \text{$A\in\SSfrak$, $U\in\Uu_0$}\}
     \subset 2^{\Ll(X,Y)}.
\]
The notation reminds us that each element $\Bb_{A,U}$ of $\Bb(0)$
contains the zero operator. For $T\in\Ll(X,Y)$ let $\Bb(T):=T+\Bb(0)$
be the translated family. We denote by
\begin{equation*}
     \Bb=\Bb_\SSfrak^{\Uu_0}
     :=\bigcup_{T\in\Ll(X,Y)} T+\Bb(0)\quad \subset 2^{\Ll(X,Y)}
\end{equation*}
the family of all translated basic collections.

\begin{theorem}\label{thm:Bb-basis}
Let $X$ and $Y$ be topological vector spaces. Let $\Uu_0\subset 2^Y$
be the collection of open sets containing the origin of $Y$.
Suppose $\SSfrak\subset 2^X$ is a non-empty family
of \underline{bounded}\,\footnote{
  Boundedness leads to $\Tt_\SSfrak$-continuity of ``$+$'' and scalar
  multiplication on $\Ll(X,Y)$.
  }
subsets $A$ of $X$ which is closed under finite unions. Then the
following is true (not using closedness of points in $X,Y$).
\begin{labeling}{\rm (local basis)}
\item[\rm (local basis)]
  The family $\Bb(0)$ of all basic collections $\Bb_{A,U}\subset\Ll(X,Y)$
  forms a local basis at $0$ of a topology $\Tt_{\SSfrak}$ on $\Ll(X,Y)$
  for which addition and scalar multiplication are
  continuous; cf. Remark~\ref{rmk:TVS-1}.
\item[\rm (basis)]
  The family $\Bb=\Bb_\SSfrak^{\Uu_0}$ is a basis for
  a topology $\Tt_\SSfrak$ on $\Ll(X,Y)$.
\end{labeling}
The\index{$\LlSSfrak(X,Y)$}\index{topology!$\SSfrak$- --}
topology\index{$\SSfrak$-topology on $\Ll(X,Y)$}
$\Tt_\SSfrak$ on $\Ll(X,Y)$,
called \textbf{\boldmath$\SSfrak$-topology}, is
\begin{itemize}
\item[--]
    Hausdorff whenever the linear span of $\bigcup\SSfrak$ is dense in
    $X$ and if $Y$ is Hausdorff;
\item[--]
    locally convex whenever $Y$ is.
\end{itemize}
\end{theorem}

\begin{proof}
See e.g.~\citet[Thm.\,11.2.2]{Narici:2011a}.
\end{proof}

By Exercise~\ref{exc:gguy676}\index{TVS!locally convex --}
any normed vector space $Y$ is a\index{locally!convex TVS}
\textbf{locally convex TVS},
i.e. a TVS such that any neighborhood of $0$ contains a
convex\footnote{
  A\index{convex set}
  subset $C$ of a vector space is called a \textbf{convex set} if $C$
  contains every line segment $\{tx+(1-t)y\mid t\in[0,1]\}$ connecting two
  of its points $x,y\in C$.
  }
one.
\\
In contrast to the basis property of $\Bb$ in the \emph{linear} setting,
recall from Remark~\ref{rmk:only-sub-basis}
that in the general case of topological \emph{spaces}
even for the family of \emph{compact} subsets
the basic collections do not form a basis, only a sub-basis.

\begin{corollary}\label{cor:TVS}
Let $X$ and $Y$ be topological vector spaces.
If $\SSfrak$ covers $X$
(e.g. if $\SSfrak=\SSfrak_{\rm p},\SSfrak_{\rm c},\SSfrak_{\rm b}$),
then $\LlSSfrak(X,Y):=\left(\Ll(X,Y),\Tt_\SSfrak\right)$ is a TVS.
\end{corollary}

The following topologies associated to the indicated
families $\SSfrak$ are called
\begin{itemize}
\item[-]
  $\Llpo(X,Y):=\Ll_{\SSfrak_{\rm p}}(X,Y)$
  point-open\index{topology!point-open --}
  or\index{$\Llpo(X,Y)$ point-open topology}\index{point-open topology}
  \textbf{\Index{p-topology}}
\item[-]
  $\Llco(X,Y):=\Ll_{\SSfrak_{\rm c}}(X,Y)$
  compact-open\index{topology!compact-open --}
  or\index{$\Llco(X,Y)$ compact-open topology}\index{compact-open topology}
  \textbf{\Index{c-topology}}
\item[-]
  $\Llbo(X,Y):=\Ll_{\SSfrak_{\rm b}}(X,Y)$
  bounded-open\index{topology!bounded-open --}
  or\index{$\Llbo(X,Y)$ bounded-open topology}\index{bounded-open topology}
  \textbf{\Index{b-topology}}
\end{itemize}

\begin{remark}[Continuous vector operations]\label{rmk:TVS-1}
Suppose $X$ and $Y$ are topological vector spaces.
Continuity of addition and scalar multiplication under a
$\SSfrak$-topology is equivalent to boundedness of every image
$TA\subset Y$ where $T\in\Ll(X,Y)$ and $A\in\SSfrak$; see
e.g.~\citet[III \S 3.1]{Schaefer:1999a} or~\citet[III\,\S 3 Prop.\,1]{Bourbaki:1987a}.
\end{remark}

\begin{exercise}[Families of compact sets]\label{exc:Bb-family}
For each of the three families
$\SSfrak=\SSfrak_{\rm p},\SSfrak_{\rm c},\SSfrak_{\rm b}$
show the particular assertion of Theorem~\ref{thm:Bb-basis}
that $\Tt_\SSfrak$ is a topology on $\Ll(X,Y)$
and $\Bb=\Bb_\SSfrak^{\Uu_0}$ is a basis --
in contrast to Remark~\ref{rmk:only-sub-basis}.

\vspace{.1cm}\noindent
[Hints:
Theorem~\ref{thm:basis-some-topology} or Exercise~\ref{exc:jhj45667}.
Lemma~\ref{le:inv-trans-dil} iv).]
\end{exercise}

\subsubsection*{Continuity properties}

\begin{proposition}\label{prop:cont-TVS}
Suppose $M$ is a topological space and $X,Y,Z$ are topological vector
spaces. Then the following is true.
\begin{itemize}
\item[a)]
  If the map $\varphi\colon M\times Y\to Z$ is continuous and, moreover, linear in the
  second variable, then the induced map
  \begin{equation}\label{eq:a1}
     \Phi\colon M\to\Llco(Y,Z),\quad p\mapsto \varphi(p,\cdot)
  \end{equation}
  is 
  continuous, in symbols $\Phi\in C(M,\Llco(Y,Z))$.
\item[b)]
  If $S\colon X\to Y$ is a compact linear operator, then the induced map
  \begin{equation}\label{eq:a2}
     \Llco(Y,Z)\to\Llbo(X,Z),\quad T\mapsto TS:=T\circ S
  \end{equation}
  is continuous.
\item[c)]
  For $\Phi$ and $S$ as in a) and b) the induced map
  \begin{equation}\label{eq:a3}
     \Psi\colon M\to\Llbo(X,Z),\quad p\mapsto \Phi(p) S
  \end{equation}
  is continuous. (Juxtaposition of linear maps
  means composition.)
\end{itemize}
\end{proposition}

For normed vector spaces $X$ and $Y$ both topological vector spaces
$\Llbo(X,Y)$ and $\Ll(X,Y)$ with the operator norm topology
coincide by Proposition~\ref{prop:norm-top=bo-top}.
\\
Operators similar to the one in~(\ref{eq:a3})
are well known in non-linear analysis under the name
\textbf{\boldmath\Index{Nemitski operator}s associated to $\varphi$};
see e.g.~\citet[\S 1.2]{ambrosetti:1993a}.

\begin{proof}[Proof of Proposition~\ref{prop:cont-TVS}]
a) is even true for topological \emph{spaces} $M,Y,Z$ and
continuous functions $\varphi\colon M\times Y\to Z$, not necessarily
linear in the second variable; see e.g.~\citet[XII.3.1]{Dugundji:1966a}
or~\citet[Le.\,7.9.5]{Muger:2016a}. Now the conclusion 
is that $\Phi$ is continuous as a map $M\to \Cco(Y,Z)$;
cf.~(\ref{eq:Cco_c}).
\\
To prove this we must show that for all $p_0\in M$ and
sub-basis elements $\Ff_{K,V}\subset\Cco(Y,Z)$
that contain $\Phi(p_0)$ there is an open neighborhood $U_{p_0}$ of
$p_0$ in $M$ whose image under $\Phi$ lies in $\Ff_{K,V}$, too.
Equivalently, we have to show that $\varphi(p_0\times K)\subset V$ implies
$\varphi(U_{p_0}\times K)\subset V$ for some open set $p_0\in U_{p_0}\subset M$.
Continuity of $\varphi$ guarantees an \emph{open} pre-image
$\varphi^{-1}(V)\subset M\times Y$ which contains $p_0\times K$ by assumption.
By\index{Lemma!Slice --}
compactness of $K$ the \Index{Slice Lemma},
see e.g.~\citet[XI.2.6]{Dugundji:1966a} or~\citet[Prop.\,7.5.1]{Muger:2016a},
provides an open neighborhood $U_{p_0}$ of $p_0\in M$
such that the thickening $U_{p_0}\times K$ of $p_0\times K$
is still contained in $\varphi^{-1}(V)$.

b) Let's show that the pre-image $\Phi^{-1}\Bb_{A,U}$
of any (open) basis element of the bounded-open topology $\Ttbo(X,Z)$
is open in $\Llco(Y,Z)$, i.e. contains some basis element
$\Bb^\prime_{K,V}\in\Ttco(Y,Z)$ of the compact-open topology.
Given $A\subset X$ bounded and $0\in U\subset Z$ open, note that
$\Phi^{-1}\Bb_{A,U}=\Bb^\prime_{S(A),U}\supset\Bb^\prime_{K,U}\in\Ttco(Y,Z)$
where by definition the compact set $K$ is the closure
of the pre-compact set $S(A)\subset Y$.

c) The composition of continuous maps is continuous. But composing
the continuous maps (\ref{eq:a1}) and (\ref{eq:a2}) is the map (\ref{eq:a3}).
\end{proof}

\subsubsection*{Fr\'{e}chet and G\^{a}teaux  derivative on TVS}

\begin{definition}[Fr\'{e}chet derivative on TVS]\label{def:Frechet-derivative}
Suppose\index{Fr\'{e}chet derivative!on TVS}\index{derivative!on TVS}
$f\colon X\supset U\to Y$ is a map between topological vector
spaces defined on an open subset $U$.

In case $0\in U$ and $f(0)=0$ one says that $f$
\textbf{has derivative zero at \boldmath$0$} if
for each open neighborhood $W_0\subset Y$ of $0$
there is an open neighborhood $V_0\subset X$ of $0$
and a function $o\colon (-1,1)\to\R$ such that
\[
     \lim_{t\to 0} \frac{o(t)}{t}=0,\qquad
     tV_0\subset U,\qquad
     f(tV_0)\subset o(t) W_0
\]
for every $t\in(-1,1)$.

In general, one calls $f$\index{differentiable!map between TVS}
\textbf{differentiable at \boldmath$x\in U$} if there is
a continuous linear operator $D\colon X\to Y$ such that the map
\[
     h\mapsto f(x+h)-f(x)-Dh
\]
has derivative zero at $0$.
In this case $df(x):=D\in\Ll(X,Y)$ is called
the \textbf{\boldmath\Index{derivative} of $f$ at $x$}.
If $f$ is differentiable at every point of $U$ one calls
$f$ (Fr\'{e}chet) \textbf{\boldmath differentiable on $U$}.
In this case the map
\[
     f^\prime:=df\colon U\to\Ll(E,F), \quad x\mapsto df(x)
\]
into the vector space of continuous linear maps $\Ll(X,Y)$ is called
the (Fr\'{e}chet)\index{differential!of $f$}
\textbf{\boldmath differential of $f$}. 
\end{definition}

By Corollary~\ref{cor:TVS}
endowing $\Ll(X,Y)$ with the topology $\Tt_\SSfrak$
associated to any of the families
$\SSfrak=\SSfrak_{\rm p},\SSfrak_{\rm c},\SSfrak_{\rm b}$
results in a TVS denoted by $\LlSSfrak(X,Y)$.
Hence $df\colon U\to \LlSSfrak(X,Y)$ is a map between TVS
and one defines iteratively the higher order differentials
\[
     f^{(\ell)}:=d^\ell f\colon U\to \LlSSfrak (E,\LlSSfrak (E,\dots \LlSSfrak (E,F))).
\]
For normed vector spaces $X$ and $Y$ the bounded-open topology
$\Tt_{\SSfrak_{\rm p}}$ and the operator norm topology on $\Ll(X,Y)$
coincide by Proposition~\ref{prop:norm-top=bo-top} below.

We say that a map $f\colon X\supset U\to Y$
admits directional derivative at $x\in U$ in direction $\xi\in X$,
if there are $\eps>0$ and $\eta\in Y$ such that the map 
\[
     (-\eps,\eps)\to Y,\quad t\mapsto f(x+t\xi)-f(x)-\eta
\]
has derivative zero at $0$.
In this case $\p_\xi f(x):=\eta$ is called
the derivative of $f$ at $x$ in direction $\xi$.
If the map $\p f(x)\colon X\to Y$, $\xi\mapsto \p_\xi f(x)$, is defined for every
$\xi\in X$ and is linear and continuous, then $f$
is said \textbf{\boldmath G\^{a}teaux differentiable at $x$}
with\index{G\^{a}teaux!differentiable}\index{G\^{a}teaux!derivative}
\textbf{G\^{a}teaux derivative} $\p f(x)\in\Ll(X,Y)$.

The (Fr\'{e}chet) derivative on topological vector spaces
enjoys some basic properties such as the chain rule
and the fact that Fr\'{e}chet differentiability
implies continuity and G\^{a}teaux differentiability.
However, other fundamentals are not available, for instance 
the implicit function theorem
and Cartan's last theorem may fail on TVS;
see examples in~\citet[\S 4.2]{Cieliebak:2018a}.

\subsection{Metric spaces}\label{sec:metric-spaces}

\begin{definition}\label{def:metric}
A \textbf{\Index{metric}} on a set $M$ is a function
$d\colon M\times M\to [0,\infty)$ that satisfies the following three axioms
whenever $ x,y,z\in M$.
\begin{enumerate}
\item[(i)]
  $d(x,y)=d(y,x)$
  \hfill (Symmetry)
\item[(ii)]
  $d(x,z)\le d(x,y)+d(y,z)$
  \hfill (Triangle inequality)
\item[(iii)]
  $d(x,y)=0$ $\Leftrightarrow$ $x=y$
  \hfill (Non-degeneracy)
\end{enumerate}
Such pair $(M,d)$ is called a \textbf{metric space}.\index{metric!space}
\end{definition}

The prototype example of a metric $d$ is the distance between two points
in euclidean space. Hence a metric is also called a 
\textbf{\Index{distance function}}.
We often use the notation $M_d$ for a metric space, meaning that
$M$ is a set endowed with the metric $d$.

\begin{definition}\label{def:metric-space}
  A metric space $M_d$ comes naturally with the 
  \textbf{metric topology}\index{metric!topology}
  $\Tt_d$\index{$\Tt_d$ metric topology}
  whose basis $\Bb_d$ consists of the open balls $B^d_x(\eps)$
  of\index{$B_x(\eps)$ open ball of radius $\eps$ about $x$}
  all radii $\eps>0$ about all points $x$ of $M_d$.
A metric space will be automatically endowed with the metric topology,
unless mentioned otherwise. 
\end{definition}

\begin{exercise}[Metric topology]
Check that the collection $\Bb_d$ of all open balls in $M_d$ indeed forms a
basis for a topology, and not just a sub-basis.
\end{exercise}

As mentioned earlier, metric spaces are normal, thus Hausdorff. Moreover,
second countability (countable basis) is equivalent to separability
(dense sequence).

\begin{exercise}[Convergent sequence]
Check that $x_n\to y\in M_d$, in the sense of
Definition~\ref{def:top-conv}, if and only if
any $\eps$-ball about $y$ contains all but finitely
many sequence members $x_n$, in symbols
\[
     \forall\eps>0\;\;\exists N\in\N\colon\quad
     d(x_n,y)<\eps\quad \forall n\ge N.
\]
\end{exercise}

\subsubsection*{Sequential convergence properties}

\begin{proposition}\label{prop:co-topology-sup-metric}
Let $Q$ be a compact topological space and $M_d$ a metric space.
Then the compact-open topology $\Ttco$ on $C(Q,M_d)$
coincides with the metric topology $\Tt_{d_\infty}$ associated
to\index{$d_\infty$ supremum metric}
the \textbf{\Index{supremum metric}}
\[
     d_\infty(f,g):=\sup_{q\in Q} d(f(q),g(q)),\quad
     f,g\in C(Q,M_d).
\]
\end{proposition}

\begin{proof}
\citet[Prop.\,7.9.2]{Muger:2016a}.
(To show equality of two topologies one shows that the members of a
basis, or of a sub-basis, of the first topology are
open with respect to the second topology, and vice versa.)
\end{proof}

\begin{exercise}
If $Q$ is compact, then
$d_\infty$ is a metric on $C(Q,M_d)$.

\vspace{.1cm}\noindent 
[Hint: If stuck, consult e.g.~\citet[Prop.\,2.1.25]{Muger:2016a}.]
\end{exercise}

Convergence $f_\nu\to g$ with respect to
$\Tt_{d_\infty}=\Tt_{d_\infty}(Q,M_d)$,  that is
\[
     \forall \eps>0\;\exists\nu_\eps\colon\quad
     d\left(f_\nu(q),g(q)\right)\le\eps\quad
     \text{whenever $\nu\ge\nu_\eps$ and $q\in Q$}
\]
is called\index{uniform!convergence}
\textbf{uniform convergence} on the compact set $Q$.

\begin{exercise}
Let $N$ be a topological space.
If $M_d$ is a metric space, the compact-open topology $\Ttco$
on $C(N,M_d)$ coincides with the topology
\[
     \bigcap_{Q\subset N\;\text{compact}} \Tt_{d_\infty}(Q,M_d)
\]
of\index{uniform!convergence!on compact subsets}
\textbf{uniform convergence on all compact subsets} $K$ of $N$.

\vspace{.1cm}\noindent 
[Hint: If $N\supset Q$ compact, then
$\Ttco(N,M_d)\supset\Ttco(Q,M_d)=\Tt_{d_\infty}(Q,M_d)$.]
\end{exercise}

\begin{definition}[Equicontinuous family]
Let $N$ be a topological space and $M_d$ a metric space.
A family $\Ff\subset\Map(N,M_d)$ of maps, a-priori continuous or not,
is\index{family!equicontinuous --}
called \textbf{\Index{equicontinuous}}
if for every $x\in N$ and every $\eps>0$
there is an open neighborhood $U_x$ of $x$ such that for all
neighborhood elements $x^\prime\in U_x$ and family members $f\in\Ff$
both values $f(x)$ and $f(x^\prime)$ are $\eps$-close, in symbols
\[
     d\left(f(x),f(x^\prime)\right)<\eps,\quad
     \text{$x^\prime\in U_x$, $f\in\Ff$.} 
\]
\end{definition}

\begin{exercise}
The members of an equicontinuous family $\Ff$
are continuous.
\end{exercise}

\subsubsection{Complete metric spaces -- 
Theorem of Baire and Arzel\`{a}--Ascoli}

\begin{definition}
A sequence $(x_n)$ in a metric space $M_d$ is called a
\textbf{\Index{Cauchy sequence}}
if for every $\eps>0$ there is a sequence member $x_N$ such that
any two subsequent members are within distance $\eps$ of one another,
in symbols
\[
     \forall\eps>0\;\;\exists N\in\N\colon\quad
     d(x_n,x_m)<\eps\quad \forall n,m\ge N.
\]
\end{definition}

\begin{exercise}
Check that every convergent sequence in a metric space is a Cauchy
sequence, but the converse is not true.
\end{exercise}

\begin{definition}[Complete metric space]
A metric space in which every Cauchy sequence converges
is called \textbf{complete} and so is
the\index{complete!metric space}\index{metric!space!complete --}
metric.
\end{definition}

\begin{exercise}
Let $Q$ be a compact topological space.
Then the metric space $\left(C(Q,M_d),d_\infty\right)$
is complete iff the target metric space $M_d$ is complete.

\vspace{.1cm}\noindent 
[Hint: If stuck, consult
e.g.~\citet[Prop.\,3.1.18 and Rmk.\,5.2.12]{Muger:2016a}.]
\end{exercise}

\begin{theorem}[{\Index{Baire's Theorem}}]\label{thm:Baire}
Let\index{Theorem of!Baire}
$M_d$ be a complete metric space and
$(U_n)$ a sequence of open and dense subsets.
Then the intersection
\[
     \bigcap_{n=1}^\infty U_n
\]
is dense in $M_d$.
\end{theorem}

\begin{proof}
See e.g.~\citet[Thm.\,3.3.1]{Muger:2016a}.
\end{proof}

Among the many applications of Baire's Theorem
are the open mapping theorem and the Banach--Steinhaus
Theorem~\ref{thm:Banach-Steinhaus}, also called the
principle of uniform boundedness.

\begin{theorem}[Arzel\`{a}--Ascoli Theorem]\label{thm:AA-F-orbit}
Let\index{Theorem of!Arzel\`{a}--Ascoli}
$Q$ be a compact topological space and
$M_d$ a complete metric space.
Then the following is true.
A family
\[
     \Ff\subset\Cco(Q,M_d)
\]
of continuous maps is
pre-compact (with respect to the supremum metric $d_\infty$)
if and only if the family $\Ff$ is equicontinuous
and the\index{$\Ff(q) $ orbit through $q$}\index{orbit!family $\Ff$- --}
\textbf{\boldmath$\Ff$-orbit} through each domain point $q\in Q$,
namely each subset
\[
     \Ff(q):=\{f(q)\mid f\in \Ff\}\subset M_d,\quad q\in Q
\]
is pre-compact.
\end{theorem}

\begin{proof}
See e.g.~\citet[Thm.\,7.7.67]{Muger:2016a}.
\end{proof}

\subsection{Normed vector spaces}\label{sec:NVS}

\begin{definition}
A \textbf{\Index{norm}} on a (real) vector space $X$ is a function
$\norm{\cdot}\colon X\to [0,\infty)$ that satisfies the following three axioms
\begin{enumerate}
\item[(i)]
  $\norm{\alpha x}=\abs{\alpha}\norm{x}$
  \hfill (Homogeneity)
\item[(ii)]
  $\norm{x+y}\le \norm{x}+\norm{y}$
  \hfill (Triangle inequality)
\item[(iii)]
  $\norm{x}=0$ $\Leftrightarrow$ $x=0$
  \hfill (Non-degeneracy)
\end{enumerate}
for all $ x,y\in X$ and $\alpha\in\R$.
Such\index{vector space!normed --}
pair $(X,\norm{\cdot})$ is called a \textbf{\Index{normed vector space}},
often just denoted by $X$.
If one drops the requirement $\norm{x}=0$ $\Rightarrow$ $x=0$ in (iii),
one\index{norm!semi- --}
obtains the definition of a \textbf{\Index{semi-norm}} on $X$.
\end{definition}

The prototype example of a norm $\norm{\cdot}$ is the distance of a
point in euclidean space from the origin.

\begin{exercise}[Normed $\Rightarrow$ metric and TVS with convex basis]
\label{exc:gguy676}
Suppose $(X,\norm{\cdot})$ is a normed vector space.
Show the following.

a) The definition
\[
     d_{\|}(x,y):=\Norm{x-y},\quad x,y\in X
\]
provides a (translation invariant: $d(x+z,y+z)=d(x,y)$)
metric on $X$.
\\
{\small \textit{(So normed vector spaces are endowed with a
natural topology, the metric topology  $\Tt_{d_{\|}}$. Because metric
topologies are Hausdorff, limits are unique.)}}

b) Both vector operations,
addition and scalar multiplication, are continuous.
\\
{\small \textit{(So any normed vector space $X$ is a TVS.)}}

c) Open balls $B_\eps(x)$ of radius $\eps>0$ centered at
$x\in X$ are convex sets. So the natural basis of the topology of a
TVS $X$ given by all open balls consists of convex sets.
\\
{\small \textit{(So by Theorem~\ref{thm:Bb-basis}
the space $\Ll(X,Y)$ of continuous linear operators
between normed vector spaces is a locally convex TVS under
the point-open, compact-open, and bounded-open topologies;
with respect to the latter it is even normed as we will see.)}}

\vspace{.1cm}\noindent 
[Hints: b) Addition: triangle inequality,
scalar multiplication: homogeneity.]
\end{exercise}

\begin{exercise}[The normed vector space $\Ll(X,Y)$]
Let $X$ and $Y$ be normed vector spaces.
Recall Definition~\ref{def:TVS-linear-maps} on boundedness.
Show that

a) A\index{linear operator!bounded}
linear map $T\colon X\to Y$ is continuous iff it is bounded
iff it maps the open unit ball about $0$ into one of finite radius $r$,
in symbols $TB_1\subset B_r$.

b) Now consider the vector space $\Ll(X,Y)$ that consists of all
bounded linear operators $T\colon X\to Y$ with addition
$T+S\colon x\mapsto Tx+Sx$ and scalar multiplication
$\alpha T\colon x\mapsto \alpha Tx$ for $\alpha\in\R$.
Taking the infimum of all radii $r>0$ of balls $B_r\supset TB_1$
still containing the image under $T$ of the unit ball
defines\index{$(norm$@$\norm{T}$ operator norm}
a norm 
\begin{equation*}
\begin{split}
     \norm{\cdot}=\norm{\cdot}_{\Ll(X,Y)}\colon \Ll(X,Y)&\to[0,\infty)
     \\
     T&\mapsto \inf\{r>0\mid TB_1\subset B_r\}
\end{split}
\end{equation*}
called\index{norm!operator --}
the \textbf{\Index{operator norm}}.
Alternatively,
it is given by
\[
     \norm{T}
     =\sup_{\norm{x}\le 1} \norm{Tx}
     =\sup_{\norm{x}=1} \norm{Tx}
     =\sup_{\norm{x}<1} \norm{Tx}.
\]
[Hints: a) Cf.~\citet[1.29]{Rudin:1991b}.]
\end{exercise}

By Exercise~\ref{exc:gguy676} the normed vector space
$\left(\Ll(X,Y),\norm{\cdot}\right)$ carries a natural
metric $d_{\|}$ and is a locally convex TVS.
The metric topology $\Tt_\|=\Tt_\|(X,Y)$ is called the
\textbf{operator norm topology} or
the\Index{operator norm!topology}\index{uniform!topology}
\textbf{uniform topology}, also indicated by $\Ll_\|(X,Y)$.
\\
\textbf{Convention:} Whenever we speak of $\Ll(X,Y)$ as a normed vector space
it is automatically endowed with the operator norm topology.

\begin{proposition}[Operator norm topology is bounded-open topology]
\label{prop:norm-top=bo-top}
For normed vector spaces $X$ and $Y$ the bounded-open and the operator
norm topology on $\Ll(X,Y)$ coincide, in symbols $\Ttbo=\Tt_\|$.
\end{proposition}

\begin{proof}
Let $\norm{\cdot}$ be the operator norm on $\Ll(X,Y)$.
Balls are centered at $0$.

$\Tt_\|\subset\Ttbo\colon$
It suffices to show that norm open balls $B_r:=\{\norm{\cdot}<r\}$ are
open with respect to $\Ttbo$. This means that $B_r$ must contain together
with any element $S$ a whole $\Ttbo$-open neighborhood $S+\Bb_{A,U}$
where $\Bb_{A,U}=\{T\in\Ll(X,Y)\mid T(A)\subset U\}$ with $A\subset X$
bounded and $0\ni U\subset Y$ open; cf.~(\ref{eq:basic-collection}).
\\
To see this abbreviate $s:=\norm{S}\in[0,r)$ and let $A$ be the closed unit
ball in $X$ and $U$ the open ball in $Y$ of radius $\frac{r-s}{2}$.
For $T\in\Bb_{A,U}$ we get
\[
     \Norm{S+T}
     \le\Norm{S}+\Norm{T}
     =s+\sup_{x\in A}\Norm{Tx}_Y
     \le s+\frac{r-s}{2}=\frac{r+s}{2}<r.
\]
Hence the $\Ttbo$-open neighborhood $S+\Bb_{A,U}$
of $T$ is contained in $B_r$.

$\Ttbo\subset\Tt_\|\colon$
By translation invariance of both topologies it suffices to
show\footnote{
  We could have localized to $0$ already in order to prove
  $\Tt_\|\subset\Ttbo$ above.
  }
that each element $\Bb_{A,U}\in\Bb(0)\subset\Ll(X,Y)$ of the local basis of
$\Ttbo$ at $0$ contains an open ball $B_r\in\Tt_\|$ about $0$.
Given $A\subset X$ bounded and $0\in U\subset Y$ open,
pick open balls $A\subset B_r\subset X$ and $B_\eps\subset U\subset Y$.
Then $\norm{T}_{\Ll(X,Y)}<r=\eps/R$ implies that $T\in \Bb_{A,U}$.
Indeed $TA\subset TB_R=RTB_1\subset R B_{\eps/R}=B_\eps\subset U$.
\end{proof}

\subsubsection*{Sequential convergence properties}

\begin{lemma}[Convergence in compact-open topology means
convergence of the orbit through each point]\label{le:co-top-lin}
Let $X$ and $Y$ be normed vector spaces. Consider operators
$(T_\nu)_{\nu\in\N}\subset\Ll(X,Y)\ni T$. Then $T_\nu\to T$ in
$\Llco(X,Y)$ iff for each domain element $\xi$ the image sequence
$T_\nu\xi$ converges to $T\xi$ in
$Y$.
\end{lemma}

\section{Analysis on Banach spaces}\label{sec:Ana-B-sp}

All vector spaces are over the real numbers.
Throughout any linear structure is with respect to the real
numbers and, as a rule of thumb, by $X$ and $Y$ we denote normed
linear spaces and by $E$ and $F$ Banach spaces.
In the context of linear spaces \textbf{\Index{subspace}}
means \emph{linear subspace}.

\subsection{Banach spaces}\label{sec:B-sp}

\begin{definition}\label{def:Banach-space}
A \textbf{\Index{Cauchy sequence}} is a sequence $x_\nu$
in a normed linear space $X$ such that $\norm{x_n-x_m}\to 0$
whenever $n,m\to\infty$.
The norm is called \textbf{complete}\index{complete!norm}
if every Cauchy sequence converges (admits a limit).
A linear space $E$ endowed with a complete norm\footnote{
  also called a complete normed linear space
  }
is called a \textbf{Banach space}.\index{Banach!space}
Any closed linear subspace $F\subset E$
endowed with the norm of $E$ is a Banach space itself,\index{Banach!subspace}
called a \textbf{Banach subspace}.
\end{definition}

Relevant examples of Banach spaces are enlisted in
Theorem~\ref{thm:Lp-spaces}.

\subsubsection*{Direct sum and topological complements}

\begin{definition}[Direct sum]
The\index{direct sum!of Banach spaces}
\textbf{direct sum of Banach spaces} $X\oplus Y$ is the set of
pairs $\{(x,y)\mid x\in X,\, y\in Y\}$ which is equipped with and
complete under the norm
$\norm{(x,y)}:=\norm{x}+\norm{y}$.\footnote{
  Alternatively, use any of the equivalent norms
  $\norm{(x,y)}_p^p:=\norm{x}^p+\norm{y}^p$ for $1\le p<\infty$
  or $\norm{(x,y)}_\infty:=\max\{\norm{x},\norm{y}\}$.
  }
\end{definition}

\begin{definition}[Banach space complement]
A closed subspace $X$ of a Banach space $Z$
is said to be \textbf{complemented} if there
is a \underline{closed} subspace $Y$ of $Z$ such that
$X\cap Y=\{0\}$ and $X+Y=Z$.
In this case we write $X\oplus Y=Z$
and call $Y$ a \textbf{Banach space complement}
or\index{complement!of Banach subspace}
a \textbf{topological complement} of $X$,
one also says that the Banach space\index{Banach!space!splits}
$X$\index{splitting!of Banach space}
\textbf{splits}.
\end{definition}

\begin{example}[Not every closed subspace is complemented]
\label{ex:not-complemented}
Consider\index{counter-examples:!not complemented}
the Banach space
$\ell^\infty:=\{\text{$x\colon \N\to\R$, $\nu\mapsto x_\nu$, bounded}\}$
of bounded real sequences equipped with the sup norm. The subspace
$c_0$ of sequences that converge to zero is closed, but does not admit
a\index{topological!complement}
\textbf{topological!complement}: There is no closed
subspace\index{complement!topological --}
$d$ such that $c_0\oplus d=\ell^\infty$; see~\citet{Whitley:1966a}.
\end{example}

\subsubsection*{Quotient spaces}

\begin{definition}[Quotient space]\label{def:Banach-quotient}
Suppose $X$ is a normed linear space
and $A\subset X$ is a closed linear subspace.
The \textbf{quotient space} of $X$ by $A$\index{quotient!space}
is the set of cosets\footnote{
  equivalently, the set of equivalence classes $\{[x]\mid x\in X\}$
  where $x\sim y$ if $x-y\in A$
  }
denoted and defined by
\[
     X/A:=\{ x+A\mid x\in X\}\subset 2^X.
\]
The function $X/A\to[0,\infty)$ given by the distance of any
point representing the coset $x+A$
to the closed subspace $A$, namely
\[
     \Norm{x+A}_{X/A}:=d(x,A):=\inf_{a\in A} \Norm{x-a}
     =\inf_{y\in x+A}\Norm{y} ,
\]
is called the \textbf{quotient norm}.\index{quotient!norm}
Often we use the shorter
notation $\Norm{x+A}$.
\end{definition}

\begin{exercise}
a) Check that the operations $\alpha(x+A):=\alpha x+A$ for $\alpha\in\R$
and $(a+A)+(b+A):=(a+b)+A$ are well defined on $X/A$ and endow
the set of cosets with the structure of a linear space.
Here closedness of $A$ is actually not needed.
b) Check that $\norm{x+A}_{X/A}=\norm{y+A}_{X/A}$ whenever
$x+A=y+A$ or, equivalently, whenever $x-y\in A$.
c) Show that the function $x+A\mapsto \Norm{x+A}$ is a norm on
the linear space $X/A$.
\vspace{.05cm}\newline
[Hint: c) Non-degeneracy ($\Norm{x+A}=0$ $\Rightarrow$ $x\in A$)
relies on closedness of $A$.]
\end{exercise}

\begin{proposition}[Quotient Banach spaces]
\label{prop:Banach-quotient}
Suppose $E$ is a Banach space and $A$ is a closed subspace. Then the
following is true.
\begin{itemize}
\item[\rm (i)]
  The quotient norm on $E/A$ is complete.
\item[\rm (ii)]
  The map between Banach spaces defined by
  \begin{equation}\label{eq:quot-proj}
     \pi\colon E\to E/A,\quad
     x\mapsto x+A
  \end{equation}
  is linear,\index{quotient!space!projection onto --}
  surjective, continuous, and of norm $\norm{\pi}\le 1$ at most one.
  It is called the\index{projection!onto quotient space}
  \textbf{projection onto the quotient space} $E/A$.
\item[\rm (iii)]
  Suppose, in addition, that $E$ is reflexive, then $E/A$ is reflexive.
\end{itemize}
\end{proposition}

\begin{proof}
(i) Given a Cauchy sequence $x_\nu+A$ in the coset space $E/A$,
by the Cauchy property it suffices to extract a subsequence that
converges to a limit element $e+A$ in $E/A$.
Forgetting sequence members, if necessary, there is a subsequence,
still denoted by $x_\nu+A$, that satisfies
\[
     \frac{1}{2^\nu}>\norm{(x_{\nu+1}+A)-(x_\nu+A)}
     =\norm{(x_{\nu+1}-x_\nu)+A}
     :=d(x_{\nu+1}-x_\nu,A).
\]
Thus there is a sequence of points $a_\nu\in A$
whose distance to $x_{\nu+1}-x_\nu$ satisfies
$\norm{x_{\nu+1}-x_\nu-a_\nu}<1/2^\nu$.
Consider the partial sum sequence $z_{\nu+1}:=a_\nu+\dots+a_1\in A$.
As the sequence $x_\nu-z_\nu$ is Cauchy in $E$,
indeed
\[
     \Norm{(x_{\nu+1}-z_{\nu+1})-(x_\nu-z_\nu)}
     =\Norm{x_{\nu+1}-a_\nu-x_\nu}
     <1/2^\nu
\]
it admits a limit $e$ in the Banach space $E$.
It follows that the sequence $x_\nu+A$
converges to $e+A$ in $E/A$ and we are done. Indeed
\begin{equation*}
\begin{split}
     \Norm{(x_\nu+A)-(e+A)}
   &=\Norm{(x_\nu-e)+A}\\
   :&=\inf_{a\in A}\Norm{x_\nu-e-a}\\
   &\le \Norm{x_\nu-e-z_\nu}
     <1/2^\nu.
\end{split}
\end{equation*}
(ii) The map $\pi$ is linear by definition
of addition in the coset space $E/A$. Surjectivity is obvious.
To see continuity and $\norm{\pi}\le 1$,
given $x\in E$, pick $a=0\in A$ to get that
$
     \norm{\pi(x)}:=\inf_{a\in A}\norm{x-a}\le \norm{x}
$.
(iii)~\citet[Prop.\,11.11]{brezis:2011a}.
\end{proof}

For more details about quotients see e.g.~\citet[\S 11.2]{brezis:2011a}
or~\citet[\S 18.14]{Rudin:1987a}.

\subsection{Linear operators}\label{sec:lin-op}

Given normed linear spaces $X$ and $Y$, 
recall that a linear map $T\colon X\to Y$ is continuous
iff it is continuous at one point iff it is \textbf{bounded};
see e.g.~\citet[Thm.\,I.6]{reed:1980a}.
To be bounded\index{bounded!linear operator}
means that the\index{norm!operator --}
\textbf{\Index{operator norm}} of $T$, defined by
\begin{equation*}
\begin{split}
    \Norm{T}=\Norm{T}_{\Ll(X,Y)}
  :&=\sup_{\Norm{x}=1} \Norm{T x}\\
   &=\inf\left\{c\ge 0 \colon  
     \text{$\Norm{Tx}\le c\Norm{x}$ for every $x\in X$}\right\}
\end{split}
\end{equation*}
is finite. By $\Ll(X,Y)$ we denote the linear space
of\index{$\Ll(E,F)$ bounded linear operators with operator norm topology}
continuous linear operators $T\colon X\to Y$. Juxtaposition $ST\colon X\to Y\to Z$
denotes composition. The \textbf{invertible elements} $T$ of $\Ll(X,Y)$,
that is $TS=\1$ and $ST=\1$ for some (unique) $S\in\Ll(Y,X)$,
are called \textbf{isomorphisms} or\index{isomorphism!toplinear}
\textbf{\Index{toplinear isomorphism}s}\footnote{
  A toplinear isomorphism is a continuous linear bijection 
  whose inverse is continuous, too. The notion
  makes sense in the general context of topological vector spaces.
  }
to emphasize context. In case of Banach spaces $E$ and $F$
invertible elements of $\Ll(E,F)$, aka toplinear isomorphisms,
aka isomorphisms, are precisely the continuous linear bijections;
cf. e.g.~\citet[I \S 2]{lang:2001a} and~\citet[IV \S 1]{Lang:1993a}.
(The inverse is continuous by the closed graph theorem.)
\\
Abbreviate $\Ll(X):=\Ll(X,X)$.
By\index{$\Llm{k}(X,Y)$ $k$-fold multilinear maps}
$\Llm{k}(X,Y)$ we denote the linear space
of $k$-fold multilinear maps $T\colon X\oplus\dots\oplus X\to Y$.
If the norm of $Y$ is complete, then the operator norm is complete, so
$\Ll(X,Y)$ is a Banach space.
Thus the dual space $X^*=\Ll(X,\R)$ of a
normed linear space is a Banach space.

\subsubsection*{Unique extension}

\begin{theorem}[B.L.T. theorem]\label{thm:BLT}
Suppose $T$ is a bounded linear map
from a normed linear space $X$
to a complete normed linear space $F$.
Then $T$ extends uniquely
to a bounded linear map $\tilde T$
from the completion of~$X$ to~$F$.
\end{theorem}

\begin{proof}
See e.g.~\citet[Thm.\,I.7]{reed:1980a}.
\end{proof}

\subsubsection*{Compact operators and projections}

\begin{definition}[Compact operator]
    A\index{compact!linear operator}
    linear operator $S\colon X\to Y$ between normed linear spaces is called
    \textbf{compact} if for every bounded sequence
    in the domain, the image sequence has a convergent
    subsequence or, equivalently, if it maps bounded sets
    to \textbf{pre-compact sets}\index{pre-compact!sets}
    (sets whose closure is compact).    
    Compact linear operators are automatically continuous.
\end{definition}

\begin{definition}[Projection]
A continuous linear operator $P\colon X\to X$ is called a
\textbf{\Index{projection}} if it is idempotent, in symbols $P\circ P=P$.
\end{definition}

\begin{exercise}[Continuous projections split]
Let $E$ be a Banach space and $P\in\Ll(E)$ a projection.
Then
  the image $F:=\im P$ is closed and complemented by the closed
  image $G:=\im Q$ of the continuous projection $Q:=\1-P$, that is

  \[
     E=F\oplus G=\im P\oplus \im(\1-P).
  \]
[Hint: Kernels of continuous maps are closed and $\im P=\ker Q$ and
vice versa.]
\end{exercise}

\subsubsection*{Principle of uniform boundedness}

The Hahn--Banach theorem and the Banach--Steinhaus theorem
are two pillars of functional analysis. The latter is also known
as the principle of uniform boundedness. Its proof
is based on the Baire category theorem which requires
a non-empty complete metric space, for instance a Banach space $E$.

\begin{theorem}[Banach--Steinhaus]\label{thm:Banach-Steinhaus}
Suppose\index{Theorem of!Banach--Steinhaus}
$E$ is a Banach space. Let $\Ff$ be a family of bounded linear
operators $T\colon E\to Y$ to some normed linear space.
Suppose that the \textbf{\boldmath$\Ff$-orbit} through each point
$x\in E$, namely each set
\[
     \Ff x:=\{Tx\,\colon\, T\in \Ff\}\subset Y 
\]
is a bounded subset of $Y$. Then the operator norm is uniformly bounded
along the family $\Ff$: There is a constant $c_\Ff\ge 0$ such that
\[
       \Norm{T}=\Norm{T}_{\Ll(E,Y)}\le c_\Ff\qquad\forall\, T\in\Ff.
\]
\end{theorem}

\begin{proof}
See e.g.~\citet[Thm.\,III.9]{reed:1980a}.
\end{proof}

Recall that $\Ll(E,F)$ carries the operator norm topology.
How to utilize the principle of uniform boundedness 
is illustrated in the proof of

\begin{proposition}\label{prop:cont-BS}
Suppose $E_1,E_0,F_0$ are Banach spaces and $U_1\subset E_1$ is an
open subset. Then the following is true.
\begin{itemize}
\item[a)]
  Let the map $\Phi\colon U_1\oplus E_0\to F_0$,
  $(x,\eta)\mapsto\Phi(x,\xi)=:\Phi(x)\eta$,
  be continuous and, moreover, linear in the second variable.
  Then the induced map
  \begin{equation}\label{eq:a11}
     U_1\to\Llco(E_0,F_0),\quad x\mapsto \Phi(x)\cdot
  \end{equation}
  is 
  continuous. 
  (The target carries the
  compact-open topology.)\footnote{
    The target carrying the\index{topology!compact-open --}
    \textbf{\Index{compact-open topology}} means that a sequence
    $T_\nu\in\Ll(E_0,F_0)$ converges to an element $T\in\Ll(E_0,F_0)$
    iff for each domain element $\xi$ the sequence $T_\nu\xi$ converges
    to $T\xi$ in $F_0$; see Lemma~\ref{le:co-top-lin}.
    }
\item[b)]
  If $S\colon E_1\to E_0$ is a compact linear operator, then the induced map
  \begin{equation}\label{eq:b2}
     \Llco(E_0,F_0)\to\Ll(E_1,F_0),\quad T\mapsto T\circ S
  \end{equation}
  is continuous. (The target carries the norm topology.)
\item[c)]
  For $\Phi\colon U_1\oplus E_0\to F_0$ and $S\colon E_1\to E_0$ as in a) and b)
  the induced map
  \begin{equation}\label{eq:b3}
     U_1\to\Ll(E_1,F_0),\quad x\mapsto \Phi(x) S\,\cdot
  \end{equation}
  is continuous.
\end{itemize}
\end{proposition}

\begin{proof}
a) Proposition~\ref{prop:cont-TVS} a).

b) Given $T\in\Ll(E_0,F_0)$ and a sequence $T_\nu\in\Ll(E_0,F_0)$ with
$T_\nu\zeta\to T\zeta$ in $F_0$ for each $\zeta\in E_0$, assume by
contradiction that there is a constant $\eps>0$ and a sequence in
$E_1$ of bounded norm, say $\norm{\xi_\nu}_{E_1}=1$, such that
$\norm{T_\nu S\xi_\nu-TS\xi_\nu}_{F_0}\ge\eps$.
Because the linear operator $S\colon E_1\to E_0$ is compact, there is
$\eta\in E_0$ and subsequences, still denoted by $T_\nu$ and $\xi_\nu$,
such that $S\xi_\nu\to \eta$ in $E_0$. Hence
\begin{equation*}
\begin{split}
     \lim_{\nu\to\infty}\Norm{T_\nu S\xi_\nu-TS\xi_\nu}_{F_0}
   &\le
     \lim_{\nu\to\infty}\Norm{T_\nu S\xi_\nu-T_\nu\eta}_{F_0}\\
   &\quad
     +\lim_{\nu\to\infty}\Norm{T_\nu\eta-T\eta}_{F_0}\\
   &\quad
     +\lim_{\nu\to\infty}\Norm{T\eta-TS\xi_\nu}_{F_0}\\
   &\le
     \lim_{\nu\to\infty}\Norm{T_\nu}_{\Ll(E_0,F_0)}\Norm{S\xi_\nu-\eta}_{E_0}\\
   &\quad
     +\lim_{\nu\to\infty}\Norm{T_\nu\eta-T\eta}_{F_0}\\
   &\quad
     +\lim_{\nu\to\infty}\Norm{T}_{\Ll(E_0,F_0)}\Norm{\eta-S\xi_\nu}_{E_0}\\
   &=0.
\end{split}
\end{equation*}
Contradiction. Here the two inequalities are obtained by first
adding twice zero and applying the triangle inequality,
then using the definition of the operator norm.
It remains to understand the vanishing of the three limits.
For limit three this is obvious and limit two vanishes by hypothesis.
Concerning limit one consider the family
$\Ff:=\{T_\nu\}_{\nu\in\N}\CUP\{T\}\subset\Ll(E_0,F_0)$.
Each $\Ff$ orbit
\[
     \Ff \zeta:=\{T_\nu\zeta\}_{\nu\in\N}\CUP\{T\zeta\}\subset F_0
     ,\quad \zeta\in E_0
\]
is bounded in $F_0$, even compact, as $T_\nu\zeta\to T\zeta$ by hypothesis.
By the Banach--Steinhaus Theorem~\ref{thm:Banach-Steinhaus}
the family $\Ff$ is bounded in the operator norm.

c) The composition of continuous maps is continuous. But composing
the continuous maps (\ref{eq:a11}) and (\ref{eq:b2}) is the map (\ref{eq:b3}).
\end{proof}

\subsubsection*{Dual spaces and Reflexivity}

\begin{definition}[Dual space]
Given a normed linear space $X$, 
its \textbf{\Index{dual space}} is the Banach space
$X^*:=\Ll(X,\R)$ of continuous\index{$X^*:=\Ll(X,\R)$ dual space}
linear functionals $\lambda\colon X\to\R$.
\end{definition}

\begin{theorem}[Hahn--Banach]\label{thm:HB}
Suppose $V$ is a linear subspace, closed or not,
of a Banach space $X$ and $\lambda\in V^*$
is a continuous linear functional on $V$.
Then there is a linear functional $\Lambda\in X^*$
that extends $\lambda$ and such that
\[
     \Norm{\Lambda}_{X^*}
     =\sup_{v\in V\atop\Norm{v}=1}\Abs{\lambda(v)}
     =:\Norm{\lambda}_{V^*}.
\]
\end{theorem}

\begin{proof}
See e.g.~\citet[Cor.\,1.2]{brezis:2011a}.
\end{proof}

\begin{definition}[Reflexive]\label{def:reflexive}
A normed linear space $X$ is called \textbf{\Index{reflexive}}
if the canonical isometric linear map $J\colon X\to (X^*)^*$
given by evaluation
\[
     J(x)(x^*):=\INNER{x^*}{x}
\]
is surjective; see e.g.~\citet[\S 2.4]{Buhler:2018b}.
(Note that any linear isometry is injective.)
\end{definition}

\begin{remark}\label{rem:reflexive}
We highly recommend~\citet[\S 3.5]{brezis:2011a}.
\begin{itemize}
\item[a)]
  \textbf{Kakutani's Theorem}: Reflexivity of a Banach space $E$ is
  equivalent\index{Theorem of!Kakutani}
  to compactness of the closed unit ball of $E$ in the weak topology.
\item[b)]
  Closed linear subspaces of reflexive Banach spaces are reflexive.
\item[c)]
  A uniformly convex Banach space, so any Hilbert space, is reflexive.
\end{itemize}
\end{remark}

\begin{example}[Non-reflexive Banach spaces]
\label{ex:not-reflexive}\mbox{ }
\begin{itemize}
\item[(i)]
  The closed linear subspace $c_0$ of the Banach space $\ell^\infty$ in
  Example~\ref{ex:not-complemented} is not reflexive; see
  e.g.~\citet[Exc.\,4.37]{Salamon:2016a}.
  Hence $\ell^\infty$ is not reflexive either by
  Remark~\ref{rem:reflexive} b).
\item[(I)] More generally, let $C^0_{\rm bd}(X)$ be the space of
  bounded continuous functions on a locally compact topological
  space $X$ endowed with the sup norm (e.g. $\ell^\infty$). Then the
  Banach space $C^0_{\rm bd}(X)$ is reflexive iff $X$ is a finite set.
  See e.g.~\citet{Conway:1985a} (III $\S$11 Exc.\,2 and  V $\S$4 Exc.\,3).
\item[(ii)]
  Consequently $C^0(Q)$ is not reflexive for compact manifolds $Q$
  of $\dim Q\ge 1$. Neither is $C^k(Q)$ for $k\in\N$; this
  follows by reduction to the case $k=0$ using the graphs maps of
  differentials, see e.g.~\citet[App.\ A]{Weber:2017b}.
\end{itemize}
\end{example}

The following theorem can be viewed as a substitute
in the Banach space universe of the orthogonal projections
available in the Hilbert space world.

\begin{theorem}[Projection theorem for reflexive Banach spaces]
\label{thm:Banach-proj-thm}
Let\index{projection!theorem for Banach spaces}
$E$\index{theorem!projection -- for Banach spaces}
be a reflexive Banach 
space and $C\subset E$ a closed convex subset.
For every $x\in E$ there is an element $y\in C$ 
which minimizes the distance to $x$, that is
\[
     \Norm{x-y}=d(x,C):=\inf_{z\in C} \Norm{x-z}.
\]
\end{theorem}

\begin{proof}
The proof uses Kakutani's theorem,
see 
e.g.~\citet[Cor.\,3.23]{brezis:2011a}.
\end{proof}


\subsubsection{Arzel\`{a}--Ascoli -- convergent subsequences}

\begin{theorem}[Arzel\`{a}--Ascoli Theorem]\label{thm:AA}
Suppose $(Q,d)$ is a compact metric space
and $C(Q)$ is the Banach space of
continuous functions on $Q$
equipped with the sup norm. Then the following is true.
A\index{Theorem of!Arzel\`{a}--Ascoli}
subset\index{family!equicontinuous}
$\Ff$\index{family!pointwise bounded}
of $C(Q)$ is pre-compact if and only if
the family $\Ff$ is \textbf{\Index{equicontinuous}}\,\footnote{
  $\forall\eps>0$ $\exists\delta>0$ such that
  $\abs{f(x)-f(y)}<\eps$
  whenever $d(x,y)<\delta$ and $f\in\Ff$.
  }
and \textbf{\Index{pointwise bounded}}\,\footnote{
  $\sup_{f\in\Ff}\abs{f(x)}<\infty$ for every $x\in Q$.
  }.
\end{theorem}


For a proof see e.g.~\citet[Thm.~A.5]{Rudin:1991b}
or~\citet[App.\,C]{Salamon:2017a}.
By Theorem~\ref{thm:AA} this generalizes to maps
taking values in a metric space.

\subsection{Calculus}\label{sec:calculus}

An efficient presentation of the Fr\'{e}chet derivative
in Banach spaces $E,F$ is given in \S 1.1 of~\citet{ambrosetti:1993a}
where \S 2.2 deals with
the\index{implicit function theorem (IFT)}
implicit function theorem (IFT).
We follow~\citet[PART FOUR]{Lang:1993a}.

\subsubsection{Fr\'{e}chet or total derivative  \boldmath$df(x)$}

Consider Banach spaces $E$ and $F$ and let $U$ be open in $E$.
One says that \textbf{\boldmath a map $f\colon U\to F$ is
differentiable at a point $x$} of $U$ if there is a
continuous\index{differentiable!at a point}
linear map $D\colon E\to F$ and a map $\psi$ defined for
all sufficiently small elements $h$ in $E$ and with values in $F$
such that
\[
     \lim_{h\to 0}\psi(h)=0
\]
and such that $f$ near $x$ is given by the sum
\begin{equation*}
     f(x+h)=f(x)+D h+\Norm{h}\psi(h).
\end{equation*}
Set $h=0$ to see that it makes sense to set $\psi(0):=0$.
Equivalently, denoting $o(h):=\Norm{h}\psi(h)$
the condition becomes
\begin{equation}\label{eq:Frechet-differential}
     0
     =\lim_{h\to 0}\frac{\Norm{o(h)}}{\Norm{h}} 
     =\lim_{h\to 0}\frac{\Norm{f(x+h)-f(x)-D h}}{\Norm{h}} \; .
\end{equation}

\begin{exercise}
a) Differentiability at $x$ implies continuity at $x$.
b) If $D\in\Ll(E,F)$ satisfies~(\ref{eq:Frechet-differential}),
then it is uniquely determined by $f$ and $x$.
\end{exercise}

\begin{definition}[Derivative and differential]
\label{def:Frechet-differential}
Let\index{Fr\'{e}chet derivative!on Banach space}\index{derivative!on Banach space}
$f\colon E\supset U\to F$ be differentiable at a point $x\in U$. Then
the unique continuous linear operator $D$
satisfying~(\ref{eq:Frechet-differential}) is called the (Fr\'{e}chet)
\textbf{\boldmath derivative of $f$ at $x$}\index{derivative!of $f$ at $x$}
and\index{$df(x)\in\Ll(E,F)$ derivative of $f$ at $x$}
denoted by $df(x):=D\in\Ll(E,F)$. 
If $f$ is differentiable at every point of $U$ one says that
$f$ is \textbf{\boldmath differentiable on $U$}.
In this case the map
\[
     f^\prime:=df\colon U\to\Ll(E,F), \quad x\mapsto df(x)
\]
into the Banach space of continuous linear maps $\Ll(E,F)$ endowed
with the operator norm is called the (Fr\'{e}chet)\index{differential!of $f$}
\textbf{\boldmath differential of $f$}. If $df$ is continuous one
says that\index{$C^1(U,F)$} \textbf{\boldmath$f$ is of class $C^1$},
in symbols $f\in C^1(U,F)$.
Higher derivatives 
\[
     f^{(\ell)}:=d^\ell f\colon U\to \Ll(E,\Ll(E,\dots \Ll(E,F)))\simeq \Ll^\ell(E,F)
\]
are defined iteratively.
If they exist and are continuous for $\ell=0,\dots k$,
one says that $f$ is of class $C^k$.
Here $\Ll^\ell(E,F)$ denotes the Banach space of $k$-fold multilinear
maps $E\oplus\dots\oplus E\to F$.
One says that $f$ is a\index{smooth!map}
\textbf{smooth map}, or of class $C^\infty$, if
$f$ is of class $C^k$ for every $k\in\N_0$.
\end{definition}

\subsubsection{G\^{a}teaux or all-directional derivative \boldmath$\p f(x)$}

A map $f\colon E\subset U\to F$ between Banach spaces with $U$
open\index{derivative!G\^{a}teaux --}\index{derivative!directional --}
is\index{$(deriva$@$\p_\xi f$ directional derivative}
said\index{G\^{a}teaux!differentiable!at $x\in U$}
\textbf{\boldmath G\^{a}teaux differentiable at $x\in U$}
if for each $\xi\in E$ the \textbf{\Index{directional derivative}}
\[
     \p_\xi f(x):=\lim_{t\to 0} \frac{f(x+t\xi)-f(x)}{t}
\]
exists and defines a continuous \emph{linear} map
$\p f(x)\colon E\to F$, $\xi\mapsto \p_\xi f(x)$.

\begin{exercise}
Show that a) (Fr\'{e}chet) differentiable
implies G\^{a}teaux differentiable,
but b) not vice versa.
\vspace{.05cm}\newline
[Hint: b) Define $f\colon\R^2\to\R$ by $f(0,0):=0$ and by
$f(u,v):=u^4v/(u^6+v^3)$ off the origin.
Show that $\p_\zeta f(0,0)=0$ for every $\zeta\in\R^2$.
So each directional derivative not only exists,
but also the map $\zeta\mapsto \p_\zeta f(0,0)$ is linear.
Is $f$ continuous at 
$(0,0)$?]
\end{exercise}

\subsection{Banach manifolds}\label{sec:B-mfs}

Roughly speaking, a Banach manifold is a 
topological space (Hausdorff and paracompact)
which is locally modeled on some Banach space
such that all transition maps between the local
models are differentiable.
Differentiability of maps between Banach manifolds is
defined in terms of differentiability of the
corresponding maps between the local model Banach spaces.
We recommend the book by~\citet{lang:2001a}
concerning differential geometry on Banach manifolds.

Suppose $X$ is a topological space and $k\in\N_0$ or $k=\infty$.
A \textbf{Banach chart} $(V,\phi,E)$ for $X$\index{chart!Banach --}
consists of a Banach space $E$ and a homeomorphism 
$$
     \phi\colon X\supset V\to U\subset E
$$
between open subsets.
Two charts are called \textbf{\boldmath$C^k$ compatible}
if the \textbf{\Index{transition map}}
\[
     \phi_j\circ\phi_i^{-1}\colon \phi_i(V_i\CAP V_j)\to \phi_j(V_i\CAP V_j)
\]
is a \textbf{\boldmath$C^k$ \Index{diffeomorphism}}
(an invertible $C^k$ map with $C^k$ inverse).
A \textbf{\boldmath$C^k$ Banach atlas for $X$} is a collection
$\Aa$ of pairwise $C^k$ compatible Banach charts for $X$
such that the chart domains form a cover $\{V_i\}_i$ of $X$.
Two atlases are called \textbf{equivalent} if their union forms an atlas.

Given such pair $(X,\Aa)$, then $X$ is connected iff it is path
connected. Furthermore, for $k\ge 1$ connectedness of $X$ implies
that all model Banach spaces $E_i$ in the charts of $\Aa$ are
isomorphic to one and the same Banach space $E$.
In this case we say that $(X,\Aa)$ is
\textbf{modeled on \boldmath$E$}.

\begin{remark}[Starting from just a set $X$]\label{rem:B-mf-set}
Alternatively starting with just \emph{a set} $X$
one can construct a $C^k$ Banach atlas as follows.
Choose a collection of bijections (the future coordinate charts)
\[
     \phi\colon X\supset V\to U\subset E
\]
from a subset $V$ of $X$ onto an \emph{open} subset $U$ of a Banach space $E$.
There are two requirements: Firstly, the sets $V$ of all the charts
must cover $X$ and, secondly, for each pair of charts the set
$\phi(V\CAP\widetilde{V})$ must be open in $E$. The notions
$C^k$ compatibility and $C^k$ Banach atlas are unchanged.
Given a $C^k$ Banach atlas~$\Aa$, consider the collection $\Bb\subset 2^X$ of all
subsets $\phi^{-1}(U^\prime)$ of $X$ where $(\phi,V)$ runs through all charts
of $\Aa$ and $U^\prime\subset E$ runs through all open subsets of $\phi(V)$.
One checks that $\Bb$ is a basis of a topology
and endows $X$ with that topology.
Then $\Aa$ is an atlas on the topological space $X$ in the earlier sense.
For an application see Exercise~\ref{exc:tangent-bdl-to-sc-mf}.
\end{remark}

\begin{definition}
A \textbf{\boldmath$C^k$ Banach manifold}
is a paracompact\index{Banach!manifold}
Hausdorff space $X$ endowed with an equivalence class of $C^k$ Banach
atlases. If $k=0$ one speaks of a \textbf{topological} and if
$k=\infty$ of a \textbf{smooth} Banach manifold. We often abbreviate
smooth Banach manifold by \textbf{Banach manifold}.
In case all model spaces are Hilbert spaces one speaks of a
\textbf{\Index{Hilbert manifold}}.
\end{definition}

\begin{definition}[Maps between Banach manifolds]
A continuous map $$f\colon X\to Y$$ between Banach manifolds
is said to be of \textbf{class \boldmath$C^k$} if for all charts
$\phi\colon X\supset V\to E$ and $\psi\colon Y\supset W\to F$ the chart
representative $\psi\circ f\circ \phi^{-1}$ is of class $C^k$ as a map
between open subsets of the Banach spaces $E$ and $F$.
\end{definition}

\section{Function spaces}\label{sec:function-spaces}

\begin{theorem}[Properties of $L^p$ and Sobolev spaces]
\label{thm:Lp-spaces}      \mbox{ }
\begin{labeling}{\rm (separable)}
\item[\rm (complete)]
  \textbf{Fischer--Riesz Theorem}:\index{Theorem of!Fischer-Riesz}
  The spaces $L^p(\R,\R^n)$ with norm
  $\norm{\cdot}_p:=\norm{\cdot}_{L^p}$ are Banach spaces
  whenever $1\le p\le \infty$.
\item[\rm (separable)]
  The spaces $L^p(\R,\R^n)$ are separable\footnote{
    A topological space is called \textbf{\Index{separable}}
    if it admits a dense sequence.
    }
  for $1\le p<\infty$,
  but not separable for $p=\infty$.
\item[\rm (reflexive)]
  The spaces $L^p(\R,\R^n)$ are reflexive
  for $1<p<\infty$,
  but not reflexive for $p=1,\infty$.
\item[\rm (Sobolev)]
  The Sobolev spaces $W^{k,p}(\R,\R^n)$ have analogous properties in $p$.
\end{labeling}
\end{theorem}

For proofs of the three properties of $L^p$ see e.g. Theorems 4.8, 4.13, and
4.10, respectively, in~\citet{brezis:2011a}, for $W^{k,p}$
see~\citet[Prop.\,8.1]{brezis:2011a}.
Concerning Sobolev spaces see also~\citet{Adams:2003a}.



\cleardoublepage
\phantomsection

%

\backmatter 

\addcontentsline{toc}{chapter}{Bibliography}
\bibliography{library_math}{}

\begin{thebibliography}{51}
\providecommand{\natexlab}[1]{#1}
\providecommand{\url}[1]{\texttt{#1}}
\expandafter\ifx\csname urlstyle\endcsname\relax
  \providecommand{\doi}[1]{doi: #1}\else
  \providecommand{\doi}{doi: \begingroup \urlstyle{rm}\Url}\fi

\bibitem[Abbondandolo and Schlenk(2018)]{Abbondandolo:2018f}
Alberto Abbondandolo and Felix Schlenk.
\newblock Floer homologies, with applications.
\newblock \emph{{\rm \href{https://arxiv.org/abs/1709.00297}{ArXiv e-prints.}}
  Jahresbericht der Deutschen Mathematiker-Vereinigung}, Nov 2018.
\newblock ISSN 1869-7135.
\newblock \doi{10.1365/s13291-018-0193-x}.
\newblock URL \url{https://doi.org/10.1365/s13291-018-0193-x}.

\bibitem[Adams and Fournier(2003)]{Adams:2003a}
Robert~A. Adams and John J.~F. Fournier.
\newblock \emph{Sobolev spaces}, volume 140 of \emph{Pure and Applied
  Mathematics (Amsterdam)}.
\newblock Elsevier/Academic Press, Amsterdam, second edition, 2003.
\newblock ISBN 0-12-044143-8.

\bibitem[Ambrosetti and Prodi(1993)]{ambrosetti:1993a}
Antonio Ambrosetti and Giovanni Prodi.
\newblock \emph{A primer of nonlinear analysis}, volume~34 of \emph{Cambridge
  Studies in Advanced Mathematics}.
\newblock Cambridge University Press, Cambridge, 1993.
\newblock ISBN 0-521-37390-5.

\bibitem[Bourbaki(1987)]{Bourbaki:1987a}
N.~Bourbaki.
\newblock \emph{Topological vector spaces. {C}hapters 1--5}.
\newblock Elements of Mathematics (Berlin). Springer-Verlag, Berlin, 1987.
\newblock ISBN 3-540-13627-4.
\newblock \doi{10.1007/978-3-642-61715-7}.
\newblock URL \url{https://doi.org/10.1007/978-3-642-61715-7}.
\newblock Translated from the French by H. G. Eggleston and S. Madan.

\bibitem[Brezis(2011)]{brezis:2011a}
Ha{\"{\i}}m Brezis.
\newblock \emph{Functional analysis, {S}obolev spaces and partial differential
  equations}.
\newblock Universitext. Springer, New York, 2011.
\newblock ISBN 978-0-387-70913-0.

\bibitem[B\"{u}hler and Salamon(2018)]{Buhler:2018b}
Theo B\"{u}hler and Dietmar~A. Salamon.
\newblock \emph{Functional analysis}, volume 191 of \emph{Graduate Studies in
  Mathematics}.
\newblock American Mathematical Society, Providence, RI, 2018.
\newblock ISBN 978-1-4704-4190-6.
\newblock
  \href{https://people.math.ethz.ch/~salamon/PREPRINTS/funcana-ams.pdf}{Manuscript}.

\bibitem[Cartan(1986)]{Cartan:1986a}
Henri Cartan.
\newblock Sur les r\'etractions d'une vari\'et\'e.
\newblock \emph{C. R. Acad. Sci. Paris S\'er. I Math.}, 303\penalty0
  (14):\penalty0 715, 1986.
\newblock ISSN 0249-6291.

\bibitem[Cieliebak(2018)]{Cieliebak:2018a}
Kai Cieliebak.
\newblock {Nonlinear Functional Analysis}.
\newblock Lecture Notes, manuscript in progress, February 2018.

\bibitem[Conway(1985)]{Conway:1985a}
John~B. Conway.
\newblock \emph{{A course in functional analysis}}, volume~96 of \emph{Graduate
  Texts in Mathematics}.
\newblock Springer-Verlag, New York, 1985.
\newblock ISBN 0-387-96042-2.
\newblock \doi{10.1007/978-1-4757-3828-5}.
\newblock URL \url{http://dx.doi.org/10.1007/978-1-4757-3828-5}.

\bibitem[Dugundji(1966)]{Dugundji:1966a}
James Dugundji.
\newblock \emph{Topology}.
\newblock Allyn and Bacon, Inc., Boston, Mass., 1966.

\bibitem[Eliashberg et~al.(2000)Eliashberg, Givental, and
  Hofer]{Eliashberg:2000a}
Y.~Eliashberg, A.~Givental, and H.~Hofer.
\newblock Introduction to symplectic field theory.
\newblock \emph{Geom. Funct. Anal.}, Special Volume, Part II:\penalty0
  560--673, 2000.
\newblock ISSN 1016-443X.
\newblock \doi{10.1007/978-3-0346-0425-3_4}.
\newblock URL \url{http://dx.doi.org/10.1007/978-3-0346-0425-3_4}.
\newblock GAFA 2000 (Tel Aviv, 1999).

\bibitem[Fabert et~al.(2016)Fabert, Fish, Golovko, and Wehrheim]{Fabert:2016a}
Oliver Fabert, Joel~W. Fish, Roman Golovko, and Katrin Wehrheim.
\newblock Polyfolds: a first and second look.
\newblock \emph{EMS Surv. Math. Sci.}, 3\penalty0 (2):\penalty0 131--208, 2016.
\newblock ISSN 2308-2151.
\newblock \doi{10.4171/EMSS/16}.
\newblock URL \url{http://dx.doi.org/10.4171/EMSS/16}.

\bibitem[{Filippenko} et~al.(2018){Filippenko}, {Zhou}, and
  {Wehrheim}]{Filippenko:2018b}
B.~{Filippenko}, Z.~{Zhou}, and K.~{Wehrheim}.
\newblock {Counterexamples in Scale Calculus}.
\newblock \emph{ArXiv e-prints}, July 2018.
\newblock URL \url{https://arxiv.org/abs/1807.02591}.

\bibitem[Floer(1986)]{Floer:1986a}
Andreas Floer.
\newblock Proof of the {A}rnol$'$d conjecture for surfaces and generalizations
  to certain {K}\"ahler manifolds.
\newblock \emph{Duke Math. J.}, 53\penalty0 (1):\penalty0 1--32, 1986.
\newblock ISSN 0012-7094.
\newblock \doi{10.1215/S0012-7094-86-05301-9}.
\newblock URL \url{http://dx.doi.org/10.1215/S0012-7094-86-05301-9}.

\bibitem[Floer(1988{\natexlab{a}})]{floer:1988a}
Andreas Floer.
\newblock Morse theory for {L}agrangian intersections.
\newblock \emph{J. Differential Geom.}, 28\penalty0 (3):\penalty0 513--547,
  1988{\natexlab{a}}.
\newblock ISSN 0022-040X.
\newblock URL \url{http://projecteuclid.org/euclid.jdg/1214442477}.

\bibitem[Floer(1988{\natexlab{b}})]{floer:1988c}
Andreas Floer.
\newblock The unregularized gradient flow of the symplectic action.
\newblock \emph{Comm. Pure Appl. Math.}, 41\penalty0 (6):\penalty0 775--813,
  1988{\natexlab{b}}.
\newblock ISSN 0010-3640.
\newblock \doi{10.1002/cpa.3160410603}.
\newblock URL \url{http://dx.doi.org/10.1002/cpa.3160410603}.

\bibitem[Floer(1989)]{floer:1989a}
Andreas Floer.
\newblock Symplectic fixed points and holomorphic spheres.
\newblock \emph{Comm. Math. Phys.}, 120\penalty0 (4):\penalty0 575--611, 1989.
\newblock ISSN 0010-3616.
\newblock URL \url{http://projecteuclid.org/euclid.cmp/1104177909}.

\bibitem[{Frauenfelder} and Weber(2018)]{Frauenfelder:2018a}
Urs {Frauenfelder} and Joa Weber.
\newblock {The shift map on Floer trajectory spaces}.
\newblock \emph{ArXiv e-prints {\rm
  \href{https://arxiv.org/abs/1803.03826}{1803.03826}}}, March 2018.

\bibitem[Fukaya(1993)]{Fukaya:1993a}
Kenji Fukaya.
\newblock Morse homotopy, {$A^\infty$}-category, and {F}loer homologies.
\newblock In \emph{Proceedings of {GARC} {W}orkshop on {G}eometry and
  {T}opology '93 ({S}eoul, 1993)}, volume~18 of \emph{Lecture Notes Ser.},
  pages 1--102. Seoul Nat. Univ., Seoul, 1993.

\bibitem[Fukaya and Ono(1999)]{Fukaya:1999a}
Kenji Fukaya and Kaoru Ono.
\newblock Arnold conjecture and {G}romov--{W}itten invariant.
\newblock \emph{Topology}, 38\penalty0 (5):\penalty0 933--1048, 1999.
\newblock ISSN 0040-9383.
\newblock \doi{10.1016/S0040-9383(98)00042-1}.
\newblock URL \url{http://dx.doi.org/10.1016/S0040-9383(98)00042-1}.

\bibitem[Fukaya et~al.(2009)Fukaya, Oh, Ohta, and Ono]{Fukaya:2009a}
Kenji Fukaya, Yong-Geun Oh, Hiroshi Ohta, and Kaoru Ono.
\newblock \emph{Lagrangian intersection {F}loer theory: anomaly and
  obstruction. {P}art {I}}, volume~46 of \emph{AMS/IP Studies in Advanced
  Mathematics}.
\newblock American Mathematical Society, Providence, RI; International Press,
  Somerville, MA, 2009.
\newblock ISBN 978-0-8218-4836-4.

\bibitem[Gromov(1985)]{gromov:1985a}
M.~Gromov.
\newblock Pseudo holomorphic curves in symplectic manifolds.
\newblock \emph{Invent. Math.}, 82:\penalty0 307--347, 1985.
\newblock \doi{10.1007/BF01388806}.

\bibitem[{Hofer}(2006)]{Hofer:2006a}
Helmut {Hofer}.
\newblock A general {F}redholm theory and application.
\newblock In \emph{Current developments in mathematics, 2004}, pages 1--71.
  (See also \href{https://arxiv.org/abs/math/0509366}{arXiv:0509366}). Int.
  Press, Somerville, MA, 2006.

\bibitem[{Hofer} et~al.(2005){Hofer}, Wysocki, and Zehnder]{Hofer:2005a}
Helmut {Hofer}, Krzysztof Wysocki, and Eduard Zehnder.
\newblock {Polyfolds and Fredholm Theory Part 1}.
\newblock Workshop SFT I,
  \href{http://www.mathematik.uni-leipzig.de/~schwarz/sft/manuscriptHWZ-leipzig.pdf}{manuscript},
  2005.
\newblock URL \url{http://www.mathematik.uni-leipzig.de/~schwarz/sft/}.

\bibitem[Hofer et~al.(2007)Hofer, Wysocki, and Zehnder]{Hofer:2007a}
Helmut Hofer, Krzysztof Wysocki, and Eduard Zehnder.
\newblock A general {F}redholm theory. {I}. {A} splicing-based differential
  geometry.
\newblock \emph{J. Eur. Math. Soc. (JEMS)}, 9\penalty0 (4):\penalty0 841--876,
  2007.
\newblock ISSN 1435-9855.
\newblock \doi{10.4171/JEMS/99}.
\newblock URL \url{http://dx.doi.org/10.4171/JEMS/99}.

\bibitem[Hofer et~al.(2009{\natexlab{a}})Hofer, Wysocki, and
  Zehnder]{Hofer:2009a}
Helmut Hofer, Krzysztof Wysocki, and Eduard Zehnder.
\newblock A general {F}redholm theory. {II}. {I}mplicit function theorems.
\newblock \emph{Geom. Funct. Anal.}, 19\penalty0 (1):\penalty0 206--293,
  2009{\natexlab{a}}.
\newblock ISSN 1016-443X.
\newblock \doi{10.1007/s00039-009-0715-x}.
\newblock URL \url{http://dx.doi.org/10.1007/s00039-009-0715-x}.

\bibitem[Hofer et~al.(2009{\natexlab{b}})Hofer, Wysocki, and
  Zehnder]{Hofer:2009b}
Helmut Hofer, Krzysztof Wysocki, and Eduard Zehnder.
\newblock A general {F}redholm theory. {III}. {F}redholm functors and
  polyfolds.
\newblock \emph{Geom. Topol.}, 13\penalty0 (4):\penalty0 2279--2387,
  2009{\natexlab{b}}.
\newblock ISSN 1465-3060.
\newblock \doi{10.2140/gt.2009.13.2279}.
\newblock URL \url{http://dx.doi.org/10.2140/gt.2009.13.2279}.

\bibitem[Hofer et~al.(2010)Hofer, Wysocki, and Zehnder]{Hofer:2010b}
Helmut Hofer, Krzysztof Wysocki, and Eduard Zehnder.
\newblock {Sc-smoothness, Retractions and New Nodels for Smooth Spaces}.
\newblock \emph{Discrete Contin. Dyn. Syst.}, 28\penalty0 (2):\penalty0
  665--788, 2010.
\newblock ISSN 1078-0947.
\newblock \doi{10.3934/dcds.2010.28.665}.
\newblock URL \url{http://dx.doi.org/10.3934/dcds.2010.28.665}.

\bibitem[Hofer et~al.(2017)Hofer, Wysocki, and Zehnder]{Hofer:2017a}
Helmut Hofer, Krzysztof Wysocki, and Eduard Zehnder.
\newblock {Polyfold and Fredholm Theory}. 714 pages, 
\newblock \emph{ArXiv e-prints}, page
  \href{https://arxiv.org/abs/1707.08941}{arXiv:1707.08941}, July 2017.

\bibitem[Hohloch et~al.(2009)Hohloch, Noetzel, and Salamon]{hohloch:2009a}
Sonja Hohloch, Gregor Noetzel, and Dietmar~A. Salamon.
\newblock Hypercontact structures and {F}loer homology.
\newblock \emph{Geom. Topol.}, 13\penalty0 (5):\penalty0 2543--2617, 2009.
\newblock ISSN 1465-3060.
\newblock \doi{10.2140/gt.2009.13.2543}.
\newblock URL \url{http://dx.doi.org/10.2140/gt.2009.13.2543}.

\bibitem[Khaleelulla(1982)]{Khaleelulla:1982a}
S.~M. Khaleelulla.
\newblock \emph{Counterexamples in topological vector spaces}, volume 936 of
  \emph{Lecture Notes in Mathematics}.
\newblock Springer-Verlag, Berlin-New York, 1982.
\newblock ISBN 3-540-11565-X.

\bibitem[Lang(1993)]{Lang:1993a}
Serge Lang.
\newblock \emph{Real and functional analysis}, volume 142 of \emph{Graduate
  Texts in Mathematics}.
\newblock Springer-Verlag, New York, third edition, 1993.
\newblock ISBN 0-387-94001-4.
\newblock \doi{10.1007/978-1-4612-0897-6}.
\newblock URL \url{http://dx.doi.org/10.1007/978-1-4612-0897-6}.

\bibitem[{Lang}(2001)]{lang:2001a}
Serge {Lang}.
\newblock \emph{{Fundamentals of differential geometry}}.
\newblock Springer-Verlag, New York, corr. printing 2nd edition, 2001.
\newblock ISBN 0-387-98593-X/hbk.

\bibitem[M{\"u}ger(2016)]{Muger:2016a}
Michael M{\"u}ger.
\newblock {Topology for the working mathematician (Working title), last
  accessed 13/10/2017 on}
  \href{http://www.math.ru.nl/~mueger/topology.pdf}{Webpage}, {Manuscript
  v19/10/2016}, 2016.
\newblock URL \url{http://www.math.ru.nl/~mueger/}.

\bibitem[Munkres(2000)]{Munkres:2000a}
James~R. Munkres.
\newblock \emph{Topology}.
\newblock Prentice Hall, Inc., Upper Saddle River, NJ, 2000.
\newblock ISBN 0-13-181629-2.
\newblock Second edition of [ MR0464128].

\bibitem[Narici and Beckenstein(2011)]{Narici:2011a}
Lawrence Narici and Edward Beckenstein.
\newblock \emph{Topological vector spaces}, volume 296 of \emph{Pure and
  Applied Mathematics (Boca Raton)}.
\newblock CRC Press, Boca Raton, FL, second edition, 2011.
\newblock ISBN 978-1-58488-866-6.

\bibitem[Reed and Simon(1980)]{reed:1980a}
Michael Reed and Barry Simon.
\newblock \emph{Methods of modern mathematical physics. {I}}.
\newblock Academic Press, Inc. [Harcourt Brace Jovanovich, Publishers], New
  York, second edition, 1980.
\newblock ISBN 0-12-585050-6.
\newblock Functional analysis.

\bibitem[Rudin(1987)]{Rudin:1987a}
Walter Rudin.
\newblock \emph{Real and complex analysis}.
\newblock McGraw-Hill Book Co., New York, third edition, 1987.
\newblock ISBN 0-07-054234-1.

\bibitem[Rudin(1991)]{Rudin:1991b}
Walter Rudin.
\newblock \emph{Functional analysis}.
\newblock International Series in Pure and Applied Mathematics. McGraw-Hill,
  Inc., New York, second edition, 1991.
\newblock ISBN 0-07-054236-8.

\bibitem[Salamon(2017)]{Salamon:2017a}
Dietmar Salamon.
\newblock {Analysis II}.
\newblock
  \href{https://people.math.ethz.ch/~salamon/PREPRINTS/analysis2.pdf}{Vorlesungsmanuskript},
  ETH Z\"urich, 09 2017.
\newblock URL
  \url{https://people.math.ethz.ch/~salamon/PREPRINTS/analysis2.pdf}.

\bibitem[Salamon and Weber(2006)]{salamon:2006a}
Dietmar Salamon and Joa Weber.
\newblock Floer homology and the heat flow.
\newblock \emph{Geom. Funct. Anal.}, 16\penalty0 (5):\penalty0 1050--1138,
  2006.
\newblock ISSN 1016-443X.
\newblock \doi{10.1007/s00039-006-0577-4}.
\newblock URL \url{http://dx.doi.org/10.1007/s00039-006-0577-4}.

\bibitem[Salamon(2016)]{Salamon:2016a}
Dietmar~A. Salamon.
\newblock \emph{Measure and integration}.
\newblock EMS Textbooks in Mathematics. European Mathematical Society (EMS),
  Z\"urich, 2016.
\newblock ISBN 978-3-03719-159-0.
\newblock \doi{10.4171/159}.
\newblock URL \url{http://dx.doi.org/10.4171/159}.

\bibitem[Schaefer and Wolff(1999)]{Schaefer:1999a}
H.~H. Schaefer and M.~P. Wolff.
\newblock \emph{Topological vector spaces}, volume~3 of \emph{Graduate Texts in
  Mathematics}.
\newblock Springer-Verlag, New York, second edition, 1999.
\newblock ISBN 0-387-98726-6.
\newblock \doi{10.1007/978-1-4612-1468-7}.
\newblock URL \url{https://doi.org/10.1007/978-1-4612-1468-7}.

\bibitem[Schochetman et~al.(2001)Schochetman, Smith, and
  Tsui]{Schochetman:2001a}
Irwin~E. Schochetman, Robert~L. Smith, and Sze-Kai Tsui.
\newblock On the closure of the sum of closed subspaces.
\newblock \emph{Int. J. Math. Math. Sci.}, 26\penalty0 (5):\penalty0 257--267,
  2001.
\newblock ISSN 0161-1712.
\newblock \doi{10.1155/S0161171201005324}.
\newblock URL \url{https://doi.org/10.1155/S0161171201005324}.

\bibitem[Treves(1967)]{Treves:1967a}
Fran\c{c}ois Treves.
\newblock \emph{Topological vector spaces, distributions and kernels}.
\newblock Academic Press, New York-London, 1967.

\bibitem[Triebel(1978)]{Triebel:1978a}
H.~Triebel.
\newblock \emph{Interpolation theory, function spaces, differential operators}.
\newblock VEB Deutscher Verlag der Wissenschaften, Berlin, 1978.

\bibitem[Weber(2013{\natexlab{a}})]{weber:2013a}
Joa Weber.
\newblock {M}orse homology for the heat flow -- {L}inear theory.
\newblock \emph{Math. Nachr.}, 286\penalty0 (1):\penalty0 88--104,
  2013{\natexlab{a}}.
\newblock ISSN 0025-584X.
\newblock \doi{10.1002/mana.201100319}.
\newblock URL \url{http://dx.doi.org/10.1002/mana.201100319}.

\bibitem[Weber(2013{\natexlab{b}})]{weber:2013b}
Joa Weber.
\newblock {M}orse homology for the heat flow.
\newblock \emph{Math. Z.}, 275\penalty0 (1-2):\penalty0 1--54,
  2013{\natexlab{b}}.
\newblock ISSN 0025-5874.
\newblock \doi{10.1007/s00209-012-1121-x}.
\newblock URL \url{http://dx.doi.org/10.1007/s00209-012-1121-x}.

\bibitem[Weber(2017{\natexlab{a}})]{Weber:2017b}
Joa Weber.
\newblock \emph{Topological methods in the quest for periodic orbits}.
\newblock Publica\c{c}\~oes Matem\'aticas do IMPA. [IMPA Mathematical
  Publications]. Instituto Nacional de Matem\'atica Pura e Aplicada (IMPA), Rio
  de Janeiro, 2017{\natexlab{a}}.
\newblock ISBN 978-85-244-0439-9.
\newblock 31${^{\rm{o}}}$ Col\'oquio Brasileiro de Matem\'atica. vii+248 pp.
  ISBN: 978-85-244-0439-9.
  \href{http://www.math.stonybrook.edu/~joa/PUBLICATIONS/CBM31-TOPMETDYN.pdf}{Access
  book}. Version on \href{https://arxiv.org/abs/1802.06435}{arXiv:1802.06435}.
  An extended version to appear in the EMS Lecture Series of Mathematics.

\bibitem[Weber(2017{\natexlab{b}})]{weber:2014c}
Joa Weber.
\newblock {S}table foliations and semi-flow {M}orse homology.
\newblock \emph{Ann. Sc. Norm. Super. Pisa Cl. Sci. (5)}, Vol. XVII\penalty0
  (3):\penalty0 853--909, 2017{\natexlab{b}}.
\newblock \doi{10.2422/2036-2145.201510_017}.

\bibitem[Whitley(1966)]{Whitley:1966a}
Robert Whitley.
\newblock Mathematical {N}otes: {P}rojecting {$m$} onto {$c_0$}.
\newblock \emph{Amer. Math. Monthly}, 73\penalty0 (3):\penalty0 285--286, 1966.
\newblock ISSN 0002-9890.
\newblock \doi{10.2307/2315346}.
\newblock URL \url{https://doi.org/10.2307/2315346}.

\end{thebibliography}

\addcontentsline{toc}{chapter}{Index}
\printindex

\end{document}